\documentclass[a4paper,twoside]{article}

\usepackage[english]{babel}
\usepackage[utf8x]{inputenc}
\usepackage{amsmath}
\usepackage{graphicx}
\usepackage[colorinlistoftodos]{todonotes}

\usepackage{amsmath}
\usepackage{epsfig,amsthm}
\usepackage{latexsym}
\usepackage{amsfonts}
\usepackage{amssymb}
\usepackage{amscd}
\usepackage{mathrsfs}
\usepackage[all,cmtip]{xy}
\usepackage{enumerate}
\usepackage{tikz}
\usepackage{arydshln}
\usepackage{bbm}

\usepackage{mathdots}
\usepackage[shortlabels]{enumitem}

\usepackage{color}

\usepackage{lipsum}

\usepackage[colorlinks]{hyperref}
\hypersetup{pdfstartview={FitH},         linkcolor=blue,  citecolor=green }

\numberwithin{equation}{section} 
\usepackage[a4paper,left=1in,right=1in,top=1in,bottom=1in,footskip=.35in]{geometry}

\usepackage{fancyhdr}
\AtEndDocument{%
  \par
  \medskip
  \begin{tabular}{@{}l@{}}%
  \textsc{Xiamen University Malaysia}
  \\
    \textsc{Jalan Sunsuria, Bandar Sunsuria, 43900 Sepang, Selangor, Malaysia
}\\
    \textit{E-mail address}: \texttt{\href{mailto:kamfai.tam@xmu.edu.my}{kamfai.tam@xmu.edu.my}
    }
    
  \end{tabular}}

\newtheorem{thm}{Theorem}[section]
\newtheorem{cor}[thm]{Corollary}
\newtheorem{lem}[thm]{Lemma}
\newtheorem{prop}[thm]{Proposition}

\theoremstyle{definition}
\newtheorem{dfn}[thm]{Definition}
\newtheorem{rmk}[thm]{Remark}

\newcommand{\mapsfrom}{\mathrel{\reflectbox{\ensuremath{\mapsto}}}}

\newcommand{\Fo}{{F_{\bullet}}}
\newcommand{\Eo}{{E_{\bullet}}}
\newcommand{\qo}{{q_{\bullet}}}

\newcommand{\GL}{{\mathrm{GL}}}

\newcommand{\SO}{{\mathrm{SO}}}
\newcommand{\SP}{{\mathrm{Sp}}}

\newcommand{\sgn}{{\mathrm{sgn}}}

\newcommand{\Ind}{{\mathrm{Ind}}}
\newcommand{\cInd}{{\mathrm{cInd}}}
\newcommand{\Res}{{\mathrm{Res}}}

\newcommand{\antidiag}{{\text{anti-diag}}}
\newcommand{\diag}{{\mathrm{diag}}}
\newcommand{\Sym}{{\mathrm{Sym}}}

\newcommand{\Aut}{{\mathrm{Aut}}}

\newcommand{\Hom}{{\mathrm{Hom}}}
\newcommand{\End}{{\mathrm{End}}}

\newcommand{\Gal}{{\mathrm{Gal}}}

\newcommand{\Ad}{{\mathrm{Ad}}}

\usepackage{tocloft}

\usepackage{musicography}

\title{Endoscopic liftings and parameters of epipelagic representations for classical groups}
\author{
 Geo Kam-Fai Tam (\href{mailto:kamfai.tam@xmu.edu.my}{kamfai.tam@xmu.edu.my})
 }
\date{Xiamen University Malaysia
\\[2ex]
\today}

\begin{document}
\maketitle


\begin{abstract}
Let $G$ be a $p$-adic classical group (orthogonal, symplectic, or unitary) and $\pi$ be an epipelagic representation of $G$ in the sense of Reeder-Yu. Using M{\oe}glin's theory of extended cuspidal supports and Bushnell-Kutzko's theory of covering types, we determine explicitly the endoscopic lift of $\pi$ to the general linear group, whose Langlands dual expresses the dual group of $G$ as a complex matrix group,  in terms of the inducing type of $\pi$ that extends the character of the first Moy-Prasad filtration subgroup defined by a stable functional. We interpret the inducing type of $\pi$ via Stevens' construction of supercuspidal representations by skew semisimple strata and introduce what we will call epipelagic strata, requiring only that the residual characteristic $p$ be odd. As an application, we reprove M. Oi's results on the endoscopic lifts of simple supercuspidal representations, in the sense of Gross-Reeder, of quasi-split classical groups. Finally, on the Galois side, we show that the epipelagic Langlands parameters of $G$ constructed by Reeder-Yu can be recovered via the self-duality of Bushnell-Henniart's admissible triples and endoscopic embeddings of L-groups. We also establish some related results concerning epipelagic parameters that were previously proved under restrictions on $p$, including a parity result on rectifying characters and an adjoint Swan conductor result related to the Hiraga-Ichino-Ikeda conjecture.

\end{abstract}

\setcounter{tocdepth}{2}
\tableofcontents

\section{Introduction}

The local Langlands correspondence (LLC) states that, in layman's terms, the irreducible admissible representations of a connected reductive group $G$ over a p-adic field $F$ may be parameterized by morphisms, known as Langlands parameters, from the absolute Galois group $\Gal(F^{\mathrm{sep}}/F)$ of $F$ into the L-group ${}^LG := \hat{G}\rtimes \Gal(F^{\mathrm{sep}}/F)$ of $G$, where the latter group is defined by a related group $\hat{G}$ known as the dual of $G$ over the field of complex numbers.  Although currently a conjecture in general, the LLC has been established when $G$ is a general linear group \cite{Harris-Taylor, Henniart-simple-proof, Laumon-Rapoport-Stuhler, Scholze-LLC-GLn} or a classical group \cite{Arthur-book, Mok-unitary, Kaletha-Minguez-Shin-White, Ishimoto-odd-orthogonal}. Recently, \cite{Fargues-Scholze}  announced a construction of Langlands parameters associated with irreducible representations of a general reductive group using deep algebraic-geometric methods, compatible with many expected properties of the LLC.

One promising approach to proving the LLC, within the scope of the representation theory of p-adic reductive groups, is the theory of endoscopy \cite{Langlands-Shelstad} and its twisted analogue \cite{Kottwitz-Shelstad}. A major idea of this approach, among other important ones, states that, again in laymen's terms, if an L-group ${}^LG$ is contained in another such group ${}^LH$, then viewing a parameter of $G$ as a parameter of $H$ via this inclusion corresponds to a lifting process from the parametrized representations of $G(F)$ to those of $H(F)$, known as the {\bf endoscopic lift} from $G$ to $H$. As a hypothetical strategy, if the LLC of $H$ is known, one may hope that the knowledge from $H$ can descend to yield information about, and eventually prove, the LLC of $G$. This strategy was shown to be successful in proving the LLC for general linear groups and classical groups in the cited works.

In this paper, we study the endoscopic lifting of epipelagic representations. These representations were first introduced in \cite{Reeder-Yu}. They are irreducible, supercuspidal, and compactly induced from so-called types, a colloquial term for irreducible representations of compact-mod-center open subgroups, which are built from relatively simple data, known as stable functionals. These functionals can be viewed as dual vectors in the quotient of the two shallowest Moy-Prasad filtration subgroups of a parahoric subgroup at the barycenter of a specific facet in the Bruhat-Tits building $\mathcal B(G,F)$ of $G$. They generalize the simple supercuspidal representations introduced in \cite{Gross-Reeder}, in which case the facet is an alcove. We choose to focus on epipelagic representations not only because of their simplicity and popularity \cite{Romano-thesis, BK-epipelagic, Kaletha-epipelagic, epipelagic-unitary}, but also because of their importance to other areas of  representation theory of reductive groups \cite{epipelagic-theta-correspondence, Epipelagic-rigid-local-systems, epipelagic-Brylinski-Deligne-cover}.

We therefore investigate how the inducing types of epipelagic representations of an endoscopic group $G$ are related to those of a target group $H$. The methodology used in this paper applies to a classical group $G$ (orthogonal, symplectic, or unitary) with the target group $H$ being the general linear group expressing the dual group of $G$ as a matrix group. According to the theory of endoscopy, since representations of $G(F)$ with the same endoscopic lift comprise an L-packet, i.e., a finite set of representations related by the endoscopic character identity (meaning they have the same Langlands parameter under the assertion of the LLC), our final results provide an explicit description of the LLC for epipelagic representations of classical groups.

Our methodology combines two theories which are purely local in nature: (a) M{\oe}glin's theory \cite{Moeglin-classification-classical-groups,Moeglin-classification-unitary-groups,Moeglin-twisted-endoscopy-Langlands-parameters} of the reducibility of parabolically induced representations, which is used to determine the cuspidal supports of endoscopic liftings, also known as extended cuspidal supports, and (b) Bushnell-Kutzko's theory \cite{BK-types} of covering types, which is used to translate the aforementioned reducibility into an analogous property for modules of (affine) Hecke algebras. This combination has been shown to be successful in computing endoscopic liftings in specific cases \cite{BHS, Lust-Stevens, Tam-unramified-unitary, BT-ramified}, the last two of which were studied previously by the author. In a noteworthy case \cite{BHS} where $G$ is a symplectic group, this methodology successfully describes the inertial class of the endoscopic lifting of {any} supercuspidal representation of $G(F)$, i.e., the endoscopic lifting is determined up to twists by unramified characters. By focusing on epipelagic representations, we restrict ourselves to relatively simple descriptions of their inducing types, as well as their associated covering types.

\subsubsection*{The methodology}

We now explain our methodology more precisely and describe the main results in Theorem \ref{main theorem in the intro}. For the simplicity of this introduction, let $G$ be special odd orthogonal or symplectic, defined by a non-degenerate symmetric or skew-symmetric $F$-bilinear form $h$. Let $K$ be the maximal unramified extension of $F$, and $\mathbb F$ be the residual field of $F$ whose characteristic $p$ is odd. The first two steps of the Moy-Prasad filtration $G(K)_{x}\supset G(K)_{x,0_+}\supset G(K)_{x,0_{++}}$ yield successive quotients $\mathsf G_x := G(K)_{x}/G(K)_{x,0_+}$ and $\mathsf{V}_{x} := G(K)_{x,0_+}/G(K)_{x,0_{++}}$. The conjugacy action of $G(K)_{x}$ on the filtration subgroups defines a representation $\mathsf G_x\rightarrow \GL(\mathsf{V}_{x})$ over ${\mathbb F}$.

Let $\beta\in \mathsf{V}_{x}^*(\mathbb F)$ be an $\mathbb F$-linear functional on $\mathsf{V}_{x}$. By fixing a non-trivial character $\psi$ of $\mathbb F$, we define a character $\psi_\beta$ of $G(F)_{x,0_+}$ which is trivial on $G(F)_{x,0_{++}}$. With our choices of $G$, the normalizer $N(\psi_\beta)$ of $\psi_\beta$ in $G(F)_{x}$ has the quotient by $G(F)_{x,0_+}$ isomorphic to an elementary abelian 2-group $\{\pm1 \}^{k}$ for some $k\in \mathbb Z_{>0}$. (This $k$ will be known to be the cardinality of $I\smallsetminus\{o\}$, where $I$ is the set of indices parametrizing the components of $\beta$ defined in the next paragraph.) With the assumption on $p$ being odd, $\psi_\beta$ extends to a character $\lambda$ of $N(\psi_\beta)$. We now assume:
\begin{equation*}
\text{the functional $\beta\in \mathsf{V}_{x}^*(\mathbb F)$ is stable, in the sense of  geometric invariant theory (GIT) \cite{Mumford-Stabilityofprojectivevarieties}. 
}
\end{equation*}
The stability of $\beta$ then implies that the compactly induced representation $\pi_{}:=\cInd_{N(\psi_\beta)}^{G(F)}\lambda$ is irreducible, and is an epipelagic representation in the sense of \cite{Reeder-Yu}. We note that the character $\lambda$ is uniquely determined by $\beta$ and a tuple of signs $\{\lambda(\omega_i)\}_{i=1}^k$, where $\omega_i\in G(F)$ corresponds to the element $((1_{j})_{j\neq i},-1_i)\in N(\psi_\beta)/G(F)_{x,0_+}$.

A similar construction applies to general linear groups. For our purpose, it is more desirable to express an inducing type of a supercuspidal representation using a maximal simple type defined by \cite[Sec 5]{BK}. Indeed, the stable functional $\beta$ can be viewed as an elliptic regular  semisimple element in the dual Lie algebra $\mathfrak g^*$ of $G$, and admits a decomposition $\oplus_{i\in I}\beta_i$, where $F[\beta_i]$ generates a field over $F$ except when $G$ is orthogonal and for a unique index $o\in I$, in which case $\beta_i=0$ (note that $F[\beta_o] = F$ with our convention). Assume that $i\neq o$. Let $\tilde G_i = \GL_{[F[\beta_i]:F]}$ and $\tilde G_{i,\tilde x}$ be a parahoric subgroup of $\tilde G_i (F)$ whose pro-p unipotent radical $\tilde G_{i,\tilde x,0_+}$ affords the character $\psi_{2\beta_i}$. The normalizer $\tilde N(\psi_{2\beta_i})$ of $\psi_{2\beta_i}$ in $\tilde G_i(F)$ has the quotient by $\tilde G_{i,\tilde x,0_+}$ isomorphic to $F[\beta_i]^\times/(1+\mathfrak p_{F[\beta_i]/F}) \cong {\boldsymbol{\mu}}_F \times \left<\varpi_i\right>$, where $ {\boldsymbol{\mu}}_{E_i}$ is the subgroup of ${E_i}^\times$ consisting of roots of unity of order coprime to $p$, and $\varpi_i$ is a chosen uniformizer of $F[\beta_i]$. By the general construction of maximal simple types (see \cite[Sec 6.1]{BK}), $\psi_{2\beta_i}$ can always be extended to a character $\tilde{\boldsymbol{\lambda}}_i$ of $\tilde N(\psi_{2\beta_i})$, and $\tilde\pi_i:=\cInd_{\tilde N(\psi_{2\beta_i})}^{\tilde G_i(F)}\tilde{\boldsymbol{\lambda}}_i$ is also a supercuspidal representation of $\tilde G_i(F)$, which is called quasi-epipelagic (and is epipelagic if $F[\beta_i]/F$ is totally ramified).

We now view $M:=\tilde G_i\times G$ as a Levi subgroup of a parabolic subgroup $P$ of a classical group $\mathcal G$, of the same type as $G$ but of higher rank. We define a family of normalized parabolically induced representations $I(s,\tilde\pi_i,\pi):=\iota_{P}^{\mathcal G(F)}(\tilde\pi_i|\det|^s\times \pi)$ parametrized by $s\in \mathbb C$, which is reducible at a half-integer $s\in \tfrac{1}{2}\mathbb Z_{\geq 0}$. Generally, M{\oe}glin's theory in \emph{loc. cit.} asserts that $\tilde\pi_i$ belongs to the extended cuspidal support of $\pi$ if and only if $I(s,\tilde\pi_i,\pi)$ is reducible at a half-integer $s\geq 1$ (with an extra parity condition that is irrelevant in our present setup). When $\pi$ is epipelagic,  $I(s,\tilde\pi_i,\pi)$ is then reducible at exactly one point in $\{0,\tfrac{1}{2},1\}$. Therefore, our objective is to determine those $\tilde\pi_i$ such that $I(1,\tilde\pi_i,\pi)$ is reducible.

\subsubsection*{The main results}

We can now state our main result on the reducibility points of $I(s,\tilde\pi_i,\pi)$, as a slight simplification of Propositions \ref{first general form of lifting, unitary group} and \ref{first general form of lifting}.

\begin{thm}
\label{main theorem in the intro}
Let $G$ be an odd special orthogonal group or a symplectic group over $F$, and $\pi_{}:=\cInd_{N(\psi_\beta)}^{G(F)}\lambda$ be an epipelagic representation constructed from the data $(\beta = \oplus_{i\in I}
\beta_i, \{\lambda(\omega_i)\}_{i\in I})$. Put $$\hat I = \begin{cases}
I\smallsetminus \{o\}&\text{ when $G$ is orthogonal, and }
\\
I\sqcup \{o\}&\text{ when $G$ is symplectic}.
\end{cases}
$$
In both cases, for each index $i\in \hat I$ with $i\neq o$, let $\tilde\pi_i:=\cInd_{\tilde N(\psi_{2\beta_i})}^{\tilde G_i(F)}\tilde{\boldsymbol{\lambda}}_i$  be the representation of $\tilde G_i(F)$ induced from the character $\tilde{\boldsymbol{\lambda}}_i$ of $\tilde N(\psi_{2\beta_i})$ defined by:
\begin{equation*}
\tilde{\boldsymbol{\lambda}}_i|_{\tilde G_{i,\tilde x,0_+}} = \psi_{2\beta_i},\quad 
\tilde{\boldsymbol{\lambda}}_i|_{{\boldsymbol{\mu}}_{F}}= \left(\frac{\cdot}{{\boldsymbol{\mu}}_{F}}\right)^{(\epsilon_G-1)/2},
\quad\text{and}\quad
\tilde{\boldsymbol{\lambda}}_i(\varpi_i)= \tilde{\boldsymbol{\lambda}}_i(-2)\lambda(\omega_i )
\mathfrak{n}_z(\varpi_i,\beta,\psi,h).
\end{equation*}
Here 
\begin{itemize}
\itemsep0em 
\item $\epsilon_G=1$ (resp. $-1$) if $G$ is orthogonal (resp. symplectic), i.e., $h$ is an $\epsilon_G$-bilinear form,
\item $\left(\tfrac{\cdot}{{\boldsymbol{\mu}}_{F}}\right)$ is the quadratic character of ${\boldsymbol{\mu}}_{F}$, and 

\item $\mathfrak{n}_z(\varpi_i,\beta,\psi,h)$ is a normalized quadratic Gauss sum (i.e., a fourth root of unity), defined by a quadratic form on a certain $\mathbb F$-space $\mathfrak{W}_z$, which depends on the given data $(\varpi_i,\beta,\psi,h)$ and is related to a covering type for $I(s,\tilde\pi_i,\pi)$ in $\mathcal G(F)$.
\end{itemize}
When $G$ is symplectic and $i=o$, we also define a character $\tilde\pi_o$ of $F^\times$ by 
$$\tilde\pi_o|_{1+\mathfrak p_F}\equiv 1,\quad  \tilde\pi_o|_{{\boldsymbol{\mu}}_F} = \left(\tfrac{\cdot}{{\boldsymbol{\mu}}_F}\right)^{\#I} 
\quad\text{and}\quad
\tilde\pi_o(\varpi_{}) = \prod_{i\in I}\left(\tfrac{\varpi\det\beta_i}{{\boldsymbol{\mu}}_F}\right).$$
Then $I(s,\tilde\pi_i,\pi)$ is reducible at $s=1$ for all $i\in \hat I $.
\qed\end{thm}

The main technicality of the proof is presented in Section \ref{section Main calculation}, especially in Section \ref{subsection Expanding the intertwining operator as a sum}. A little more work in Section \ref{subsection Non-degeneracy of a quadratic form} leads us to compute the values of $\mathfrak n_z(\varpi_i,\beta,\psi,h)$. A simple application of M{\oe}glin's theory on the bound for the number of such $\tilde\pi_i$ then implies that $\{\tilde\pi_i\}_{i\in I}$ comprises the extended cuspidal support of $\pi$, i.e., the endoscopic lift of $\pi$ has cuspidal support precisely $\{\tilde\pi_i\}_{i\in I}$.

We remark that analogous results hold for other types of $G$, although the result is much simpler when $G$ is unramified unitary (in which case any epipelagic representation is indeed simple supercuspidal) and a bit more complicated when $G$ is even special orthogonal or ramified unitary. For the last two types of $G$, there is an additional structural complicacy in the inducing types arising from the stability of the functionals. We omit the details for the sake of this introduction, but refer interested readers to the classification results of stable gradings in \cite[Sec 7.2]{Gross-Levy-Reeder-Yu} and their interpretations in terms of semisimple strata in Section \ref{subsection Stable functionals for classical groups}.

We therefore conclude that, with our methodology, we can compute the endoscopic lifts, or more precisely their inducing types, and consequently the L-packets of \emph{all} epipelagic representations of \emph{all} classical groups. The \emph{only} requirement on $p$ is that $p\neq 2$. These results are presented in Section \ref{subsection Reducibility results for different classical groups}.

\subsubsection*{Further details}

To compute the reducibility points stated in Theorem \ref{main theorem in the intro}, we apply Bushnell-Kutzko's theory to  identify functorially the Bernstein component of $\tilde\pi_i\times \pi$ in $\mathcal G$ with the category of modules over the intertwining algebra, i.e., the desired Hecke algebra $\mathcal H(\mathcal G,\lambda_P)$, of a covering type $\lambda_P$ over $\tilde\lambda_i\times \lambda$ in $\mathcal G$, where $\tilde\lambda_i $ is the restriction of $\tilde{\boldsymbol{\lambda}}_i$ to the maximal compact subgroup of $\tilde N(\psi_{2\beta_i})$. The structure of $\mathcal H(\mathcal G,\lambda_P)$ is well known due to Lusztig \cite{Lusztig-finite-classical-groups}: it is of generic type on an infinite dihedral group \cite[Th 1.2]{Stevens-Miyauchi}. We may take from $\mathcal H(\mathcal G,\lambda_P)$ two generators, denoted by $T_y$ and $T_z$ in this paper, each of which satisfies a quadratic equation whose coefficients can be determined by the structure of $\lambda_P$ (Section \ref{subsection Structures of Hecke algebras}).

The reducibility of $I(s,\tilde\pi_i,\pi)$ is hence converted into that of the corresponding module $X_s$ over $\mathcal H(\mathcal G,\lambda_P)$. By the aforementioned functoriality, $X_s$ is induced from the character $D_s$ of the intertwining algebra $\mathcal H(M,\tilde\lambda_i\times \lambda)$ of $\tilde\lambda_i\times \lambda$ in $M$ corresponding to $\tilde\pi_i|\det|^s\times \pi$. A crucial observation from \cite[(1.13)]{Blasco-Blondel-SP4} implies that, when $X_s$ is reducible, there are two  ways to express the eigenvalues of the convolution product $T_y* T_z$: one by multiplying the respective eigenvalues of the generators from the quadratic equations, and the other from the inducing character $D_s$. Equating the two expressions leads to a formula (\ref{main formula of reducibility points}), which we will examine in Section \ref{subsection Expanding the intertwining operator as a sum}, for computing the points of reducibility; see Section \ref{section Reducibility} for complete details of how this formula is derived.

One major advantage of considering epipelagic representations is that the relevant inducing types are all one-dimensional characters whose ramification filtrations are minimal (i.e., single-level). This advantage renders the calculation in Section \ref{subsection Expanding the intertwining operator as a sum} manageable. The analogous calculation for general supercuspidal representations, especially when the ramification filtration is multi-level, exhibits substantial complications arising from the matching of  conjugacy classes between $G$ and $\tilde G_i$. See \cite{BT-ramified} for a calculation when $G$ is a ramified unitary group with the same absolute rank as $\tilde G_i$.

Cuspidal inducing types and the related covering types can be constructed from arithmetic data known as semisimple strata, based on the general theories developed in \cite{BK, Stevens-supercuspidal}. To apply these theories to epipelagic representations, we will translate the language of stable functionals into that of semisimple strata in Section \ref{subsection Stable functionals for classical groups}. For instance, the stability condition on functionals translates into an elliptic-regularity conditions on vectors in $\mathfrak g^*$, and the classification in \cite[Sec 7.2]{Gross-Levy-Reeder-Yu} of points $x\in \mathcal B(G,F)$ for which $\mathsf{V}^*_{x}$ contains stable functionals translates into conditions on the semisimple strata that define the relevant types of parahoric subgroups admitting characters arising from stable functionals (see Proposition \ref{equivalent conditions for stability} and the list that follows). Since these translated conditions can be intrinsically stated within the framework of semisimple strata, we define in Definition \ref{definition of Epipelagic stratum} what we call {\bf epipelagic strata}, which are then extended to characters inducing epipelagic representations (see Section \ref{subsection Epipelagic inducing types for classical groups}). This translation process has the advantage that the residual characteristic $p$ of the base field $F$ is required only to be odd, the same condition on $p$ for exhaustively constructing supercuspidal representations of classical groups in \cite{Stevens-supercuspidal}.

\subsubsection*{Related literatures}

With our methodology explained above, in Section \ref{section Examples for simple supercuspidals} we reprove, and compare with ours, the results of M. Oi \cite{Oi-SO-odd,Oi-Sp-and-SO-even,Oi-U-unram} on  endoscopic liftings of simple supercuspidal representations of $G$. Oi's methodology for computing endoscopic liftings mainly uses the endoscopic character identity to compare the character expansions of $\pi$ and $\tilde\pi_i$ at affine generic elements in terms of Gauss and Kloostermann sums. His method applies to quasi-split classical groups, except for ramified unitary groups, whereas our method applies to all classical groups, or more precisely, pure inner forms of quasi-split classical groups, as long as the cohomological invariant of the Hermitian form defining $G$ is considered (see Sections \ref{subsection Cohomological classification} and \ref{section Embeddings of lattices}).

Moreover, our method applies to a more general class of representations, namely epipelagic representations. In fact, we use the inducing types of skew epipelagic strata that are simple as `building blocks' to derive results for the types that are more generally semisimple. These simple characters closely resemble the affine generic characters underlying simple supercuspidals. Similar to Oi's results, our inducing types of epipelagic representations can be expressed in terms of signs and quadratic Gauss sums.


Speaking of the endoscopic character identity, Kaletha's theory applies to tamely ramified reductive p-adic groups to  construct explicitly the L-packets of epipelagic representations \cite{Kaletha-epipelagic} and regular supercuspidal representations \cite{Kaletha-regular-supercuspidal}. One may also consult the recent preprints \cite{Oi-2025, Oi:2026aa} concerning an endoscopic character relation for toral supercuspidal representations. In those works, the residual characteristic $p$ is required to satisfy  stricter conditions than just being odd. Since we construct the inducing types of epipelagic representations based on the theory of semisimple strata in \cite{Stevens-supercuspidal}, our method requires only that $p$ be odd. It would be interesting to compare our endoscopic liftings with those obtained by various authors.

Following the spirit of \cite{Kaletha-epipelagic} again, \cite{epipelagic-unitary} computed the LLC for certain epipelagic representations of odd ramified unitary groups whose underlying stable functionals admit a very specific form. We will provide a brief comparison of their results with ours in Subsection \ref{Comparison with results in epipelagic-unitary}.

Recently, \cite{BHS2} computed the endoscopic liftings of simple supercuspidal representations of symplectic groups, using a method similar to the one presented in this paper. They confirmed that their lifts match those of Oi. Although one may argue by transitivity, we directly compare the lifting results in \cite{BHS2} with ours in Subsection \ref{subsection Results in BHS2} for the convenience of readers, and verify that the two results coincide.

\subsubsection*{The Galois side}

In the final section, Section \ref{section Langlands parameters}, we shift to the Galois side and explain how endoscopic lifting from classical groups to general linear groups works. By doing so, we can view Langlands parameters of classical groups as self-dual representations of the Weil group $\mathcal W_F$ of $F$ via endoscopic lifting, allowing us to work primarily with self-dual parameters of general linear groups. We then specialize to parameters of (quasi-)epipelagic representations, and compare our endoscopic lifting results with Reeder-Yu's template for epipelagic Langlands parameters in \cite[Sec 7.2 and 7.3]{Reeder-Yu}.

The following two results are consequences of our comparison. Strictly speaking, these results are expected because they have been established in a broader generality, but under various ``tameness" conditions on the residual characteristic $p$, the group $G$, or the parameters. Here, we establish these results without conditions on $p$ other than being odd, while restricting our attention to classical groups and their epipelagic parameters. For this reason, we record them here in itemized form and state them as propositions in Subsections \ref{subsubsection A parity result} and \ref{subsubsection Minimal conductor property}, rather than formulating them as theorems.
\begin{enumerate}[(i)]
\item 
Knowing that an endoscopic-lifted parameter from a classical group $G$ is self-dual with a parity ($+$ or $-$) depending on the type of $G$, we show in Proposition \ref{last parity result} that the parity of a self-dual  (quasi-)epipelagic parameter $\tilde\varphi$, constructed from Reeder-Yu's template, matches the type of $G$ only if it is twisted by a specific character, which we call a rectifying character. The idea of rectifying characters originates in the essentially tame case \cite{BH-ET1,BH-ET2,BH-ET3}, was extended to the effective setting \cite{BH-Eff} for the LLC for general linear groups, and was studied in the context of endoscopy by the author \cite{Tam-ETLLC-GLn,Tam-ETLLC-GLn-inner}. We extend this idea here for classical groups via endoscopic lifting. See also \cite{Kaletha-epipelagic} for an alternative formulation of rectifying characters for general tamely ramified reductive groups.

\item Let $\pi$ be a supercuspidal representation of $G$ with epipelagic parameter $\varphi$. If $\pi$ contains a stable functional $\beta\in \mathsf V_{x}^*$ for a point $x\in \mathcal B(G,F)$, and $\mathsf G_{x}$ is the reductive quotient, then we show in Proposition \ref{adjoint swan is dim of Gx} (see also Propositions \ref{adjoint swan conductor GL-case} and \ref{adjoint swan conductor simple supercuspidal for classical groups}) that the adjoint swan conductor $\mathrm{Adsw}$ of $\varphi$ is given by 
$$\mathrm{Adsw}(\varphi) = \dim \mathsf G_{x}.$$
This identity is consistent with the formal degree conjecture formulated by Hiraga-Ichino-Ikeda \cite{Hiraga-Ichino-Ikeda}, which predicts the formal degree of a supercuspidal representation in terms of invariants attached to its Langlands parameter. In the epipelagic case, one of these invariants, $\mathrm{Adsw}(\varphi)$, is shown to equal $\dim \mathsf G_{x}$ derived from the inducing type of $\pi$. We remark that this conductor result was previously known or could be deduced under certain conditions on $p$ or on the parameters \cite{Reeder-Yu, Oi-Sp-and-SO-even, SchweinFormaldegreeofregularsupercuspidals}.
\end{enumerate}

The novelty of our results stems from the inputs of \cite{BH-Eff} and \cite{BK-epipelagic}, which allows arguments valid for all odd primes $p$. Specifically, using these inputs, we obtain a description of an arbitrary irreducible parameter $\tilde\varphi = \Ind_{\mathcal W_K}^{\mathcal W_F}(\tilde\rho\otimes\tilde \tau)$ in terms of an admissible triple $[K/F,\tilde\rho,\tilde \tau]$, which is an equivalence class of data consisting of a tamely ramified subextension $K/F$, which can be taken as the maximal tamely ramified subextension of $F[\beta]$, and two parts $\tilde\rho$ and $\tilde \tau$ of the parameter, which we colloquially refer to as the wild and the tame parts respectively. We then compare the data $[K/F,\tilde\rho,\tilde \tau]$ with Reeder-Yu's template which, loosely speaking, can likewise be separated into wild and tame parts.

For the wild part, in Proposition \ref{ramification theorem in RY template}, we establish a bijection $\beta\leftrightarrow [\hat\psi_\beta]$ between stable functionals and certain $\mathcal W_K$-conjugacy classes of representations $\hat\psi_\beta$ of the wild inertia group $\mathcal P_F$ of Heisenberg type, slightly generalizing \cite[Lem 7.1]{Reeder-Yu} in the context of general linear groups. Extending $\hat\psi_\beta$ to its normalizer in $\mathcal W_F$ yields the wild part $\tilde\rho$. To avoid the residual characteristic constraint in \cite{Reeder-Yu}, we apply the Ramification Theorem of \cite{BH-Eff} as an intermediary step, which provides a canonical bijection 
$\Theta \leftrightarrow \Psi$ between intertwining classes $\Theta $ of simple characters, also known as endo-classes, and $\mathcal W_F$-orbits $\Psi$ of irreducible representations of $\mathcal P_F$, specializing to $\beta\leftrightarrow [\hat\psi_\beta]$ in the quasi-epipelagic case. The ramification structure of the above representations of Heisenberg type, which is known from \cite{BK-epipelagic}, allows us to  compute the adjoint Swan conductor under only the odd residual characteristic constraint.

For the tame part, we continue along the lines of Reeder-Yu's template to construct two automorphisms $\hat n_s$ and $\hat n_t$ in the normalizer of the image of $\tilde\rho$ in $\hat G$, derived from the generators $s$ of inertia and $t$ of Frobenius, respectively, to obtain the tame part $\tilde\tau$ of $\tilde\varphi$. If $\tilde\varphi$ is self-dual and lifted from an epipelagic parameter of $G$, we recover in Proposition \ref{prop ReederYu parameter as induced representation} the rectifying character $\hat \mu$ of $\mathcal W_K$, specializing to the epipelagic case from the general rectifying character appearing in the Comparison Theorem of \cite[Sec 7.3]{BH-Eff}, which measures the discrepancy between the canonical constructions of supercuspidal representations and their parameters from the  ramification data $\Theta$ and $\Psi$ in the last paragraph. The rectifying character carries a parity itself, and we show that one must twist $\tilde\tau$ by $\hat \mu$ to achieve the correct parity, guaranteeing that the twisted parameter $\Ind_{\mathcal W_K}^{\mathcal W_F}(\tilde\rho\otimes(\tilde \tau\hat \mu))$ is lifted from $G$ (but not $\Ind_{\mathcal W_K}^{\mathcal W_F}(\tilde\rho\otimes\tilde \tau)$ if $\hat \mu$ is non-trivial).

As a final remark, in the tame case, we can place the above discussions within the context of \cite{Kaletha-epipelagic} using admissible embeddings of L-tori into the L-group of $G$, and interpret an epipelagic parameter as an induced representation using endoscopic chi-data. This interpretation was established in the author's previous work \cite{Tam-ETLLC-GLn,Tam-ETLLC-GLn-inner} for general linear groups and their inner forms. For classical groups, we find that the combinatorial arguments in Lemmas \ref{characterization of n_t} and \ref{mu_ns_nt as induced representation} are sufficient for our purposes.

\subsection{Acknowledgements}

This research was supported by the open competition of NWO under Grant No. OCENW.M20.132, and partially by the Radboud Excellence Initiative. We would like to thank Maarten Solleveld for reading several draft versions of this article with great interest. During the preparation of the first draft, the author was supported by the encouragement of Julia Gordon, to whom the author is deeply indebted. This article would not exist without the contributions of these two mathematicians. Thanks are also due to Anne-Marie Aubert, Corinne Blondel, and Guy Henniart for their careful readings of the first draft and insightful comments.

This work was partially presented at the Oberwolfach Workshop 2449: \emph{Representations of p-adic groups} in December 2024, and at the seminar of the Institute of Mathematics, Academia Sinica, Taipei in March 2025. The author would like to thank the respective organizers for their invitations and hospitality. During the trip to Taipei, the author was inspired by discussions of Cheng-Chiang Tsai and Masao Oi to expand and improve the article. The writing of the second draft was supported by Xiamen University Malaysia Research Fund (Grant No: XMUMRF/2024-C14/IMAT/0033).

{\bf Disclaimer.} Portions of the text and grammar were refined with the assistance of AI
tools, but no mathematical ideas were generated by them.

\subsection{Notations and conventions}

The cardinality of a finite set $X$ is denoted by $\#X$. The real and imaginary parts of a complex number $s\in \mathbb C $ are denoted by $\Re(s)$ and $\Im(s)$ respectively.

Given a group $G$, we denote its center by $ZG$. If $G$ is a finite abelian group, the notation $G^\wedge$ stands for the Pontryagin dual of $G$. The action of $G$ on a set $X$ is denoted by $x\mapsto {}^gx$, for $g\in G$ and $x\in X$.

Given a matrix $A$, we denote by ${}^tA$ its transpose. Suppose that $A$ has entries in a field $K$, and $K/K_\bullet$ is a quadratic extension, we denote by $\overline A$ the entrywise $K/K_\bullet$-conjugation of $A$.

We denote a diagonal matrix with entries $A_1,\dots,A_n$ from NW to SE by $\diag(A_1,\dots,A_n)$, and an anti-diagonal matrix with entries $A_1,\dots,A_n$ from NE to SW by $\antidiag(A_1,\dots,A_n)$. Each entry $A_i$ may itself be a matrix.

A representation of a topological, locally profinite group $G$ is assumed to be smooth. We write $(G,\pi)$ if we want to emphasize the underlying group of the representation $\pi$. In many cases, $G$ is the subgroup of rational points of a reductive group over a non-Archimedean local field, and $\pi$ is a smooth representation of $G$, usually irreducible and supercuspidal. To simplify the writing, all supercuspidal representations of a reductive group are presumed to be irreducible unless otherwise specified.

Given a non-Archimedean local field $F$, denote $\mathfrak o_F$ the ring of integers with maximal ideal $\mathfrak p_F$. The residual field $\mathbb F:=\mathfrak o_F/\mathfrak p_F$ has cardinality $q = q_F$, which is a power of a prime number $p$. By fixing a uniformizer $\varpi$ of $F$, the multiplicative subgroup $F^\times$ of $F$ decomposes as $\left<\varpi\right>\times \mathfrak o_F^\times$, and also $ \mathfrak o_F^\times =  {\boldsymbol{\mu}}_F\times U^1_F$, where ${\boldsymbol{\mu}}_F$ is the subgroup of root of unity of order coprime to $p$, and $U^k_F= 1+\mathfrak p^k_F$ for all $k\in \mathbb Z_{\geq 1}$.

If $H$ is a finite cyclic group of even order, denote by $\left(\frac{\cdot }{H}\right):H\rightarrow \{\pm 1\}$ the quadratic character of $H$. Let $\psi:\mathbb F\rightarrow \mathbb C^\times$ be a non-trivial additive character. We denote the normalized quadratic Gauss sum by $\mathfrak{n}_{\psi}  := q^{-1/2}\sum_{x\in \mathbb F^\times }\psi(x)\left(\frac{x}{\mathbb F^\times}\right)$. It is known that $\mathfrak{n}_\psi^2  = \left(\frac{-1}{\mathbb F^\times}\right)$, so that $\mathfrak{n}_\psi$ is a 4th root of unity.

The extended real line \cite[(6.4.1)]{Bruhat-Tits-reductive-group-1} is $\tilde{\mathbb R} = \{r,r_+:r\in \mathbb R\}$, equipped with an order $r>s$ if and only if $r>r_+\geq  s$. We work with filtrations of groups over $F$ compatibly parametrized by $\tilde{\mathbb R}$. To match with the enumeration of Moy-Prasad filtrations, we parametrize any lattice sequence $\Lambda$ using the normalized valuation, i.e., $\varpi\Lambda(r) = \Lambda(r+1) $ for all $r\in \mathbb R$. If $E/F$ is a field extension of ramification degree $m$, we separate the normalized filtration on the ideals generated by $\mathfrak p_E$ from the standard filtration by using $\mathfrak p^r(E) = \mathfrak p_E ^{\lceil rm\rceil} :=(\mathfrak p_E )^{\lceil rm\rceil}$ and $\mathcal U^r(E) = U^{\lceil rm\rceil}_E$.

 \section{Epipelagic Representations}
 \label{section Epipelagic Representations}

 Let $F$ be a non-Archimedean local field with residual characteristic $p$, and $G$ be a connected reductive group over $F$. Let $K$ be the maximal unramified extension of $F$ in its separable closure $F^{\text{sep}}$. Fix a maximal $F$-split torus $S$ of $G$, and take a maximal $K$-split $F$-torus $T$ containing $S$.

The torus $T$ determines an apartment $\mathcal{A} =\mathcal{A}(T,K)$ in the Bruhat-Tits building $\mathcal{B}(G) = \mathcal{B}(G,K)$ of $G$, as well as an affine root system $\Phi_{\mathrm{aff}}$ containing the underlying (relative) root system $\Phi = \Phi(G,T)$. Fix a special point $o$ and an alcove $\mathcal{C}$ in $\mathcal A$ containing $o$. The boundary hyperplanes of $\mathcal C$ determine a set of simple affine roots $\Delta_{\mathrm{aff}}$, consisting of the roots in a simple root system $\Delta\subset \Phi$ together with an affine root $\alpha_0:=1-\alpha_l$, where $\alpha_l$ is the longest root in $\Phi$.

We take a point $x\in \mathcal{B}(G,F) = \mathcal B(G)^{\Gal(K/F)}$ and assume it lies in $\mathcal{A}$ by $G$-translation. For $r\in \tilde{\mathbb R}_{\geq 0}$, denote by $G(K)_{x,r}$ the Moy-Prasad filtration subgroup in $G(K)$, and put $G_{x,r}=G(F)_{x,r}:=G(K)_{x,r}^{\Gal(K/F)}$. The stabilizer group $G(F)_{x}$ of $x$ in $G(F)$ contains the parahoric subgroup $G_{x,0}$ with finite index modulo center. Now take $ r(x)$ to be the minimal positive value in $\{\psi(x):\psi\in \Phi_{\mathrm{aff}}\}$. We call an irreducible representation of $G$ {\bf epipelagic} if it has depth $r(x)$ for some $x\in \mathcal{B}(G,F)$ and contains a non-zero vector fixed by $G_{x,r(x)_+}$.

 \subsection{Stable functionals}
 \label{subsection Stable functionals}

The maximal compact subgroup $T_0$ of $T(K)$ acts on the affine root subgroup $U_{\psi} = U_{\psi}(\mathfrak o_K)$ for each $\psi\in \Phi_{\mathrm{aff}}$, and on the quotient $\mathfrak g_{\dot\psi}:=U_{\psi}/U_{\psi_+}$, where $U_{\psi_+}$ is the next filtration subgroup of $U_{\psi}$ in $U_{\dot\psi}$ and ${\dot\psi}\in \Phi$ is the direction of $\psi$. Denote by $T_{0_+}$ the kernel of the action of $T_0$ on all $\mathfrak g_{\alpha}$ for all $\alpha\in \Delta$, so that $\mathfrak g_{\alpha}$ is the $\alpha$-weight subspace of the Lie algebra $\mathfrak g$ of $G(\overline{\mathbb F})$ under the action of $T_0/T_{0_+}$.

Put
$\Phi_{x,r} = \{ \alpha\in \Phi:
  \alpha =   \dot\psi\text{ for some $\psi \in \Phi_{\mathrm{aff}}$}
    \text{ such that $\psi(x)= r$}\}$ for $r\in \tilde{\mathbb R}$. We have a decomposition of the quotient
    $$
\mathsf{V}_{x,r} := G(K)_{x,r}/G(K)_{x,r_+}\cong \bigoplus_{\alpha\in \Phi_{x,r}
}\mathfrak g_{\alpha}.$$ 
A functional $\beta$ in ${\mathsf{V}}^*_{x,r}$, the linear dual space of ${\mathsf{V}_{x,r}}$, is called {\bf stable} for the action of $\mathsf G_x := G(K)_{x,0}/G(K)_{x,0_+}$ if the $\mathsf G_x $-orbit of $\beta$ in ${\mathsf{V}}^*_{x,r}$ is closed and the stabilizer of $\beta$ in $\mathsf G_x $ is finite-mod-center (i.e., a finite algebraic group after modulo the center of $\mathsf G_x$).

 Fix an additive character $\psi$ of $F$ throughout the paper, which is trivial on $\mathfrak p_F$ but non-trivial on $\mathfrak o_F$, and denote also by $\psi$ the induced character of $\mathbb F = \mathfrak o_F/\mathfrak p_F$. Given $x\in \mathcal B(G,F)$ and $\beta\in {\mathsf{V}}^*_{x,r}(\mathbb F) := ({\mathsf{V}}^*_{x,r})^{\Gal(\bar{\mathbb F}/\mathbb F)}$, we define a character on the compact subgroup $G(F)_{x,r}$ by 
 $$\psi_\beta: G(F)_{x,r} \xrightarrow{} (G_{x,r}/G_{x,r_+} )^{\Gal(\bar{\mathbb F}/\mathbb F)} \cong \mathsf{V}_{x,r}(\mathbb F) \xrightarrow{\beta} \mathbb F\xrightarrow{\psi}\mathbb C^\times.$$ 
Denote by $G(F)_{x,\beta}$ the stabilizer of $\beta$ in $G(F)_x$. We now take $r=r(x)$. If  $\beta$ is moreover stable, we take an irreducible representation $(G(F)_{x,\beta},\lambda)$ containing $( G_{x,r}, \psi_\beta)$, i.e., $\lambda$ being an irreducible constituent in $\Ind_{G(F)_{x,r}}^{G(F)_{x,\beta}}\psi_\beta$. Then $\cInd_{G(F)_{x,\beta}}^{G(F)} \lambda$ is an irreducible supercuspidal representation of $G(F)$, and is moreover epipelagic \cite[Prop 2.4]{Reeder-Yu}.

\begin{rmk} The existence of rational stable functions in ${\mathsf{V}}^*_{x,r}$ was first shown in \cite[Sec 5 and 6]{Reeder-Yu} for large $p$, and was then established in \cite{Fintzen-Romano-stable-vectors} for arbitrary $p$ if one extends  from $F$ to a sufficiently large unramified extension of $F$.
 \qed\end{rmk}

\subsection{Example: Simple supercuspidal representations}
\label{subsection Example: Simple supercuspidal representations}

Simple supercuspidal representations are examples of epipelagic representations. These representations were first constructed in \cite{Gross-Reeder} for simply connected simple $F$-split groups, using affine generic characters. For quasi-split reductive groups, we recall the statements from \cite{Oi-SO-odd}.

With the setup of the previous section, we take $x\in \mathcal C$ to be the barycenter, i.e., the point where the affine roots in $\Delta_{\mathrm{aff}}$ attain a common value, which is $r = 1/h$, where $h$ is the twisted Coxeter number \cite{Reeder-torsion-autom}. Put $\mathcal{I}^+ = G_{x,r}$ and $\mathcal{I}^{++} = G_{x,r_+}$, and denote by $Z$ the center of $G(F)$. We call a character $\chi$ of $\mathcal{I}:= Z\mathcal{I}^+$ {\bf affine generic} if it is trivial on $\mathcal{I}^{++}$ and non-trivial on $U_{\alpha}/U_{\alpha_+}$ for all $\alpha\in \Delta_{\text{aff}}$. In this paper, we also call the functional $\beta$ that gives rise to $\psi_\beta = \chi|_{\mathcal I^+}$ {\bf affine generic} if $\chi$ is.

Denote by $N(\chi)$ the stablizer of $\chi$ in $G(F)$. Given an affine generic character $\chi$ of $Z\mathcal{I}^+$, the compact induction $\cInd_{Z\mathcal{I}^+}^{G(F)}\chi$ admits a decomposition
$$\cInd_{\mathcal{I}}^{G(F)}\chi = \bigoplus_{\chi'} (\dim {\chi'} )\pi_{\chi'},$$
where $\chi'$ ranging over the irreducible constituents of $\cInd_{Z\mathcal{I}^+}^{N(\chi)}\chi$ and $\pi_{\chi'} = \cInd_{N(\chi)}^{G(F)} \chi'$. Each $\pi_{\chi'}$ is an irreducible supercuspidal representation of $G(F)$. \cite[Prop 2.8]{Oi-SO-odd}.

Recall the decomposition $N_G(T)/T_0 = W_{\mathrm{aff}}\rtimes \bar\Omega$, where $\bar \Omega$ is the (finite) group $\Gamma_{\mathcal C}$ defined in \cite[p.189]{Bourbaki-Lie-4-6}, and let $\Omega$ be a set of representatives of $\bar\Omega$ in $N_G(T)$. If $\pi_{\chi'}$ and $\pi_{\xi'}$ are two such representations arising respectively from affine generic characters $\chi$ and $\xi$ of $\mathcal{I}$, then
\begin{equation}
\label{criterion of isomorphic simple supercuspidal}
   \begin{split}
    & \text{$\pi_{\chi'}$ and $\pi_{\xi'}$ are isomorphic if and only if }
     \\
     &\text{there exists $t\in T_0\Omega$ such that ${}^t(\chi,\chi')=  (\xi,\xi')$.}
   \end{split}
\end{equation}
 This gives an upper bound 
$$\mathcal{I}\subseteq N(\chi) \subseteq \mathcal{I}\Omega = N_{G(F)}(\mathcal{I}^+)$$
 for the stabilizer $N(\chi)$. Note that in \cite[Sec 8]{Gross-Reeder}, the group $G$ is assumed to be simply-connected, so that $\bar\Omega$ is trivial.

\section{Covering types}
\label{section Covers}

Let $P= MU$ be a parabolic subgroup of a connected reductive group $\mathcal G$ over $F$, where $M$ is a Levi subgroup in $P$ and $U$ is the unipotent radical of $P$. Let $P^- = M U^-$ be the opposite of $P$.

A compact subgroup $\mathcal J_P$ of $\mathcal G(F)$ is called {\bf decomposed with respect to} $(M,P)$, or just $(M,P)$-decomposed, if 
$\mathcal{J}_P = \mathcal{J}_P^- \mathcal{J}_M \mathcal{J}_P^+$, where 
$$\mathcal{J}_P^- = \mathcal{J}_P\cap U^-, \quad \mathcal{J}_M = \mathcal{J}_P\cap M,\quad \text{and}\quad \mathcal{J}_P^+ = \mathcal{J}_P\cap U^+.$$
Let $\mathcal{J}_P$ be $(M,P)$-decomposed, and $Z_M$ be the center of $ M$. We call an element ${\mathbbm z}\in Z_M(F)$  {\bf strongly positive with respect to }$(P,\mathcal{J}_P)$ if 
$${\mathbbm z} \mathcal{J}_P^+{\mathbbm z}^{-1} \subset \mathcal{J}_P^+, \quad {\mathbbm z}^{-1} \mathcal{J}_P^-{\mathbbm z} \subset \mathcal{J}_P^-,$$
and for any compact subgroups $\mathcal{J}_1,\mathcal{J}_2\subset U^+$ and $\mathcal{J}_3,\mathcal{J}_4\subset U^-$, there exists an integer $N\geq 0$ such that 
$${\mathbbm z}^N \mathcal{J}_1{\mathbbm z}^{-N} \subset \mathcal{J}_2, \quad {\mathbbm z}^{-N} \mathcal{J}_3{\mathbbm z}^N \subset \mathcal{J}_4.$$
An irreducible representation $\lambda_P$ of $\mathcal{J}_P$ is called {\bf decomposed with respect to} $(M,P)$ if $\mathcal{J}_P$ is $(M,P)$-decomposed and both $\mathcal{J}_P^-$ and $\mathcal{J}_P^+$ are contained in $\ker \lambda_P$. We also call $(\mathcal{J}_P,\lambda_P)$ an $(M,P)$-decomposed pair.

\begin{dfn}
\label{definition of cover}
Suppose that $\lambda_M = \lambda_P|_{\mathcal{J}_M}$ is irreducible. We call $(\mathcal{J}_P,\lambda_P)$ a {\bf covering type}, or simply a {\bf cover}, of $(\mathcal{J}_M,\lambda_M)$ in $\mathcal G(F)$ if 
\begin{enumerate}[(a)]
\item $(\mathcal{J}_P,\lambda_P)$ is a decomposed pair,  and 

\label{existence of decomposed pair}

\item  there exist a strongly $(P,\mathcal{J}_P)$-positive element ${\mathbbm z}\in Z_M(F)$ and an invertible element in the Hecke algebra $ \mathcal H(\mathcal G(F),\lambda_P)$ supported on $\mathcal{J}_P{\mathbbm z} \mathcal{J}_P$.\qed
\label{existence of invertible supported on positive} 
\end{enumerate} \end{dfn}

We will use the strongly positive element in \ref{existence of invertible supported on positive} above for computations in Section \ref{section Main calculation}. There is an equivalent criterion in \cite[4.2]{Kim-Yu-tame-types}, which is more convenient for confirming directly that a constructed $(\mathcal{J}_P,\lambda_P)$ is a cover of $(\mathcal{J}_M,\lambda_M)$ and will be recalled in Proposition \ref{results in Kim-Yu-tame-types}. We remark that the criterion is based on the following proposition.

\begin{prop}
\cite{Blondel_Injectivity-of-jacquet} $(\mathcal{J}_P,\lambda_P)$ is a cover of $(\mathcal{J}_M,\lambda_M)$ if and only if criteria \ref{existence of decomposed pair} in Definition \ref{definition of cover} and 
\begin{equation*}
   \begin{split}
     \text{(b')}\quad &\text{for any smooth representation $V$ of $\mathcal G(F)$, the Jacquet map $V \rightarrow V_U$} \\
        &\text{induces an injection on the $(\mathcal{J}_P,\lambda_P)$-isotypic subspace of $V$, }
   \end{split}
\end{equation*}
are satisfied. \hfill$\square$
\end{prop}

\subsection{Preliminaries on the construction of covering types}
\label{subsection Preliminaries on the construction of covering types}

Consider the following setup: suppose that $\mathcal{J}_P$ is a compact open subgroup of $\mathcal G(F)$ containing a parahoric subgroup $\mathcal{J}_{P,0} $ (whose associating facet is unimportant here and hence ignored) and that the pro-unipotent radical $\mathcal{J}_{P,0_+} $ is a normal subgroup of $\mathcal{J}_P$. Assume that $\mathcal{J}_P, \mathcal{J}_{P,0}$, and $\mathcal{J}_{P,0_+}$ are all decomposed with respect to $(M,P)$, and moreover that $\mathcal{J}_P/\mathcal{J}_{P,0_+} \cong \mathcal{J}_M /\mathcal{J}_{M,0_+}  $, where $\mathcal{J}_M = \mathcal{J}_P\cap M $ and $\mathcal{J}_{M,0_+}  = \mathcal{J}_{P,0_+} \cap M$.

Let $(\mathcal{J}_{P,0_+},\theta_P)$ be a character trivial on both $\mathcal{J}_P^{+ }$ and $\mathcal{J}_P^{- }$. Denote by $(\mathcal{J}_{M,0_+},\theta_M)$ its restriction. Take an irreducible representation $(\mathcal{J}_{M},\lambda_M)$ which is $\theta_M$-isotypic, and an $(M,P)$-decomposed pair $(\mathcal{J}_P,\lambda_P)$ with $\lambda_P|_{\mathcal{J}_M}\cong \lambda_M$. 

\begin{prop}
\label{results in Kim-Yu-tame-types}
Under the above setup, we have the following.
\begin{enumerate}[(i)]

\item \cite[Th 6.3]{Kim-Yu-tame-types} The pair $(\mathcal{J}_{P,0_+}, \theta_P)$ is a cover of $(\mathcal{J}_{M,0_+}, \theta_M)$ in $\mathcal G(F)$. 
\label{extension of simple character}

\item \cite[Cor 6.4]{Kim-Yu-tame-types} If $(\mathcal{J}_P,\lambda_P)$ satisfies criterion \ref{existence of decomposed pair} of Definition \ref{definition of cover} and $\lambda_P|_{\mathcal{J}_{P,+}}$ is $\theta_P$-isotypic, then it is a cover of $(\mathcal{J}_M,\lambda_M)$ in $\mathcal G(F)$. (This is indeed a consequence of \ref{extension of simple character}). 
\label{cover of simple character extends to cover of cuspidal type}
 \hfill$\square$
\end{enumerate}

\end{prop}

When constructing covers for classical groups in the next section, we will implicitly identify the notions of lattice filtrations and Moy-Prasad filtrations. Interested readers may consult \cite{Broussous-Stevens-buildings} or \cite{Lemaire-compare-lattice-Moy-Prasad} for details of this identification. In particular, we will take a stable functional $\beta$ over $F$ in $M$ and lift it to a functional in $\mathcal G$, such that the building of the centralizer $\mathcal G(F)_\beta = \mathcal G_\beta(F)$ is a tree. We will take two facets  $\mathcal F_y,\mathcal F_z\subset \mathcal B(\mathcal G,F)$ such that each
$$w := \mathcal F_w\cap \mathcal B( \mathcal G_\beta,F) ,\quad w\in \{y,z\},$$ 
is a vertex in $\mathcal B( \mathcal G_\beta,F)$, and moreover $y$ and $z$ are adjacent of each other. Let $\mathcal F_{\mathfrak m} \subset \mathcal B(\mathcal G,F)$ be a facet such that $\mathfrak m:= \mathcal F_{\mathfrak m}\cap \mathcal B( \mathcal G_\beta,F)  $ is the edge connecting $y$ and $z$.

The parahoric subgroups $\mathcal G(F)_{\mathcal F_w,0}$, for $w\in \{y,z\}$, and also $\mathcal G(F)_{\mathcal F_{\mathfrak m},0}$ are decomposed with respect to a given $(M,P)$. Upon restricting to $ \mathcal G_\beta$, the parahoric subgroup $\mathcal G_\beta(F)_w$ is also decomposed with respect to $(M\cap \mathcal G_\beta,P\cap \mathcal G_\beta)$, which is a pair of Levi and parabolic subgroups in $\mathcal G_\beta$. The indices $\{y,z\}$ will be labelled to match the pair of parabolics $\{P,P^-\}$, in the sense that
 $$\mathcal G_\beta(F)_{\mathfrak m} = (\mathcal G_\beta(F)_y\cap P)\mathcal G_\beta(F)_{y,0+} = (\mathcal G_\beta(F)_z\cap P^-)\mathcal G_\beta(F)_{z,0+}.$$
For each $w\in \{y,z\}$, the quotient ${\mathsf P}_{\beta,\mathfrak m,w}:= \mathcal G_\beta(F)_{\mathfrak m}/\mathcal G_\beta(F)_{w,0_+}$ is then a maximal parabolic subgroup of $\overline{\mathcal G}_{\beta,w} :=\mathcal G_\beta(F)_{w}/\mathcal G_\beta(F)_{w,0_+}$.

Later in Section \ref{subsection Lattices and covers for classical groups} when we study the reducibility of certain parabolically induced representations of classical groups, we will define a character $\theta_{\mathfrak{m}}$ on the pro-p-subgroup $\mathcal G(F)_{\mathcal F_{\mathfrak m},0_+}$. If we can extend this character to an irreducible representation of its normalizer $\mathcal G(F)_{\mathcal F_\mathfrak m}$, then according to the conditions in Proposition \ref{results in Kim-Yu-tame-types} it will be automatically a covering type (of its restriction to a suitable Levi subgroup). In fact, our constructed covering type will also be a character.

\section{Classical groups}
\label{section Classical groups}

This section translates the general language of Moy-Prasad filtrations and stable characters into the language of lattices and strata, after providing the preliminaries of classical groups.

Let $\Fo$ be a non-Archimedean local field with residual characteristic $p\neq 2$. For an extension $F/\Fo$ which is either trivial or quadratic, we denote by $(x\mapsto \bar x)\in \Gal(F/\Fo)$ the trivial or involutive automorphism respectively. We choose representatives $\{1,\zeta,\varpi,\zeta\varpi\}$ of $F^\times/F^{\times2}$, where $\zeta$ is a generator of the group $ {\boldsymbol{\mu}}_F$ of roots of unity of order coprime to $p$, and $\varpi$ is a uniformizer of $F$, both fixed throughout the paper.

Given a finite dimensional vector space $V$ over $F$, we denote by $\tilde{G}=\tilde{G}(V)$ the $F$-algebraic group of linear automorphisms of $V$, which is isomorphic to the general linear group $\GL_{n}$ over $F$, where $n=\dim _FV$. Suppose now $V$ is equipped with a non-degenerate $\epsilon$-Hermitian form $h=h_V$ relative to the extension $F/\Fo$ and $\epsilon=\epsilon_G\in \{\pm 1\}$, which means that
$$h(av,bw)={}\bar ab\cdot h(v,w)\quad\text{ and }\quad
h(w,v)=\epsilon\overline{ h(v,w)},
\quad \text{for all }a,b\in F\text{ and }v,w\in V.$$
A classical group is the group $G^\sharp = G^\sharp(V,h)$ of isometries of $h$, which is an $\Fo$-algebraic group of type either 
\begin{equation*}
\begin{split}
&\text{orthogonal ($\mathrm{O_{odd}}$, 
$\mathrm{O_{even}}$) }
\\
&\text{symplectic ($\mathrm{Sp}$)}
\\
&\text{unitary ($\mathrm{U_{N}}$)}
\end{split}
\qquad\text{when}\qquad
\begin{split}
&\text{$[F/\Fo]=1$, $\epsilon=1$,}
\\
&\text{$[F/\Fo]=1$, $\epsilon=-1$,\quad or}
\\
&\text{$[F/\Fo]=2$, $\epsilon=(-1)^{N-1}$}.
\end{split}
\end{equation*}
Denote by $\sigma $ the  involutive automorphism of $\tilde{{G}}$ such that $\tilde{{G}}^\sigma = {{G}}^\sharp$. Let ${G}$ be the identity component of ${G}^\sharp$, so that ${G}={G}^\sharp$ except when ${G}^\sharp$ is an orthogonal group, in which case ${G}$ is the underlying special orthogonal group ($\mathrm{SO_{odd}}$, $\mathrm{SO_{even}}$).

The involution $\sigma$ on $\tilde G$ induces the adjoint action $\alpha$ on its Lie algebra $\tilde{\mathfrak g}=\tilde{\mathfrak g}(V):=\mathrm{End}_F(V)$, so that $\mathfrak g:=\tilde{\mathfrak g} ^\alpha$ is the Lie algebra of ${G}$.

Suppose now that $(V,h_V)$ is fixed. Let $\tilde V = \tilde V_-$ be another finite dimensional vector space over $F$, and $\tilde V_+$ be its $F$-linear dual. Put  $\tilde V_\pm =\tilde V_-\oplus \tilde V_+$, equipped with the structure of a hyperbolic space with respect to the Hermitian form
\begin{equation*}
h_{\tilde V}:\left<(z-,z+),(w_-,w_+)\right> \mapsto (z_-,w_+) +\epsilon\cdot      {}\overline{(w_-,z_+)},
\quad 
z_-,w_-\in \tilde V_-\text{ and }z_+,w_+\in \tilde V_+.
\end{equation*} 
It will be convenient to fix a Witt basis on $\tilde V_\pm$, consisting of a basis of $\tilde V_-$ and its dual basis in $\tilde V_+$.

We define the Hermitian space 
$W := V \perp \tilde V_\pm $, equipped with the form $h_W = h_{\tilde V_\pm } \perp h_V$. The isometry group $\mathcal G^\sharp = G^\sharp(W,h_W)$ is a classical group of the same type as $G^\sharp$ but of higher rank, which is rank$(G)+\dim \tilde V$. We denote the involutive and adjoint automorphisms respectively associated to $\mathcal G^\sharp$ and its Lie algebra $\mathrm{Lie}(\mathcal G)$ also by $\sigma$ and $\alpha$.

Fix a chain of subspaces $\tilde V\subset V \oplus \tilde V \subset W$ and denote by $P$ the corresponding parabolic subgroup of $\mathcal G$, with unipotent radical $U$ and the opposite $U^-$. Denote by $M^\sharp\cong \tilde{G}({\tilde V})\times G^\sharp(V)$ the Levi component of $P$, and fix an embedding 
\begin{equation*}
\label{embedding of Levi subgroup}
i_M:\tilde{G}\times G^\sharp\rightarrow M^\sharp\subset \mathcal G^\sharp, \qquad (g,h)\mapsto (g,h,^\sigma g).
\end{equation*}
Denote by $\mathcal G$ the identity component of $\mathcal G^\sharp$, and put $M = M^\sharp\cap \mathcal G$.

For computations in Section \ref{section Main calculation}, it is desirable to write down some matrix presentations for elements in $U$ and $U^-$. By choosing a suitable basis of $V$, let $H$ be the $\epsilon$-Hermitian matrix, i.e., $H = \epsilon{}^{t}\overline H$, that defines the Hermitian form $h_V$. Then  ${}^\sigma g = H^{-1}{}^t\overline g^{-1} H$ and ${}^\alpha X = -H^{-1}{}^t\overline X H$. Let $\tilde H$ be the matrix that identifies the dual space $\tilde V_+$ with $\tilde V_-$ relative to the Witt basis. Then the matrix $H_W$ that defines $h_W$ is given by 
\begin{equation}
\label{The big matrix JW}
H_W = \text{anti-diag}(\tilde H, H ,  \epsilon {}^{t}\overline{\tilde H}).\end{equation}
If we denote a typical blocked unipotent matrix by 
\begin{equation*}
(X,Y,Z)^+ := \begin{bmatrix}
I_{\tilde V_-} & X & Y
\\
& I_V& Z
\\
&& I_{\tilde V_+}
\end{bmatrix},
\end{equation*}
then $\alpha:(X,Y,Z)^+ \mapsto ({}^\alpha Z,{}^\alpha Y,{}^\alpha X)^+$, where
\begin{equation}
\label{alpha on X and Y}
{}^\alpha X :=  -H^{-1} {}^{t} \overline X \tilde H,\qquad
{}^\alpha Y := -\epsilon {}^{t}\overline{\tilde H}^{-1}   {}^{t}\overline Y  \tilde H,
\qquad\text{and}\qquad
{}^\alpha Z := -\epsilon {}^{t}\overline{\tilde H}^{-1}  {}^{t} \overline Z  H.
\end{equation}
We have $(X,Y,Z)^+\in U$ if and only if both relations $Z={}^\alpha X$ and \begin{equation}
\label{X-alpha-X-equals-Y-minus-alpha-Y}
X{}^\alpha X = Y-{}^\alpha Y
\end{equation} hold, in which case we simply denote $(X,Y,{}^\alpha X)^+$ by $(X,Y)^+$. Similar results hold for $
(X,Y,Z)^- := \begin{bmatrix}
I_{\tilde V_-} & &
\\
Z& I_V& 
\\
Y&X& I_{\tilde V_+}
\end{bmatrix}\in U^-
$, and we put $(X,Y)^- = (X,Y,{}^\alpha X)^-\in U^-$.

\subsection{Semisimple strata}
\label{subsection Semisimple strata}

We extract from \cite{BK, Stevens-supercuspidal} the definitions on  lattices and strata. Let $V$ be a finite dimensional vector space over $F$. Given an $\mathfrak o_F$-lattice sequence $\Lambda$ in $V$ and $r\in \mathbb R$, we put  
$$\mathfrak P^{r}(\Lambda):=\{X\in \End_{F}(V):X\Lambda(t)\subseteq \Lambda(t+r)\text{ for all }t\in \mathbb R\}.$$
We emphasize that, unlike in \emph{loc. cit.}, the filtration of any lattice sequence in this paper is normalized, i.e., $\varpi\Lambda(t) = \Lambda(t+1) $ for all $t\in \mathbb R$. Define the valuation map $v_\Lambda$ on $\End_{F}(V)$ by $v_\Lambda(X) = \max\{t\in \mathbb R: X\in \mathfrak P^{t}(\Lambda)\}$.

A stratum in an $F$-vector space $V$ is a triple $\mathbf{s}=[\Lambda,t,\beta]$ consisting of an $\mathfrak o_F$-lattice sequence $\Lambda$ in $V$, an element $\beta\in \tilde {\mathfrak g}(V)$, and a real number $t\leq -v_\Lambda(\beta)$. With $\Lambda$ and $t$ fixed, two strata $\mathbf{s}_i=[\Lambda,t,\beta_i]$, for $i=1,2$, are called equivalent if $\beta_1-\beta_2\in \mathfrak P^{-t}(\Lambda)$. A stratum $\mathbf{s}$ is {\bf simple} if $F[\beta]$ is a field and $\Lambda$ is $\mathfrak o_{F[\beta]}$-invariant.

Suppose that $V$ admits a decomposition $\oplus_{i\in I}V_i$, with $\mathbf{1}_i: V\rightarrow V_i$ the projection onto $V_i$ with kernel $\oplus_{j\neq i}V_j$, and $\Lambda_i = \Lambda\cap V_i$ is a lattice sequence in $V_i$. A stratum $\mathbf{s}=[\Lambda,t,\beta]$ is called {\bf semisimple} if $\beta = \sum_{i\in I}\beta_i$, with $\beta_i =\mathbf{1}_i\circ \beta\circ \mathbf{1}_i $, each $\mathbf s_i = [\Lambda_i,t,\beta_i]$ is simple or null (i.e., $\beta_i=0$), and $[\Lambda_i\oplus \Lambda_j,t,\beta_i\oplus \beta_j]$ is not equivalent to a simple stratum for all $i,j\in I$ with $i\neq j$.

We can formulate the definitions above in the self-dual setting. Suppose now that $p\neq 2$ and $V$ is equipped with an $F/\Fo$-Hermitian form $h$. We call a semisimple stratum $\mathbf s = \oplus_{i\in I}\mathbf s_i$ self-dual if $\Lambda$ is self-dual, i.e., $\Lambda = \Lambda^*$ where
$$\Lambda^*(t) = \Lambda(t)^* :=\{v\in V: h(v,\Lambda(t))\subseteq \mathfrak p_F \},\quad t\in \mathbb R,$$
 and $\beta\in \tilde {\mathfrak g}^\alpha$. A self-dual stratum $\mathbf s$ is {\bf skew} if moreover the decomposition $\oplus_{i\in I}V^i$ is orthogonal with respect to $h$ and each $\beta_i\in \tilde {\mathfrak g}(V_i)^\alpha$. Note that $\mathbf s_i$ is then a skew simple stratum for each $i\in I$.

The idea of strata is that, given a simple stratum $\mathbf{s}=[\Lambda,t,\beta]$ with $t\geq -v_{\Lambda}(\beta)/2$, the equivalence class of $\mathbf s$ corresponds to an additive character 
$$\psi_\beta:X\mapsto \psi\circ\mathrm{tr}_{\tilde {\mathfrak g}/F}(\beta X), \quad X\in \mathfrak P^{t_+}(\Lambda),$$
which is trivial on $\mathfrak P^{-v_{\Lambda}(\beta)_+}(\Lambda)$. For all $t\in \mathbb R_{>0}$, denote $\mathcal{U}^{t}(\Lambda): = I+\mathfrak P^{t}(\Lambda)$. We use the same symbol for the character $I+X\mapsto \psi_\beta(X)$ on $\mathcal{U}^{t_+}(\Lambda)$, which is trivial on $\mathcal{U}^{-v_{\Lambda}(\beta)_+}(\Lambda)$. If $\mathbf s$ is furthermore self-dual, then by restriction $\psi_\beta$  corresponds to characters on $\mathfrak P^{t_+}(\Lambda)^\alpha$ and on $\mathcal{U}^{t_+}(\Lambda)^\sigma$ respectively.

We remark that if in contrast $t<-v_{\Lambda}(\beta)/2$, then we have to undergo an approximation process to produce suitable characters for constructing supercuspidal representations, \emph{cf.} \cite[Sec 2.4]{BK} and \cite[Sec 3.1]{Stevens-semi-simple-char}. Moreover, during the construction we may need to extend each character into an irreducible representation of Heisenberg type, which is of dimension $>1$ in general. These processes do not concern us when we later consider representations which have the shallowest depths, and so we will keep the related discussions minimal.

\subsection{Compact subgroups}
\label{subsection Stable functionals for classical groups}

Fix a connected classical group $G =G(V,h)$ and let $\mathbf{s}=[\Lambda,0,\beta]$ be a skew semisimple stratum in $V$. By definition, the lattice sequence $\Lambda$ admits an orthogonal decomposition $\Lambda = \oplus_{i\in I}\Lambda_i$, and $\beta = \oplus_{i\in I}\beta_i$. Put $E_i = F[\beta_i]$, then each $\Lambda_i$ is $\mathfrak o_{E_i}$-invariant. If $e_i:=e_{\Lambda_i/\mathfrak o_F}$ is the $\mathfrak o_F$-period of $\Lambda_i$, then $e_i = e_{E_i/F}$.

Starting from such a stratum $\mathbf s$, we define a subgroup $\mathcal H^1(\Lambda,\beta)$ containing $\mathcal{U}^{(-v_\Lambda(\beta)/2)_+}(\Lambda)$ as in \cite[Sec 3.2]{Stevens-semi-simple-char}. In general, $\mathcal H^1(\Lambda,\beta)$ is a product of groups of the form $\prod_{j}\mathcal{U}^{t_j/2_+}(\Lambda)\cap Z_{\tilde G(F)}({\gamma_j})$, where $\mathbf s_j = [\Lambda,t_j,\gamma_j]$ is a finite sequence of semisimple strata which approximate $\mathbf s$.

 We now state an equivalent condition for $\mathcal  H^1(\Lambda,\beta)$ to have only one factor in its product form, i.e., to be equal to $\mathcal{U}^{(-v_\Lambda(\beta)/2)_+}(\Lambda)$. 
 
 \begin{prop}
 \label{epipelagic and parahoric forcing stability}
Under the above conditions on $\mathbf s$, the equality $\mathcal  H^1(\Lambda,\beta) = \mathcal{U}^{(-v_\Lambda(\beta)/2)_+}(\Lambda)$ holds if and only if the periods $e_i$ are equal for all $i\in I$ with $\beta_i\neq 0$.
 \end{prop}
 \proof
To prove the sufficiency, it suffices to recall the definitions from \cite[Sec 3.1 and 3.2]{Stevens-semi-simple-char} and observe that $\mathbf s$ has only one step of approximation, namely itself. Hence, by definition, the product form of $\mathcal H^1(\Lambda,\beta) $ has only one factor. Conversely, suppose that  we have ordered the indices in $I$ such that $e_i\leq e_{i+1}$ for all $I$., If there exists $i\in I$ with both $\beta_{i}$ and $\beta_{i+1}$ non-zero such that $e_i\neq e_{i+1}$, then there are at least two steps of approximation of the form $\gamma_j = \oplus_{k\leq i+1}\beta_k$ and $\gamma_{j+1} = \oplus_{k\leq i}\beta_k$, which produce at least two distinct factors in the product form of $\mathcal H^1(\Lambda,\beta)$. This proves the necessity.
 \hfill$\blacksquare$

Under the conditions in Proposition \ref{epipelagic and parahoric forcing stability}, if $x$ is the barycenter of the facet in $\mathcal B(G,F)$ corresponding to $\Lambda$ (\emph{cf.} \cite{Broussous-Stevens-buildings} or \cite{Lemaire-compare-lattice-Moy-Prasad}), then 
$\mathcal{U}^{(-v_\Lambda(\beta)/2)_+}(\Lambda) = G(F)_{x,r}$ where $r = -v_\Lambda(\beta)$ is just $r(x)$ defined in Section \ref{subsection Stable functionals}.

\section{(Quasi-)epipelagic strata}
\label{subsection Epipelagic strata}

In this section, we will show that the definition of stable functionals and the epipelagic condition in \cite{Reeder-Yu} lead to the same forms of semisimple strata as above.

For a connected classical group $G$, let $x\in \mathcal B(G,F)$ be the barycenter of a facet. Take $r\in {\mathbb R}$ and a functional $\beta\in {\mathfrak g}^*_{x,(-r)_+} $. We view $\beta$ as in $  {\mathfrak g}_{x,r} \subseteq \tilde {\mathfrak g}_{x,r}^\alpha$ via the trace form on $\tilde {\mathfrak g}$. If $\Lambda$ is the self-dual lattice sequence corresponding to $x$, then the coset $\beta +\mathfrak P^{(-r)_+}(\Lambda)$ defines a character $\psi_\beta$ on ${\mathsf{V}}_{x,r}$ via the identifications $\mathfrak P^{t}(\Lambda) = \tilde{\mathfrak g}_{x,t}(F)$ and $\mathfrak P^{t}(\Lambda)^\alpha = \tilde{\mathfrak g}_{x,t}(F)^\alpha$ for $t= -r$ and $(-r)_+$, respectively, and the Moy-Prasad isomorphism $\mathfrak P^{-r}(\Lambda)^\alpha/\mathfrak P^{(-r)_+}(\Lambda)^\alpha\cong {\mathsf{V}}^*_{x,r}(\mathbb F)$.

Suppose further that $\beta$ is semisimple in $\mathfrak g^*$ over $K$; in particular, we have the usual definition of regularity of $\beta$. Let $K_\beta$ be the field extension of $K$ that splits $\beta$. Put $e = e_{K_\beta/K}$ and ${\mathfrak g}^*_{K_\beta} = {\mathfrak g}^*\times_F K_\beta$. Since $x\in \mathcal B(G,F)$ is a rational point of order dividing $e$, it becomes a hyperspecial vertex in $\mathcal B^{}(G,K_\beta)$, and the dual space ${\mathsf{V}}^*_{x,r}$ can be regarded as a subspace of $\overline{ {\mathfrak g}^*_{K_{\beta}}} := ({\mathfrak g}^*_{K_{\beta}})_{x,r}/({\mathfrak g}^*_{K_{\beta}})_{x,r_+}$. The semi-simplicity of $\beta$ implies that it lies in a Cartan subspace, in the sense of \cite[Sec 5]{Reeder-Yu}, of ${\mathsf{V}}^*_{x,r}$, so that we may regard $\beta$ as a split semisimple element in $\overline{ {\mathfrak g}^*_{K_{\beta}}}$.

The base-changed space  ${\mathfrak g}^*_{K_\beta} $ is equipped with the action of $\Gal(K_\beta/K)$. Let $\mathfrak s_\beta$ be the centralizer of $\beta$ in ${\mathfrak g}^*_{K_{\beta}}$, which is $\Gal(K_\beta/K)$-invariant.

 The following proposition is the first step towards defining what we will call an epipelagic stratum (see Definition \ref{definition of Epipelagic stratum} below). 
 
 \begin{prop}
 \label{equivalent conditions for stability}
(\cite[Lemma 13]{Gross-Levy-Reeder-Yu} or \cite[Prop 3.1]{Kaletha-epipelagic}) Under the above setup, the functional $\beta\in {\mathsf{V}}^*_{x,r}$ is stable if and only if 
\begin{enumerate}[(i)]
\item $\beta$ is regular as an element in $\overline{ {\mathfrak g}^*_{K_{\beta}}}$, and

\label{equivalent conditions for stability - regularity}
\item the $\Gal(K_\beta/K)$-fixed point subspace of $\mathfrak s_\beta$ is in the center of ${\mathfrak g}^*_{K_{\beta}}$.
 \hfill$\square$
\label{equivalent conditions for stability - ellipticity}
\end{enumerate}
\end{prop}

Suppose now that $r=r(x)$ and let $\mathbf s = [\Lambda,0,\beta]$ be a stratum. From the viewpoint of semisimple strata, the regularity condition \ref{equivalent conditions for stability - regularity} in Proposition \ref{equivalent conditions for stability} on $\beta$ implies that $\mathbf s$ satisfies the regularity condition for semisimple strata in \cite[Def. 2.4]{Stevens-supercuspidal}. Assuming we have corresponding decompositions $\beta = (\beta_i)_{i\in I}$ and $\Lambda = \oplus_{i\in I}\Lambda_i$, condition \ref{equivalent conditions for stability - ellipticity} implies that each component $\mathbf s_i:=[\Lambda_i,0,\beta_i]$ is minimal, in the sense of \cite[1.4.15]{BK}. In particular, $\mathbf s$ is a skew semisimple stratum \cite[Def. 2.5]{Stevens-supercuspidal}.

We hence translate the tables in \cite[Sec 7.2]{Gross-Levy-Reeder-Yu} (for groups of type A, B, C, or D) into a list of structures of a stable functional $\beta$.  
  \begin{enumerate}[(i)]
\item For $G = \GL_N$ or $\mathrm U_{N,F/\Fo}$ where $F/\Fo$ is unramified, the only allowable order $e$ is the Coxeter number, which is $N $, and $F[ \beta]/F$ is a totally ramified extension of degree $N$. 
\label{Table GL and unram-U}
 \end{enumerate}
 
For other classical groups, there is possibly an index $o\in I$ such that $\beta_o=0$, which is unique by regularity if it exists. There are addition conditions listed as follows.

   \begin{enumerate}[resume*]
\item  For $G = \SO_{2n+1}$ or $\SP_{2n}$,  all extensions $E_i/F$ have a constant even ramification degree $e$ for all $i\neq o$. The index $o$ exists when $G$ is odd orthogonal, in which case $E_o=F$, and does not exist if $G$ is symplectic. We also require that $\sum_{i\in I\smallsetminus \{o\}}[E_i:F] = e\sum_{i\in I\smallsetminus \{o\}} f_{E_i/F} = 2n$.
\label{Table SO-odd}

\item For $G = \mathrm{U}_{N,F/\Fo}$ where $F/\Fo$ is ramified, there exists at most one index $o\in I$ such that $E_o = F$, and all other extensions $E_i/F$ have a constant odd ramification degree (which agrees with the result in \cite[Sec 3.1]{BT-ramified}). We require that $\sum_{i\in I}[E_i:F] = N$.
\label{Table ram-U}

\item  For $G = \SO_{2n}$, there exists at most one index $o\in I$ such that $E_o/F$ is ramified quadratic, all other extensions $E_i/F$ have a constant even ramification degree, and furthermore:

   \begin{itemize}
\item if $G$ is split or unramified, then $\#I$ is even;

\item if $G$ is ramified, then $\#I$ is odd.

   \end{itemize}
We require that $\sum_{i\in I}[E_i:F]  = 2n$.
\label{Table SO-even}
\end{enumerate}
Finally, for all classical groups in (ii)-(iv), the self-duality of $\mathbf s$ induces a Galois-involution on each $E_i$ with fixed-field $E_{i\bullet}$ such that $E_i/E_{i\bullet}$ is ramified.

\begin{rmk}
Indeed, the conclusions in Proposition \ref{equivalent conditions for stability}  already determine the shape of $\beta$ in the above list. This implication is clear for general linear groups. The same implication applies to unramified unitary groups, as they are isomorphic to general linear groups over $K$. For other classical groups, the parities of the degrees $e_{E_i/F}$ are already known from the classification in \cite{Stevens-supercuspidal} (and see also \cite{BT-ramified} for ramified unitary groups). For even orthogonal groups, let $d_{E_i/F}$ be the discriminant of $E_i/F$, then we must have $\prod_{i\in I}d_{E_i/F} = \mathrm{disc}(h) \bmod (F^\times)^2$, where $\mathrm{disc}(h)$ is the discriminant of the Hermitian form $h$ defining $G$. Represent the possible discriminants, i.e., $F^\times/F^{\times 2}$, by the quartet $\{1,\zeta,\varpi,\zeta\varpi\}$. Since every $e_{E_i/F}>1$ is even, each $d_{E_i/F} $ can only be $\varpi$ or $\zeta\varpi$. This implies that if $\mathrm{disc}(h)  = 1$ or $\zeta$ (i.e., $G$ is split or unramified), then $\#I$ must be even; if $\mathrm{disc}(h)  = \varpi$ or $\zeta\varpi$ (i.e., $G$ is ramified), then $\#I$ must be odd. \qed
\end{rmk}

Let's summarize the above discussion in the following definition.

\begin{dfn}
\label{definition of Epipelagic stratum}
A skew semisimple stratum $\mathbf s = [\Lambda,0,\beta]$ is called {\bf epipelagic} (over $F$) if  
\begin{enumerate}[(i)]
\item $v_\Lambda(\beta)=-r(x)$ (where $x$ is the barycenter of the facet in $\mathcal B^{\mathrm{red}}(G,F)$ corresponding to $\Lambda$) and 

\item its components $\beta_i$ satisfy the conditions in the above list \ref{Table GL and unram-U}-\ref{Table SO-even}. \end{enumerate}
A simple stratum $\mathbf s = [\Lambda,0,\beta]$ for $G = \GL_N$ is called {\bf quasi-epipelagic}  if $[F[\beta]:F] = N$ and $v_\Lambda(\beta)=-1/e_{F[\beta]/F}$.\qedhere
\end{dfn}
Quasi-epipelagic strata for general linear groups arise from endoscopic liftings of epipelagic representations. The (quasi-)epipelagic condition on $\mathbf s$ implies that, with fixed $\Lambda_i$, the positive depth $-v_{\Lambda_i}(\beta_i)$ for each $\beta_i\neq 0$ is as small as possible, i.e., $-v_{\Lambda_i}(\beta_i) = e_{K[\beta_i]/K}^{-1}$, and is constant. Therefore, a (quasi-)epipelagic stratum  gives rise to the compact subgroup $\mathcal H^1(\Lambda,\beta) = \mathcal{U}^{0_+}(\Lambda)$ as in Proposition \ref{epipelagic and parahoric forcing stability}.

As a final remark in this section, the stability of the functional $\beta$ was studied under the setting of \cite{Gross-Levy-Reeder-Yu}, which requires $F[\beta]/F$ to be tamely ramified. Our definition of epipelagic strata is a particular instance of semisimple strata generally defined in \cite{Stevens-semi-simple-char}, which requires only that $p$ be odd.

\subsection{(Quasi-)epipelagic inducing types}
\label{subsection Epipelagic inducing types for classical groups}

We now begin with a skew semisimple stratum $\mathbf s$. Following \cite[Sec 3.2]{Stevens-semi-simple-char}, we define a semisimple character $\theta$ on $\mathcal H^1(\Lambda,\beta)$ whose restriction to $G(F)_{x,0_+}=\mathcal{U}^{(-v_\Lambda(\beta)/2)_+}(\Lambda)$ is $\psi_\beta$. If $\mathbf s$ is furthermore  (quasi-)epipelagic (as defined in Definition \ref{definition of Epipelagic stratum}), which is assumed until the end of this section, then we see that $\theta = \psi_\beta$ is a character of $G(F)_{x,0_+}$.

The stabilizer group $G(F)_{x,0,\beta}$ of $\psi_\beta$ in $G(F)_{x,0}$ has quotient $\mathsf A_{x,0,\beta}: = G(F)_{x,0,\beta} / G(F)_{x,r}$ isomorphic to
\begin{subequations}
\label{eq:A-group}
\begin{align}
& \GL_1(\mathbb F_{F[\beta]})\cong {\boldsymbol{\mu}}_{F[\beta]} 
 && G = \GL_m,
  \label{eq:A-group,GL}
\\
& \mathrm U_1(\mathbb F_{F[\beta]}/(\mathbb F_{F[\beta]_\bullet}))\cong {\boldsymbol{\mu}}_{F[\beta]}^\sigma   
&\text{if}\qquad&\text{$G$ is unramified $\mathrm U_N(F/\Fo)$, or }
\label{eq:A-group,unram-unitary}
\\
&\text{a product of $\mathrm O_1(\mathbb F_{F[\beta]})\cong \{\pm 1\}$} 
&&\text{$G$ is of other types of classical groups}.
\label{eq:A-group,other-classical-groups}
\end{align}
\end{subequations}
Since all these quotients are abelian and have order coprime to $p$, we can extend $\psi_\beta$ to a character of $G(F)_{x,0,\beta}$ and obtain a bijection
$$\{\lambda\in [G(F)_{x,0,\beta}]^\wedge\text{ where }\lambda|_{G(F)_{x,r}}=\psi_\beta\}\xrightarrow{\sim }  {\mathsf A}^\wedge_{x,0,\beta}.$$
In the first case (\ref{eq:A-group,GL}) where $G = \GL_m$, the full stabilizer group $G(F)_{x,\beta}$ of $\psi_\beta$ in $G(F)$ is $\tilde{\boldsymbol{{\mathcal{J}}}} = F[\beta]^\times G(F)_{x,0_+}$, which contains $\tilde{{\mathcal{J}}} = G(F)_{x,0,\beta}$ as its maximal compact subgroup. Note that $\tilde{\boldsymbol{\mathcal{J}}}$ is the normalizer denoted by $\tilde N(\psi_\beta)$ in the Introduction. By \cite[6.2.2]{BK}, $\tilde{\boldsymbol{\mathcal{J}}} $ is equal to the intertwining set $I_{\tilde{G}(F)}(\tilde{\lambda})$ for any extension $(\tilde{{\mathcal{J}}},\tilde\lambda)$ of $\psi_\beta$, and by \cite[6.1.2]{BK}, we can extend $\tilde\lambda$ to a character $(\tilde{\boldsymbol{{\mathcal{J}}}},\tilde{\boldsymbol{\lambda}})$, so that $\tilde\pi_{\tilde{\boldsymbol{\lambda}}} = \cInd_{\tilde{\boldsymbol{\mathcal{J}}}}^{\tilde G(F)}\tilde{\boldsymbol{\lambda}}$ is a supercuspidal representation. Both the inducing type $\tilde{\boldsymbol{\lambda}}$
and the representation $\tilde\pi_{\tilde{\boldsymbol{\lambda}}}$ are deemed to be quasi-epipelagic.

Continue from above, still with $G = \GL_m$, if $\beta$ is moreover self-dual, then we can extend $\psi_\beta$ to a self-dual character $(\tilde{\mathcal{J}}, \tilde\lambda)$ by requiring $\tilde\lambda|_{{\boldsymbol{\mu}}_F}$ to be of order $\leq 2$. Choosing a uniformizer $\varpi_\beta$ of $F[\beta]$ such that ${}^\sigma\varpi_\beta = -\varpi_\beta^{-1}$, we can further extend $\tilde\lambda$ to $\tilde{\boldsymbol{\lambda}}$ such that $\tilde{\boldsymbol{\lambda}}(\varpi_\beta)^2  = \tilde\lambda(-1)$, i.e., $\tilde{\boldsymbol{\lambda}}$ is also self-dual, and so is $\tilde\pi$.

The situation in the second case (\ref{eq:A-group,unram-unitary}) is similar to the first case, because $\tilde G = \GL_N$ and $G = \mathrm U_{N,F/\Fo} = \tilde G^\sigma$ are isomorphic over $F$. We have $\mathcal J = \tilde{\boldsymbol{{\mathcal{J}}}}^\sigma = \mathrm U_1(F[\beta]/F[\beta]_\bullet) G(F)_{x,0_+}$. We extend $\psi_\beta$ on  $G(F)_{x,0_+}$ to a cuspidal type $\lambda$ on $\mathcal J$, which is again a character since $\#(\mathcal J/G(F)_{x,0_+}) = \#  \mathrm U_1(\mathbb F_{F[\beta]}/(\mathbb F_{F[\beta]_\bullet}))$ is coprime to $p$. 

In the last case (\ref{eq:A-group,other-classical-groups}), the precise numbers of $\{\pm 1\}$-components can be read from the tables in \cite[Sec 7.2]{Gross-Levy-Reeder-Yu}, taking into account the simply-connectedness of the groups. The number of components is $\#I-2$ for $\SO_{\mathrm{odd}}$ (note that the index $o\in I$ does not contribute to a $\{\pm 1\}$-component), $\#I-1$ for $\SO_{\mathrm{even}}$ or ramified $\mathrm{U}$ with $o\in I$, and $\# I$ for all other cases. The element
\begin{equation}
\label{definition of omega_i element}
\omega_i = \diag((I_j)_{j\in I\smallsetminus \{i\}}, -I_i)\in \left(\textstyle\prod_{i\in I}G(V_i)(F)\right)\cap G(F)_{x,0}\subset G(F)_{x,0,\beta}
\end{equation}
corresponds to $((1)_{j\neq i},-1_i)\in \mathsf A_{x,0,\beta}$, and an extension $\lambda$ of a fixed $\psi_\beta$ is determined by $\{\lambda(\omega_i)\}_{i\in I}$.

For both cases (\ref{eq:A-group,unram-unitary}) and (\ref{eq:A-group,other-classical-groups}), the quotient $\mathsf A_{x,\beta}: = G(F)_{x,\beta} / G(F)_{x,r}$ of the full stabilizer group ${\mathcal{J}} = G(F)_{x,\beta}$ is isomorphic to $\mathsf A_{x,0,\beta}$ for unramified unitary groups, to $\{\pm 1\}^{\# I-1}$ for $\SO_{\mathrm{odd}}$, and to $\{\pm 1\}^{\# I}$ for other types of classical groups. Note that $\mathcal J$ is the normalizer denoted by $N(\psi_\beta)$ in the Introduction. Again $\lambda$ can be extended to a cuspidal type on ${\mathcal{J}}$, deemed to be epipelagic, so that $\pi_\lambda = \cInd_{\mathcal J}^{G(F)}\lambda$ is an epipelagic supercuspidal representation.

 \section{Reducibility}
\label{section Reducibility}

Recall the setup at the beginning of Section \ref{section Classical groups}: let $G$ be a classical group defined by a Hermitian space of dimension $n$, and $\tilde G$ be a general linear group $\GL_m$ for some $m$. We denote by $\mathcal G $ the classical group of the same type as $G$ but of a higher rank and containing $M=\tilde G\times G$ as a Levi subgroup.

 Let $\pi_M = \tilde\pi\times \pi$ be an irreducible supercuspidal representation of $M$. Take a parabolic subgroup $P$ containing $M$. We are interested in the points $s\in \mathbb {C}$ where the normalized  parabolic induction 
 \begin{equation}
\label{reducibility of parabolic induction with complex point}
I(s,\tilde\pi,\pi):=\iota_{P}^{\mathcal G}(\tilde\pi|\det|^s\times \pi)
\end{equation}
is reducible. Due to the obvious reason, we restrict ourselves to the complex domain 
\begin{equation}
\label{complex domain}
\{s\in\mathbb C: 0 \leq\Im(s) < \tfrac{2\pi }{f_{\tilde\pi} \log q}\},
\end{equation}
where $f_{\tilde\pi}$ is the order of the subgroup of unramified characters $\chi$ of $F^\times$ such that $(\chi\circ\det)\tilde\pi\cong \tilde\pi$. If $\mathbf s = [\Lambda,0,\beta]$ is an underlying stratum of $\tilde\pi$, then $f_{\tilde\pi} = f_{F[\beta]/F}$. 

If $I(s,\tilde\pi,\pi)$ is reducible for some $s\in \mathbb C$, then, by twisting by an unramified character, we  assume that $\tilde{\pi}$ is self-dual. There is then a unique \emph{real} $s_{\tilde{\pi},{\pi}}\geq 0$, indeed a half integer \cite[Sec 4]{Moeglin-classification-classical-groups},\cite[Sec 3]{Moeglin-classification-unitary-groups}, such that $I(s,\tilde\pi,\pi)$ is reducible at $s=\pm s_{\tilde{\pi},{\pi}}$. By  \cite[6.2.5]{BK}, there are exactly two self-dual representations in the inertial class of $\tilde{\pi}$, namely $\tilde\pi $ and its twist  $\tilde\pi' $ by the unramified character $|\det|^{ \frac{\pi \sqrt{-1}}{f_{\tilde\pi}\log q}}$. The {complex} points of reducibility of $I(s,\tilde\pi,\pi)$ are therefore in the set \begin{equation*}
\label{4 reducibility points}
\mathrm{Red}(\tilde{\pi},{\pi}) := \left\{\pm s_1,\,\pm s_2 +\frac{\pi \sqrt{-1}}{f_{\tilde\pi}\log q}\right\}\qquad\text{ for some }s_1,\,s_2\in \frac{1}{2}\mathbb{Z}_{\geq 0},
\end{equation*}
and those in $\mathrm{Red}(\tilde{\pi}',{\pi})$ take   the same form, with $s_1$ and $s_2$ exchanged.

 The reason we are interested in the reducibility of (\ref{reducibility of parabolic induction with complex point}) is due to the theory of M{\oe}glin \cite{Moeglin-classification-classical-groups,Moeglin-classification-unitary-groups,Moeglin-twisted-endoscopy-Langlands-parameters}; summarizing it in one sentence, it asserts that if $s_{\tilde{\pi},{\pi}}\geq 1$, then \begin{equation*}
\text{ the Langlands parameter of $\tilde\pi$ is a component of the Langlands parameter of $\pi$.}
\end{equation*}
To be more precise, let $\mathcal E (\pi)$ be the set 
 $$\left\{(\tilde G,\tilde\pi): 
     \begin{matrix}   \text{$\tilde G$ a general linear group, $\tilde\pi$ an isomorphism class of}
     \\
     \text{irreducible supercuspidal representations of $\tilde G(F)$,}
     \\
    \text{such that  (\ref{reducibility of parabolic induction with complex point}) is reducible for some $s\in  \tfrac{1}{2}\mathbb N$ with $s\geq 1$}     \end{matrix}
\right\}.$$
Suppose that $s=1$ for all $(\tilde G,\tilde\pi)\in \mathcal E(\pi)$, which is the only case that concerns us for epipelagic representations of classical groups. (We refer the reader for  the general situation to \cite[Sec 4]{Moeglin-twisted-endoscopy-Langlands-parameters}, in which certain parity conditions on $2s-1$ are required.) This set $\mathcal E(\pi)$ is the {\bf extended cuspidal support} of $\pi$, in the sense that the endoscopic lifting $\Pi$ of $\pi$ is the parabolically induced representation of $\GL_{\hat N}(F)$ with cuspidal support $\mathcal E(\pi)$, where 
$$\hat N = \sum_{(\GL_m(F),\tilde\pi)\in \mathcal E(\pi)}m.$$ 
 (In fact, $\hat N$ is known to be equal to the rank of the matrices in the Langlands dual group of $G$, and is equal to $N$ or $N\pm 1$ depending on the type of $G$.) The lifting $\Pi$ is $\sigma$-elliptic in the sense of \cite[Sec 2.8]{Moeglin-twisted-endoscopy-Langlands-parameters}, and is also irreducible. Hence understanding the set $\mathcal E(\pi)$ suffices for understanding the lifting $\Pi$ of an epipelagic $\pi$.

 \subsection{The Hecke algebra approach}
 
 The harmonic analysis for the parabolically induced representation (\ref{reducibility of parabolic induction with complex point}), especially concerning character expansions, can be rather difficult. One algebraic approach, stemming from \cite{BK-types}, is to view this representation in its Bernstein component of the category of smooth representations of $\mathcal G$ and relate this component to the module category of an associated Hecke algebra, as explained below.

 Let $\mathfrak s$ be the inertial class of $\pi_M := \tilde\pi\times \pi$ in $M$, and denote by $ \Ind_M^{\mathcal G}\mathfrak s$ the induced class in $\mathcal G$. Let $\mathcal{R}^{\mathfrak{s}}(M)$ be the full subcategory of representations of $M(F)$  whose irreducible subquotients have cuspidal support lying in $\mathfrak s$, and $\mathcal{R}^{\mathfrak{s}}(\mathcal G)$ be defined similarly using $ \Ind_M^{\mathcal G}\mathfrak s$. Suppose that $\pi = \cInd_{\mathcal J}^{G(F)}\lambda$ and $\tilde\pi = \cInd_{\tilde{\boldsymbol{\mathcal{J}}}}^{\tilde G(F)}\tilde{\boldsymbol{\lambda}}$ are the (quasi-)epipelagic representations constructed as in Section \ref{subsection Epipelagic inducing types for classical groups}, with $\tilde{\boldsymbol\lambda}$ being extended from an irreducible representation $\tilde\lambda$ of the maximal compact subgroup $\tilde {\mathcal J}$ of $\tilde{\boldsymbol{\mathcal{J}}}$. We put $\mathcal J_M = \tilde{\mathcal{J}}\times {\mathcal J}$ and $\lambda_M = \tilde\lambda\times \lambda$, and let $\mathcal{H}(M,\lambda_M)$ be the associated Hecke algebra \cite[Sec 4.1]{BK}, i.e., the space of compactly supported, two-sided $\lambda_M$-invariant functions $M\rightarrow \End_{\mathbb C}(\mathbb V_{\lambda_M})$ equipped with the convolution as multiplication (defined using a fixed Haar measure). Since $\pi_M$ is supercuspidal, it is known \cite{BK-types} that $({\mathcal J}_M,\lambda_M)$ is an $\mathfrak{s}$-type, which means that there is an equivalence of categories
\begin{equation*}
\mathcal{M}_M:\mathcal{R}^{\mathfrak{s}_M}(M)\rightarrow \text{Mod-}\mathcal{H}(M,\lambda_M),\,\qquad\tau\mapsto \Hom_{{\mathcal J}_M}(\lambda_M,\tau).
\end{equation*}
If $({\mathcal J}_P,\lambda_P)$ is a cover of $({\mathcal J}_M,\lambda_M)$ in $\mathcal G(F)$, then by \cite{BK-types} we know that $({\mathcal{J}}_P,\lambda_P)$ is an $\Ind_M^{\mathcal G}\mathfrak s$-type, and there is an analogous equivalence of categories
$
\mathcal{M}_{\mathcal G}:\mathcal{R}^{\mathfrak{s}}(\mathcal G)\rightarrow \text{Mod-}\mathcal{H}(\mathcal G,\lambda_P)
$. 
Let
\begin{equation}
\label{Morphism t_P}
t_P:\mathcal{H}(M,\lambda_M)\rightarrow
\mathcal{H}(\mathcal G,\lambda_P)
\end{equation}
be the injective morphism of algebras defined in \cite[(8.3, 8.4)]{BK-types}, and denote by $(t_P)_*:\text{Mod-}\mathcal{H}(M,\lambda_M)\rightarrow \text{Mod-}\mathcal{H}(\mathcal G,\lambda_P)$ the co-induction functor between the module categories. The above constructions render the following commutative diagram:
\begin{equation}\label{commutative diagram}
  \xymatrixcolsep{5pc}\xymatrix{
\mathcal{R}^{\mathfrak s}(\mathcal G)
\ar[r]^{\mathcal{M}_{\mathcal G}}    
&\text{Mod-}\mathcal{H}(\mathcal G,\lambda_P)
\\
\mathcal{R}^{\mathfrak s}(M) \ar[u]_{\iota_P^{\mathcal G}}
\ar[r]^{\mathcal{M}_M }
&\text{Mod-}\mathcal{H}(M,\lambda_M).  
\ar[u]_{(t_P)_*}
}
\end{equation}
Hence, we can study the reducibility of (\ref{reducibility of parabolic induction with complex point}) by understanding the structures of $\mathcal{H}(M,\lambda_M)$ and $
\mathcal{H}(\mathcal G,\lambda_P)$, as well as their module categories.

\subsection{Structures of Hecke algebras}
\label{subsection Structures of Hecke algebras}

 We refer the reader to \cite[Ch 5]{BK} and \cite[Sec 6 and 7]{Stevens-supercuspidal} for the fine structure of inducing types of supercuspidal representations of general linear groups and classical groups. In the subsequent sections, we will be concerned only with epipelagic representations to simplify the discussions.

Continuing with the setup from the previous section, suppose now that $\tilde\lambda$ is constructed from a quasi-epipelagic simple stratum 
$[\tilde\Lambda,0,\tilde\beta]$, so that $\tilde\lambda$ is a character. We decompose $\tilde \lambda = \tilde \rho \otimes \tilde \kappa$, where $\tilde \kappa$ has order a power of $p$, and $\tilde \rho $ has order coprime to $p$. Hence $\tilde \rho $ is a depth zero character, i.e., a character of $\mathcal{U}(\tilde\Lambda)_{F[\tilde\beta]} $ trivial on $\mathcal{U}^{0_+}(\tilde\Lambda)_{F[\tilde\beta]}$. Similarly, $\lambda$ is constructed from the epipelagic semisimple stratum 
$[\Lambda,0,\beta]$, and we decompose $ \lambda =  \rho \otimes  \kappa$, where $\kappa$ has order a power of $p$ and $ \rho $ is a depth zero character of $\mathcal{U}(\Lambda)_{F[\beta]}^\sigma $ trivial on $\mathcal{U}^{0_+}(\Lambda)_{F[\beta]}^\sigma$. The character $\tilde{\rho}\times \rho$ hence descends to 
\begin{equation}
\label{product of finite reductive groups}
\mathcal{U}(\tilde\Lambda)_{F[\tilde\beta]} /\mathcal{U}^{0_+}(\tilde\Lambda)_{F[\tilde\beta]} \times \mathcal{U}(\Lambda)_{F[\beta]}^\sigma /\mathcal{U}^{0_+}(\Lambda)_{F[\beta]}^\sigma,
\end{equation}
 which is a product of finite groups, each of whose factors is the subgroup of rational points of a reductive (not necessarily connected) group over the finite fields $\mathbb F_{F[\tilde \beta]}$ and $\mathbb F_{F[\beta_i]}$ for $i\in I$ respectively, or all over the base finite field $\mathbb F_{\Fo}$ by restriction of scalars.

We now describe the structures of the related Hecke algebras. Recall that $\tilde{\boldsymbol{\mathcal J}}$ is equal to the intertwining set $I_{\tilde{G}(F)}(\tilde{\lambda})$. We have $\tilde{\boldsymbol{\mathcal J}} = F[\tilde\beta]^\times\tilde{\mathcal J}$, and so $
\tilde{\boldsymbol{\mathcal J}}/\tilde{\mathcal J} \cong F[\tilde\beta]^\times/\mathfrak o^\times_{F[\tilde\beta]} \cong \left<\varpi_{\tilde \beta}\right>$, where $\varpi_{\tilde \beta}$ is a uniformizer of $F[\tilde\beta]$. The Hecke algebra $\mathcal{H}(\tilde{G},\tilde{\lambda})$ is isomorphic to $\mathbb{C}[Z,Z^{-1}]$, where $Z$ is supported on the single coset $\varpi_{\tilde\beta} \tilde{\mathcal J}$. As for $\lambda$, we have $I_{{G}(F)}({\lambda})={\mathcal J}$, and so $\mathcal{H}({G},{\lambda})\cong \mathbb{C}$. Therefore, $\mathcal{H}(M,\lambda_M)\cong \mathbb{C}[Z,Z^{-1}]$.

We now describe the structure of $\mathcal{H}(\mathcal G,\lambda_P)$. If $({\mathcal J}_P,\lambda_P)$ is a cover of $({\mathcal J}_M,\lambda_M)$ but $\tilde{\lambda}$ is not self-dual, then 
$\mathcal{H}(\mathcal G,\lambda_P)\cong \mathcal{H}(M,{\lambda}_M)\cong \mathbb{C}[Z,Z^{-1}]$. This case is not of interest to us. In contrast, when $\tilde{\lambda}$ is self-dual, it is known \cite[Cor 6.16]{Stevens-supercuspidal},  \cite[Prop 3.3]{Blondel-Weil} that     
\begin{equation}
\label{HG is rank 2 over HM}
\mathrm{rank}_{\mathcal{H}(M,\lambda_M)}
(\mathcal{H}(\mathcal G,\lambda_P))=
\#(N_{\mathcal G(F)}(\mathfrak s)/M(F))=2,
\end{equation}
so that by \cite[(11.4)]{BK-types} the Hecke algebra $\mathcal{H}(\mathcal G,\lambda_P)$ is a rank-2 module over $ \mathcal{H}(M,{\lambda}_M)$ via the injective morphism $t_P$ in (\ref{Morphism t_P}). We may describe this algebra by choosing two special generators $T_w$, for $w\in \{y,z\}$, as follows. We pick two elements $s_y$ and $s_z$ in $\mathcal G(F)$ (for example, we may choose $s_1$ and $s^\varpi_1$ in \cite{Stevens-supercuspidal} or \cite{Blondel-Weil}), each of which represents a non-trivial coset of the normalizer group $N_{\mathcal G(F)}(\mathfrak s)$ mod $M(F)$, satisfying the following conditions.
\begin{enumerate}[(i)]
\item $s_y{\mathcal J}_P^-s_y^{-1}\subset {\mathcal J}_P^+$\text{ and }$s_z{\mathcal J}_P^+s_z^{-1}\subset {\mathcal J}_P^-$.

\item If we put $\tilde{\mathbbm z} :=s_ys_z=i_M(\varpi_{\tilde\beta}I_{\tilde V_-},I_V)$, then ${\mathcal J}_P \tilde{\mathbbm z}^{e_{F[\tilde\beta]/F}}{\mathcal J}_P = {\mathcal J}_P {\mathbbm z}{\mathcal J}_P  $ for a strongly $(P,{\mathcal J}_P)$-positive element ${\mathbbm z}$ in the center of $M$.

\item The generator $T_w$, for each $w\in \{y,z\}$, is supported on the double coset ${\mathcal J}_Ps_w{\mathcal J}_P$.

\item  Each $T_w$ satisfies a quadratic relation of the form
\begin{equation}
\label{quadratic relation general b and c}
T_w*T_w=b_wT_w+c_w     \mathbbm{1}  
\end{equation}
for certain real numbers $b_w$ and $c_w$ (here $\mathbbm{1}$ is the unit element of $\mathcal{H}(\mathcal G,\lambda_P)$, which is the two-sided $\lambda_P$-invariant function supported  on ${\mathcal J}_P$ with $\mathbbm{1}(1)=I_{\lambda_P}$ the identity operator on the representation space $\mathbb V_{\lambda_P}$ of $\lambda_P$).

\item The relation $T_y*T_z=t_P(Z)$ holds.\end{enumerate}
In particular, we have $t_P(Z) (\tilde{\mathbbm z}) =  T_y(s_y) \circ T_z (s_z) $. Moreover, $t_P(Z) ^{* e_{F[\tilde\beta]/F}}$ (the power taken using the convolution product) is supported on the strongly positive element ${\mathbbm z}$, confirming our assumption that $({\mathcal J}_P,\lambda_P)$ is a cover of $({\mathcal J}_M,\lambda_M)$ (see Definition \ref{definition of cover}).

There are two approaches to study the real coefficients $b_w$ and $c_w$. 
On the one hand, we follow \cite[Sec 1]{Blasco-Blondel-SP4} to compute these coefficients directly: $c_y =[{\mathcal J}_P^+:s_y{\mathcal J}_P^- s_y^{-1}]$, 
\begin{equation}
\label{formula for b and c for y}
b_y=\sum_{u\in \mathcal S_y} {T_y(u)}, \quad\text{where \quad $\mathcal S_y:= { \frac{s_y{\mathcal J}_P^+s_y^{-1}\cap {\mathcal J}_Ps_y{\mathcal J}_P}{{\mathcal J}_P^- }}$},
\end{equation}
and similarly $c_z  =[s_z {\mathcal J}_P^-s_z^{-1}:{\mathcal J}_P^+] $,
\begin{equation}
\label{formula for b and c for z}
b_z =\sum_{u\in \mathcal S_z} { T_z(u)}{},\quad 
\text{
where \quad $ \mathcal S_z:=\frac{s_z^{-1}{\mathcal J}_P^-s_z\cap {\mathcal J}_Ps_z{\mathcal J}_P}{{\mathcal J}_P^+}$}.
\end{equation}
Under the condition $\dim\lambda_P=1$, the summands in $b_w$ are just scalars. (In general, $b_w$ is a sum of traces of operators in $\End_{\mathbb C}( \mathbb V_{\lambda_P})$.) We have implicitly chosen $s_w$, for $w\in \{y,z\}$, such that $s_w^2\in {\mathcal J}_M$, so that we may and do normalize each $T_w$, up to a sign,
 such that $
T_w(s_w)^2=\lambda_P(s_w^2)$. The choices of $s_w$ will become clear in (\ref{representative-sy-and-sz}).


On the other hand, we view $[\Lambda_W,0,\beta_W]:=[(\tilde\Lambda\oplus \tilde\Lambda)\oplus \Lambda,0,(\tilde\beta \oplus {}^\alpha\tilde\beta) \oplus \beta]$  as a semisimple self-dual stratum in $\mathcal G(F)$. There are a maximal and two minimal $((\mathfrak o _{F[\tilde \beta]}\oplus \mathfrak o _{F[\tilde \beta]}) \oplus \mathfrak o_{F[\beta]})$-lattice sequences properly contained in $\tilde\Lambda_W$, denoted by $\mathfrak m$ and $\mathfrak M^w$ with $w\in \{y,z\}$. Their precise definitions will be given in Section \ref{subsection Lattices and covers for classical groups}; in fact, they correspond to the edge $\mathfrak m$ and the vertices $\{y,z\}$ of the building $\mathcal B(\mathcal G_\beta,F)$ mentioned in Section \ref{subsection Preliminaries on the construction of covering types}. Put 
$$\overline{\mathcal G}_{\beta,w} = \mathcal G_\beta(F)_{w}/\mathcal G_\beta(F)_{w,0_+} = \mathcal{U}(\mathfrak{M}^w)_{F[\beta]} / \mathcal{U}^{0_+}(\mathfrak{M}^w)_{F[\beta]},$$ which is the group of rational points of a product of (not necessarily connected) reductive groups, over $\mathbb F_{{F[\tilde\beta]} }$ and  $\mathbb F_{{F[\beta]} }$ respectively, containing the product group (\ref{product of finite reductive groups}) as a Levi subgroup and supporting the character $\tilde\rho\times\rho$. By \cite[Section 7.1]{Stevens-supercuspidal}, there exist injections 
$$\mathcal{H}(\overline{\mathcal G}_{\beta,w}  ,\tilde{\rho}\times \rho) \cong \mathcal{H}(\mathcal{U}(\mathfrak{M}^w)_{F[\beta]},\tilde{\rho}\times \rho)\hookrightarrow  \mathcal{H}(\mathcal G,\lambda_P), \quad w\in \{y,z\},$$ 
whose images together generate the whole algebra $ \mathcal{H}(\mathcal G,\lambda_P)$. Hence, we reduce to considering Hecke algebras for reductive groups over finite fields. The structure of these algebras is well-known \cite{Lusztig-finite-classical-groups}: under suitable normalizations of $T_w$, that we denote by $\mathcal T_w$  to avoid confusion, the quadratic relation can be written as 
\begin{equation}
\label{normalized quadratic relation general b and c}
(\mathcal T_w+\mathbbm{1})*(\mathcal T_w -q_{F[\tilde\beta]}^{r_w}\mathbbm{1})=0
\end{equation}
for certain integers $r_w\geq 0$. The values of $r_w$ can be read from \cite[Table II]{Lusztig-Chevalley-groups}. 

Therefore, using (\ref{formula for b and c for y}) and (\ref{formula for b and c for z}), and comparing (\ref{quadratic relation general b and c}) and (\ref{normalized quadratic relation general b and c}), we can determine the value of $b_w$ for $w\in \{y,z\}$ up to a sign. We will analyze this sign in the next section and relate it with the reducibility of $I(s,\tilde\pi,\pi)$ as well.

\subsection{Eigenvalues of Hecke algebra  elements}
\label{subsection Eigenvalues of Hecke algebra  elements}

Continuing from the previous section in the epipelagic case, for $s\in \mathbb C$, let $D_s = \mathcal M_M(\tilde\lambda|\det|^s\times \lambda)$ and  $X_s = \mathcal {M}_{\mathcal G}(I(s,\tilde\pi,\pi))$ be respectively the $\mathcal  H(M,\lambda_M)$- and $\mathcal  H(\mathcal G,\lambda_P)$-modules under the categorial equivalences in the diagram (\ref{commutative diagram}). Note that $D_s$ is just a line over $\mathbb C$, and $X_s$ is two dimensional over $\mathbb C$ by (\ref{HG is rank 2 over HM}).

The following proposition establishes a crucial relation between the reducibility of parabolically induced representations and the eigenvalues for modules over Hecke algebras.
\begin{prop} 
\label{Blondel-Blasco-reducibility-criterion} \cite[(1.13)]{Blasco-Blondel-SP4}
The module $X_s$ is reducible if and only if the product of the eigenvalues of $T_y$ and $T_z$ on $X_s$ is equal to the scalar action of $Z$ on $D_s$.
 \hfill$\square$
\end{prop}

As in the (quasi-)epipelagic case, we continue to assume that $\dim \lambda_P=1$. On the one hand, since $\tilde\pi$ is induced from a self-dual extension $(\tilde{\boldsymbol{\mathcal J} }, \tilde{\boldsymbol\lambda})$ of $( \tilde{\mathcal  J}, \tilde{\lambda})$, the scalar action of $Z$ on $D_s$ is given by \cite[(1.15)]{Blasco-Blondel-SP4}:
$$
q_{F[\tilde\beta]}^{s} \Delta_P(\tilde{\mathbbm z})^{1/2} {\tilde{\boldsymbol{\lambda}}}(\varpi_{\tilde\beta})^{-1} T_y(s_y)T_z(s_z).$$
Here $\Delta_P$ is the modular character of $P$ arising from the normalized parabolic induction, and we have
 \begin{equation*}
 \Delta_P(\tilde{\mathbbm z})=[{\mathcal J}_P^+: \tilde{\mathbbm z} {\mathcal J}_P^+\tilde{\mathbbm z}^{-1}]=[s_y {\mathcal J}_P^+s_y^{-1}:{\mathcal J}_P^-][s_z {\mathcal J}_P^-s_z^{-1}:{\mathcal J}_P^+] = c_yc_z. 
 \end{equation*}
On the other hand, by comparing (\ref{quadratic relation general b and c}) and (\ref{normalized quadratic relation general b and c}), if we put $\epsilon_w = \sgn(b_w)\in \{\pm 1\}$  for $w\in\{y,z\}$ (which are still undetermined), then the eigenvalues of $T_w$ are 
\begin{equation*}
     \epsilon_w c_w^{1/2} q_{F[\tilde\beta]}^{r_w/2  }
     \qquad\text{and}\qquad
      -\epsilon_w  c_w^{1/2} q_{F[\tilde\beta]}^{-r_w/2  }.
\end{equation*}
Here the square roots $c_w^{1/2} $ are taken to be positive. The possible products of the eigenvalues of $T_y$ and   $T_z$ are  
\begin{equation*}
\label{product of eigenvalues of Ty and Tz}
   \epsilon_y \epsilon_z  (c_y c_z)^{1/2}  q_{F[\tilde\beta]}^{\pm (r_y+r_z)/2  }
        \qquad\text{and}\qquad
         -\epsilon_y \epsilon_z  (c_yc_z)^{1/2}  q_{F[\tilde\beta]}^{\pm (r_y-r_z)/2  }.
\end{equation*}
When $X_s$, and hence $I(s,\tilde{\pi},\pi)$, is reducible, we can use Proposition \ref{Blondel-Blasco-reducibility-criterion} to conclude that
 \begin{equation*}
\label{values of real parts}
   \Re \mathrm{Red}(\tilde{\pi},{\pi}) = \{\pm \frac{r_y+r_z}{2},\,\pm \frac{r_y-r_z}{2}\}.
\end{equation*}
This gives exactly four values of $s\in \mathbb C$ (counted with multiplicity) in the domain (\ref{complex domain}) at which $I(s,\tilde\pi,\pi)$ is reducible: 
\begin{equation}
\label{main formula of reducibility points}
q_{F[\tilde\beta]}^s = \delta\epsilon_y T_y(s_y)\epsilon_z T_z(s_z)^{-1} {\tilde{\boldsymbol{\lambda}}}(\varpi_{\tilde\beta})  q_{F[\tilde\beta]}^{\epsilon (r_y+\delta r_z)/2},\quad \epsilon, \delta \in \{\pm 1\}.
\end{equation}

\begin{prop}
\label{main prop of reducibility points}
In the situation above, if ${\tilde{\boldsymbol{\lambda}}}(\varpi_{\tilde\beta}) = \delta\epsilon_y T_y(s_y)\epsilon_z T_z(s_z)$ for some $\delta \in \{\pm 1\}$, then 
$$\mathrm{Red}(\tilde\pi,\pi) = \left\{\pm \frac{r_y+\delta r_z}{2},\pm \frac{r_y-\delta r_z}{2}+\frac{\pi \sqrt{-1}}{f_{\tilde\pi}\log q} \right\}. $$
 \hfill$\square$
\end{prop}
 We remark that this result is independent of the choice of $\varpi_{\tilde\beta}$, because  $s_y$ and $s_z$ are chosen such that $s_ys_z=i_M(\varpi_{\tilde\beta}I_{\tilde V_-},I_V)$ (see also Proposition \ref{prop of independence: uniformizer}). We will compute the signs $\epsilon_wT_w(s_w)$, for $w\in \{y,z\}$, in Section \ref{section Main calculation}.

\subsection{Examples of liftings of characters}
\label{Example: depth zero characters}

We provide two simple examples to illustrate the methodology explained in Sections \ref{subsection Structures of Hecke algebras} and \ref{subsection Eigenvalues of Hecke algebra  elements} for computing $\mathrm{Red}(\tilde\pi,\pi) $ in Proposition \ref{main prop of reducibility points}. The results will also be useful in the calculations in Section \ref{subsection Expanding the intertwining operator as a sum}. Readers unfamiliar with the underlying ideas may first 
consult the next section.

\subsubsection{Example 1: U(1)}

Consider $\tilde G = \GL_1$ and $G = \mathrm U_{1,F/\Fo}$. Take characters $(F^\times, {\tilde{\boldsymbol{\lambda}}})$ and $(\mathrm U_1(F/\Fo), {\lambda})$, and put ${\tilde{{\lambda}}} = 
{\tilde{\boldsymbol{\lambda}}}|_{\mathfrak o^\times_F}$. Put $\mathcal{U}^{0_+}(F)_1 = \mathcal{U}^{0_+}(F) \cap \mathrm U_1(F/\Fo)$ and ${\boldsymbol{\mu}}(F)_1 = {\boldsymbol{\mu}}_F \cap \mathrm U_1(F/\Fo)$.

Denote by $(1-c)$ the map $x\mapsto x \bar x^{-1}:F^\times\rightarrow \mathrm U_1(F/\Fo)$. Assume that $\tilde{\lambda}|_{\mathcal{U}^{0_+}(F)} = {\lambda}|_{\mathcal{U}^{0_+}(F)_1} \circ(1-c)$, and abbreviate $\tilde{\lambda} \leftrightarrow {\lambda}$ when the following condition holds:
$$\tilde{\lambda}|_{{\boldsymbol{\mu}}_F } = {\lambda}|_{{\boldsymbol{\mu}}(F)_1} \circ(1-c).$$
First, consider the case where ${F}/{\Fo}$ is unramified. The self-duality condition on $\tilde{\lambda}$ implies that $\tilde{\lambda}|{\boldsymbol{\mu}}_{{\Fo}}\equiv 1$. If $\zeta\in {\boldsymbol{\mu}}_{{F}}\smallsetminus{\boldsymbol{\mu}}_{{\Fo}}^\times$ with $\zeta^2\in {\boldsymbol{\mu}}_{\Fo}$, then $\tilde{\lambda}(\zeta)  = \lambda(-1) \in \{ \pm 1\}$ whenever $\tilde\lambda\leftrightarrow \lambda$. We also pinpoint a character $\boldsymbol{\tilde{\lambda}}$  by taking $\boldsymbol{\tilde{\lambda}}(\varpi)=1$, and denote by $\boldsymbol{\tilde{\lambda}}'$ the self-dual unramified twist of $\boldsymbol{\tilde{\lambda}}$, i.e., $\boldsymbol{\tilde{\lambda}}(\varpi)=-1$.

We use (\ref{formula for b and c for y}) to calculate  $c_y = q_{}^{3/2}$ and
$$b_yT_y(s_y) = \sum_{\begin{smallmatrix}
X\in \mathfrak{o}_{F}/\mathfrak{p}_{F}
\\
Y \in {\boldsymbol{\mu}}_{F}
\\
Y+{}\overline Y = -X\overline X
\end{smallmatrix}}\tilde{\lambda}(Y)\lambda(1+\overline XY^{-1} X) 
 = \begin{cases}
 \lambda(-1) (\qo^3 -1) & \text{if }\tilde{\lambda}\leftrightarrow \lambda,
 \\
 - \lambda(-1) \qo(\qo-1) & \text{if }\tilde{\lambda}\not\leftrightarrow \lambda \text{ and $\tilde{\lambda}$ is self-dual}.
 \end{cases} $$
The calculation for $z$ is easier; using (\ref{formula for b and c for z}), we have
$c_z = q_{}^{1/2} $ 
and 
$$b_z T_z(s_z) = \sum_{\begin{smallmatrix}
Y \in {\boldsymbol{\mu}}_{F}
\\
Y= -{}\overline Y \end{smallmatrix}}\tilde{\lambda}(Y)
 =  \lambda(-1) (\qo-1) \quad \text{ if $\tilde{\lambda}$ is self-dual}. $$
 As a remark, one can compute that $b_wT_w(s_w)=0$ for both $w\in \{y,z\}$ when $\tilde\lambda$ is not self-dual, but we do not need this result later.

 When $\tilde{\lambda}$ is self-dual, we have $r_z = 1/2$, and $r_y = 3/2$ or $ 1/2$ depending on whether $\tilde{\lambda}\leftrightarrow \lambda$ or not. Proposition \ref{main prop of reducibility points} implies that 
$$\mathrm{Red}(\boldsymbol{\tilde{\lambda}},\lambda)= 
   \begin{cases}     \{\pm1,\pm\frac{1}{2} +\frac{\pi \sqrt{-1}}{\log q}\}
   & \text{$\tilde{\lambda}\leftrightarrow \lambda$}, \\
   \{0,\pm \frac{1}{2}+\frac{\pi \sqrt{-1}}{\log q}\} & \tilde{\lambda}\not\leftrightarrow \lambda.
   \end{cases}
$$
In the first case $\tilde{\lambda}\leftrightarrow \lambda$, we also have $\mathrm{Red}(\boldsymbol{\tilde{\lambda}}',\lambda)= 
   \{\pm\frac{1}{2} ,\pm1+\frac{\pi \sqrt{-1}}{\log q}\}$.

The calculation for ramified quadratic ${F}/{\Fo}$ is simpler, so we simply highlight some key points. If $\tilde{\lambda}$ is self-dual, then $\tilde{\lambda}|_{{\boldsymbol{\mu}}_{\Fo}}$ has order at most 2. Assuming $\varpi^2 = \varpi_{\bullet}$, we have $\boldsymbol{\tilde{\lambda}}(-\varpi_{}^2) = 1$. We compute that 
$$b_yT_y(s_y)= \tilde\lambda(-2)\lambda(-1)(q-1)
\quad\text{and}\quad 
b_z T_z(s_z) = \begin{cases}
q-1 & \tilde{\lambda}|_{{\boldsymbol{\mu}}_{\Fo}} \equiv 1,
\\
0 & \tilde{\lambda}|_{{\boldsymbol{\mu}}_{\Fo}}\text{ is quadratic.}
\end{cases}$$  
Since $c_y = c_z=q$, we obtain 
$$r_y = 1,\quad r_z=\begin{cases}
1& \tilde{\lambda}|_{{\boldsymbol{\mu}}_{\Fo}} \equiv 1,
\\
0& \tilde{\lambda}|_{{\boldsymbol{\mu}}_{\Fo}}\text{ is quadratic,}
\end{cases} 
$$
and hence
$$
\mathrm{Red}(\boldsymbol{\tilde{\lambda}},\lambda)
 = 
 \begin{cases}
\{\pm 1,\frac{\pi \sqrt{-1}}{\log q}\} & \tilde{\lambda}|_{{\boldsymbol{\mu}}_{\Fo}} \equiv 1\text{ and ${\tilde{\boldsymbol{\lambda}}}(\varpi_{})=\lambda(-1)$},
\\
\{0,\pm 1+\frac{\pi \sqrt{-1}}{\log q}\} & \tilde{\lambda}|_{{\boldsymbol{\mu}}_{\Fo}} \equiv 1\text{ and ${\tilde{\boldsymbol{\lambda}}}(\varpi_{})=-\lambda(-1)$},
\\
\{\pm \frac{1}{2},\pm \frac{1}{2}+\frac{\pi \sqrt{-1}}{\log q}\}
 & \tilde{\lambda}|_{{\boldsymbol{\mu}}_{\Fo}}\text{ is quadratic.}
\end{cases}$$
The result agrees with the calculations in \cite[(3.5) and Cor 3.6]{BT-ramified}.

\subsubsection{Example 2: ramified SO(2)}
\label{subsubsection Example 2: ramified SO(2)}

Fix a uniformizer $\varpi$ of $F$ and let $G = \SO_{2,F[\sqrt\varpi]/F}$, the special orthogonal group of absolute rank 1 that splits over $F[\sqrt\varpi]$ but not over $F$. Then $G(F)$ is just isomorphic to $\mathrm U_1(F[\sqrt\varpi]/F) = F[\sqrt\varpi]^\times_1:=\ker N_{F[\sqrt\varpi]/F}$. Take a depth zero character $(G(F),\lambda)$. Viewing $G$ as an $F$-group, the theory of endoscopy asserts that there is an irreducible representation $\Pi$ of $\GL_2(F)$ lifted from $\lambda$. It turns out that $\Pi$ is non-supercuspidal, and we determine its cuspidal support, which is necessarily a set of two characters of $F^\times$.

The group $G$ is the connected isometry group defined by the matrix $H = \diag(-\varpi,1)$ and consists of elements represented by $$a+b\sqrt{\varpi} = \begin{bmatrix}
a&b \\ b\varpi & a
\end{bmatrix}, \quad a,b\in F\text{ and }a^2-b^2\varpi = 1.$$
 The element $\mathbbm p : = \diag(1,-1)\in G^\sharp(F)\smallsetminus G(F)$ conjugates $a+b\sqrt{\varpi}\mapsto a-b\sqrt{\varpi}$. We take a self-dual lattice sequence $\Lambda$ in $V$ determined by 
\begin{equation*}
     \Lambda(0) = \mathfrak{o}_F\oplus  \mathfrak{o}_F, \quad \Lambda(1) = \mathfrak{o}_F\oplus  \mathfrak{p}_F.
\end{equation*}
As we will test reducibility using characters of $\tilde G = \GL_1$, we take $\tilde V_\pm $ with $\dim \tilde V_{-} = \dim \tilde V_{+}=1$. The lattice sequence
$$\mathfrak m (k) = \Lambda_0(\tfrac{k-1}{3})\oplus\Lambda(\tfrac{k}{3})\oplus\Lambda_0(\tfrac{k+1}{3}) $$ 
renders the compact groups in $\mathcal G(F) = \SO_{4,F[\sqrt\varpi]/F}(F)$:
\begin{equation*}
{\mathcal J}_P = {\begin{bmatrix}
\mathfrak o^\times&\mathfrak p&\mathfrak o&\mathfrak o
\\
\mathfrak o&\mathfrak o^\times&\mathfrak o&\mathfrak o
\\
\mathfrak p&\mathfrak p&\mathfrak o^\times&\mathfrak o
\\
\mathfrak p&\mathfrak p&\mathfrak p&\mathfrak o^\times
\end{bmatrix}}\cap \mathcal G(F) 
\quad \text{and}\quad 
{\mathcal J}_{P,0_+} =I+\begin{bmatrix}
\mathfrak p&\mathfrak p&\mathfrak o&\mathfrak o
\\
\mathfrak o&\mathfrak p&\mathfrak o&\mathfrak o
\\
\mathfrak p&\mathfrak p&\mathfrak p&\mathfrak o
\\
\mathfrak p&\mathfrak p&\mathfrak p&\mathfrak p
\end{bmatrix}\cap \mathcal G(F) .
\end{equation*}
Let $(\tilde G(F), \tilde{\boldsymbol \lambda})$ be a depth zero character whose values on ${\boldsymbol{\mu}}_F\times \left<\varpi\right>$ are yet to be determined.

First, take $w=y$. For $(X,Y)^-\in \mathcal S_y$, representing $X $ by $[0,b]$ with $b\in \mathfrak o\bmod \mathfrak p$ and $Y\in \mathfrak o^\times \bmod \mathfrak p$, we have $2Y = -b^2$. One computes that $I - {}^\alpha X Y^{-1}X =\diag(1,-1) = \mathbbm p$. Hence
$$b_y T_y(s_y) = \sum_{-b^2=2Y}\tilde\lambda (Y) \lambda( \mathbbm p^2) = \tilde\lambda(-2) (q-1).$$
Now, take $w=z$. For $(X,Y\varpi^{-1})^+\in \mathcal S_z$, with representatives $X = [a,0]$ where $a\in \mathfrak o\bmod \mathfrak p$ and $Y\in \mathfrak o^\times \bmod \mathfrak p$, we have $2Y\varpi^{-1} = a^2\varpi^{-1}$ and $I - {}^\alpha X Y^{-1}X = \diag(-1,1) =- \mathbbm p$. Hence 
$$b_z T_z(s_z) = \tilde\lambda(-1)\sum_{a^2=2Y}\tilde\lambda (Y) \lambda(- \mathbbm p^2) = \tilde\lambda(-2)\lambda(-1) (q-1).$$
We therefore have two possibilities for the character $\tilde\lambda$:
$$\tilde \lambda_1 |_{{\boldsymbol{\mu}}_F}\equiv \mathbf 1_{{\boldsymbol{\mu}}_F}, \tilde{\boldsymbol \lambda}_1(\varpi) = \lambda(-1)
\quad\text{and}\quad
\tilde \lambda_2  |_{{\boldsymbol{\mu}}_F}\equiv \left(\frac{\cdot}{{\boldsymbol{\mu}}_F}\right), \tilde{\boldsymbol \lambda}_2(\varpi) = \lambda(-1).
$$
Both $\mathrm{Red}\{\tilde {\boldsymbol \lambda}_1, \lambda\}$ and $\mathrm{Red}\{\tilde {\boldsymbol \lambda}_2, \lambda\}$ contain 1. The representation $\iota^{\GL_2(F)}_{\GL_1(F)\times \GL_1(F)}(\tilde{\boldsymbol \lambda}_1\times \tilde{\boldsymbol \lambda}_2)$ is clearly irreducible and is the endoscopic lifting of $\lambda$.

\section{Main calculations}
\label{section Main calculation}

We dive into the calculations of the signs $\epsilon_w T_w(s_w)$, with $w\in \{y,z\}$, appearing in Proposition \ref{main prop of reducibility points} and eventually deduce the relation between $\tilde\pi$ and $\pi$ such that $I(s,\tilde\pi,\pi)$ is reducible at $s=1$. The technical calculations are provided in Section \ref{subsection Expanding the intertwining operator as a sum}, following the preparatory material in Sections \ref{subsection Cohomological classification}
 - \ref{subsection Lattices and covers for classical groups}. The main results are given in Propositions \ref{first general form of lifting, unitary group} and \ref{first general form of lifting}. The required properties of normalized quadratic Gauss sums are provided in Section \ref{subsection Non-degeneracy of a quadratic form}.

\subsection{Cohomological classification}
\label{subsection Cohomological classification}

We provide a classification of classical groups by their pure inner forms. Let $G$ be a connected classical group over $\Fo$ determined by a Hermitian form $h$, and view the first Galois cohomology $\mathrm{H}^1(\Fo,{G})$ as a pointed set (where the distinguished point gives a quasi-split form of $G$). Let $\hat G$ be the Langlands dual group of $G$, with center $Z\hat G$. We use the Kottwitz isomorphism $$\mathrm{H}^1(\Fo,{G}) \cong \pi_0(Z\hat{{G}})^{\Gal(F^{\text{sep}}/\Fo)}$$
to enumerate the set $\mathrm{H}^1(\Fo,G)$ explicitly.

\begin{itemize}

\item If $G$ is a unitary group, then $\mathrm{H}^1(\Fo,G) \cong \{\pm 1\} $ corresponds bijectively to the two isomorphism classes of Hermitian spaces of the same dimension. These two classes are distinguished by the discriminant map:
\begin{equation}
\label{discriminant of Hermitian forms}
\mathrm{disc}:\{\text{non-degenerate }F/\Fo\text{-Hermitian forms on }V\}\rightarrow \Fo^\times /N_{F/\Fo}(F^\times)\cong \{\pm 1\}.
\end{equation}

\item For $G=\mathrm{SO}_{N}$, we view $G$ as defined by the quadratic space $(V,q)$ with $\dim V = N$ and $q(x) = h(x,x)$. Then $\mathrm{H}^1(F,G) $ is a singleton when $N\leq 2$; otherwise, $\mathrm{H}^1(F,G) \cong \{\pm 1\} $ corresponds bijectively to the isomorphism classes of quadratic spaces of fixed dimension and discriminant (defined analogously to (\ref{discriminant of Hermitian forms}) with image in $F^\times / F^\times {}^2$). The two classes of quadratic spaces can be distinguished by the {Hasse-Witt invariant} \cite[IV, Sec 2.1]{Serre-A-course-in-arithmetic} defined as follows. If $\dim V=1$, we define $e(q)=1$; if $\dim V\geq 2$, by choosing an ordered orthogonal basis $\{v_i\}$ of $V$ (whose choice is irrelevant), we define
$$e(q) = \prod_{i<j} (q(v_i), q(v_j))$$
where $(\cdot,\cdot)$ is the Hilbert symbol. Its value lies in $\{\pm 1\}$ for non-Archimedean $F$.

\item For $G=\mathrm{Sp}$, the set $\mathrm{H}^1(F,G)  $ is trivial, so there is no non-trivial pure inner form of $G$.

\end{itemize}
From now on, we choose a quasi-split form of $G$ to be parametrized by $+1\in \{\pm 1\}$, and denote this form by $G_+$; the other form is hence denoted by $G_-$. For fixed $(\epsilon,F/\Fo) $ and $N = \dim_FV$, as well as a discriminant $d\in F^\times/F^{\times 2}$ when $\epsilon=1$ and $F=\Fo$, we temporarily define a map
$$\mathrm{H}^1:\{\text{Hermitian forms with fixed }(\epsilon, F/\Fo, N, d)\}\rightarrow \{\pm 1\}$$
to unify the classification above. Hence $\mathrm{H}^1$ has image $\{1\}$ in the symplectic case, $ \{\pm 1\}$ via the discriminant in the unitary case, and $\{\pm 1\}$ via the Hasse-Witt invariant in the orthogonal case.

For computational convenience, we choose the following Hermitian matrices (under some choices of bases of $V$ that are eventually irrelevant) 
to define the forms representing their isometry classes. 
\begin{itemize}
\item If $G$ is symplectic, we simply take
$$H_+ = H_- = \antidiag(1,-1,1,-1,\dots).$$

\item If $G = \mathrm U_N$ is unramified unitary, let $\zeta $ be a primitive $2(q-1)$th root of unity in $\boldsymbol{\mu}_F$, and $\varpi$ be a uniformizer in $\Fo$. The two classes of forms can be represented by
\begin{equation*}
   \begin{split}
     H_+ =  \zeta^{(\epsilon-1)/2}\antidiag(1,-1,1,-1,\dots), \quad H_- = \diag(I_{\lfloor N/2\rfloor},\varpi,I_{\lfloor (N-1)/2\rfloor})H_+ .\end{split}
  \end{equation*}
We see that $\mathrm{H}^1(H_+) = 1$ and $\mathrm{H}^1( H_-)=-1$, and $H_+$ is quasi-split over $\Fo$.

\item If $G = \mathrm U_{n}$ is odd ramified unitary, let $\zeta $ be a primitive $(q-1)$th root of unity in $\boldsymbol{\mu}_F$, and $\varpi$ be a uniformizer in $F$ such that $ \varpi_\bullet = -\varpi^2 $ is a uniformizer in $\Fo$. The two classes of forms can be represented by
\begin{equation*}
   \begin{split}
     H_+ = \varpi^{(\epsilon-1)/2}\antidiag(1,-1,1,-1,\dots,1), \quad H_- = \diag(I_{(n-1)/2},\zeta,I_{(n-1)/2})H_+ .     \end{split}
  \end{equation*}
We see that $\mathrm{H}^1(H_+) = 1$ and $\mathrm{H}^1( H_-)=-1$, and $H_+$ is quasi-split over $\Fo$.

\item Let $G$ be ramified $\SO_{2n}$ with discriminant $-\varpi'$. Here we allow the value of $\varpi'$ to vary, which will be fixed in different cases in Section \ref{section Examples for simple supercuspidals}. The two classes of forms can be represented by
\begin{equation*}
   \begin{split}
       H_+ & = \antidiag(1,\dots,1,\diag((-1)^{n}\varpi',1),1,\dots,1),       \\
         H_- & = \antidiag(1,\dots,1,\zeta^{}\diag((-1)^{n}\varpi',1),1,\dots,1).   \end{split}
  \end{equation*}
It is easy to show that $\mathrm{H}^1(H_+) =  - \mathrm{H}^1( H_-) $, and both forms are quasi-split over $F$.
\end{itemize}

\subsection{Embeddings}
\label{section Embeddings of lattices}

Fix a connected classical group $G = G(V,h)$. Let $\mathbf s$ be a skew semisimple stratum in $V$ as in Definition \ref{definition of Epipelagic stratum}, and $F[\beta] = \oplus_{i\in I}F[\beta_i]$ be a maximal elliptic Cartan subspace in $\tilde {\mathfrak g}(V)(F)$. Each summand $E_i:=F[\beta_i]$ is $\alpha$-invariant, and the restriction of $-\alpha$ on $E_i$ defines a Galois involution simply denoted by $c$, whose fixed field is denoted by $E_{i\bullet}$. Put $(E_i^\times)_1 = \mathrm U_1(E_i/E_{i\bullet})$ and $E_1^\times = \prod_{i\in I}(E_i^\times)_1$.

To describe the $\Fo$-embeddings of $E_1^\times$ into $G(\Fo)$ up to $G(\Fo)$-conjugacy, we first impose a Hermitian structure on $E_i$ (as an $F$-space) for each $i\in I$: 
\begin{equation*}
\label{hermitian forms on E_i quadratic}
 h_{E_i}(x_i,y_i)
      =  
      \begin{cases}
   {}^cx_iy_i
& \text{ when }\epsilon_G = 
   1,
\\
(\beta_i - {}^c\beta_i){}^cx_iy_i   
& \text{ when }\epsilon_G = -1,
\end{cases}
\quad  \text{for all $x_i,y_i\in E_i$,}
\end{equation*}
We then put $h_{F[\beta]} = \oplus_{i\in I}h_{E_i}$ on $F[\beta]$. Each $h_{E_i}$ can be viewed as an $F$-form by taking $\mathrm{tr}_{E_i/F}\circ h_{E_i}$. 

In our setting for epipelagic representations, we have $\dim _FV = [F[\beta]:F] = \sum_{i\in I}[E_i:F]$, and each $E_i/E_{i\bullet}$ is ramified if $F/\Fo$ is either trivial or ramified. Consider the different isometry classes of Hermitian forms as follows: for $\delta\in \{1,\zeta\}$, let $h_i^\delta$ be the Hermitian form defined by the matrix representative $H_+$ when $\delta = 1$ and by $H_-$ when $\delta = \zeta$, with $H_\pm$ defined in Section \ref{subsection Cohomological classification}. Given a partition of the index set 
$$I = I_1\sqcup I_\zeta$$
(we often just say that $I_\zeta$ is a partition of $I$, as $I_1 = I\smallsetminus I_\zeta$ is hence determined), denote by $(V,h_{I_\zeta})$ the resulting Hermitian space defined by the orthogonal direct sum:
$$ h_{I_\zeta} = \bigoplus^{\perp}_{i\in I_1}h_i^1 \oplus^\perp \bigoplus^{\perp}_{i\in I_\zeta}h_i^\zeta.$$
By the classification using $\mathrm{H}^1$ in Section \ref{subsection Cohomological classification}, 
$(F[\beta],h_{F[\beta]})$ is $F/\Fo$-isometric to any $(V,h_{I_\zeta})$ with partition $I_\zeta$ such that $\mathrm{H}^1(h_{I_\zeta})=\mathrm{H}^1(h_{F[\beta]})$. An isometry map $(F[\beta],h_{F[\beta]})\rightarrow (V,h_{I_\zeta})$ then induces an embedding of $F[\beta]^\times$ into $\tilde G(V)(F)$, and by restriction an embedding 
$$m_{I_\zeta}: E_1^\times\hookrightarrow  G(\Fo)=G(V,h_{I_\zeta})(\Fo).$$
It is routine to show that the $G(\Fo)$-isometry class of $m_{I_\zeta}$ is independent of the choice of the isometry map, and so depends only on the partition $ I_\zeta$ of $I$. Given two partitions $I_\zeta$ and $ I'_\zeta$ of $I$, the embeddings $m_{I_\zeta}$ and $m_{I'_\zeta}$ are $G(\Fo)$-conjugate if and only if ${I_\zeta}={I'_\zeta}$. (Similar results are shown in \cite[Rem 3,26(ii) and Def 9.15]{KSS-endo-parameters}, using the language of concordance.)

The classification of $G$ using $\mathrm{H}^1(\Fo,G)$ in Section \ref{subsection Cohomological classification} implies that, since these invariants take values in $\{\pm 1\}$, switching an index from one of the sets $I_1$ and $I_\zeta$ to another changes the pure inner form from one to another, whereas switching two indices yields a form equivalent to the original one. Therefore, given a fixed pure inner form of $G$, we can parametrize the $\Fo$-embeddings $E^\times_1\rightarrow G(\Fo)$ by partitions of $I$ such that one of $\#I_1$ and $ \#I_\zeta$ has a fixed parity. For instance, we parametrize the embeddings into $G=G(V,h)$ with $\mathrm{H}^1(h)=1$ by the partitions where $\#I_\zeta$ is even.

The quotient $\mathsf A_{x,\beta} = G(F)_{x,\beta} / G(F)_{x,r}$ was defined at the end of Section \ref{subsection Epipelagic inducing types for classical groups}. If $s_{I_\zeta}\in G$ conjugates $m_{I_\zeta}$ into $m_{I_{\O}}$, then $s_{I_\zeta}$ induces an isomorphism 
$$\mathsf A_{x,m_{I_\zeta}(\beta)}\rightarrow \mathsf A_{x,\beta} = \mathsf A_{x,m_{I_{\O}}(\beta)},$$
and therefore a permutation of the set $I$, which is still denoted by $s_{I_\zeta}$ below.

\subsection{Lattices and covers for classical groups}
\label{subsection Lattices and covers for classical groups}

In this section, we examine how the choices of embeddings in Section \ref{section Embeddings of lattices} affects the relevant lattice sequences. 

Let $V=F[\beta]$ equipped with the form $h_{I_\zeta}$ defined in Section \ref{section Embeddings of lattices} for a fixed choice of partition $I_\zeta$ of $I$. For each $j\in I$ and $\delta\in \{1,\zeta\}$, let $H_j = H^\delta_j\in \tilde{\mathfrak g}(V_j)(F)$ be the Hermitian matrix (i.e., ${}^{t}\overline {H}_j= \epsilon_G H_j$) corresponding to the form $h_j^\delta$ on $E_j$. Then the form $h_{I_\zeta}$ can be presented by
\begin{equation*}
H =H_{I_\zeta} =  \diag(\underbrace{\dots,  H^1_j,\dots}_{j\in I_1},\underbrace{\dots, H^\zeta_j,\dots}_{j\in I_\zeta}).
\end{equation*}
Note that if $X\in \tilde{\mathfrak g}(V_j)(F)$, then $X\mapsto H_j^{-1}{}^{t} \overline X H_j$ defines the Galois-involution on $E_j$.

We define the dual index set
$$
\hat I := 
I\smallsetminus \{o\}\text{ when $G$ is orthogonal,\quad  and }\quad 
I\sqcup \{o\}\text{ when $G$ is symplectic}.
$$
Fix an index $i\in \hat I$ and put $\tilde V_- $ to be $ E_i$ if $i\in I$ and $ F$ if $i=o$. Take $\tilde H = H_i$ and, with $H$ defined above, form the matrix $H_W$ as in (\ref{The big matrix JW}), thereby defining a Hermitian form $h_W$ on $W =  (\tilde V_- \oplus \tilde V_+)\perp V$. The actions of the operator $\alpha$ on the entries $(X,Y)$ in (\ref{alpha on X and Y}) become
 \begin{equation*}
     {}^\alpha X  =
-\sum_{j\in I}  H_j^{-1}{}^{t} \overline X_j H_i, \quad 
{}^\alpha Y = -\epsilon {}^{t}\overline{\tilde H}^{-1}  {}^{t}\overline  Y \tilde H_i =   -  H_i^{-1}{}^{t} \overline Y H_i .
\end{equation*}
Given an embedding $m_{I_\zeta}: E_1^\times \hookrightarrow G(\Fo)$ defined in Section \ref{section Embeddings of lattices} corresponding to the partition $I_\zeta$, we define an embedding
$$\mathbf m_{I_\zeta}: E_1^\times \hookrightarrow \mathcal G(\Fo) = G(W,h_W)(\Fo), \quad g = (g_j )_{j\in I}\mapsto (g_i,m_{I_\zeta}(g),g_i).
$$
We now construct self-dual lattice sequences in $W$. Let $\Lambda$ and $\Lambda_i$ be the self-dual lattice sequences appearing in the strata $\mathbf s$ and its component $\mathbf s_i$ respectively. Put $V_o=0$ and denote by $W_i$  the space $(\tilde V_- \oplus \tilde V_+)\perp V_i$, which is isomorphic, as an $F$-space, to $E_i^{\oplus 3}$ if $i\in I$ and $E_i^{\oplus 2}$ if $i=o$. Following \cite[Sec 6.2]{Stevens-supercuspidal}, we define two minimal self-dual $\mathfrak o_{E_i}$-lattice sequences $\mathfrak M_i^w{}$, with $w\in \{y,z\}$, such that $q^y_+=q^y_-=0$ and $q^z_+=-q^z_-=-1/2$ are the unique numbers in the interval $[0,1)$ giving
$$\mathfrak M_i^w(r)\cap \tilde V_\delta \supsetneq \mathfrak M_i^w(r_+)\cap \tilde V_\delta\quad\Leftrightarrow \quad r = q^w_\delta, \quad \text{ for }\delta\in \{+,-\}\text{ and }w\in \{y,z\},$$
as well as $\mathfrak M_i^w(0)\cap V_i\supsetneq \mathfrak M_i^w(0_+)\cap  V_i$ (a void condition if $i=o$). They are contained in the maximal self-dual $\mathfrak o_{E_i}$-lattice sequence $\mathfrak m_i$ such that 
$q^\mathfrak m_+=-q^\mathfrak m_-=-1/3$ are the unique numbers in the interval $[0,1)$ giving
$$\mathfrak m_i(r)\cap \tilde V_\delta \supsetneq \mathfrak m_i(r_+)\cap \tilde V_\delta\quad\Leftrightarrow \quad r = q^\mathfrak m_\delta, \quad \text{ for }\delta\in \{+,-\},$$
and also $\mathfrak m_i(0)\cap V_i\supsetneq \mathfrak m_i(0_+)\cap  V_i$. Finally, we define  
$$\mathfrak L = \mathfrak L_i \oplus \bigoplus_{j\in I\smallsetminus \{i\}} \Lambda_j$$
for $\mathfrak L \in \{\mathfrak M^y,\mathfrak M^z,\mathfrak m\}$. The lattice sequences $\mathfrak M_i^w{}$, for $w\in \{y,z\}$, are the ones mentioned in Section \ref{subsection Structures of Hecke algebras}, with corresponding facets $\mathcal F_w$ in the Bruhat-Tits building of $\mathcal G $ mentioned in Section \ref{subsection Preliminaries on the construction of covering types}.

For $j,k\in I\cup\{+,-\}$ and $r\in \mathbb R$, we define
$$\mathfrak P^r_{(j,k)}=\{X\in \Hom_F(V_k,V_j): X\Lambda_k(s)\subset \Lambda_j(s+r)\text{ for all }s\in \mathbb R\},$$
so that $\mathfrak P^r_{(j,j)} = \mathfrak P^r(\Lambda_{j})$. In the (quasi-)epipelagic case, since $E_j$ is a maximal subfield in $\tilde{\mathfrak g}(V_j)(F)$, we simply have $\mathfrak P^r_{(j,j)}\cap E_j = \mathfrak p^r({E_j}) $.

The compact subgroup $ {\mathcal J}_P$ of the covering type is defined byx 
   $$ {\mathcal J}_P :=  (\mathfrak P^0(\mathfrak m) \cap Z_{\mathcal G(\Fo)}(\text{image }\mathbf m_{I_\zeta}))(\mathfrak P^{0_+}(\mathfrak m) \cap \mathcal G(\Fo)).$$  
To present it as a matrix group, we recall the form from \cite[1.3. Prop 1]{Blondel-cover-propag} with block size $t=3$ and obtain a presentation of the lattice \begin{equation*}
 {\mathfrak J}_P 
   = \begin{bmatrix}
    \mathfrak o_{E_i}+ \mathfrak{P}^{0_+}(\Lambda_-) & \mathfrak o_{E_i}+\mathfrak P^{0_+}_{(-,I)}
     &\mathfrak P^0_{(-,+)}
     \\
\mathfrak P^{0_+}_{(I,-)} &  \mathfrak o_{E}  +\mathfrak{P}^{0_+}(\Lambda) &  \mathfrak o_{E_i}+\mathfrak P^{0_+}_{(I,+)}
      \\
  \mathfrak  p_{E_i}+ \mathfrak P^{0_{++}}_{(+,-)} & \mathfrak P^{0_+}_{(+,I)}  &    \mathfrak o_{E_i} + \mathfrak{P}^{0_+}(\Lambda_+)
    \end{bmatrix}{} \subset \mathfrak P^0(\mathfrak m),
     \end{equation*}
where at the $(+,-)$-corner, $\mathfrak P^{0_{++}}_{(+,-)} = \mathfrak P^{2/e_i}_{(+,-)}$ if $\mathfrak P^{0_{+}}_{(+,-)} = \mathfrak P^{1/e_i}_{(+,-)}$. We then put $\mathcal J_P = \mathfrak J_P \cap \mathcal G(\Fo)  $.

We now define the elements $s_y$ and $s_z$ that appeared in Section \ref{subsection Structures of Hecke algebras}. Let $I_{(+,-)}$ be the matrix that maps the basis of $\tilde V_-$ defining the Hermitian matrix $H_i$ of $h_i$ into its dual basis in $\tilde V_+$; it is represented by the identity matrix $I_{\tilde V}$ with the above bases. Let $I_{(-,+)}$ be defined similarly. Fix a uniformizer $\varpi_i\in E_i$ and define
\begin{equation}
\label{representative-sy-and-sz}
s_y  = \text{anti-diag}(I_{(-,+)},{\mathbbm p},I_{(+,-)}),
\quad s_z =\text{anti-diag}(\varpi_{i}^{-1}I_{(-,+)},  {\mathbbm p},  - \varpi_{i}^{}I_{(+,-)}),
\end{equation}
where $\mathbbm p \in G^\sharp(\Fo) $ is just the identity unless $G$ is orthogonal, in which case $\mathbbm p$ is a specifically chosen element in $ G^\sharp(\Fo)$ such that $\mathbbm p^2=1$ and $\det s_y = \det s_z = 1$. The elements in (\ref{representative-sy-and-sz}) satisfy $s_y^2=1$ and  $s_z^2=\diag(-I_{\tilde V_-},I_{V},-I_{\tilde V_+})$.

\subsection{Expanding the intertwining operator as a sum}
\label{subsection Expanding the intertwining operator as a sum}

Let $\pi = \cInd_{\mathcal J}^{G(F)}\lambda$ and $\tilde\pi = \cInd_{\tilde{\boldsymbol{\mathcal{J}}}}^{\tilde G(F)}\tilde{\boldsymbol{\lambda}}$ be (quasi-)epipelagic representations whose underlying cuspidal types $\lambda$ and $\tilde{{\lambda}}$ are extended from characters with underlying (quasi-)epipelagic strata $[\Lambda,0,\beta]$ and $[\tilde\Lambda,0,\tilde \beta]$ respectively. In this section, we determine the relations between the extensions ${\lambda}$ and $\tilde{\boldsymbol{\lambda}}$ such that $I(s,\tilde\pi,\pi)$ is reducible at $s=1$. The results will be given in Propositions \ref{first general form of lifting, unitary group} and \ref{first general form of lifting}, which will be used to determine the endoscopic lift of $\pi$ in Section \ref{subsection Reducibility results for different classical groups}.

As a necessary condition for $I(s,\tilde\pi,\pi)$ to be reducible at $s=1$, we take $\tilde\beta = 2\beta_i$ for a fixed $i\in \hat I $, where $\beta_o = 0$. For $w\in \{y,z\}$, given $s_w$ in (\ref{representative-sy-and-sz}), we put $T_w(s_w) = \tilde T_w(s_w)\times \lambda(\mathbbm p)$  for some intertwining operator $\tilde T_w(s_w)\in \End(\mathbb V_{\tilde\lambda})$ and $\lambda(\mathbbm p)\in\End(\mathbb V_{\lambda}) $. Again, in the (quasi-)epipelagic case, these operators are just scalars. We normalize $T_w$ such that $\tilde T_y(s_y)^2 = 1$ and $\tilde T_z(s_z)^2 = \tilde\lambda(-1)$. We compute the value of $b_w T_w(s_w)$ for $w\in \{y,z\}$ from (\ref{formula for b and c for y}) and (\ref{formula for b and c for z}) in the subsections below.

\subsubsection{Computing $b_y$}

We first compute $b_y$. By applying the calculations in \cite[Sec 4.1]{BT-ramified}, we expand (\ref{formula for b and c for y}) as
\begin{equation}
\label{main calculation: by begins}
b_y\tilde T_y(s_y)=\sum_{(X,Y)\in \mathcal S_y}\tilde{\lambda}(Y)\lambda (I-{}^\alpha X Y^{-1}X),
\end{equation}
where $\mathcal S_y = 
      ( s_y{\mathcal J}_P^+s_y^{-1}\cap {\mathcal J}_Ps_y{\mathcal J}_P)/{\mathcal J}_P^-$. The quotient $s_y{\mathcal J}_P^+s_y^{-1}/{\mathcal J}_P^-$ consists of elements of the form 
\begin{equation*}
   \begin{split}
     (X,Y)&\in  \frac{\mathfrak o_{E_i}+\mathfrak P^{0_+}_{(-,I)}}{\mathfrak P^{0_+}_{(-,I)}}
\oplus 
\frac{\mathfrak P^{0}_{(-,+)}}{\mathfrak  p_{E_i}+ \mathfrak P^{0_{++}}_{(-,+)}}
   \end{split}
\end{equation*}
satisfying relation (\ref{X-alpha-X-equals-Y-minus-alpha-Y}). Writing $X = (X^j)_{j\in I}$, we have $X^j=0$ for all $j\neq i$ and $X^i\in \mathfrak o_{E_i}/\mathfrak p_{E_i}$. The condition $(X,Y)\in {\mathcal J}_Ps_y{\mathcal J}_P$ forces $Y\in \tilde{\mathcal J} = \mathfrak o_{E_i}^\times+ \mathfrak P^{0_{+}}({\Lambda_i})$, so that we can write $Y = Y_0(1+Y_1)$ with $Y_0 \in \mathfrak o_{E_i}^\times/ \mathfrak p_{E_i}$ and $Y_1\in  \mathfrak P^{0_+}({\Lambda_i})/(\mathfrak p_{E_i}+\mathfrak P^{0_{++}}({\Lambda_i}))$. 
Hence (\ref{X-alpha-X-equals-Y-minus-alpha-Y}) implies that 
\begin{subequations}
\begin{align}
     & Y_0+  {}^c Y_0 = - X^i {}^c X^i,
     \label{relation of X and Y at level 0 and w=y}
     \\
     & Y_0Y_1- {}^\alpha Y_1{}^cY_0=0\quad \text{i.e., }{}^\alpha Y_1 =  Y_0Y_1{}^cY_0^{-1}.
     \label{relation of X and Y at level 1 and w=y}
\end{align}
\end{subequations}
We first obtain
\begin{equation}
\label{main calculation: tilde-lambda-Y in by}
 \tilde{\lambda}(Y) = \tilde\lambda(Y_0)\psi_{\tilde\beta}(Y_1).
\end{equation}
We then compute $\lambda(I-{}^\alpha X Y^{-1}X)$. We represent $\beta = \oplus_{i\in I}\beta_i$ by diagonal blocks and look at the corresponding blocks of $I-{}^\alpha X Y^{-1}X$, which are $I_j$ for all $j\neq i$, while the $i$-th diagonal block is 
\begin{equation}
\label{value W_0(1+W_1)}
   I_i-{}^\alpha X^i Y^{-1}X^i =  I_i - {}^\alpha X^i(I+Y_1)^{-1}Y_0^{-1}X^i \equiv I_i - {}^\alpha X^i(I-Y_1)Y_0^{-1}X^i \mod \mathcal{U}^{0_{++}}(\Lambda_i).
\end{equation}
From here we branch into two cases, depending on the type of $G$.

If $G$ is orthogonal, symplectic, or ramified unitary, we follow arguments similar to those in \cite[Lem 4.2]{BT-ramified}. In this case, (\ref{relation of X and Y at level 0 and w=y}) becomes $2Y_0 = -(X^i)^2$, so that $X^i\in \mathfrak o_{E_i}^\times$. Consequently,
$I_i-{}^\alpha X^i Y_0^{-1}X^i =-I_i$, and we obtain, for $i\neq o$,
\begin{equation}
\label{main calculation: lambda-p-W in by, not unramified unitary}
\lambda(I-{}^\alpha X Y^{-1}X) =\lambda(\omega_i)\psi_{\beta_i}(2Y_1)^{-1},
\end{equation}
where $\omega_i=\diag((I_j)_{j\neq i},-I_i)$ as in (\ref{definition of omega_i element}). Since $\tilde\beta = 2\beta_i$, putting (\ref{main calculation: tilde-lambda-Y in by}) and (\ref{main calculation: lambda-p-W in by, not unramified unitary}) into (\ref{main calculation: by begins}) yields
\begin{equation*}
 c_y = q_i^2, \quad b_y T_y(s_y) = \tilde\lambda(-2)\lambda(\omega_i)(q_i-1) = \tilde\lambda(-2)\lambda(\omega_i) (q_i^{1/2}-q_i^{-1/2})(c_y/q_i)^{1/2}. 
\end{equation*}
Here $q_i = q_{F[\beta_i]}$ is the cardinality of the residue field $\mathbb E_i$ of $E_i$. The above implies that
\begin{equation}
\label{results of ry and eyTy}
 r_y =1\quad\text{and} \quad \epsilon_y T_y(s_y)= \tilde\lambda(-2)\lambda(\omega_i ).
\end{equation}
This result holds regardless of whether $\tilde\lambda$ is trivial or quadratic.

If $G$ is unramified unitary, then $I$ is just a singleton. We express the value in (\ref{value W_0(1+W_1)}) as $W_0(I+W_1)$, where $W_0\in {\boldsymbol{\mu}}_F$ and $I+W_1\in \mathcal{U}^{0_+}(\Lambda)$; then we can take 
\begin{equation*}
W_0 = I - {}^\alpha XY_0^{-1}X   = 
   \begin{cases}
     1& \text{if $X=0$,}\\
        -Y_0^{-1}{}^cY_0& \text{if $X\neq 0$},
   \end{cases}
   \end{equation*}
and $ W_1 =  W_0^{-1}{}^\alpha XY_1Y_0^{-1}X$. We continue by separating into cases $X\neq 0$ and $X=0$. In the former case, we have $
 W_1 =  -(1+ {}^cY_0^{-1}Y_0)Y_1$. Using (\ref{relation of X and Y at level 1 and w=y}) and ${}^\alpha \beta = \beta$, we have $\psi_{\beta}(W_1) =\psi_{\beta}(-2Y_1)$, and so
$$\lambda(I-{}^\alpha X Y^{-1}X) =\lambda(W_0)\psi_{\beta}(W_1)=\lambda(-{}^cY_0 Y_0^{-1})\psi_{\beta}(2Y_1)^{-1}.$$
Since $\tilde\beta = 2\beta$, the summand with $X\neq 0$ is 
\begin{equation}
\label{by unramified unitary X not 0 summand}
\sum_{
   \begin{smallmatrix} X\neq 0 \\ Y_1   \end{smallmatrix}} \tilde{\lambda}(Y) \lambda(I-{}^\alpha X Y^{-1}X) = \lambda(-1)(\#Y_1)^{1/2}\sum_{X\neq 0}\tilde\lambda(Y_0)\lambda({}^cY_0 Y_0^{-1})
\end{equation}
(the factor being $(\#Y_1)^{1/2}$ instead of $(\#Y_1)$ because of (\ref{relation of X and Y at level 1 and w=y})). In the latter case, we have $W_0=1$ and $W_1=0$. The relation (\ref{relation of X and Y at level 0 and w=y}) becomes ${}^cY_0 = -Y_0$, and so $\psi_{\tilde\beta}(Y_1) = \psi_{\tilde\beta}(^\alpha Y_1) =  \psi_{\tilde\beta}(- Y_0Y_1Y_0^{-1})  = \psi_{\tilde\beta}(- Y_1) $ which implies that $\psi_{\tilde\beta}(Y_1) =1$ if $p\neq 2$. Hence the summand with $X=0$ is 
\begin{equation}
\label{by unramified unitary X equal 0 summand}
\sum_{ \begin{smallmatrix} X= 0 \\ Y_1   \end{smallmatrix}} \tilde{\lambda}(Y) \lambda(I-{}^\alpha X Y^{-1}X) = \lambda(-1 )(\#Y_1)^{1/2}\sum_{X= 0}\tilde\lambda(Y_0).
\end{equation}
The total sum (\ref{by unramified unitary X not 0 summand})+(\ref{by unramified unitary X equal 0 summand}) is $\lambda(-1)(\#Y_1)^{1/2}$ times the sum considered in the first example of Section \ref{Example: depth zero characters}, and the results there imply that, when $G$ is unramified unitary: 
\begin{equation*}
c_y = \qo^3(\#Y_1), \quad 
b_y \tilde T_y(s_y)= \begin{cases}
\lambda(-1) (\qo^3-1)(c_y/\qo^3)^{1/2} & \tilde\lambda \leftrightarrow \lambda,
\\
-\lambda(-1) \qo(\qo-1)(c_y/\qo^3)^{1/2}
 & \tilde\lambda \not\leftrightarrow \lambda\text{ but is self-dual},
\end{cases}
\end{equation*}
which implies that
\begin{equation}
\label{results of ry and eyTy for unramified unitary group}
r_y = \begin{cases}
3/2
\\
1/2
\end{cases}
\quad\text{and}\quad 
\epsilon_y T_y(s_y)= \begin{cases}
\lambda(-1)
\\
-\lambda(-1).
\end{cases}
\end{equation}
Note that the results in (\ref{results of ry and eyTy}) and (\ref{results of ry and eyTy for unramified unitary group}) are independent of whether $i\in I_1$ or $I_\zeta$.

\subsubsection{Computing $b_z$}

We then compute $b_z$, again with $i\in I\smallsetminus\{o\}$ fixed. By applying the calculations in \cite[Sec 4.2]{BT-ramified}, we expand (\ref{formula for b and c for z}) as \begin{equation}
\label{main calculation: bz begins}
b_z \tilde T_z(s_z)=\sum_{(X,Y)\in \mathcal S_z}\tilde{\lambda}(Y\varpi_i) \lambda(I-{}^\alpha X Y^{-1}X),
\end{equation}
where $
  \mathcal S_z = 
    (s_z{\mathcal J}_P^-s_z^{-1}\cap {\mathcal J}_Ps_z{\mathcal J}_P)/{\mathcal J}_P^+ $. The quotient $s_z{\mathcal J}_P^-s_z^{-1}/{\mathcal J}_P^+$ consists of elements 
\begin{equation*}
   \begin{split}
      (X,Y)& \in     \frac{\mathfrak P^0_{(+,I)}}{\mathfrak o_{E_i}+\mathfrak P^{0_{+}}_{(+,I)}}
\oplus
\frac{\mathfrak  p_{E_i}^{-1}+\mathfrak P^{0_{}}_{(+,-)}}{ \mathfrak P^{0_{}}_{(+,-)}},
   \end{split}
\end{equation*}
satisfying relation (\ref{X-alpha-X-equals-Y-minus-alpha-Y}). The entries $X^j$ for $j\in I\smallsetminus\{i\}$ lie in $\mathfrak W^{i}_{z,j}:={\mathfrak P^{0_{}}_{(+,j)}}/{\mathfrak P^{0_{+}}_{(+,j)}}$ while $X^i\in \mathfrak W^{i}_{z,i}:={\mathfrak P^{0_{}}(\Lambda_{i})}/({\mathfrak o_{E_i}+ \mathfrak P^{0_{+}}(\Lambda_{i}}))$. The condition $(X,Y)\in {\mathcal J}_Ps_z{\mathcal J}_P$ forces $Y\in \varpi_i^{-1}\tilde {\mathcal J}$, so that we write $Y = \varpi_i^{-1}Y_0(I+Y_1)$ with auxiliary $Y_1$, and choose $Y_0\in {\boldsymbol{\mu}}_{E_i}$ such that 
$$
Y-{}^\alpha Y = \varpi_i^{-1}
Y_0(I+Y_1)- (I-{}^\alpha Y_1 ){}^cY_0\varpi_i^{-1} = 
X{}^\alpha X.$$
Comparing valuations, the above relation implies that
\begin{subequations}
\begin{align}
  & Y_0 = {}^cY_0\quad \text{ (i.e., $Y_0\in {\boldsymbol{\mu}}_{E_{i\bullet}}$, since ${}^\alpha\varpi_i = \varpi_i$)},
  \label{relation of X and Y at level 0 and w=z}
   \\
&\varpi_i^{-1}
Y_0Y_1+ {}^\alpha Y_1 {}^cY_0\varpi_i^{-1} = 
X{}^\alpha X . 
     \label{relation of X and Y at level 1 and w=z}
\end{align}
\end{subequations}
We have $\tilde{\lambda}(Y\varpi_i)=\tilde \lambda(Y_0)\psi_{\tilde \beta}(1+Y_1) $, and using (\ref{relation of X and Y at level 1 and w=z}), the last factor is
\begin{equation}
\label{bz X sum 1}
\psi_{\tilde \beta}(1+Y_1) =  \psi(\beta_i(Y_1+{}^\alpha Y_1)) = \psi( \beta_i\varpi_i Y_0^{-1}
\sum_{j\in I}X^j{}^\alpha X^j ).
\end{equation}
For  $\lambda(I-{}^\alpha X Y^{-1}X)$, it is easy to see that $I-{}^\alpha X Y^{-1}X = I-{}^\alpha X Y_0^{-1}\varpi_i X \mod \mathcal{U}^{0_{++}}(\Lambda)$. Reducing to diagonal blocks as before, we obtain
\begin{equation}
\label{bz X sum 2}
\lambda(I-{}^\alpha X Y^{-1}X)= \prod_{j\in I_{}}\psi(- \beta_j{}^\alpha X^j\varpi_i Y_0^{-1} X^j).
\end{equation}
Putting (\ref{bz X sum 1}) and (\ref{bz X sum 2}) into (\ref{main calculation: bz begins}), we obtain
\begin{equation}
\label{bz expanded as quadratic form}
   \begin{split}
b_z \tilde T_z(s_z) = \sum_{Y_0}\tilde \lambda(Y_0)\sum_X 
\prod_{j\in I_{}}\psi( \varpi_iY_0^{-1}( \beta_i X^j - X^j \beta_j) {}^\alpha X^j  ).
   \end{split}
\end{equation}
We will show in Section \ref{subsection Non-degeneracy of a quadratic form} that the last sum $\sum_X$ is a quadratic Gauss sum on the space 
\begin{equation}
\label{direct sum of mathfrak W}
\mathfrak W^{i}_z :={\mathfrak W}^{i}_{z,i}\oplus
\bigoplus_{j\in I\smallsetminus\{i\}} 
\mathfrak W^{i}_{z,j}  = {\mathfrak P^{0_{}}_{i}}/({\mathfrak o_{E_i}+ \mathfrak P^{0_{+}}(\Lambda_{i}})) \oplus
\bigoplus_{j\in I\smallsetminus\{i\}}
{\mathfrak P^{0_{}}_{(+,j)}}/{\mathfrak P^{0_{+}}_{(+,j)}}
\end{equation}
containing $X$, equipped with a non-degenerate quadratic form. We will also compute this Gauss sum, which takes the form
\begin{equation}
\label{Gauss sum of the form}
\left(\frac{Y_0}{{\boldsymbol{\mu}}_{E_{i\bullet}}}\right)^
{\dim_{\mathbb E_{i\bullet}}\mathfrak{W}^{i}_z}
q_i^{\dim_{\mathbb  E_{i\bullet}} \mathfrak{W}^{i}_z /2}
\mathfrak{n}_z(\varpi_i,\mathbf{s},\psi,h) , 
\end{equation}
 where 
$\mathfrak{n}_z(\varpi_i,\mathbf{s},\psi,h)$ is a 4th root of unity.

Substituting the form (\ref{Gauss sum of the form}) of the Gauss sum into (\ref{bz expanded as quadratic form}), we obtain
\begin{equation}
\label{b_z T_z(s_z) general form}
b_z \tilde T_z(s_z)=\mathfrak{n}_z(\varpi_i,\mathbf{s},\psi,h)
q_i^{\dim_{\mathbb F_{E_{i\bullet}}} \mathfrak{W}^{i}_z /2}
\sum_{Y_0\in {\boldsymbol{\mu}}_{E_{i\bullet}}}\tilde \lambda(Y_0)\left(\frac{Y_0}{{\boldsymbol{\mu}}_{E_{i\bullet}}}\right)^{\dim_{\mathbb F_{E_{i\bullet}}}\mathfrak{W}^i_z},
\end{equation}
and therefore
\begin{equation*}
\text{$b_z=0$ \quad if and only if \quad $\tilde \lambda|_{{\boldsymbol{\mu}}_{E_{i\bullet}}} \neq \left(\frac{\cdot}{{\boldsymbol{\mu}}_{E_{i\bullet}}}\right)^{\dim_{\mathbb F_{E_{i\bullet}}}\mathfrak{W}^i_z}$.}
\end{equation*}
Henceforth, we assume $b_z\neq 0$. In this case, we have
\begin{equation*}
b_z \tilde T_z(s_z)= \mathfrak{n}_z(\varpi_i,\mathbf{s},\psi,h)
q_i^{\dim_{\mathbb F_{E_{i\bullet}}} \mathfrak W^i_z /2}
(q_{i\bullet}-1).
\end{equation*}

\subsubsection{Preliminary results on endoscopic liftings}

When $G$ is an unramified unitary group, the index set $I$ is a singleton, and $E_i/F$ is totally ramified. As $\mathfrak{W}_z = \mathfrak{W}^i_z$ is an $\mathbb F$-space, $\dim_{\mathbb F_\bullet}\mathfrak{W}_z$ is always even, so the sum $\sum_{Y_0}$ in  (\ref{b_z T_z(s_z) general form}) equals $\qo-1$. We thus have 
\begin{equation*}
\label{}
c_z = q^{(\dim_{\mathbb F_\bullet} \mathfrak W_z -1)/2}, \quad r_z = 1/2 
\quad\text{and}\quad 
\epsilon_z T_z(s_z)= \mathfrak{n}_z(\varpi_i,\mathbf{s},\psi,h_{})
\end{equation*}
which gives the following proposition when combined with (\ref{results of ry and eyTy for unramified unitary group}).

\begin{prop}
\label{first general form of lifting, unitary group}
Let $G=G(V,h)$ be an unramified unitary group, and $\pi = \cInd_{\mathcal J}^{G(F)}\lambda$ be an epipelagic representation constructed from an epipelagic skew simple stratum $[\Lambda,0,\beta]$ with $\lambda|_{\mathcal{U}^{0_+}(\Lambda)^\sigma} = \psi_\beta$, where $F[\beta]/F$ is totally ramified. We construct a character $( F[\beta]^\times \mathcal{U}^{0_+}(\Lambda),\tilde{\boldsymbol{\lambda}})$ as follows,
\begin{equation*}
\tilde \lambda|_{\mathcal{U}^{0_+}(\Lambda)} = \psi_{2\beta},\quad 
 \tilde \lambda|_{{\boldsymbol{\mu}}_F}\leftrightarrow \lambda|_{{\boldsymbol{\mu}}(F)_1},
\quad\text{and}\quad
\tilde{\boldsymbol{\lambda}}( \varpi_i)= \lambda(-1)
\mathfrak{n}_z(\varpi_\beta,\mathbf{s},\psi,h).
\end{equation*}
Let $\tilde \pi = \tilde \pi_{\tilde{\boldsymbol{\lambda}}}$ be the induced epipelagic representation of $\tilde G(F)$, then we have $$\mathrm{Red}(\tilde\pi,\pi) = \{\pm 1, \pm \tfrac{1}{2}+\tfrac{\pi \sqrt{-1}}{\log q} \}.$$
 In particular, $I(s,\tilde\pi,\pi)$ is reducible at $s=1$. 
  \hfill$\square$
\end{prop}

When $G$ is not of unramified unitary type, we interpret the results in (\ref{b_z T_z(s_z) general form}) directly and obtain 
\begin{equation*}
\label{}
c_z = q_i^{(\dim_{\mathbb E_i} \mathfrak W^i_z -1)/2}, \quad r_z = 1/2
\quad\text{and}\quad 
\epsilon_z T_z(s_z)= \mathfrak{n}_z(\varpi_i,\mathbf{s},\psi,h).
\end{equation*}
Combining it with (\ref{results of ry and eyTy}), we obtain the following proposition.

\begin{prop}
\label{first general form of lifting}
Let $G=G(V,h)$ be a connected classical group not of unramified unitary type, and $\pi = \cInd_{\mathcal J}^{G(F)}\lambda$ be an epipelagic representation constructed from an epipelagic skew semisimple stratum $[\Lambda,0,\beta]$ with $\lambda|_{\mathcal{U}^{0_+}(\Lambda)^\sigma} = \psi_\beta$. Writing $\beta = \sum_{i\in I}\beta_i$, for a fixed $i\in I\smallsetminus \{o\}$, we construct a character $( F[\beta_i]^\times \mathcal{U}^{0_+}(\Lambda_i), \tilde{\boldsymbol{\lambda}}_i)$ as follows,
\begin{equation*}
\tilde\lambda_i|_{\mathcal{U}^{0_+}(\Lambda_i)} = \psi_{2\beta_i},\quad 
\tilde \lambda_i|_{{\boldsymbol{\mu}}_{E_i}}= \left(\frac{\cdot}{{\boldsymbol{\mu}}_{E_i}}\right)^{\dim_{\mathbb E_i}\mathfrak{W}^i_z},
\quad\text{and}\quad
\tilde{\boldsymbol{\lambda}}_i(\varpi_i)= \tilde\lambda_i(-2)\lambda(\omega_i )
\mathfrak{n}_z(\varpi_i,\mathbf{s},\psi,h).
\end{equation*}
Let $\tilde \pi_i = \tilde \pi_{\tilde{\boldsymbol{\lambda}}_i}$ be the induced quasi-epipelagic representation of $\tilde G(V_i)(F)$, then we have $$\mathrm{Red}(\tilde\pi_i,\pi) = \left\{\pm 1,\pm \tfrac{1}{2}+\tfrac{\pi \sqrt{-1}}{\log q_i} \right\}.$$ In particular, $I(s,\tilde\pi_i,\pi)$ is reducible at $s=1$. 
 \hfill$\square$
\end{prop}

In Sections \ref{section Examples for simple supercuspidals} and \ref{subsection Reducibility results for different classical groups} below, we further analyze the value of $\tilde{\boldsymbol{\lambda}}(\varpi_i)$ for different types of $G$. Note that the choice of the uniformizer $\varpi_i$ of $E_i$ is arbitrary, and our result is indeed independent of this choice.

\begin{prop}
\label{prop of independence: uniformizer}
The results in  Propositions \ref{first general form of lifting, unitary group} and \ref{first general form of lifting} relating $\tilde{\boldsymbol{\lambda}}(\varpi_i)$ and $\mathfrak{n}_z(\varpi_i,\mathbf{s},\psi,h_{})$ are independent of the choice of $\varpi_i$. 
\end{prop}
\proof
The proof is identical to that of \cite[Prop 4.8(ii)]{BT-ramified}. The main idea is: to obtain the value in (\ref{b_z T_z(s_z) general form}), we assign $\tilde\lambda|_{{\boldsymbol{\mu}}_{E_{i\bullet}}}$ to be the character defined on $Y_0\in {\boldsymbol{\mu}}_{E_{i\bullet}}$ extracted from the Gauss sum in (\ref{bz expanded as quadratic form}).  
 \hfill$\blacksquare$

\subsection{Properties of a quadratic form}
\label{subsection Non-degeneracy of a quadratic form}

We first recall a summary of general properties of quadratic Gauss sums from \cite[Sec 4.5]{BH-ET2}. Let $\mathbb F$ be a finite field and $(\mathbf V,\mathbf q)$ be a non-degenerate quadratic form on an $\mathbb F$-vector space. For a non-trivial additive character $\psi$ of $\mathbb F$, we define the normalized quadratic Gauss sum by
$$\mathfrak n(\mathbf q) = q^{-\dim_{\mathbb F}\mathbf V/2}\sum_{X\in \mathbf V}\psi(\mathbf q(X)).$$
Its value can be expressed as follows. Define the symmetric bilinear form $\mathbf h$ associated to $\mathbf q$ by
\begin{equation}
\label{definition of bilinear form}
\mathbf h(X,Y) = \tfrac{1}{2}(\mathbf q(X+Y) - \mathbf q(X)-\mathbf q(Y)),\quad X,Y\in \mathbf V.
\end{equation}
If $\mathbf H$ is the symmetric matrix associated to $\mathbf h$, then the discriminant $\det \mathbf H\neq 0$, and we denote $\det \mathbf q := \det \mathbf H $. Putting
\begin{equation}
\label{simple quadratic Gauss sum}
\text{$\mathfrak n_\psi = \mathfrak n(\mathbf q_0)$\quad  where \quad $\mathbf q_0:x\mapsto x^2$ on $\mathbf V = \mathbb F$, 
}
\end{equation}
we have
\begin{equation}
\label{factorizing the Gauss sum of a quadratic form}
\mathfrak n(\mathbf q) = \left(\frac{\det \mathbf q}{\mathbb  F^\times}\right)\mathfrak n_\psi^{\dim_{\mathbb F}\mathbf V}.
\end{equation}
It is well known that $\mathfrak n_\psi$ is a 4th root of unity, and hence so is $\mathfrak n(\mathbf q)$. It turns out that $\mathfrak n(\mathbf q)$ is independent of the base field defining $\mathbf q$.

\begin{prop}
\label{Gauss sum independent of base fields}
Let $\mathbb E/\mathbb F$ be a finite extension, and $(\mathbf V_{\mathbb E}, \mathbf q_{\mathbb E})$ be a quadratic form over $\mathbb E$. If $(\mathbf V_{}, \mathbf q_{})$ is the same as $(\mathbf V_{\mathbb E}, \mathbf q_{\mathbb E})$ regarded as a quadratic form over $\mathbb F$, then $\mathfrak n(\mathbf q_{\mathbb E})= \mathfrak n(\mathbf q_{})$.\end{prop}
\proof This is a simple exercise, using (\ref{factorizing the Gauss sum of a quadratic form}) and the Hasse-Davenport relation $\mathfrak n_{\psi_{\mathbb E}} = (-1)^{f-1}\mathfrak n_{\psi_{}}^f$, where $\psi_{\mathbb E} = \psi\circ \mathrm{tr}_{\mathbb E/\mathbb F}$ and $f = f_{\mathbb E/\mathbb F}$.
\hfill$\blacksquare$

We now analyze the normalized Gauss sum $\mathfrak{n}_z(\varpi_i,\mathbf{s},\psi,h_{})$ that appeared in the previous section. Under the setup of (\ref{bz expanded as quadratic form}), for a fixed index $i\in I\smallsetminus \{o\}$, we define a quadratic form on $\mathfrak{W}^i_{z,j}$, for $j\in I$, by 
\begin{equation*}
\mathbf q^i_{z,\mathbf s,j}(X_j) =\mathrm{tr}_{\tilde{ \mathfrak g}(V_i)/E_i}(\varpi_i
(\beta_iX_{j}{}-  X_{j}\beta_j ) {}^\alpha X_j) \mod \mathfrak{p}_{E_i},\quad \text{$X_j\in \mathfrak{W}^i_{z,j}$},
\end{equation*}
Note that $\mathfrak{W}^i_{z,i}$ is structurally different from the other summands $\mathfrak{W}^i_{z,j}$: from (\ref{direct sum of mathfrak W}), we have  
\begin{equation*}
     \mathfrak W^{i}_{z,i}
      = 
{\mathfrak P^{0_{}}_{i}}/({\mathfrak o_{E_i}+ \mathfrak P^{0_{+}}(\Lambda_{i}})) \quad \text{ and }\quad \mathfrak W^{i}_{z,j} = 
{\mathfrak P^{0_{}}_{(+,j)}}/{\mathfrak P^{0_{+}}_{(+,j)}} \quad \text{ for }j\neq i.
\end{equation*}
The space $\mathfrak{W}^i_{z}=\bigoplus_{j\in I}\mathfrak{W}^i_{z,j}$ is therefore equipped with the orthogonal sum
$
\mathbf q^i_{z,\mathbf s} = \bigoplus^\perp_{j\in I}\mathbf q^i_{z,\mathbf s,j}
$. 
\begin{prop}
\label{properties of quadratic forms}
\begin{enumerate}[(i)]
\item The quadratic form $\mathbf q^i_{z,\mathbf s}$ is non-degenerate. \label{non-degenerate quad form}

\item 
The form $\mathbf q^i_{z,\mathbf s}$ is Hermitian-symmetric, i.e., ${}^\alpha\mathbf q^i_{z,\mathbf s} = -\mathbf q^i_{z,\mathbf s}$.
\label{Hermitian symmetry of the quadratic form}

\end{enumerate}

\end{prop}
\proof
To show \ref{non-degenerate quad form}, it suffices to show that each $\mathbf q^i_{z,\mathbf s,j}$ is non-degenerate. When $\beta_i$ is non-null, this is equivalent to showing that 
$$\beta_i X_{j}{}- X_{j}\beta_j \in \mathfrak{P}^{-s}_{(i,j)}
\quad\Rightarrow \quad
X\in 
\begin{cases}
\mathfrak{p}^{s}({E_i})+
\mathfrak{P}^{s_+}_{(i,j)}  & j=i,
\\
\mathfrak{P}^{s_+}_{(i,j)} & j\neq i.
\end{cases} 
$$
In the epipelagic case, $s=0$. The first case ($j=i$) is proved in \cite[Prop 4.7]{BT-ramified}, and the second case ($j\neq i$) follows analogously from \cite[Lem 3.7(i)]{Stevens-semi-simple-char}. When $\beta_i=0$, for $j\neq o$, we require 
$$X_{j}\beta_j \in \mathfrak P^0_{(i,j)}
\quad\Rightarrow \quad
X_j\in 
\mathfrak{P}^{0_+}_{(i,j)},$$
which holds obviously since $v_{\Lambda_j}(\beta_j)<0$. \ref{Hermitian symmetry of the quadratic form} follows from a direct calculation showing that $${}^\alpha(
( X_{j}{}- \beta_i^{-1} X_{j}\beta_j ) {}^\alpha X_j) = -X({}^\alpha X -\beta_j{}^\alpha X \beta_i^{-1} )$$
 using the $\alpha$-invariance of $\beta_j$ for all $j\in I\smallsetminus\{o\}$.
\hfill$\blacksquare$

Using the independence result in Proposition \ref{prop of independence: uniformizer}, we choose $\varpi_i = \beta_i^{-1}$ when $i\neq o$ and obtain
$$\mathbf q^i_{z,\mathbf s}:X\mapsto \sum_{j\in I}\mathrm{tr}_{\tilde{ \mathfrak g}(V_i)/ E_i}(
( X_{j}{}- \beta_i^{-1} X_{j}\beta_j ) {}^\alpha X_j) \mod \mathfrak{p}_{E_i},\quad \text{$X\in \mathfrak{W}^i_{z}$}.$$
The Gauss sum $\mathfrak{n}_z(\varpi_i,\mathbf{s},\psi,h)$ in Proposition \ref{first general form of lifting} equals $\mathfrak n(\mathbf q^i_{z,\mathbf s})$, or equivalently the product $\prod_{j\in I}\mathfrak n(\mathbf q^i_{z,\mathbf s,j})$.

Expanding via definition (\ref{definition of bilinear form}) and using the symmetry in Proposition \ref{properties of quadratic forms}\ref{Hermitian symmetry of the quadratic form}, the associated symmetric bilinear form of $\mathbf q^i_{z,\mathbf s,j}$ is given by 
\begin{equation*}
\mathbf h_{z,\mathbf s,j}(X_j,Y_j) =\mathrm{tr}_{\tilde{ \mathfrak g}(V_i)/E_i}((X_{j}{}- \beta_i^{-1} X_{j}\beta_j ) {}^\alpha Y_j) \mod \mathfrak{p}_{E_i}, \quad \text{$X_j,Y_j\in \mathfrak{W}^i_{z,j}$}.
\end{equation*}
The discriminant of $\mathbf h_{z,\mathbf s,j}$ is equal to $\gamma_j\theta_j\mod \mathfrak{p}_{E_i}$, where 
\begin{equation*}
\gamma_j = \det{}(I_{\mathfrak{W}^i_{z,j}}-  {}{\beta_i^{-1}}\otimes {\beta_j})= \prod_{   \begin{smallmatrix} \gamma\in \Gal(\mathbb E_i/\mathbb F),
\\
\delta\in \Gal(\mathbb E_j/\mathbb F),
\\
\gamma\neq\delta\text{ (when $i=j$)}
   \end{smallmatrix}}(1-{}^\gamma{}\beta^{-1}_i {}^\delta\beta_j),
\end{equation*}
and $\theta_j = \det H_j \det H_i^{-1}$. Note that $\gamma_j\in {\boldsymbol{\mu}}_{F}$ and, with our choices of $H_j$ in Section \ref{subsection Cohomological classification}, $\theta_j\in {\boldsymbol{\mu}}_F$ too. Therefore, by (\ref{factorizing the Gauss sum of a quadratic form}), for $j\neq i$,
$$\mathfrak n(\mathbf q^i_{z,\mathbf s,j}) = \left(\frac{\gamma_j\theta_j}{{\boldsymbol{\mu}}_{F}}\right)\mathfrak n_{\psi,\mathbb F}^{2ef_if_j} = \left(\frac{(-1)^{ef_if_j}\gamma_j\theta_j}{{\boldsymbol{\mu}}_{F}}\right).$$
Here $2e$ is the common ramification degree $e_{E_j/\Fo}$ for $j\in I\smallsetminus \{o\}$ and $f_j = f_{E_j/F}$ is the residue degree.

To simplify notations, for each $i\in I\neq \{o\}$, we simply put
\begin{equation}
\label{a sign kappa_i with j not o}
\kappa_i =  \prod_{j\in I\smallsetminus\{o\}} \left(\frac{(-1)^{ef_if_j}\gamma_j\theta_j}{{\boldsymbol{\mu}}_{F}}\right)
\end{equation}
then $\mathfrak n(\mathbf q_{\neq o})$ is equal to 
$\kappa_i \mathfrak n_\psi^{-f_i}$.

The calculations for $j\in \{i,o\}$ can be made explicit, but depend on the type of $G$. We first provide the calculations for simple supercuspidal representations in Section \ref{section Examples for simple supercuspidals} and compare them with Oi's results. We then provide more explicit endoscopic liftings as well as the resulting L-packets in Section \ref{subsection Reducibility results for different classical groups}.

\section{Examples: simple supercuspidals}
\label{section Examples for simple supercuspidals}

In a series of papers \cite{Oi-SO-odd,Oi-Sp-and-SO-even,Oi-U-unram}, M. Oi computed the endoscopic liftings of simple supercuspidals \cite{Gross-Reeder} of classical groups using the endoscopic character identity. In this section (or more precisely Section \ref{subsection Comparison with Oi's results}), we compare his results with our reducibility results in Propositions \ref{first general form of lifting, unitary group} and \ref{first general form of lifting}, and show that both results yield the same liftings. Some of these results are also required to compute further the endoscopic liftings of epipelagic representations of classical groups. 

\subsection{Construction of simple supercuspidals}
\label{subsection Constructions simple supercuspidals}

We first recall the explicit construction of the inducing types of simple supercuspidals. Our treatment here is slightly more general than Oi's, where he chose convenient representatives of conjugacy classes of affine generic characters in advance. We consider these characters in general and show that our results agree with those of Oi. This generality has the advantage of allowing transitions between equivalent simple strata.

We provide descriptions of simple supercuspidals for all types of quasi-split classical groups; the case of ramified unitary groups (which was not covered by Oi) is postponed to Section \ref{subsection Ramified unitary groups}. We then apply (\ref{criterion of isomorphic simple supercuspidal}) to provide conditions for two affine generic characters to be induced to isomorphic supercuspidals. We also modify the setup to accommodate liftings for non-quasi-split pure inner forms, which are necessary for the calculations in Section \ref{subsection Reducibility results for different classical groups}.

\subsubsection{General linear groups} 
\label{subsection Construction of quasi-simple supercuspidals GL-case}

Let $\tilde G = \GL_{r}$ and $\mathbf s=[\Lambda,0,\tilde\beta]$ be a simple stratum in $\tilde G$. Denote $E=F[\tilde\beta]$ with $r=[E:F] = mf$, where $m=e_{E/F}$ and $f=f_{E/F}$. We furthermore assume that $v_\Lambda(\tilde \beta) = -1$, so that $\tilde\varpi = \tilde\beta^{-1}$ is a uniformizer of $E$. Let $\tilde{\boldsymbol{{\mathcal{J}}}} = E^\times \mathcal U_{0_+}(\Lambda)$ be the subgroup defined in Section \ref{subsection Epipelagic inducing types for classical groups}. Take the character $(\psi_{\tilde\beta},\mathcal U_{0_+}(\Lambda))$ defined by $\tilde \beta$. Let $\phi$ be a character of ${\boldsymbol{\mu}}_E$ and take $\xi \in \mathbb C^\times$. We define the character $\tilde{\boldsymbol{\lambda}} = \tilde{\boldsymbol{\lambda}}(\tilde\beta,\phi,\xi)$ on $\tilde{\boldsymbol{{\mathcal{J}}}}$ by 
$$\tilde{\boldsymbol{\lambda}}(\tilde\varpi^j z (I+X)) = \xi^j\phi(z)\psi_{\tilde\beta} (X),\,\qquad \tilde\varpi^j z(I+X)\in \tilde{\boldsymbol{{\mathcal{J}}}}  = \left<\tilde\varpi\right>{\boldsymbol{\mu}}_E \mathcal U_{0_+}(\Lambda).$$
The induced representation $\tilde\pi = \tilde\pi{(\tilde\beta,\phi,\xi)}=\cInd_{\tilde{\boldsymbol{{\mathcal{J}}}} }^{\tilde G(F)}\tilde{\boldsymbol{\lambda}}$ is a supercuspidal representation, and two such representations $\tilde\pi{(\tilde\beta_j,\phi_j,\xi_j)}$ for $j=1,2$ are isomorphic if and only if $\tilde\beta_1 - \tilde\beta_2\in \mathfrak{P}^0(\Lambda)$ and $(\phi_1 ,\xi_1 )= (\phi_2, \xi_2)$.

When $f=1$, we are in the simple supercuspidal case. Following Section \ref{subsection Example: Simple supercuspidal representations}, we ontain a similar construction of simple supercuspidal representations as follows. The simple affine roots $\Delta_{\mathrm{aff}} = \{\alpha_i\}_{i=1}^{m}$ are $\alpha_i = e_i- e_{i+1}$ for  $1\leq i\leq m-1$ and $\alpha_0 = 1-(e_1-e_m)$. We express elements in the pro-p Iwahori subgroup as 
$$u =I+\antidiag( 
\diag( u_1, \dots, u_{m-1}), 
u_0\varpi )\in \mathcal{I}^+,$$ 
 and the affine generic character $\tilde{\lambda}$ restricts to each simple affine root space as $u_i\mapsto  \psi_{a_i}(u_i)$ for some $a_i\in \mathfrak{o}_F^\times$ for all $i$. Put $\tilde a = \prod_{i=0}^{m-1}a_i\bmod \mathcal{U}^1(F)$.

Let $Z = Z_{\tilde G}$, and $\Omega $ be as in Section \ref{subsection Example: Simple supercuspidal representations}. It is easy to show that  $N(\tilde{\lambda})= \Omega Z \mathcal{I}^+$. The generator of the cyclic group $\bar \Omega:=\Omega/(\Omega\cap Z \mathcal{I}^+)$ of order $m$ can be lifted to 
$$\tilde\varpi = \antidiag( 
\diag( a_1^{-1}, \dots, a_{m-1}^{-1}), 
a_0^{-1}\varpi),$$
so that $\tilde\varpi^m = (\varpi \tilde a^{-1})I\in Z$. Choose a character $\phi$ of ${\boldsymbol{\mu}}_F= {\boldsymbol{\mu}}_{F[\tilde\varpi]}\subset Z$ and any $\xi \in \mathbb C^\times$ to define the character $\tilde{\boldsymbol{\lambda}} = \tilde{\boldsymbol{\lambda}}((a_{i})_{i=0}^{m-1},\phi,\xi)$ by 
$$\tilde{\boldsymbol{\lambda}}(\tilde\varpi^j z u) = \xi^j\phi(z)\prod_{i}\psi_{a_i} (u_i),\,\qquad \tilde\varpi^j zu \in N(\tilde\lambda) =  \left<\tilde\varpi\right>{\boldsymbol{\mu}}_F \mathcal{I}^+.$$
The two simple supercuspidal representations $\tilde \pi_{\tilde{\boldsymbol{\lambda}}_j}$ for $j=1,2$, where $\tilde{\boldsymbol{\lambda}}_j =\tilde{\boldsymbol{\lambda}}((a_{i}^j)_{i=0}^{m-1},\phi_j,\xi_j)$, are isomorphic if and only if $(\tilde a^1,\phi_1 ,\xi_1 )= (\tilde a^2,\phi_2, \xi_2)$.

\begin{rmk}
\cite[2.2 Prop]{BK-epipelagic} asserts that $\tilde \pi_{\tilde{\boldsymbol{\lambda}}_1}\cong \tilde \pi_{\tilde{\boldsymbol{\lambda}}_2}$ if and only if $(\tilde a^1,\phi_1  )= (\tilde a^2,\phi_2)$ and $$\epsilon(\tilde \pi_{\tilde{\boldsymbol{\lambda}}_1},1/2,\psi) = \epsilon(\tilde \pi_{\tilde{\boldsymbol{\lambda}}_2},1/2,\psi),$$
 where $\epsilon(\tilde \pi,s,\psi)$, for $s\in \mathbb C$, is the Godement-Jacquet local constant of $\tilde\pi$. By \emph{loc. cit.} [2.2 Lem(1)], once $(\tilde a,\phi  )$ is fixed, the last condition is equivalent to $\xi_1=\xi_2$.
\qed\end{rmk}

An Oi representative of $\tilde{\boldsymbol{\lambda}}$ takes the form $(\tilde a,\tilde\lambda|_{{\boldsymbol{\mu}}_F},\tilde{\boldsymbol{\lambda}}(\tilde\varpi))\in \mathbb F^\times \times \hat{\mathbb F}^\times \times \mathbb C^\times$, where $\tilde a$ represents $(a_0,\dots,a_{m-1})  $ $= (\tilde a,1,\dots,1)$.

\subsubsection{Unramified unitary groups}  

Let $G = \mathrm{U}_{N,F/\Fo}$, where $F/\Fo$ is unramified, and take $\varpi\in \Fo$. We take $H = \diag(1,-1,1,-1,\dots)$, so that ${}^t\overline H = (-1)^{n-1} H$. Put $n=\lfloor N/2\rfloor$. The simple affine roots $\Delta_{\mathrm{aff}} = \{\alpha_i\}_{i=1}^{n}$ are $\alpha_i = e_{i+1}-e_i$ for $i<n$, 
$\alpha_n = e_n$ (resp. $2e_n$) if $N$ is odd (resp. even), and $\alpha_0 = 1-2\sum_{i=1}^n{\alpha_i}$. We express
\begin{equation*}
   \begin{split}
     u &=I+ 
     \antidiag( 
\diag( u_1, \dots, u_n, \overline{u_n}, \dots, \overline{u_1}), 
u_0\varpi ),\quad \text{$\overline{u_0} = -u_0$, when $N$ is odd, }
\\
    & = I+ \antidiag( 
\diag( u_1, \dots,  u_{n-1}, u_n, \overline{u_{n-1}}, \dots, \overline{u_1}), 
u_0\varpi )\quad \text{$u_0,u_n\in \mathfrak o_{\Fo}\bmod \mathfrak p_{\Fo}$, when $N$ is even.}   \end{split}
\end{equation*}
The affine generic character ${\lambda}$ restricts to each simple affine root space as $u_i\mapsto  \psi_{a_i}(u_i)$ for some $a_i\in \mathfrak{o}_F^\times$ for all $i$, with the additional conditions that $a_0\in \ker\mathrm{tr}_{F/\Fo}$ when $N$ is odd, and $a_0,a_n\in {\Fo}$ when $N$ is even.

 The group $\Omega$ is trivial, and $N_G(\lambda) = Z \mathcal{I}^+$.  Take a character $\phi$ of ${\boldsymbol{\mu}}(F)_1\subset Z$ and define the character $\lambda = \lambda((a_{i})_{i=0}^{n},\phi)$ by 
$$\lambda( z u) = \phi(z)\prod_{i}\psi_{a_i} (u_i),\,\qquad zu \in N_G( \lambda) =  {\boldsymbol{\mu}}(F)_1 \mathcal{I}^+.$$
We put
  \begin{equation*}
   \begin{split}
  a&=a_{0}\left(\textstyle\prod_{i =1}^{n} N_{F/\Fo}a_{i}\right)
  \mod (F^\times)^2\mathcal{U}^1(F)
    \quad \text{when $N$ is odd, }
\\
&=a_{0}\left(\textstyle\prod_{i =1}^{n-1} N_{F/\Fo}a_{i}\right)a_{n} \mod \mathcal{U}^1(F)
    \quad \text{when $N$ is even.}   \end{split}
\end{equation*}
 Two representations $\pi_{\lambda_j}$ for $j=1,2$, where $\lambda_j = \lambda((a_{i}^j)_{i=0}^{n},\phi_j)$, are isomorphic if and only if $(a^1 ,\phi_1)= (a^2, \phi_2)$.

From \cite{Oi-U-unram}, an Oi representative of $\lambda$ takes the form $(a,\lambda|_{{\boldsymbol{\mu}}(F)_1})$, where $a$ represents $(a_0,\dots,a_{n}) = (a,1,\dots,1)$ with $a\in \mathbb F_\bullet^\times$ (resp. $\ker\mathrm{tr}_{\mathbb F/\mathbb F_\bullet}\cap \mathbb F^\times$) if $N=2n$ (resp. $N=2n+1$). Its lifting is given by 
$$\pi(a,\lambda|_{{\boldsymbol{\mu}}(F)_1})\mapsto \tilde\pi( a , \tilde\lambda|_{{\boldsymbol{\mu}}_F},(-1)^{n-1}\lambda(-1)),$$ where  $\tilde\lambda|_{{\boldsymbol{\mu}}_F} = \lambda|_{{\boldsymbol{\mu}}(F)_1}\circ(1-c)$.

\subsubsection{Odd special orthogonal groups.} 
Let $G $ be the split $\SO_{2n+1}$ with symmetric matrix $H = \antidiag(1,-1,1\dots,-1,1)$. The simple affine roots of $G$ are $\alpha_i = e_{i}-e_{i+1}$ for $1\leq i\leq n-1$, $\alpha_n = e_n$, and $\alpha_0 = 1-\alpha_l = 1-(e_1+e_2)$. Here the root subgroup of $\alpha_0$ occupies the $(2n,1)$ and $(2n+1,2)$-entries, i.e., an element $u\in \mathcal{I}^+$ is of the form 
$$u = I+\begin{bmatrix}
&u_1&&
\\
&\cdot&\diag( u_2, \dots, u_{n},u_n,\dots,u_2)&
\\
u_0\varpi&\cdot&\cdot&u_1
\\
&u_0\varpi&&
\end{bmatrix}.
$$
 The character $\lambda$ maps $u_i\mapsto  \psi_{a_i}(u_i)$ for all $i\in \{0,1,\dots,n\}$, where $a_i\in \mathfrak{o}_F^\times$. The normalizer is $N(\lambda)= \Omega \mathcal{I}^+$. The order of $\bar\Omega$ is 2, and the non-trivial element in $\bar\Omega$ is represented by $\omega =  -\mathbbm p$, where
   $$ \mathbbm p= \antidiag ((a_0/a_1)\varpi^{-1},I_{2n-1},(a_1/a_0)\varpi) , $$
so that $\det \omega = 1$ and $\omega^2 =I$. Take a sign $\xi\in \{\pm 1\}$ and define the character
$$\lambda(\omega^j  u) = \xi^j\prod_{i}\psi_{a_i} (u_i),\,\qquad \text{ for }j\in \mathbb Z/2\text{ and } u\in \mathcal{I}^+.$$
Put $a = a_{0}a_{1}\prod_{i\geq 2} a_{i}^2\bmod \mathcal{U}^1(F)$. Two representations $\pi_{\lambda_1}$ and $\pi_{\lambda_2}$, where $\lambda_j = \lambda(\{a_{i}^j\}_i,\xi_j)$, are isomorphic if and only if $(a^1,\xi_1) = (a^2,\xi_2)$.

From  \cite{Oi-SO-odd}, an Oi representative takes the form  $(a,\lambda(\omega))\in \mathbb F^\times \times \{\pm 1\} $, where $a$ represents $(a_0,\dots,a_{n}) = (a,1,\dots,1)$, and its lifting is given by
\begin{equation}
\label{Oi's lifting odd orthogonal}
\pi(a,\lambda(\omega)) \mapsto \tilde \pi(2a,\mathbf 1_{{\boldsymbol{\mu}}_F}, \lambda(\omega)).
\end{equation}

\subsubsection{\bf Symplectic groups} 
Let $G = \SP_{2n}$ be defined by the involution $g\mapsto H{}^tg^{-1}H^{-1} $, where $H = \antidiag(1,-1,1,-1,\dots)$. The simple affine roots are $\alpha_i = e_i - e_{i+1}$ for  $1\leq i\leq n-1$, $\alpha_n = 2e_n$, and $\alpha_0 = 1-\alpha_l = 1-2e_1$. Here the root subgroup of $\alpha_0$ occupies the $(2n,1)$-entry, and so 
$$u  =I+\antidiag( 
\diag( u_1, \dots,  u_{n-1}, u_n, u_{n-1}, \dots, {u_1}), 
u_0\varpi )\in \mathcal I^+.$$
 Since $G$ is simply connected, the group $\Omega$ is trivial, and so the normalizer $N(\lambda) $ is $Z \mathcal{I}^+ = \{\pm 1\} \mathcal{I}^+$. We take a sign $\xi\in \{\pm 1\}$
and define the character 
$$\lambda((-1)^j  u) = \xi^j\prod_{i}\psi_{a_i} (u_i),\,\qquad \text{ for }j\in \mathbb Z/2\text{ and } u\in \mathcal{I}^+.$$
Put $a = a_{0}\left(\prod_{i=1}^{n-1} a_{i}^2\right)a_{n} \mod \mathcal{U}^1(F)$. Two representations $\pi_{\lambda_1}$ and $\pi_{\lambda_2}$, where $\lambda_j = \lambda(\{a_{i}^j\}_i,\xi_j)$ are isomorphic if and only if $(a^1,\left(\frac{a_{n}^1}{{\boldsymbol{\mu}}_F}\right),\xi_1) = (a^2,\left(\frac{a_{n}^2}{{\boldsymbol{\mu}}_F}\right),\xi_2)$.

By \cite[Th 7.17]{Oi-Sp-and-SO-even}, an Oi representative takes the form $(a,\kappa,\lambda(-1))\in \mathbb F^\times \times (\mathbb F^\times / \mathbb F^{\times2}) \times \{\pm 1\} $, where $(a,\kappa)$ represents $(a_0,\dots,a_{n}) = (a\kappa^{-1},1,\dots,1,\kappa)$. Note that we can take $\kappa\in \{1,\zeta\}$, and both representations $\pi(a\kappa^{-1},\kappa,\lambda(-1))$ lift to 
\begin{equation}
\label{Oi's result for symplectic groups}
\tilde\pi (4a,\left(\tfrac{\cdot}{{\boldsymbol{\mu}}_F}\right),\lambda(-1)\left(\tfrac{-1}{{\boldsymbol{\mu}}_F}\right)\mathfrak n_\psi)\oplus \chi_{F[\sqrt{(-1)^{n-1}a\varpi}]/F}
\end{equation}
where, if $E/F$ is a quadratic extension, $\chi_{E/F}$ is the character of $F^\times$ whose kernel is the image of $N_{E/F}$.

\subsubsection{Ramified even orthogonal groups}

Consider the ramified $G = \SO^{}_{2n,F[\sqrt\varpi]/F}$ defined by the Hermitian matrix
\begin{equation}
\label{H matrix for ramified even orthogonal group G-plus}
H = \antidiag(1,\dots,1,\diag(-\varpi, 1),1,\dots,1).
\end{equation}
The simple affine roots
are
$\alpha_i = e_{i}-e_{i+1}$, for $1< i\leq n-2$, $\alpha_{n-1}=e_{n-1}$, and $\alpha_0 = \tfrac{1}{2}-e_1$ for $n>2$. If $n=2$, then $\Delta^{\text{aff}} = \{e_1, 1-e_1\}$. We express elements in $\mathcal I^+$ as $u = u(u_0,\dots,u_{n-1})$, where
$$u  =I+
\begin{bmatrix}
&\diag(u_1,\dots,u_{n-2})&&&&
\\
&&\cdot&u_{n-1}&\cdot&
\\
u_0&\cdot&\cdot&\cdot&\cdot&
\\
&&\cdot&\cdot&-u_{n-1}&
\\
&&\cdot&&&-\diag(u_{n-2},\dots,u_1)
\\
&&u_0\varpi&&&
\end{bmatrix}.
$$
The normalizer is $N(\lambda)  = Z \mathcal{I}^+ = \{\pm 1\} \mathcal{I}^+$. Take $\xi\in \{\pm 1\}$ and define
$$\lambda((-1)^k u) = \xi^k\prod_{i=0}^{n-1}\psi_{a_i} (u_i),\,\qquad \text{ for }k=1,2\text{ and } u\in \mathcal{I}^+.$$
Put $a=\prod_{i=0}^{n-1} a_{i} \bmod \mathcal{U}^1(F)$. Two representations $\pi_{\lambda_1}$ and $\pi_{\lambda_2}$, where $\lambda_j = \lambda(\{a_{i}^j\}_i,	\xi_j)$, are isomorphic if and only if $(a^1,\xi_1 )= (a^2,\xi_2)$.

An outer automorphism of $G$ is represented by an element in $G^\sharp = \mathrm O_{2n}$ not in $G$. For example, we may take \begin{equation}
\label{p element in ramified even orthogonal groups}
     \mathbbm p  = \diag(I_{n-1},\diag(1,-1),I_{n-1}),
     \end{equation}
which conjugates $u(u_0,\dots,u_{n-2},u_{n-1})$ to $u(u_0,\dots,u_{n-2},-u_{n-1})$.

From \cite[Th 7.16]{Oi-Sp-and-SO-even}, an Oi representative takes the form $(a, \lambda(-1))$, where  $a\in \mathbb F^\times$ 
represents $(a_0,\dots,a_{n}) = (a,1,\dots,1)$. Both $\pi(a,\lambda(-1))$ and $\pi(-a,\lambda(-1))$, which are conjugate with each other by an outer automorphism of $G$, lift to \begin{equation}
\label{Oi's lifting ramified even orthogonal}
\tilde\pi((-1)^{n-1}a^2, \left(\tfrac{\cdot}{{\boldsymbol{\mu}}_F}\right), \mathfrak n_\psi \lambda(-1)).
\end{equation}

\subsubsection{Split or unramified even orthogonal groups}

 Let $G=\SO_{2n}$ or $G=\SO^{\mathrm{ur}}_{2n}$ (note that Oi considered $\SO_{2n+2}$ and $\SO^{\mathrm{ur}}_{2n+2}$ instead). The calculations for these two types of groups are similar (since they become isomorphic over the unramified quadratic extension of $F$), so we treat them together in a single subsection.

Henceforth, $G$ is an even orthogonal group defined by the Hermitian matrix
$$H = \antidiag(1,\dots,1,H',1,\dots,1),$$
where 
$$H'=\antidiag(1,1) \quad \text{if $G$ is split, }\quad \diag(-\zeta,1)\quad\text{if $G$ is unramified}.$$

Assuming $n\geq 3 $, the simple affine roots are given by
\begin{equation*}
   \begin{split}
    &\text{$\alpha_i = e_i - e_{i+1}$ for $1\leq i\leq n-1$, $\alpha_n = e_{n-1} +e_n  $, and $\alpha_0 = 1-\alpha_l = 1-(e_1+e_2)$}, 
    \text{ or }\\
       &\text{$\alpha_i = e_{i}-e_{i+1}$, for $1< i\leq n-2$, $\alpha_{n-1}=e_{n-1}$, and $\alpha_0 = 1-(e_1+e_2)$,}\end{split}
\end{equation*}
according to whether $G$ is split or unramified.

We express an element in $\mathcal I^+$ as
$$u  =I+
\begin{bmatrix}
&u_1&&
\\
&&\diag(u_2,\dots,u_{n-2}, u', -u_{n-2},\dots,-u_2)&
\\
u_0\varpi&\cdot&&-u_1
\\
&-u_0\varpi&&
\end{bmatrix}. 
$$
where 
$$u'=\begin{bmatrix}
u_{n-1}&u_{n}&
\\
&&-u_{n}
\\
&&-u_{n-1}\end{bmatrix} \quad \text{if $G$ is split, }\quad \begin{bmatrix}
u_{n}\zeta&u_{n-1}&
\\
&&u_{n}
\\
&&-u_{n-1}
\end{bmatrix}\quad\text{if $G$ is unramified}. $$
The order of $\Omega$ is 2, whose non-trivial element can be represented by  
\begin{equation*}
   \omega = \antidiag(-\tfrac{a_0}{a_1}\varpi^{-1}, \diag(I_{n-2},\omega',I_{n-2}),-\tfrac{a_1}{a_0}\varpi)
  ,\end{equation*}
where 
\begin{equation*}
   \begin{split}
     \omega'&=\antidiag\left(\frac{a_{n-1}}{a_n},\frac{a_n}{a_{n-1}}\right) \quad \text{if $G$ is split, }
     \\
     \quad& \frac{1}{a_n^2-a_{n-1}^2\zeta}\begin{bmatrix}
a_n^2+a_{n-1}^2\zeta& -2a_na_{n-1}
\\
2a_na_{n-1}\zeta & -(a_n^2+a_{n-1}^2\zeta)
\end{bmatrix}\quad\text{if $G$ is unramified},   \end{split}
\end{equation*}
 so that $\omega^2 =I$. Take two signs $\eta, \xi\in \{\pm 1\}$
and define
$$\lambda(\omega^j (-1)^k u) = \eta^j \xi^k\prod_{i}\psi_{a_i} (u_i),\,\qquad \text{ for }j,k\in \{1,2\}\text{ and } u\in \mathcal{I}^+,$$
i.e., the affine generic character $\lambda$ can be extended to the full normalizer $N(\chi) = \Omega \{\pm 1\} \mathcal{I}^+$. We put 
\begin{equation}
\label{a element in split or unramified orthogonal groups}
 a=a_{0}a_{1}\left(\textstyle\prod_{i=2}^{n-2} a_{i}^2\right)N_{a_{n-1},a_n},
\end{equation}
where 
 $$N_{a_{n-1},a_n}=a_{n-1}a_n\quad \text{if $G$ is split, }\quad a_{n-1}^2-a_{n}^2\zeta^{-1}\quad\text{if $G$ is unramified}.$$
Two representations $\pi_{\lambda_1}$ and $\pi_{\lambda_2}$, where $\lambda_j = \lambda(\{a_{i}^j\}_i,\xi_j,\eta_j)$, are isomorphic if and only if $(a^1,\left(\tfrac{N_{a^1_{n-1},a^1_n}}{{\boldsymbol{\mu}}_F}\right)$, $\xi_1,\eta_1)=(a^2,\left(\tfrac{N_{a^2_{n-1},a^2_n}}{{\boldsymbol{\mu}}_F}\right),\xi_2,\eta_2)$.

A non-trivial element in $G^\sharp = \mathrm O_{2n}$ not in $G$ can be represented by  
\begin{equation}
\label{p element in split or unramified orthogonal groups}
     \mathbbm p  = \diag(I_{n-1},\omega',I_{n-1}).
     \end{equation}
Whether $G$ is split or unramified,  $\mathbbm p$ stabilizes the character $\lambda$.

To express Oi's lifting result, we first put
\begin{equation}
\label{the unramification index of G}
u_G=0 \quad \text{if $G$ is split, }\quad 1\quad\text{if $G$ is unramified}.
\end{equation}
Depending on whether $G$ is split or unramified, an Oi representative takes the form 
\begin{equation*}
(a,\kappa,\lambda(-1),\lambda(\omega))\in 
   \begin{cases}
     \mathbb {\boldsymbol{\mu}}_F \times {\boldsymbol{\mu}}_F / {\boldsymbol{\mu}}_F^{2} \times \{\pm 1\} \times \{\pm 1\} , \\
        \mathbb {\boldsymbol{\mu}}_F \times {\boldsymbol{\mu}}_{F[\sqrt\zeta]} / ({\boldsymbol{\mu}}_F \ker N_{F[\sqrt\zeta]/F} )\times \{\pm 1\} \times \{\pm 1\} , \\
   \end{cases}
\end{equation*}
where $(a,\kappa)$ represents $(a_0,\dots,a_{n}) = (a\kappa^{-1},1,\dots,1,\kappa)$ or $(a(N_{F[\sqrt\zeta]/F}\kappa)^{-1},1,\dots,1,\kappa)$, and we can take $\kappa \in \{1,\zeta\}$ or $\in \{1,\zeta'\}$ where $N_{F[\sqrt\zeta]/F}\zeta' = \zeta$.

From  \cite[Th 8.7]{Oi-Sp-and-SO-even}, both representations $\pi( a,\kappa,\lambda(-1),\lambda(\omega))$, where $\kappa \in \{1,\zeta\}$ or $\in \{1,\zeta'\}$ as above, lift to 
\begin{equation}
\label{Oi's lifting split or unramified even orthogonal}
\tilde\pi((-1)^{n}2^{2-u_G}a, \left(\tfrac{\cdot}{{\boldsymbol{\mu}}_F}\right), 
(-1)^{u_G}\lambda(\omega)\left(\tfrac{-1}{{\boldsymbol{\mu}}_F}\right)\mathfrak n_\psi )
\times 
\tilde\chi_1\times 
\tilde\chi_{2},
\end{equation}
where both $\tilde\chi_1$ and  $\tilde\chi_2$ are tamely ramified characters of $F^\times$ satisfying
\begin{equation}
\label{Oi's result for split and unramified even orthogonal groups}
   \begin{split}
  \tilde \chi_1 \equiv \mathbf 1_{{\boldsymbol{\mu}}_F}, \quad \tilde\chi_1(\varpi) = \lambda(\omega);\quad \tilde \chi_2 \equiv \left(\frac{\cdot}{{\boldsymbol{\mu}}_F}\right), \quad 
     \tilde\chi_2(\varpi) = (-1)^{u_G}\left(\frac{-2^{u_G}a}{{\boldsymbol{\mu}}_F}\right)\lambda(\omega).
         \end{split}
\end{equation}

\subsection{Comparison with Oi's results}
\label{subsection Comparison with Oi's results}

We use our reducibility results in Propositions \ref{first general form of lifting, unitary group} and \ref{first general form of lifting} to determine the endoscopic liftings of simple supercuspidals, and show that our results agree with those of Oi. The groups considered by Oi are all quasi-split, so that $I_\zeta = \O$ in the simple supercuspidal case. When $G$ is a symplectic group, we also compare our results directly with those of \cite{BHS2} (see Subsection \ref{subsection Results in BHS2}).

\subsubsection{Unramified unitary groups}
\label{subsubsection Unramified unitary groups}

Our result in Proposition \ref{first general form of lifting, unitary group} shows that $\tilde{\boldsymbol{\lambda}}(\varpi_i) = \lambda(-1) \mathfrak n(\mathbf q) $ where $\mathbf q = \mathbf q_{z,\mathbf s}$ in Section \ref{subsection Non-degeneracy of a quadratic form}, while Oi's result gives $\tilde{\boldsymbol{\lambda}}(\varpi_i) = (-1)^{n-1}\lambda(-1)$ where $n=\dim_FV$. Hence our aim is to show that 
$$\mathfrak n(\mathbf q) = (-1)^{n-1}.$$ 
 Let $\zeta\in {\boldsymbol{\mu}}_{\Fo}$ be a generator, so that $F = \Fo[\sqrt\zeta]$ with $\overline{\sqrt\zeta} = -\sqrt\zeta$. As a quadratic form over $\mathbb F_\bullet$, the discriminant of $\mathbf q$ is equal to $(-\zeta)^{n-1}N_{\mathbb{F/F}_\bullet}(\det{}_{\mathbb F} \mathbf q)$, where $\det{}_{\mathbb F} $ is the discriminant map over $\mathbb F$. Proposition \ref{properties of quadratic forms}\ref{Hermitian symmetry of the quadratic form} implies that 
 $\det{}_{\mathbb F} \mathbf q\in \mathbb F_\bullet$. Hence 
$$\mathfrak n(\mathbf q) =\left(\frac{(-\zeta)^{n-1}(\det{}_{\mathbb F} \mathbf q)^2}{{\boldsymbol{\mu}}_\Fo}\right)\mathfrak n_\psi^{2n-2} = \left(\frac{\zeta}{{\boldsymbol{\mu}}_\Fo}\right)^{n-1} = (-1)^{n-1},$$
which is the desired result. We remark that this result applies to both forms $G_+$ and $G_-$.

\subsubsection{Ramified even orthogonal groups}
\label{subsubsection Ramified even orthogonal groups}

The ramified even orthogonal case is much simpler than the split and unramified cases, so we begin with it.

Let $G = G(V,h)$ be a ramified even orthogonal group, where the Hermitian form $h$ is defined by the matrix $H = H_+$ defined in (\ref{H matrix for ramified even orthogonal group G-plus}). The lattice sequence determined by 
\begin{equation*}
   \begin{split}
     &\Lambda(\tfrac{1}{2n}-\tfrac{1}{2}) = [\mathfrak{o}_F^{\oplus 2n}],\quad \dots,
     \quad 
     \Lambda(0) = [\mathfrak{o}_F^{\oplus n+1},\mathfrak{p}_F^{\oplus n-1}],
     \quad
     \Lambda(\tfrac{1}{2n}) = [\mathfrak{o}_F^{\oplus n},\mathfrak{p}_F^{\oplus n}]
     \\
& \Lambda(\tfrac{1}{n}) = [\mathfrak{o}_F^{\oplus n-2 },\mathfrak{p}_F,\mathfrak{o}_F,\mathfrak{p}_F^{\oplus n }],\quad \dots,
\quad \Lambda(\tfrac{1}{2}) = [\mathfrak{p}_F^{\oplus n-1},\mathfrak{o}_F,\mathfrak{p}_F^{\oplus n}],
    \end{split}
\end{equation*}
is self-dual. The affine generic functional $\beta = \beta(a_0,\dots,a_{n-2}, a_{n-1}) $ in this case is represented by 
$$ \frac{1}{2}\begin{bmatrix}
&&a_0&&&
\\
\diag(a_1,\dots,a_{n-2})&&\cdot&&&
\\
&\cdot&\cdot&\cdot&\cdot&a_0\varpi^{-1}
\\
&a_{n-1}&\cdot&\cdot&&
\\
&\cdot&\cdot&-a_{n-1}&&
\\
&&&&\diag(-a_{n-2},\dots,-a_{1})&
\end{bmatrix}.$$ 
Then $[\Lambda_{},0,2\beta]$ is equivalent to $[ \Lambda_{},0,\tilde \beta]$, where 
\begin{equation}
\label{beta in ramified even orthogonal group}
\tilde\beta  = \tilde \varpi^{-1} = \antidiag((-1)^{n-1}a^2\varpi^{-1},I_{2n-1})
\end{equation}
and $a=\prod_{i=0}^{n-1} a_{i} \bmod \mathcal{U}^1(F)$. If we take $\mathbbm p = \diag(I_{n-1},1,-1,I_{n-1})$ as in (\ref{p element in ramified even orthogonal groups}), then 
\begin{equation}
\label{action of p on beta, ramified even orthogonal}
\Ad(\mathbbm p)\beta(a_0,\dots,a_{n-2}, a_{n-1}) = \beta(a_0,\dots,a_{n-2}, -a_{n-1}).
\end{equation}
By comparing Oi's result (\ref{Oi's lifting ramified even orthogonal}) with ours in Proposition \ref{first general form of lifting}, we have to show that 
$$\mathfrak n(\mathbf q) = \left(\frac{-2}{{\boldsymbol{\mu}}_F}\right)\mathfrak n_\psi.$$
The bilinear form associated with $\mathbf q$, 
$$\mathbf h:(X,X')\mapsto  \tfrac{1}{2}\mathrm{tr}_{\tilde{\mathfrak g}/F}(( X - \tilde \beta^{-1} X\tilde \beta){}^\alpha X'),$$ 
with $X=\diag(x_1,\dots,x_{2n})\in \tilde{\mathfrak g} $ mod center,
expands as 
\begin{equation}
\label{quadratic form in the even ramified orthogonal case}
   \begin{split}
    \mathbf h:(X,X')\mapsto &\tfrac{1}{2}(
(x_n-x_{2n})x_1'  +(x_{1}-x_{n})x'_{n}
+(x_{n+2}-x_{n+1})x'_{n+1}+
(x_{n+1}-x_{n-1})x'_{n+2}
\\
&+
\left(\sum_{i=2}^{n-1} +\sum_{i=n+3}^{2n}\right) (x_{2n+2-i}-x_{2n+1-i})x_i' ).
\end{split}
\end{equation}
Modulo the 1-dimensional radical, i.e., the center of $\tilde{\mathfrak g}$, it has discriminant $2^{2n-1}(-1)^{n}$. Hence 
$$\mathfrak n(\mathbf q) =  \left(\frac{2^{2n-1}(-1)^{n}}{{\boldsymbol{\mu}}_F}\right)\mathfrak n_\psi^{2n-1}  = \left(\frac{-2}{{\boldsymbol{\mu}}_F}\right)\mathfrak n_\psi$$
which is the desired result.

\begin{rmk}
The above result applies to $G_+$ defined by the Hermitian form in (\ref{H matrix for ramified even orthogonal group G-plus}). The result for $G_-$ is entirely analogous, in which case we take $H = \antidiag(1,\dots,1,\zeta\diag(-\varpi, 1),1,\dots,1)$, so that an affine unipotent element is of the form
$$u  =I+
\begin{bmatrix}
&\diag(u_1,\dots,u_{n-2})&&&&
\\
&&\cdot&u_{n-1}&\cdot&
\\
u_0&\cdot&\cdot&\cdot&\cdot&
\\
&&\cdot&\cdot&-u_{n-1}\zeta^{-1}&
\\
&&\cdot&&&-\diag(u_{n-2},\dots,u_1)
\\
&&u_0\zeta\varpi&&&
\end{bmatrix},
$$
and the affine generic functional $\beta $ is represented by 
$$ \frac{1}{2}\begin{bmatrix}
&&a_0&&&
\\
\diag(a_1,\dots,a_{n-2})&&\cdot&&&
\\
&\cdot&\cdot&\cdot&\cdot&a_0\zeta^{-1}\varpi^{-1}
\\
&a_{n-1}&\cdot&\cdot&&
\\
&\cdot&\cdot&-a_{n-1}&&
\\
&&&&\diag(-a_{n-2},\dots,-a_{1})&
\end{bmatrix}.$$ 
The remainder of the calculation is completely analogous to the case when $G=G_+$.
\qed\end{rmk}

\subsubsection{Odd orthogonal groups}
\label{subsubsection Odd orthogonal groups}

The affine generic functional $\beta$ defining a character of $\mathcal I^+$ is represented by $$\tfrac{1}{2}\begin{bmatrix}
&&a_0\varpi^{-1}&
\\
a_1&&\cdot &a_0\varpi^{-1}
\\
&\diag(a_2,\dots,a_{n},a_n,\dots,a_2)&&
\\
&&a_1&
\end{bmatrix}.$$
Put $a = a_0a_1\prod_{i\geq 2}a_i^2$ and define, with $I=\{i,o\}$, $$\beta_i = \varpi_i^{-1}  = \antidiag(a(2^{2n-3}\varpi)^{-1},I_{2n-1})\in \tilde {\mathfrak g}(V_{i}),$$ 
and $\beta_o=0\in F$. The stratum $[\Lambda_{}, 0,\beta]$ is equivalent to $[\Lambda_i\oplus \Lambda_o,0,\beta_i\oplus\beta_o]$ (note that  $\beta_i\oplus\beta_o$ has the same characteristic polynomial as $\beta$), and we take $\tilde \beta =2\beta_i$ to define $\tilde{\boldsymbol{\lambda}}_i|_{\mathcal U^{0_+}(\Lambda_i)}$.

To match the calculation (\ref{beta in ramified even orthogonal group}) for the even orthogonal group case, we modify our uniformizer by defining $\varpi'$ such that 
$$a(2^{2n-3}\varpi)^{-1}=(-1)^{n-1}b^2\varpi'^{-1} \mod \mathcal U^1(F) $$
 for some $b\in {\boldsymbol{\mu}}_F$, and take \begin{equation*}
\label{odd orthogonal group, H1 and Ho}
H_i = \antidiag(1,\dots,1,\diag(-\varpi', 1),1,\dots,1),\quad H_o=[\varpi'].
\end{equation*}
Comparing (\ref{Oi's lifting odd orthogonal}) with our result in Proposition  \ref{first general form of lifting}, our aim is to show that $\tilde{\boldsymbol{\lambda}}(\varpi_i) = \lambda (\omega)$, where $\omega\in G(V,h)$ corresponds to $\omega_i = \diag(-I_{2n},1)\in G(V,h_{H_i\oplus H_o})$. It suffices to prove that 
$$\mathfrak{n}(\mathbf q)=1,$$
where, with $H_i\oplus H_o$ defined as above, we have $\mathbf q = \mathbf q_{i}\oplus \mathbf q_{o}$, where $\mathbf q_{i}$ is defined as in (\ref{quadratic form in the even ramified orthogonal case}), and $\mathbf q_{o}(X^0)=x_0^2$. 
Here $X^i = \diag(x_1,\dots,x_{2n})\in\tilde {\mathfrak g}(V_{i})$ and $X^0\in \End(V_o,V_{i})$ is a rank-1 column vector depending on a single parameter $x_0\in \mathbb F$. The discriminant of $\mathbf q_{i}$ was computed in the previous Subsection \ref{subsubsection Ramified even orthogonal groups} for even ramified orthogonal groups, which is $2^{2n-1}(-1)^{n}$ (with the modified uniformizer $\varpi'$), while that of $\mathbf q_{o}$ is $2\bmod \mathbb F^{\times2}$. Hence, the discriminant of $\mathbf q$ is $2^{2n-1}(-1)^{n}\cdot2 \equiv (-1)^{n}\bmod \mathbb F^{\times2}$, and therefore
$$ \mathfrak{n}(\mathbf q) = \left(\frac{(-1)^{n}}{{\boldsymbol{\mu}}_F}\right)\mathfrak n_\psi^{2n}=1$$
which is the desired result.

\subsubsection{Symplectic groups with $\dim \tilde V > 1$}
\label{subsubsection Symplectic groups comparison}

We first consider the case $\dim \tilde V = \dim V$. The affine generic functional $\beta $ is represented by 
\begin{equation}
\label{beta element for symplectic groups}
\tfrac{1}{2}\antidiag(2a_0\varpi^{-1}, \diag(a_1,\dots,a_{n-1},2a_n, a_{n-1},\dots,a_1)).
\end{equation}
Putting $a = a_0(\prod_{i=1}^{n-1}a_i^2)a_n$, the stratum $[\Lambda_{},0,2\beta]$ is equivalent to $[\Lambda,0,\tilde\beta]$, where 
$$\tilde\beta = \tilde\varpi ^{-1} = \antidiag(a(2^{2n-2}\varpi)^{-1},I_{2n-1}).$$  
Our result in Proposition \ref{first general form of lifting} implies that the lifting is $\tilde\pi (4a,\left(\tfrac{\cdot}{{\boldsymbol{\mu}}_F}\right), \left(\tfrac{-2}{{\boldsymbol{\mu}}_F}\right)
\lambda (-1)\mathfrak{n}(\mathbf q))$. Compare this with (\ref{Oi's result for symplectic groups}), it suffices to show that 
$$\mathfrak{n}(\mathbf q) = \left(\frac{2}{{\boldsymbol{\mu}}_F}\right)\mathfrak n_\psi.$$
The bilinear form associated to $\mathbf q$,
$$(X,X')\mapsto  \tfrac{1}{2}\mathrm{tr}_{\tilde{\mathfrak g}/F}(( X - \tilde \beta^{-1} X\tilde \beta){}^\alpha X'),$$ 
for $X=\diag(x_1,\dots,x_{2n})\in \mathbb F^{\oplus 2n}\bmod \Delta\mathbb F$, expands as 
$$ \frac{1}{2}\left( (x_{1}-x_{2n})x_1' 
+\sum_{i=2}^{2n}(x_{2n+2-i}-x_{2n+1-i})x_{i}'\right).$$
Modulo the 1-dimensional radical, its discriminant is $2^{2n-1}(-1)^{n-1} = 2(-1)^{n-1}\bmod(\mathbb F^\times)^2$. Its normalized Gauss sum is 
\begin{equation*}
\label{Gauss sum appearing in tripling method, SP-case}
\mathfrak{n}(\mathbf q)=\left(\frac{2(-1)^{n-1}}{{\boldsymbol{\mu}}_F}\right)
\mathfrak{n}_\psi^{2n-1}  = \left(\frac{2}{{\boldsymbol{\mu}}_F}\right)\mathfrak{n}_\psi
\end{equation*}
which is the desired result.

\subsubsection{Symplectic groups with $\dim \tilde V = 1$}
\label{subsubsection Symplectic groups comparison, i=o}

We next consider the case $\dim \tilde V = 1$. The character of $\tilde G = \GL_1(F)$ suggested by Oi is the central character $\tilde\chi$ of the above lifting, which is trivial on ${\mathcal U^1(F)}$ and satisfies
\begin{equation}
\label{lifting character for symplectic groups}\tilde\chi|_{{\boldsymbol{\mu}}_F} = \left(\frac{\cdot}{{\boldsymbol{\mu}}_F}\right)
\quad \text{and} \quad
\tilde\chi(\varpi) = \tilde{\boldsymbol{\lambda}}((4a)^{-1}\varpi_i^{2n}) =  \left(\frac{a}{{\boldsymbol{\mu}}_F}\right)\mathfrak n_\psi^{2n} = \left(\frac{(-1)^na}{{\boldsymbol{\mu}}_F}\right).
\end{equation}
This is the character associated with the ramified quadratic extension $F[\sqrt{(-1)^{n-1}a\varpi}]/F$ appearing in (\ref{Oi's result for symplectic groups}).

Recall that our Hermitian space is defined by the matrix $H = \antidiag(1,-1,1,-1,\dots)$, and in this case $\dim \tilde V  =1$. The operator $\alpha$ on $X$ is given by 
$$[x_1,\dots,x_{2n}] \mapsto {}^t[x_{2n},-x_{2n-1},\dots,x_{2},-x_1],$$
and on $Y$ by identity, so that the relation $X{}^\alpha X = Y- {}^\alpha Y$ reduces to $0=0$, i.e., there is no relation between $X$ and $Y$.

The maximal lattice sequence $\Lambda$ in $V$ determined by 
$$\Lambda(0) = [\mathfrak{o}_F^{\oplus n},  \mathfrak{p}_F^{\oplus n} ],\quad  \Lambda(\tfrac{1}{2n})= [\mathfrak{o}_F^{\oplus n-1},  \mathfrak{p}_F^{\oplus n+1} ],\cdots$$
is self-dual, and we take 
$$\Lambda_- (0) = \mathfrak{o}_F
\quad\text{and}\quad
 \Lambda_- (0_+) = \Lambda_- (1) = \mathfrak{p}_F.$$
We take the null stratum $[\{\mathfrak{p}^k\}_k,0,0]$ in $F$. Then the group ${\mathcal J}_P$ can be expressed as
\begin{equation*}
\label{JP-group, with dim V=1, symplectic case}
\begin{bmatrix}
{\boldsymbol{\mu}}_F+ \mathfrak{p}_F&[\mathfrak{p}_F^{\oplus n},  \mathfrak{o}_F^{\oplus n} ]&\mathfrak{o}_F
\\
{}^t [\mathfrak{p}_F^{\oplus (n+1)},  \mathfrak{o}_F^{\oplus (n-1)} ]&{\boldsymbol{\mu}}_F + \mathcal{I}^+&{}^t[\mathfrak{o}_F^{\oplus n},  \mathfrak{p}_F^{\oplus n} ]
\\
\mathfrak{p}_F&[\mathfrak{p}_F^{\oplus (n+1)},  \mathfrak{o}_F^{\oplus (n-1)} ]&{\boldsymbol{\mu}}_F + \mathfrak{p}_F
\end{bmatrix}.
\end{equation*}
We now compute the coefficients $b_w$. For $w=y$, we have
$$X\in[\mathfrak{p}_F^{\oplus n},  \mathfrak{o}_F^{\oplus n} ]/[\mathfrak{p}_F^{\oplus (n+1)},  \mathfrak{o}_F^{\oplus (n-1)} ]\cong \mathfrak{o}_F/\mathfrak{p}_F
\qquad\text{and}\qquad
Y\in \mathfrak{o}_F^\times/\mathfrak{p}_F.
$$
For $X$ represented by $[0^n,x_{n+1},0^{n-1}]$, we have $-\mathrm{tr}(\beta{}^\alpha X Y^{-1} X) = (-1)^na_nx_{n+1}^2Y^{-1}$, so that
$$b_y \tilde T_y(s_y)= \sum_{Y}\tilde{\lambda}(Y) \sum_{x_{n+1}}\psi((-1)^na_nx_{n+1}^2Y^{-1})=\left(\frac{(-1)^n a_n}{{\boldsymbol{\mu}}_F}\right)q^{1/2}\mathfrak n_\psi(q-1)$$
when $\tilde\lambda |_{{\boldsymbol{\mu}}_F} $ is quadratic, in which case \begin{equation*}
c_y = q^2, \quad r_y = 1,
\quad\text{and}\quad 
\epsilon_y T_y(s_y)= \left(\frac{(-1)^n a_n}{{\boldsymbol{\mu}}_F}\right)\mathfrak n_\psi.
\end{equation*}
The calculation for $w=z$ is similar. We have
\begin{equation*}
\label{definition of P_(F,Lambda)}
X\in \varpi^{-1}[\mathfrak{p}_F^{\oplus (n+1)},  \mathfrak{o}_F^{\oplus (n-1)} ]/[\mathfrak{p}_F^{\oplus n},  \mathfrak{o}_F^{\oplus n} ] \quad\text{and}\qquad
Y = \varpi^{-1}Y_0\in \mathfrak{p}_F^{-1}/\mathfrak{o}_F.
\end{equation*}
For $X$ represented by $[x_1,\dots,x_{n},0,x_{n+2}\varpi^{-1},\dots,x_{2n}\varpi^{-1}]$, we obtain
$$\mathrm{tr}(-\beta{}^\alpha X Y^{-1} X )= (a_0x_1^2+2(-a_1x_2x_{2n}+\cdots+(-1)^{n-1}a_{n-1}x_nx_{n+2}) )Y_0^{-1}.$$ This form has discriminant $a_0 = aa_n$ mod square, and so
$$b_z \tilde T_z(s_z)= \sum_{Y_0}\tilde{\lambda}(Y_0) \sum_{X}\psi(\mathrm{tr}(-\beta{}^\alpha X Y_0^{-1} X ))=q^{(2n-1)/2}(q-1)\left(\frac{aa_n}{{\boldsymbol{\mu}}_F}\right)\mathfrak n_\psi $$
again when $\tilde\lambda |_{{\boldsymbol{\mu}}_F} $ is quadratic, in which case 
\begin{equation*}
c_z = q^{2n}, \quad r_z = 1,
\quad\text{and}\quad 
\epsilon_z T_z(s_z)= \left(\frac{aa_n^{}}{{\boldsymbol{\mu}}_F}\right)\mathfrak n_\psi.
\end{equation*}
Therefore, 
\begin{equation*}
\tilde\lambda|_{{\boldsymbol{\mu}}_F}  = \left(\frac{\cdot}{{\boldsymbol{\mu}}_F}\right)
\quad\text{and}\quad
\tilde{\boldsymbol{\lambda}}(\varpi)= \left(\frac{(-1)^{n}a_n(aa_n^{})}{{\boldsymbol{\mu}}_F}\right)
\mathfrak{n}_\psi^{2} =\left(\frac{(-1)^{n-1}a}{{\boldsymbol{\mu}}_F}\right) = \left(\frac{\varpi\det \beta}{{\boldsymbol{\mu}}_F}\right)
\end{equation*}
which is precisely $\chi_{F[\sqrt{(-1)^{n-1}a\varpi}]/F}$.

\subsubsection{Comparison with results in \cite{BHS2}}
\label{subsection Results in BHS2}

Let $\tilde\chi$ be our character in (\ref{lifting character for symplectic groups}). The character $\epsilon_1$ computed in \cite[Sec 3.4]{BHS2} is tamely ramified, quadratic on ${\boldsymbol{\mu}}_F$, and maps $\varpi$ to $1$. We easily see that $\tilde\chi|_{\mathfrak o_F^\times} = \epsilon_1|_{\mathfrak o_F^\times}$.

We put
 $$\beta^{\mathrm{BHS}} = \antidiag(\varpi^{-1},\diag(-1,1,\dots,(-1)^n,,\dots,1,-1)).$$
Note that this differs slightly from the $\beta$-element in \cite[Sec 2.2]{BHS2}, since their underlying symplectic form is defined by $\antidiag(J,-J)$, where $J = \antidiag(1,\dots,1)$, which is not the same as ours. Using our expression in (\ref{beta element for symplectic groups}), we have $a_0=1$ and $a_n=(-1)^n$ for $\beta^{\mathrm{BHS}} $.

Our character takes the value $\tilde\chi((-1)^na_0a_n\varpi)=1$, and so  
$$\tilde\chi(\varpi) =\left(\frac{(-1)^na_0a_n}{{\boldsymbol{\mu}}_F}\right) =\left(\frac{(-1)^n\cdot(-1)^n\cdot 1}{{\boldsymbol{\mu}}_F}\right)=1,$$ which is equal to $\epsilon_1(\varpi)$.

We then examine the inducing types. Our type  $\tilde{\boldsymbol{\lambda}}$ in Proposition \ref{first general form of lifting} is quadratic on ${\mathfrak o_{F[\beta]}^\times}$, which coincides with $\tau|_{{\mathfrak o_{F[\beta]}^\times}}$ in \cite[Sec 4.6]{BHS2}.

Our type takes the value $\tilde{\boldsymbol{\lambda}}(\varpi_i) = \lambda(-1)\left(\frac{-1}{{\boldsymbol{\mu}}_F}\right)\mathfrak{n}_\psi$. The uniformizer $\varpi^{\mathrm{BHS}}$ is $1/\beta^{\mathrm{BHS}}$, which equals $2\varpi_i$ in our notation. We have
$$\tilde{\boldsymbol{\lambda}}(\varpi^{\mathrm{BHS}}) = \lambda(-1)\left(\frac{-2}{{\boldsymbol{\mu}}_F}\right)\mathfrak{n}_\psi,$$
which is equal to $\tau(\beta^{-1})$ in \cite[Prop 4.14]{BHS2}.

We hence conclude that our liftings are the same as those obtained in \cite{BHS2}, by a direct comparison.

\subsubsection{Split or unramified even orthogonal groups with $i\neq o$}

The affine generic functional $\beta$ is represented by $$  \frac{1}{2}\begin{bmatrix}
&&a_0\varpi^{-1}&
\\
a_1&&\cdot&-a_0\varpi^{-1}
\\
&\diag(a_2,\dots,a_{n-2}, \beta', -a_{n-2},\dots,-a_{2})&&
\\
&&-a_1&
\end{bmatrix},$$
where 
$$\beta'=\begin{bmatrix}
a_{n-1}&&\\
a_{n}&&
\\
&-a_{n}&-a_{n-1}
\end{bmatrix} \quad \text{if $G$ is split, }\quad \begin{bmatrix}
a_{n}&&
\\
a_{n-1}&&
\\
&a_{n}\zeta^{}&-a_{n-1}
\end{bmatrix}\quad\text{if $G$ is unramified}. $$
The outer automorphism  $\mathbbm p $ given in (\ref{p element in split or unramified orthogonal groups}) acts on $\beta$ by switching the entries $a_{n-1}$ and $a_n$ if $G$ is split 
and by conjugation $(a_{n-1},a_n) \mapsto (a_{n-1},-a_n) $ if $G$ is unramified. The characteristic polynomial of $\beta$ is 
$$T^{2n}-(-1)^{n}a(2^{2-u_G} \varpi)^{-1}T^2,$$
where $a$ is given in (\ref{a element in split or unramified orthogonal groups}) and $u_G$ in (\ref{the unramification index of G}). Hence, $[\Lambda_{},0,2\beta]$ is equivalent to $[\Lambda_{i}\oplus \Lambda_o,0,\tilde\beta\oplus (0,  0)]$, where $\Lambda_{i}$ and $\Lambda_{o}$ are self-dual lattice sequences in $V_{i}$ and $V_{o}$, respectively, with $\dim V_{i}=2n-2$ and $\dim V_o = 2$, and 
$$\tilde\beta = \varpi_i ^{-1} = \antidiag((-1)^{n}a(2^{2-u_G}\varpi)^{-1},I_{2n-3}).$$  
As in the odd orthogonal case in Subsection \ref{subsubsection Odd orthogonal groups}, to match the calculation (\ref{beta in ramified even orthogonal group}) for the even ramified orthogonal case in Subsection \ref{subsubsection Ramified even orthogonal groups}, we modify our uniformizer by defining $\varpi'$ such that
\begin{equation}
\label{Modified uniformizer}
(-1)^{n}a(2^{2-u_G} \varpi)^{-1}=(-1)^{n-1}b^2\varpi'^{-1}\mod \mathcal U^1(F) 
\end{equation}
for some $b\in {\boldsymbol{\mu}}_F$, and take 
$$H_1 = \antidiag(1,\dots,1,\diag(-\varpi', 1),1,\dots,1),\quad H_o=\diag(-\zeta^{u_G}\varpi',1).$$
Here $H_o$ is placed at the center of $H_1$, i.e., $H = \antidiag(1,\dots,1,\diag(-\varpi', -\zeta^{u_G}\varpi', 1,1),1,\dots,1)$.

For $i\neq o$, comparing (\ref{Oi's lifting split or unramified even orthogonal}) with our result in Proposition  \ref{first general form of lifting}, it suffices to show that
$$
\mathfrak{n}(\mathbf q) = (-1)^{u_G}\left(\frac{2}{{\boldsymbol{\mu}}_F}\right) \mathfrak n_\psi
.$$ 
We express $\mathbf q$ as the sum of two quadratic forms: $\mathbf q_{i}$ defined as in the even ramified orthogonal case in (\ref{quadratic form in the even ramified orthogonal case}) but with $\dim_{\mathbb F} = 2n-3$, and $\mathbf q_{o}$ with $\dim_{\mathbb F} = 2$. The associated bilinear form is 
$(X,X')\mapsto \mathbf h_{i}(X,X')+ \mathbf h_{o}(X,X')$, where $\mathbf h_{i}$ is as in (\ref{quadratic form in the even ramified orthogonal case}) and 
$\mathbf h_{o} = -\tfrac{1}{2}(\zeta^{u_G}x_0x_0'+y_0y_0')$, and we have
 $$\mathrm{disc}(\mathbf q_{o}) = \zeta^{u_G}
\quad\Rightarrow\quad \mathfrak n(\mathbf q_{o}) = (-1)^{u_G}\left(\frac{-1}{{\boldsymbol{\mu}}_F}\right). $$
Recall that $ \mathfrak n(\mathbf q_{i})= \left(\tfrac{-2}{{\boldsymbol{\mu}}_F}\right) \mathfrak n_\psi$. Therefore, 
$$\mathfrak{n}(\mathbf q) =   \mathfrak n(\mathbf q_{i}) \mathfrak n(\mathbf q_{o})= (-1)^{u_G}\left(\frac{2}{{\boldsymbol{\mu}}_F}\right) \mathfrak n_\psi$$
which is the desired result.

\subsubsection{Split or unramified even orthogonal groups with $i=o$}

 We continue from the previous subsection, but now with $i=o$. The calculation is analogous to Section \ref{subsubsection Example 2: ramified SO(2)}; although $\dim V_o =2$, we must take a self-dual tamely ramified character $\tilde{\boldsymbol{\lambda}}$ of $\tilde G(F) = F^\times$, which contains the null stratum $[\{\mathfrak{p}^k\}_k,0,0]$ in $F$ as its maximal simple stratum (rather than an epipelagic stratum of dimension 2).

The operator $\alpha$ on $X$ is given by
\begin{equation*}
   \begin{split}
&[x_1,\dots,x_{n-2},
(x_{n-1}\varpi,x_{n}\varpi,x_{n+1},x_{n+2}),
x_{n+3},\dots,x_{2n}]
\\
&\mapsto {}^t[-x_{2n},,\dots,-x_{n+3},(\varpi'^{-1} x_{n-1},\zeta^{-u_G}\varpi'^{-1}x_{n},-x_{n+1},-x_{n+2}),-x_{n-2},\dots,-x_1],
   \end{split}
\end{equation*}
and on $Y$ by minus-identity. Here, $\varpi'$ is the modified uniformizer defined in (\ref{Modified uniformizer}). We put
$$\mathbbm p_o = \diag(I_{n-1},\diag(1,-1),I_{n-1})
\quad \text{and}\quad
\omega_o = \diag(I_{n-1},-I_2,I_{n-1}).$$
 When $w=y$, representing $X = [0_{n-1},(0,x_{n+1}),0_{n-1}]$, we have $2Y = -x_{n+1}^2$. 
One can show that $I - {}^\alpha X Y^{-1}X $ is $G(F)$-conjugate to $\mathbbm p_o$, and 
$$b_y\tilde T_y(s_y) = \sum_{-2Y = x_{n+1}^2}\tilde\lambda (Y) \lambda( \mathbbm p_o^2) = \tilde\lambda(-2) (q-1),$$
which implies that
\begin{equation*}
c_y = q , \quad 
r_y = 1,
\quad\text{and}\quad 
\epsilon_y T_y(s_y)= \tilde\lambda(-2).
\end{equation*}
When $w=z$, with representatives $X = [x_1,\dots,x_{n-1},
(x_{n},0),
x_{n+2}\varpi^{-1},\dots,x_{2n}\varpi^{-1}]$, we have 
\begin{equation*}
\label{relation in split or unramified even orthogonal groups, w=z, i=o}
2Y\varpi^{-1} = X{}^\alpha X =-\varpi^{-1}\zeta^{-u_G}x_n^2.
\end{equation*}
Regardless of the Hermitian form, one can show that indeed the trace of $\beta_j{}^\alpha X^jX^j$ is 0, which means that $\mathbf q_{j}$ is trivial. Now $I - {}^\alpha X Y^{-1}X $ is $G(F)$-conjugate to $\omega_o \mathbbm p_o$. Hence, 
$$b_z\tilde T_z(s_z) = \sum_{2Y =-\zeta^{-u_G}x_n^2.}\tilde\lambda (Y) \lambda(\omega_o \mathbbm p_o^2) = \tilde\lambda(\zeta)^{u_G}\tilde\lambda(- 2)\lambda(\omega_o) (q-1),$$
which implies that
\begin{equation*}
c_z = q, \quad 
r_z = 1,
\quad\text{and}\quad 
\epsilon_z T_z(s_z)= \tilde\lambda(\zeta)^{u_G}\tilde\lambda(- 2)\lambda(\omega_o).
\end{equation*}
We see that the deduction is independent of $\tilde\lambda|_{{\boldsymbol{\mu}}_F}$, which gives two tamely ramified characters ${\tilde{\boldsymbol{\lambda}}}_1$ and ${\tilde{\boldsymbol{\lambda}}}_2$ extending $\tilde\lambda$: 
 \begin{equation}
 \label{split or unramified orthogonal groups, the two charaacters}
   \begin{split}
      \tilde \lambda_1 \equiv \mathbf 1_{{\boldsymbol{\mu}}_F}& \quad\Rightarrow \quad {\tilde{\boldsymbol{\lambda}}}_1(\varpi') = \lambda(\omega_o); \\
       \tilde \lambda_2 \equiv \left(\frac{\cdot}{{\boldsymbol{\mu}}_F}\right)
       & \quad\Rightarrow \quad {\tilde{\boldsymbol{\lambda}}}_2(\varpi') = (-1)^{u_G}\lambda(\omega_o),
   \end{split}
\end{equation}
giving the required results in (\ref{Oi's result for split and unramified even orthogonal groups}).

\subsection{Ramified unitary groups}
\label{subsection Ramified unitary groups}

The calculation applies equally well to ramified unitary groups, the type of groups not covered in Oi's series. Let $G=\mathrm U_{N}(F/\Fo)$ where $F = \Fo[\varpi]$ with $\varpi^2=-\varpi_\bullet$. 

\subsubsection{Odd case}

If $N=2n+1$, we take $H=\antidiag(1,-1,1,\dots,1)$. An 
element $u\in \mathcal{I}^+$ takes the form 
$$u = I+\antidiag( \diag(u_1,\dots,u_{n},u_{n},\dots,u_1), u_0\varpi), \quad u_0,u_1,\dots,u_{n}\in {\boldsymbol{\mu}}_F,$$
and the affine generic functional is represented by
$$\beta = \tfrac{1}{2}\antidiag(2a_0\varpi^{-1},\diag(a_1,\dots,a_{n},{a_{n}},\dots,{a_1})),\quad a_0,a_1,\dots,a_{n}\in {\boldsymbol{\mu}}_F.$$
The normalizer is $N(\chi) =  \{\pm 1\} \mathcal{I}^+$. Take a sign $\xi\in \{\pm 1\}$
and define
$$\lambda( (-1)^k u) =  \xi^k\prod_{i}\psi_{a_i} (u_i),\,\qquad \text{ for }k=1,2\text{ and } u\in \mathcal{I}^+.$$
Putting
$ a=a_{0}(\textstyle\prod_{i=1}^{n} a_{i}^2)
$, two representations $\pi_{\lambda_1}$ and $\pi_{\lambda_2}$, where $\lambda_j = \lambda(\{a_{i}^j\}_i,\xi_j)$ for $i\in\{1,2\}$, are isomorphic if and only if $(a^1,\xi_1)=(a^2,\xi_2)$.

The quadratic form $\mathbf q = \mathbf q_{z,\mathbf s}$ is on a space of dimension $2n$ and has discriminant $(-1)^n$. The normalized Gauss sum is $\mathfrak n(\mathbf q)=\left(\tfrac{(-1)^n}{{\boldsymbol{\mu}}_F}\right)\mathfrak n_\psi^{2n}=1$. If we take $\varpi_i = \antidiag(I_{2n},2a\varpi^{-1})$, 
\begin{equation*}
\tilde \lambda|_{\tilde {\mathcal I}^+} =\psi_{2\beta} , \quad \tilde \lambda|_{{\boldsymbol{\mu}}_{F}}= \mathbf 1_{{\boldsymbol{\mu}}_F},
\quad\text{and}\quad
\tilde{\boldsymbol{\lambda}}(\varpi_i)= \tilde\lambda(-2)\lambda(-1),
\end{equation*}
then the lifting is given by 
$$\pi(a,\lambda(-1)) \mapsto \tilde \pi(2a,\mathbf 1_{{\boldsymbol{\mu}}_F}, \lambda(-1)).$$ 

\subsubsection{Even case}

If $N=2n$, we take $H=\antidiag(1,-1,1,\dots,-1)$. An 
element $u\in \mathcal{I}^+$ takes the form 
$$u = I+\antidiag(\diag(u_1,\dots,u_{n},\dots,u_1),u_0\varpi_\bullet),
$$
and the affine generic functional is represented by
$$\beta = 
\tfrac{1}{2}\antidiag(2a\varpi_\bullet^{-1}, \diag(a_1,\dots,a_{n-1},2a_n,a_{n-1},\dots,a_1)).$$
The normalizer is $N(\chi) \cong \{\pm 1\}^2  \mathcal{I}^+$. Take a pair of signs $\xi, \eta\in \{\pm 1\}$
and define
$$\lambda( ((-1)^{k_1},(-1)^{k_2}) u) =  \xi^{k_1} \eta^{k_2}\prod_{i}\psi_{a_i} (u_i),\,\qquad \text{ for }k_1,k_2\in \{0,1\}\text{ and } u\in \mathcal{I}^+.$$
Putting 
$ a=a_{0}(\textstyle\prod_{i=1}^{n-1} a_{i}^2)a_n
$, two representations $\pi_{\lambda_1}$ and $\pi_{\lambda_2}$, where $\lambda_j = \lambda(\{a_{i}^j\}_i,\xi_j,\eta_j)$ for $i\in\{1,2\}$, are isomorphic if and only if $(a^1,\xi_1,\eta_1)=(a^2,\xi_2,\eta_2)$.

The characteristic polynomial of $\beta$ is $T^{2n}-a(2^{2n-2}\varpi)^{-1}$. We hence take $I=\{j,o\}$ and the Hermitian forms $H_j=\varpi\antidiag(1,-1,1,-1\dots,1)$ and $H_o=[\varpi]$. Put $\beta_j = \tfrac{1}{2}\antidiag(a\varpi^{-1},I_{2n-2})$ and $\beta_o=0$, and define $\tilde\beta = 2[\beta_j,\beta_o]$. %

When $i=j$, $\mathbf q_{j}$ is the form $\mathbf q$ in the odd case (with $n$ replaced by $n-1$), and $\mathbf q_{o}= -x_o^2$. We hence have $\mathfrak n(\mathbf q)=\left(\tfrac{-1}{{\boldsymbol{\mu}}_F}\right)\mathfrak n_\psi^{}$, and define \begin{equation*}
\tilde \lambda_j|_{\tilde {\mathcal I}^+} =\psi_{2\beta_j} , \quad \tilde \lambda_j|_{{\boldsymbol{\mu}}_{F}}= \left(\frac{\cdot}{{\boldsymbol{\mu}}_F}\right)
,\quad\text{and}\quad
\tilde{\boldsymbol{\lambda}}_j(\varpi_j)= \lambda(\omega_j)\left(\frac{2}{{\boldsymbol{\mu}}_F}\right)\mathfrak n_\psi. 
\end{equation*}
If $i=o$, then $\mathbf q_{o}=1$ since $\mathfrak W_{z,o}^o$ is trivial, while  
$$\mathbf q_{j}=-\varpi X^j \beta_j {}^\alpha X^j = (\det \beta_j) x_1^2, \quad 
X^j\in \mathfrak W^o_{z,j}.$$
Therefore, we define $\tilde{\boldsymbol{\lambda}}_o$ to be the tamely ramified character of $F^\times$ satisfying
\begin{equation*}
\tilde \lambda_o|_{{\boldsymbol{\mu}}_{F}}= \left(\frac{\cdot}{{\boldsymbol{\mu}}_F}\right)
\quad\text{and}\quad
\tilde{\boldsymbol{\lambda}}_o(\varpi)= \lambda(-1)\left(\frac{-2\det \beta_j}{{\boldsymbol{\mu}}_F}\right)\mathfrak n_\psi.
\end{equation*}
The endoscopic lift of $\pi_\lambda$ is the parabolically induced representation 
$$\tilde\pi(a,\left(\tfrac{\cdot}{{\boldsymbol{\mu}}_F}\right), \lambda(\omega_j)\left(\tfrac{2}{{\boldsymbol{\mu}}_F}\right)\mathfrak n_\psi)\times \tilde{\boldsymbol{\lambda}}_o$$ 
of $\GL_{2n}(F)$.

\section{L-packets of epipelagic representations}
\label{subsection Reducibility results for different classical groups}

We present the L-packets of epipelagic representations of classical groups. Roughly speaking, the calculation is an inverse process of the endoscopic lifting: we begin with a representation $\tilde \pi$ of a general linear group which is an endoscopic lift of an epipelagic representation of a classical group, described via the inducing types $\{\tilde{\boldsymbol{\lambda}}_i\}_i$ of the cuspidal support of $\tilde \pi$. We then describe the inducing types of all possible representations lifting to $\tilde\pi$ using information of $\{\tilde{\boldsymbol{\lambda}}_i\}_i$. Again, we omit the discussion for unramified unitary groups, since their L-packets of epipelagic supercuspidals are singletons.

More precisely, we begin with a skew semisimple epipelagic stratum $\mathbf{s} = [\Lambda,0,\beta]$, with orthogonal decomposition  $[\Lambda_i,0,\beta_i]$ for $i\in I$. Let $\hat I$ be the dual index set 
of $I$, defined in each subsection below depending on the type of $G$. Take a tuple of signs $\delta = (\delta_i)_{i}\in \{\pm 1\}^{\#\hat I}$, where $\delta_o$ (and $\delta_{o'}$ if $o'\in \hat I$) may or may not depend on other $\delta_i$, again  depending on the type of $G$. If $G = G(V,h)$ and $N=\dim_FV$, we put 
$$\hat N = N-1 \text{ if $G$ is odd orthogonal}, \quad N+1 \text{ if $G$ is symplectic, and }\quad 
N\text{ otherwise.}$$
From the results of Proposition \ref{first general form of lifting} and Sections \ref{subsection Comparison with Oi's results} and \ref{subsection Ramified unitary groups} for various types of groups, we know that any endoscopic lift is an irreducible representation of $\GL_{\hat N}(F)$, depending only on $\mathbf{s}$ and $\delta$, and denoted simply by $\tilde \pi_{(\mathbf{s},\delta)}$, whose cuspidal support consists of quasi-epipelagic representations of the form $\tilde\pi(\beta_i, \phi_i,\xi_i)$  (with notation as in Subsection \ref{subsection Construction of quasi-simple supercuspidals GL-case}) for $i\in \hat I$, each of which is either a self-dual character of $F^\times$ or a self-dual quasi-epipelagic representation such that, for all $i\in \hat I \smallsetminus \{o,o'\}$, the characters $\phi_i = \tilde{{\lambda}}_i|_{\boldsymbol{\mu}_{E_i}}$ are
\begin{equation*}
   \begin{split}
   &\text{all trivial when $G$ is odd orthogonal or odd ramified unitary; } 
   \\
  & \text{all quadratic otherwise.}
   \end{split}
\end{equation*}
Consequently, for each $i\in \hat I \smallsetminus \{o,o'\}$, we assign $\xi_i = \tilde{\boldsymbol{\lambda}}_i(\varpi_i) $  to be  
\begin{equation*}
   \begin{split}
   &\text{$\delta_i$ when $G$ is odd orthogonal or odd ramified unitary}; 
   \\
  & \text{$\delta_i\cdot (\text{a power of the normalized quadratic Gauss sum }\mathfrak n_\psi)$ otherwise.}
   \end{split}
\end{equation*}
We then describe the corresponding L-packet $\tilde\Pi_{(\mathbf{s},\delta)}$, as a set of supercuspidal representations of $G(F)$ parametrized by partitions $I_\zeta \subseteq I$, together with an extra sign when $G$ is even orthogonal. Each representation in $\tilde\Pi_{(\mathbf{s},\delta)}$ is then determined by $(I_\zeta , \mathbf{s},\delta)$ (and a sign in the even orthogonal case).

For each $i\in I$, the common ramification degree of $F[\beta_i]/F$ is $2e$, and $\kappa_i$ is the sign appearing in (\ref{a sign kappa_i with j not o}). For each partition $I_\zeta$ of $I$, the embedding $m_{I_\zeta}$ is defined in Section \ref{section Embeddings of lattices}.

\subsubsection{Odd orthogonal groups}

For $G=\SO_{2n+1}$, given a skew semisimple epipelagic stratum $\mathbf s=[\Lambda,0,\beta]$, denote $\hat I = I\smallsetminus \{o\}$ and take $\delta \in \{\pm 1\}^{\#\hat I}$. Define an irreducible parabolically induced representation of $\GL_{2n+1}(F)$ by
$$\tilde\pi_{(\mathbf s,\delta)} = \prod_{i\in \hat I}\tilde\pi(2\beta_i,\mathbf 1_{{\boldsymbol{\mu}}_{E_i}}, \delta_i), $$
then $\tilde\pi_{(\mathbf s,\delta)}$ is the endoscopic lift of the representations in the L-packet
$$\tilde\Pi_{(\mathbf s,\delta)} := \{\pi_{I_\zeta}\}_{I_\zeta},$$
where each $\pi_{I_\zeta}$ is induced from a character $\lambda_{(I_\zeta,\mathbf s,\delta)}$ defined by 
\begin{equation*}
 \lambda_{(I_\zeta,\mathbf s,\delta)}|_{G(F)_{x,0_+}} = \psi_{m_{I_\zeta}(\beta)}\quad \text{and}\quad
\lambda_{(I_\zeta,\mathbf s,\delta)}(\omega_i) = \delta_i\kappa_i\left(\tfrac{2}{{\boldsymbol{\mu}}_{F}}\right),\quad\text{ for all }i\in I\smallsetminus\{o\}. 
\end{equation*}
Then $\lambda_{(I_\zeta,\mathbf s,\delta)}$ is an epipelagic type, and $\pi_{I_\zeta}:=\pi_{(I_\zeta,\mathbf s,\delta)}$ is an epipelagic supercuspidal representation of $G(F)$. The packet $\tilde\Pi_{(\mathbf s,\delta)} $ contains $2^{\#I-1}$ representations: half of them belong to $G_+$ and the other half to $G_-$. Indeed, if $\pi_{{\O}}$ belongs to $G_+$, then 
$$\text{$\pi_{I_\zeta}$ belongs to $G_\epsilon$ \quad$\Leftrightarrow$\quad $\#I_\zeta \equiv \tfrac{1}{2}(\epsilon-1)\bmod 2$,\quad for $\epsilon\in \{+,-\}$.}
$$

\subsubsection{Symplectic groups}

For $G=\SP_{2n}$, given $\mathbf s=[\Lambda,0,\beta]$ as before, denote $\hat I = I\sqcup \{o\}$ and take $\delta = \{\pm 1\}^{\#I}$. Define a tamely ramified character $\tilde{\boldsymbol{\lambda}}_{o}$ of $F^\times$ by 
$$ \tilde\lambda_{o}|_{{\boldsymbol{\mu}}_F} = \left(\frac{\cdot}{{\boldsymbol{\mu}}_F}\right)^{n/e} 
\quad\text{and}\quad
\tilde{\boldsymbol{\lambda}}_{o}(\varpi_{}) = \prod_{i\in I}\delta_i\left(\frac{\varpi^{f_i} N_{E_i/F}\beta_i}{{\boldsymbol{\mu}}_F}\right),$$
and  an irreducible parabolically induced representation of $\GL_{2n+1}(F)$ by
$$\tilde\pi_{(\mathbf{s},\delta)} = \left( \prod_{i\in I}\tilde\pi(2\beta_i,\left(\tfrac{\cdot}{{\boldsymbol{\mu}}_{E_i}}\right)^{}, \delta_i\mathfrak n_\psi^{f_i}) \right)\times \tilde{\boldsymbol{\lambda}}_{o}.$$ 
Then $\tilde\pi_{(\mathbf{s},\delta)}$ is the endoscopic lift of the representations in the L-packet of $2^{\#I}$ representations,
$$\tilde\Pi_{(\mathbf{s},\delta)} := \{\pi_{I_\zeta}\}_{I_\zeta},$$
where each $\pi_{I_\zeta}$ is induced from the epipelagic type $\lambda_{(I_\zeta,\mathbf s,\delta)}$ defined by 
\begin{equation*}
\lambda_{(I_\zeta,\mathbf s,\delta)}|_{G(F)_{x,0_+}} = \psi_{m_{I_\zeta}(\beta)},
\quad\text{and}\quad 
\lambda_{(I_\zeta,\mathbf s,\delta)}(\omega_i) = \delta_i\left(\tfrac{-1}{{\boldsymbol{\mu}}_{F}}\right)^{f_i}\kappa_i,\quad\text{ for all }i\in I, 
\end{equation*}
and so $\pi_{I_\zeta}:=\pi_{(I_\zeta,\mathbf s,\delta)}$ is an epipelagic supercuspidal representation of $G(F)$.

\subsubsection{Even orthogonal groups}

For $G=\SO^{}_{2n, F[\sqrt{d}]/F}$ where $d\in F^\times/(F^\times)^2$, given $\mathbf s=[\Lambda,0,\beta]$ as before and such that $\prod_{i\in I\smallsetminus\{o\}}{d_i} = d \bmod(F^\times)^2$, where $d_i$ is the discriminant of $F[\beta_i]/F$, denote $\hat I = I$ when $o\notin I$ and 
$\hat I = I\sqcup \{o'\}$ otherwise.

Take a tuple $\delta \in \{\pm 1\}^{\#I}$. We separate into cases depending on whether $o\in I$. When $o\notin I$, define the following irreducible parabolically induced representation of $\GL_{2n}(F)$ by 
$$\tilde\pi_{(\mathbf s,\delta)} =  \prod_{i\in I}\tilde\pi(\beta_i,\left(\tfrac{\cdot}{{\boldsymbol{\mu}}_{E_i}}\right), \delta_i\mathfrak n_\psi^{f_i}).
$$
When $o\in I$, we define two additional tamely ramified characters ${\tilde{\boldsymbol{\lambda}}}_{o}$ and ${\tilde{\boldsymbol{\lambda}}}_{o'}$ of $F^\times$ satisfying
 \begin{equation}
 \label{the two extra characters in the general epipelagic even orthogonal case}
   \begin{split}
      \tilde\lambda_{o} |_{{\boldsymbol{\mu}}_F}\equiv 
      \left(\frac{\cdot}{{\boldsymbol{\mu}}_F}\right), \quad {\tilde{\boldsymbol{\lambda}}}_{o}(\varpi) = {\delta_o}
      ;\quad \quad 
     \tilde\lambda_{o'} |_{{\boldsymbol{\mu}}_F}\equiv \mathbf 1_{{\boldsymbol{\mu}}_F},\quad 
  {\tilde{\boldsymbol{\lambda}}}_{o'}(\varpi) = \prod_{i\in I}\delta_i,
   \end{split}
\end{equation}
which are slight modifications of ${\tilde{\boldsymbol{\lambda}}}_{1}$ and ${\tilde{\boldsymbol{\lambda}}}_{2}$ in (\ref{split or unramified orthogonal groups, the two charaacters}). We then define the irreducible parabolically induced representation of $\GL_{2n}(F)$ by 
$$\tilde\pi_{(\mathbf s,\delta)} =  \left(\prod_{i\in I\smallsetminus \{o\}}\tilde\pi(\beta_i,\left(\tfrac{\cdot}{{\boldsymbol{\mu}}_{E_i}}\right), \delta_i\mathfrak n_\psi^{f_i})\right)
\times   \tilde{\boldsymbol{\lambda}}_{o}
\times   \tilde{\boldsymbol{\lambda}}_{o'}.$$
We introduce an additional sign $\xi\in \{+,-\}$. If $\beta_i = \beta_i(a_1,\dots,a_{k})$ for some $a_1,\dots,a_{k}\in \mathfrak o_{E_i} \bmod \mathfrak p_{E_i} $ (see (\ref{action of p on beta, ramified even orthogonal})), put
$$\beta_i(+) = \beta_i
\quad\text{and}\quad\beta_i(-) =\beta_i(a_1,\dots,a_{k-1},-a_{k}).$$ 
Choosing any tuple $(\xi_i)_i\in \{\pm 1\}^{\#I}$ such that $\prod_{i}\xi_i = \xi$, we define
$$\beta_{\xi} = \bigoplus_{i\in I\smallsetminus\{o\}}\beta_i(\xi_i ), \quad \text{for }\xi\in \{+,-\}.$$
In either case whether $o\in I$ or not, $\tilde\pi_{(\mathbf s,\delta)}$ is the endoscopic lift of the packet
$$\tilde\Pi_{(\mathbf s,\delta)} = \{\pi_{(I_\zeta,\xi)}
\}_{I_\zeta,\xi}$$
containing $2^{\#I+1}$ representations: half of them belong to $G_+$ and the other half to $G_-$. Indeed, for a fixed $\xi$, if $\pi_{(\O,\xi)}$ belongs to $G_+$, then 
$$\text{$\pi_{(I_\zeta,\xi)}$ belongs to $G_\epsilon$ \quad$\Leftrightarrow$\quad $\#I_\zeta \equiv \tfrac{1}{2}(\epsilon-1)\bmod 2$,\quad for $\epsilon\in \{+,-\}$.}
$$
Hence, each group has a packet of cardinality $2^{\#I}$.

To construct $\pi_{(I_\zeta,\xi)}$ for a partition $I_\zeta\subseteq I$ and a sign $\xi$, define a character $\lambda_{(I_\zeta,\mathbf s,\delta,\xi)}$ by 
\begin{equation*}
 \lambda_{(I_\zeta,\mathbf s,\delta,\xi)}|_{G(F)_{x,0_+}} = \psi_{m_{I_\zeta}(\beta_\xi)}\quad\text{and}\quad 
\lambda_{(I_\zeta,\mathbf s,\delta,\xi)}(\omega_i) = \delta_i\left(\tfrac{-1}{\mu_{E_i}}\right)\kappa_i,\quad\text{ for all }i\in I.  
\end{equation*}
Then $ \lambda_{(I_\zeta,\mathbf s,\delta,\xi)}$ is an epipelagic type, and the induced representation
$$\pi_{(I_\zeta,\xi)}:=\pi_{(I_\zeta,\mathbf s,\delta,\xi)}$$ is an epipelagic supercuspidal representation of $G(F)$.

When $o\notin I$, the packet $\tilde\Pi_{(\mathbf s,\delta)}$ is a disjoint union 
$$\tilde\Pi_{(\mathbf s,\delta)} = \tilde\Pi_{(\mathbf s,\delta,+)} \sqcup \tilde\Pi_{(\mathbf s,\delta,-)}$$ 
where $\tilde\Pi_{(\mathbf s,\delta,\xi)} = \{\pi_{(I_\zeta,\xi)}\}_{I_\zeta}$ for a fixed $\xi$, and both $\pi_{(I_\zeta,+)}$ and $\pi_{(I_\zeta,-)}$ induce isomorphic representations of $G^\sharp(F)$. In contrast, when  $o\in I$, all representations $\pi_{(I_\zeta,\xi)}$ are invariant under $G^\sharp(F)$. %

\subsubsection{Ramified unitary groups}
\label{subsubsection Ramified unitary groups L-packets}

Let $G=\mathrm U_{N}(F/\Fo)$ where $F = \Fo[\varpi]$ with $\varpi^2=-\varpi_\bullet$. Given $\mathbf s=[\Lambda,0,\beta]$ as before, denote $\hat I = I$ and take $\delta = \{\pm 1\}^{\#\hat I}$.  Note that $E_o = F$ and $\beta_o=0$ when $o\in I$.

The representation of $\GL_{N}(F)$ defined by 
$$\tilde\pi_{(\mathbf s,\delta)} = \prod_{i\in \hat I}\tilde\pi(\beta_i,\left(\tfrac{\cdot}{{\boldsymbol{\mu}}_{E_i}}\right)^{N-1}, \delta_i\mathfrak n_\psi^{f_i(N-1)}) $$ 
is the endoscopic lift of the L-packet
$$\tilde\Pi_{(\mathbf s,\delta)} := \{\pi_{I_\zeta}\}_{I_\zeta},$$
constructed as follows. We define a character $\lambda_{(I_\zeta,\mathbf s,\delta)}$ by 
\begin{equation*}
\lambda_{(I_\zeta,\mathbf s,\delta)}|_{G(F)_{x,0_+}} = \psi_{m_{I_\zeta}(\beta)},
\quad\text{and}\quad 
\lambda_{(I_\zeta,\mathbf s,\delta)}(\omega_i) = \delta_i\kappa_i,\quad\text{ for all }i\in I, 
\end{equation*}
where $\kappa_i$ is defined as before if $i\neq o$ and  
$$\kappa_o = \left(\frac{(-2)^{N-1}\prod_{i\in I\smallsetminus\{o\}}\varpi^{f_i} N_{E_i/F}\beta_i}{{\boldsymbol{\mu}}_{F}}\right).$$
Then $\lambda_{(I_\zeta,\mathbf s,\delta)}$ is an epipelagic type, and  $\pi_{I_\zeta}:=\pi_{(I_\zeta,\mathbf s,\delta)}$ is the induced epipelagic supercuspidal representation of $G(F)$. The packet $\tilde\Pi_{(\mathbf s,\delta)}$ contains $2^{\#I}$ representations: half of them belong to $G_+$ and the other half to $G_-$. Indeed if $\pi_{{\O}}$ belongs to $G_+$, then 
$$\text{$\pi_{I_\zeta}$ belongs to $G_\epsilon$ \quad$\Leftrightarrow$\quad $\#I_\zeta \equiv  \epsilon \bmod 2$,\quad for $\epsilon\in \{+,-\}$.}
$$

\subsubsection{Comparison with results in \cite{epipelagic-unitary}}
\label{Comparison with results in epipelagic-unitary}

Suppose that $(V,H)$ is a Hermitian space of odd dimension $N$ such that $G = G(V,H)$ is a ramified unitary group. We assume that $o\notin I$ and $\beta_i\in F$ for all $i\in I$. In this case, \cite{epipelagic-unitary} determined the criteria for two epipelagic representations to belong to the same L-packet.

Note that $(\lambda,\psi)$ in \emph{loc. cit.}, which we denote denoted by $(\lambda^{\mathrm{FRT}},\psi^{\mathrm{FRT}})$ now, corresponds to our $(\psi_\beta, \lambda)$. The Kostant section in \cite[(5.7.1)]{epipelagic-unitary} is evidently the same as ours: 
$$\beta = \diag(\beta_i)_{i\in I}.$$
Putting $1_I = (1)_{i\in I}\in \{\pm 1\}^{\#I}$, their inducing type $\rho(\lambda^{\mathrm{FRT}},\psi^{\mathrm{FRT}})$ corresponds to our $ \lambda_{(I_{\O}, \beta,1_I)}$, both takes values
$$\omega_i\mapsto \kappa_i := \prod_{j\in I\smallsetminus\{i\}}\left(\frac{ H_j H_i (1-\beta_j /\beta_i)}{\boldsymbol{\mu}_F}\right)= \prod_{j\in I\smallsetminus\{i\}}\left(\frac{ H_j(\beta_i -\beta_j )}{\boldsymbol{\mu}_F}\right),$$
where the last equality holds because $N$ is odd; the last value is given by \cite[Prop 5.15]{epipelagic-unitary}.

Suppose that $I_\zeta$ is a subset of $I$, and denote $\lambda_{I_\zeta}  = \lambda_{(I_\zeta, \mathbf s,\delta)}$ for fixed $(\mathbf s,\delta)$. Let $s_{I_\zeta}$ be the corresponding permutation on $\#I$, then $\lambda_{I_\zeta} (\omega_i)  = \delta_i \kappa_{s_{I_\zeta}(i)}$. Therefore, 
\begin{equation}
\label{FRT same L-packet}
\text{the quotient $\lambda_{I_\zeta} \lambda_{I_{\O}}^{-1}$ maps $\omega_i$  to $\kappa_{s_{I_\zeta}(i)}\kappa_{i}^{-1}$}.
\end{equation}
Since $\pi_{I_\zeta}$ and $\pi_{I_{\O}}$ belong to the same L-packet (Subsection \ref{subsubsection Ramified unitary groups L-packets}), we see that (\ref{FRT same L-packet}) precisely rephrases
\cite[Th 5.16]{epipelagic-unitary}.

\section{Epipelagic Langlands parameters}
\label{section Langlands parameters}

The aim of this section is to study the Langlands parameters of epipelagic representations for a connected classical group $G$. To this end, we assume that the characteristic of $F$ to be 0. We have deduced from the results of Section \ref{subsection Reducibility results for different classical groups} that an epipelagic parameter, after being endoscopic-lifted to a general linear group, is a direct sum of irreducible parameters corresponding to quasi-epipelagic representations. We look at these irreducible parameters in detail in this section. We emphasize that, while the construction of epipelagic parameters in \cite{Reeder-Yu} requires certain conditions on the residual characteristic $p$, we extend their results to all odd $p$ with the aid of \cite{BK-epipelagic} and \cite{BH-Eff}.

We focus mainly on the case where $F/\Fo$ is either trivial or ramified. The case where $F/\Fo$ is unramified is similar but less interesting, as seen in earlier sections (such as Subsection \ref{subsubsection Unramified unitary groups}).

As basic setup, let $\mathcal W_F$ be the Weil group of $F$. The isomorphism ${\mathcal W}_F^{ab} \cong F^\times $ from Artin reciprocity induces a canonical bijection
$$\text{\{characters of $F^\times$\}}\rightarrow \{\text{characters of ${\mathcal W}_F$}\},
\quad
\xi\mapsto \hat \xi.$$
We denote by $\mathcal P_F$ the wild inertia subgroup of $\mathcal W_F$. If $E/F$ is a field extension, the induction functor $\Ind_{\mathcal W_E}^{\mathcal W_F}$ is  denoted simply by $\Ind_{E/F}$  as in \cite{BH-Eff}, and similarly for the restriction functor $\Res_{E/F}$.

\subsection{L-groups of classical groups}
\label{subsection L-groups of classical groups}

In this section, we review basic setups on the dual side, including L-groups, their embeddings, lifting parameters, and self-duality. Towards the end, we view Langlands parameters as representations of $\mathcal W_F$ and describe them using admissible triples.

\subsubsection{Pinned automorphisms}
\label{subsubsection Pinned automorphisms}

We recall briefly the basics of L-groups from \cite[Sec 1 and 2]{Borel-L-functions}. By definition, groups that are inner forms of each other share the same L-group, and so we let $G$ be a quasi-split connected reductive group over $F$. 

We associate to $G$ a based root datum $\Psi = (X,\hat X,\Delta, \hat \Delta)$, which is invariant by the action of $\mathcal W_F$ (via the canonical morphism $\mathcal W_F\rightarrow \Gal(F^{\mathrm{sep}}/F)$). The dual group $\hat G$ of $G$ is the complex connected reductive group with based root datum is the dual $\hat \Psi = (\hat X,X,  \hat \Delta, \Delta)$ of $\Psi$. To define the L-action, we first fix a Borel-pair, i.e., a Borel subgroup of $\hat G$ and an underlying maximal torus, together with a set of simple root vectors. These fixed data $\hat {\mathcal E}$, called a pinning ({\'e}pinglage), determine an isomorphism $\Aut(\hat \Psi)\rightarrow \Aut(\hat G, \hat {\mathcal E})$, which induces the composition
\begin{equation}
\label{outer action}
\mathcal W_F\rightarrow \Aut(\Psi) \cong \Aut(\hat \Psi)\cong  \Aut(\hat G, \hat {\mathcal E}) \subset \Aut(\hat G),
\end{equation}
defining the L-action $\mathcal W_F$ on $\hat G$ as well as the L-group ${}^LG:=\hat G \rtimes {\mathcal W}_{F}$ of $G$.

Consider the groups studied in the previous sections. Viewing $ \tilde G = \GL_{m/F}$ as a split algebraic group over $F$, its dual group is $ \hat{\tilde  G} = \GL_{m}(\mathbb C)$, and the L-group is the direct product ${}^L\tilde G = {}^L\tilde G_F=\hat {\tilde G} \times {\mathcal W}_{F}$. We sometimes view $ \hat{\tilde  G}$ as the linear automorphism group of an $m$-dimensional complex vector space.

For a connected classical group $G$ with $F=\Fo$, we sometimes view the dual group $\hat G$ as the linear automorphism group of a complex vector space $\hat V$ equipped with a non-degenerate bilinear form $\hat h$. We may choose a suitable matrix $\hat H$ to define $\hat h$ as follows.
\begin{enumerate}[(i)]

\item If $G$ is $\SO_{2n+1}$, $\SP_{2n}$, or split $\SO_{2n}$, then ${}^LG := \hat G \times {\mathcal W}_F$, where $ \hat G = \SP_{2n}(\mathbb C)$, $\SO_{2n+1}(\mathbb C)$, or $\SO_{2n}(\mathbb C)$ respectively, defined by 
\begin{equation*}
\hat H = \begin{cases}\antidiag(1,\dots,1) &\text{ if $\hat G$ is orthogonal}, 
\\
 \antidiag(1,\dots,1,-1,\dots,-1) &\text{ if $\hat G$ is symplectic}.
\end{cases}
\end{equation*}

\item If $G$ is the quasi-split $\SO_{2n}$ that splits over a quadratic extension $F^\sharp/F$, then
$\hat G = \SO_{2n}(\mathbb C) $ is defined by $\hat H$ as above, and ${}^LG = \hat G \rtimes {\mathcal W}_F$, where the ${\mathcal W}_F$-action on $\hat G$ factors through ${\mathcal W}_F/{\mathcal W}_{F^\sharp} \cong \Gal(F^\sharp/F)$, and the non-trivial element of $\Gal(F^\sharp/F)$ acts by conjugation by the permutation matrix $\hat n_{\mathrm O}$ in $\mathrm{O}_{2n}(\mathbb C)$ that switches only the middle two coordinates.

\end{enumerate}
In both cases, we embed each of the $L$-groups into a general linear group to express it as a complex matrix group, i.e., we naturally embed both $  \SP_{2n}(\mathbb C)$ and $\SO_{2n}(\mathbb C)$ into $\GL_{2n}(\mathbb C)$, and  $\SO_{2n+1}(\mathbb C)$ into $\GL_{2n+1}(\mathbb C)$ in case (i), and $\mathrm{O}_{2n}(\mathbb C)$ into $\GL_{2n}(\mathbb C)$ in case (ii).

If $F/\Fo$ is quadratic and $G = \mathrm U_{m,F/\Fo}$, then $\hat G =  \GL_m(\mathbb C)$ and ${}^LG = \GL_m(\mathbb C)\rtimes {\mathcal W}_{\Fo}$, where the ${\mathcal W}_{\Fo}$-action on $\hat G$ is trivial on ${\mathcal W}_F$, and the non-trivial element $c\in {\mathcal W}_{\Fo}\smallsetminus{\mathcal W}_F $ acts by the automorphism $\theta$ defined by
\begin{equation*}
\theta:g\mapsto \hat H^{-1} {}^tg^{-1}\hat H, \quad \hat H = \begin{cases}\antidiag(1,\dots,1) &\text{ if $m$ is odd}, 
\\
 \antidiag(1,\dots,1,-1,\dots,-1) &\text{ if $m$ is even}.
\end{cases}
\end{equation*}
In all cases, the choice of $\hat H$ need not take the same form as $H$ on the p-adic side; this choice is made only for computational convenience. For the same reason, we usually pin the maximal torus in the pinning $\hat {\mathcal E}$ to be the diagonal subgroup.

\begin{rmk}
\label{Remark on Galois form}
It is sometimes convenient to replace the Weil group component in an L-group by a finite Galois group. We therefore introduce the following ad-hoc notations,
\begin{equation*}
   \begin{split}
     & \mathrm{O}_{2n,F^\sharp/F}(\mathbb C):= \mathrm{SO}_{2n} (\mathbb C)\rtimes \Gal(F^\sharp/F)\quad \text{when $G$ is even orthogonal and splits over $F^\sharp/F$;}
\\
&\mathrm{GL}_{m,F/F_\bullet}(\mathbb C):=\mathrm{GL}_{m}(\mathbb C)\rtimes \Gal(F/F_\bullet)\quad \text{when $G$ is $F/F_\bullet$-unitary.}
   \end{split}
\end{equation*}
Together with  $ \hat G = \SP_{2n}(\mathbb C)$, $\SO_{2n+1}(\mathbb C)$, or $\SO_{2n}(\mathbb C)$ when $G$ is split, we call these mentioned groups the \emph {Galois-form} of their corresponding L-groups. This form will be used whenever we see fit, particularly in (\ref{types of classical groups, dual side}) and Subsections \ref{subsubsection Unramified induction} and \ref{subsubsection Epipelagic parameters} below.
\qed\end{rmk}

In what follows, a Langlands parameter is a continuous morphism $\varphi: \mathcal W_F\rightarrow {}^LG$ such that the composition $\mathcal W_F\xrightarrow{\varphi} {}^LG \xrightarrow{}\mathcal W_F$  with the natural projection ${}^LG \rightarrow\mathcal W_F$ is the identity.

\subsubsection{Endoscopic embeddings}
\label{subsubsection Endoscopic embeddings}

Let $K/F$ be a field extension. Given two quasi-split connected reductive groups $G$ over $F$ and $H$ over $K$, a morphism between their L-groups $\iota:{}^LH\rightarrow {}^LG$ is called an L-morphism if $\iota|_{\hat H}:{\hat H}\to {\hat G}$ is a morphism of algebraic groups and the diagram 
\begin{equation*}
  \xymatrixcolsep{5pc} \xymatrixrowsep{2pc}\xymatrix{
{}^LH
\ar[r]^{\iota}  
\ar[d]_{
\begin{smallmatrix}
\text{natural}
\\
\text{projection}
\end{smallmatrix}
}  
&{}^LG
\ar[d]^{
\begin{smallmatrix}
\text{natural}
\\
\text{projection}
\end{smallmatrix}
}
\\
{\mathcal W}_K
 \ar@{^{(}->}[r]
 &{\mathcal W}_F 
}
\end{equation*}
commutes, {\it cf.} \cite[Sec 15]{Borel-L-functions}, and is called an L-embedding if it is furthermore injective. From now on, all morphisms (in particular all embeddings) between two L-groups are assumed to be L-morphisms.

Due to technical reasons concerning the matching of conjugacy classes on the p-adic side (which does not concern us in the present work), we furthermore require certain L-embeddings defined below to be {\it endoscopic}. The complete definition in \cite{Kottwitz-Shelstad} is highly involved, but for our purposes we only require the following:
\begin{enumerate}[(i)]
\item $\iota$ maps $\hat H$ isomorphically onto $(\hat G ^{\mathrm{Ad}(s)\circ\theta})^\circ $ for a prescribed $F$-automorphism $\theta$ of $\hat G$ and for some {$\theta$-quasi-semisimple} $s\in \hat G$, i.e., ${\mathrm{Ad}(s)\circ\theta}$ fixes a Borel-pair.

\item $\iota({}^LH)$ embeds into $({}^LG|_{\mathcal W_K})^{\mathrm{Ad}(s)\circ{}^L\theta_a}$ for a prescribed 1-cocycle $a:\mathcal W_F\to Z\hat G$, where ${}^LG|_{\mathcal W_K}:=\hat G\rtimes \mathcal W_K\subseteq {}^LG$ and ${}^L\theta_a$ is the automorphism $g\rtimes w \mapsto {}^\theta g a(w)\rtimes w$ on ${}^LG$. 
\end{enumerate}
An L-embedding $\iota:{}^LH\to {}^LG $ is called $(\theta,a,s)$-{\bf endoscopic} if the above conditions are satisfied. Up to $\theta$-conjugacy:  $(s,\iota)\mapsto (g s{}^\theta g^{-1},\mathrm{Ad}(g)\circ \iota)$ for some $g\in \hat G$, we assume that $s\in \hat {\mathcal T}$ and $\iota$ maps the pair $(\hat {\mathcal B}_H, \hat {\mathcal T}_H)$ in a fixed pinning of $\hat H$ to the pair $(\hat{\mathcal B}, \hat{\mathcal T})$ of $\hat G$.

 Here are two simple examples that we will use later.
 \begin{enumerate}[(i)]

\item (Levi-type) When both $(\theta,\omega) = (1,1)$, the group $\hat H = Z_{\hat G}(s)$ is a Levi subgroup of $\hat G$.

\item (Galois-type) Let $E/F$ be a cyclic extension of degree $d$, and $\omega = \omega_{E/F}$ be a character of $F^\times$ associated to $E/F$ by local class field theory. Choose $\gamma\in \mathcal W_F$ which projects to a generator of $\Gal(E/F)$. Put $G = \GL_{md,F}$ and take 
$$s = \diag(I_m, e^{2\pi \sqrt{-1}/d}I_m, \dots,  e^{2\pi \sqrt{-1}(d-1)/d}I_m)\in \hat G.$$ Put also $H = \Res_{E/F}\GL_{m,E}$, so that $\hat H = \Ind_{E/F}\GL_m(\mathbb C)$. The embedding $ {}^LH \to {}^LG$ determined by 
$$h\rtimes \gamma \mapsto h n_\gamma \times \gamma,\quad \text{ for all }h\in \hat H,$$
where $n_\gamma = \antidiag(I_m, \diag(I_m,\dots,I_m))$, is $(1,\hat \omega,s)$-endoscopic \cite[Sec 6.3]{Tam-ETLLC-GLn}.

\end{enumerate}

\subsubsection{Lifting parameters by embeddings}
\label{subsubsection Lifting parameters}

We recall from \cite[Sec 1]{Arthur-book} and \cite[Sec 2]{Mok-unitary} an embedding $\iota_G$ from the L-group ${}^LG = \hat G \rtimes {\mathcal W}_{\Fo}$ of a connected classical group $G$ into the L-group of a general linear group $\tilde G$. The motivation of defining $\iota_G$ is evident: by regarding $G$ as a twisted endoscopic group of $\tilde G$ via $\iota_G$ (with suitable endoscopic data \cite[Sec 2.1]{Kottwitz-Shelstad}), lifting parameters from $G$ to $\tilde G$ via $\iota_G$ corresponds to the endoscopic lifting on the representation side. A lifted parameter is viewed as a representation of ${\mathcal W}_F$.

When $F=\Fo$, the embedding $\iota_G: {}^LG\rightarrow  \GL_{m}(\mathbb C)$ is the natural inclusion discussed in Section \ref{subsubsection Pinned automorphisms}, which is $(\text{transpose-inverse},1,\hat H)$-endoscopic. The lifting of a parameter $\varphi$ is just  $\tilde\varphi:=\iota_G\circ \varphi$. By \cite[Th 8.1(ii)]{GGP}, the isomorphism class of $\tilde\varphi$ determines the $\hat G$-conjugacy class of $\varphi$, except in the even orthogonal case when every irreducible component of $\tilde\varphi$ is even-dimensional, in which case both the $\hat G$-conjugacy classes of $\varphi$ and of $\Ad(\hat n_O)\varphi$ give rise to $\tilde\varphi$.

When $F/\Fo$ is quadratic, we fix $c\in \mathcal W_{\Fo}\smallsetminus \mathcal W_{F}$ and view $\Ind_{F/\Fo}\hat{\tilde G} = \hat{\tilde G} \times \hat{\tilde G}$, where $\hat{\tilde G}= \GL_m(\mathbb C)$ on which $\mathcal W_F$ acts trivially, and $c$ exchanges the two $\hat{\tilde G} $-factors of $\Ind_{F/\Fo}\hat{\tilde G}$. Form the L-group $ {}^LG_{\Fo}:=( \Ind_{F/\Fo}\hat{\tilde G} )\rtimes {\mathcal W}_{\Fo}$, and define two embeddings $\iota^{\epsilon}_G: {}^LG \rightarrow {}^LG_{\Fo}:=( \Ind_{F/\Fo}\hat{\tilde G} )\rtimes {\mathcal W}_{\Fo}$, where $\epsilon\in \{\pm \}$, by 
\begin{equation*}
   \begin{split}
      &\iota^{\epsilon}_G:g\times w\mapsto (\hat \chi_\epsilon(w)g,\hat \chi_\epsilon(w)^{-1} {}^tg^{-1})\times w
\quad\text{for $w\in \mathcal W_F$, and}\quad  
\\
&1\rtimes c\mapsto ( \hat H , \epsilon\hat H^{-1})\rtimes c,
   \end{split}
\end{equation*}
where $\chi_{\epsilon}$ is a character of $F^\times$, with $\chi_{+}$ chosen to be trivial and $\chi_{-}$ chosen to be tamely ramified and satisfying
$$\chi_{-}|_{{\boldsymbol{\mu}}_F} =\left(\frac{\cdot}{{\boldsymbol{\mu}}_F}\right)^{f_{F/\Fo}}, \quad \chi_{-}(\varpi) =(-1)^{e_{F/\Fo}}\mathfrak n_\psi^{2-f_{F/\Fo}},$$
where $\mathfrak n_\psi$ is the quadratic Gauss sum (\ref{simple quadratic Gauss sum}) associated to $F$. Each embedding $\iota_G^\epsilon$ is then which is $(\text{transpose-inverse},\hat \chi_\epsilon,\hat H)$-endoscopic. For later reference, we call this embedding unitary-type endoscopic.

We put $\iota_G = \iota_G^{(-1)^{m-1}}$ for $\hat G = \GL_m(\mathbb C)$. Given a parameter $\varphi$ of $G$, its lifting can be described as follows. Denote by $\pi_1:( \Ind_{F/\Fo}\hat{\tilde G} )\rtimes {\mathcal W}_{\Fo}\rightarrow \hat{\tilde G}$ the projection onto the first factor $\hat{\tilde G} $. Then, given a parameter $\varphi: {\mathcal W}_{\Fo}\rightarrow {}^LG $, we call 
 $\tilde\varphi:= \pi_1\circ \iota_G\circ \varphi|_{ {\mathcal W}_{F}}: {\mathcal W}_{F}\rightarrow \hat{\tilde G}$ the lifting of $\varphi$, i.e., lifting $\varphi$ is in fact restricting $\varphi$ to ${\mathcal W}_{F}$. In other words, 
 $\varphi = j_G\circ \tilde\varphi$, where $j_G:{}^L\tilde G \hookrightarrow {}^LG$ is the natural inclusion. The lifted parameter $\tilde\varphi$ satisfies
 $$\tilde\varphi(cwc^{-1}) = {}^{}{}^\theta\tilde\varphi( w) \quad  \text{for all $w\in \mathcal W_F$} \quad\text{and}\quad  \tilde\varphi(c^2) = (-1)^{m-1}.$$
Again, the isomorphism class of $\tilde\varphi$ determines the $\hat G$-conjugacy class of $\varphi$.

For later reference, whether $F=\Fo$ is trivial or quadratic, we refer to the embedding $\iota_G$ as {\bf standard} for the classical groups $G$ discussed above.

\subsubsection{Self-duality}
\label{first subsubsection Self-duality}

To unify notation, we still denote by $c\in \Gal(F/\Fo)$ the trivial element when $F=\Fo$. Given a parameter $\tilde \varphi:{\mathcal W}_F\rightarrow \hat{\tilde{G}}:=\hat{\tilde{G}}(\hat V)$, we put ${}^c\tilde\varphi(w)= \tilde\varphi(c^{-1} wc)$. We call $\tilde\varphi$ self-dual if ${}^c\tilde\varphi$ is isomorphic to the contragredient $\tilde\varphi^\vee$ of $\tilde\varphi$ as a $\mathcal W_F$-representation, i.e., up to $ \hat{\tilde{G}}$-conjugacy, $\tilde\varphi$ satisfies
$$\hat h({}^c\tilde\varphi(w) u ,\tilde\varphi(w) v) = \hat h(u,v),\quad \text{for all }u,v\in\hat V\text{ and }w\in {\mathcal W}_F.$$
Let $\epsilon\in \{\pm 1\}$. If $\tilde\varphi$ is self-dual, we say that $\tilde\varphi$ is $(\epsilon,F/\Fo)$-self-dual, or just $\epsilon$-self-dual when $F=\Fo$, if 
$$\hat h(v,u) = \epsilon\hat h(\tilde\varphi(c^2)u,v),\quad \text{for all }u,v\in\hat V.$$
We make two simple observations.
\begin{enumerate}[(i)]
\item When $\dim \hat V=1$, if $F=\Fo$, then a self-dual character $\mu$ of $F^\times$ is quadratic and $+$-self-dual; whereas if $[F:\Fo]=2$, then a self-dual character $\mu$ of $F^\times$ is $(+,F/\Fo)$-self-dual (resp. $(-,F/\Fo)$-self-dual) if and only if $\mu |_{\Fo^\times}$ is trivial (resp. equal to $\omega_{F/\Fo}$).

\item Let $E/\Eo$ be an extension over $F/\Fo$, i.e., $E/\Eo$ is trivial or quadratic if $F=\Fo$, and is quadratic if $F/\Fo$ is. If $\xi$ is an $(\epsilon,E/\Eo)$- self-dual representation of $\mathcal W_E$, then $\Ind_{E/F} {\xi}$ is an $(\epsilon,F/\Fo)$-self-dual representation of $\mathcal W_F$.

\end{enumerate}

\begin{prop}
A parameter $\tilde \varphi:{\mathcal W}_F\rightarrow \hat{\tilde{G}}:=\hat{\tilde{G}}(\hat V)$ is
$$\text{$+$-self-dual \quad (resp. }\text{$-$-self-dual, \quad resp. }\text{$(+,F/\Fo)$-self-dual, \quad resp. }\text{$(-,F/\Fo)$-self-dual)}$$
if and only if it is a lifting from ${}^LG$, where $\hat G = \hat G(\hat V,\hat h)$ is
$$\text{orthogonal \quad (resp. sympelctic, \quad resp. odd unitary, \quad resp. even unitary}).$$
\end{prop}
\proof
This is \cite[Th 8.1(i)]{GGP}, which asserts that a parameter of $G$ determines uniquely the self-duality structure of its lifting as a $\mathcal W_F$-representation.  \hfill$\blacksquare$

When $F=\Fo$ (resp. $F/\Fo$ is quadratic), two representations are deemed of the same parity if they are both $+$-self-dual or both  $-$-self-dual (resp. both $(+,F/\Fo)$-self-dual or both  $(-,F/\Fo)$-self-dual), and of different parities otherwise.

\subsubsection{Admissible triples}
\label{subsubsection Admissible triples (BH descriptions of Langlands parameters)}

We recall from \cite{BH-Eff} the description of irreducible representations of $\mathcal{W}_F$ via admissible triples. An admissible triple $(K/F,\rho,\tau)$  consists of 
\begin{enumerate}[(i)]
\item 

a finite tamely ramified extension $K/F$, 
\item an irreducible representation $\rho$ of $\mathcal{W}_K$ such that the restriction
$\rho_{\mathcal P} := \rho|_{\mathcal P_F}$ remains irreducible and $Z_{\mathcal{W}_F} (\rho_{\mathcal P}) := \{g\in \mathcal{W}_F:{}^g \rho_{\mathcal P}\cong \rho_{\mathcal P}\}= \mathcal{W}_K$, and 

\item an irreducible representation $\tau$ of $\mathcal{W}_K$ such that $\tau|_{\mathcal P_F}$ is trivial.
\end{enumerate}

 Two admissible triples $(K_i/F,\rho_i,\tau_i)$, for $i \in \{ 1,2\}$, are called equivalent if there exist $g\in \mathcal{W}_F$ and a tamely ramified character $\chi$ of $\mathcal{W}_{K_2}$ such that $K_2 ={}^gK_1$, $\rho_2\cong {}^g\rho_1 \otimes \chi$, and $\tau_2\cong {}^g\tau_1\otimes \chi^{-1}$. Denote by $[K/F,\rho,\tau]$ the equivalence class of $(K/F,\rho,\tau)$.

Given an equivalence class $\Xi = [K/F,\rho,\tau]$ of admissible triples, we define $\tilde\varphi = \tilde \varphi_\Xi:=\Ind_{K/F}(\rho\otimes\tau)$, viewed as a Langlands parameter of $\GL_m$, where $m=\dim \tilde \varphi$. Then \cite[1.4]{BH-Eff} implies that 
\begin{enumerate}[(i)]
\item  $\tilde\varphi$ depends only on the equivalence class $\Xi$, and is irreducible;

\item every irreducible representation of ${\mathcal W}_F$ arises as $\tilde\varphi_\Xi$ for a unique class $\Xi = [K/F,\rho,\tau]$.

\end{enumerate}

\begin{rmk}
If a supercuspidal representation whose parameter $\tilde\varphi$ has an underlying simple stratum with field datum $\beta$, then $K$ is $F$-isomorphic to the maximal tamely ramified subextension of $F[\beta]/F$ \cite[Sec 6.3]{BH-Eff}. The extension $F[\beta]/K$ is then {\bf totally wild}, i.e., it is totally ramified and of degree a power of $p$.
\qed\end{rmk}

In another direction, let $t = t(\tilde\varphi)$ be the number of unramified characters $\chi$ of $\mathcal W_F$ such that $\chi\tilde\varphi\cong \tilde\varphi$, which is equal to $f_{K/F}\dim \tau $. If $U/F$ is the unramified extension of $F$ of degree $t/\dim \tau$, then $\tilde\varphi = \Ind_{U/F}\eta$ for some representation $\eta$ of $\mathcal W_{U}$. If we write $\tau = \Ind_{K_d/K}\xi$ for some depth zero character $\xi$ of $\mathcal W_{K_d}$, where $d=\dim\tau$ and $K_d$ is the unramified extension of $K$ of degree $d$, then 
$\eta = \Ind_{K_d/U}(\xi\otimes \Res_{K_d/K}\rho)$.

We now describe the structure of $\tilde \varphi = \tilde\varphi_\Xi$ when it is self-dual. Eventually we will often consider `irreducible components' of epipelagic representations of classical groups, and they are of depth zero only when they are characters, so we assume that $\rho_{\mathcal P}$ is non-trivial henceforth. Since $p$ is odd, $\rho_{\mathcal P}$ cannot be self-dual if it is non-trivial. There exists $g\in \mathcal W_F \smallsetminus \mathcal P_F$ such that ${}^g\rho \cong \rho^\vee$, and we assume that $g^2 \in \mathcal W_K$. Denote by $K^\flat$ the fixed field of $g$, so that $K/K^\flat$ is quadratic. We have 
$\tilde\varphi = \Ind_{K^\flat/F}(\rho\otimes \tau \oplus \rho^\vee\otimes \tau^\vee )$.

\begin{rmk}
\label{ramified quadratic extension over F/Fo}
In the case $[F:\Fo]=2$, we have ${}^\theta\rho \cong \rho^\vee$ by the uniqueness of the $\mathcal P_F$-invariant orthogonal (i.e., self-dual) structure \cite[7.5 Prop]{BHS}. Indeed, since $F/\Fo$ is tamely ramified, we can assume that $c^2 \in Z_{\mathcal{W}_\Fo}(\rho|_{\mathcal P_F}) = \mathcal W_K$, so that ${}^\theta\rho\ncong \rho$ but ${}^{\theta^2}\rho\cong \rho$. Both ${}^\theta\rho\oplus \rho$ and $\rho^\vee \oplus \rho$ are $\mathcal P_F$-invariant orthogonal representations by duality, and both contain $\rho$. Hence they are isomorphic by uniqueness, and so are  ${}^\theta\rho $ and $\rho^\vee$. But then ${}^\theta\rho \cong {}^g \rho$ for some $g\in {\mathcal{W}_\Fo}$, and we may arrange that $g|_{F} = c$. This implies that $K/K^\flat$ is an extension over $F/\Fo$.
\qed\end{rmk}

As in \cite[Sec 6]{BHS}, we are interested in whether $\tilde\varphi $ and its unique self-dual unramified twist have the same parity. Since they always have different parities when $K/K^\flat$ is unramified, which is necessarily the case when $F/\Fo$ is unramifed, we assume the contrary in the following proposition.

\begin{prop}
\label{prop effective parameter, self-dual}
Suppose that $F=\Fo$ or $F/\Fo$ is ramified. Let $\tilde\varphi =\Ind_{K/F}(\rho\otimes\tau)= \Ind_{U/F}\eta = \Ind_{K_d/F}(\xi\otimes \Res_{K_d/K}\rho)$ (where $d=\dim \tau$) be a self-dual positive depth parameter, and $\tilde\varphi'$ be the unique self-dual twist of $\tilde\varphi$ by an unramified character, then 
\begin{enumerate}[(i)]
\item $\tilde\varphi$ and $\tilde\varphi'$ have the same parity if and only if $K/K^\flat$ is ramified and $d=1$.
\label{prop effective parameter, self-dual, same parity}
\item When the conditions in (i) are satisfied, $\tilde\varphi$ and $\tilde\varphi'$ are $(-,F/\Fo)$-self-dual if and only if $\xi$ is ramified.

\item When $F=\Fo$, the conditions in (i) are satisfied if and only if $\eta$ is self-dual.
\end{enumerate}

\end{prop}

\proof
These assertions are proved in \cite[Sec 6.6]{BHS} and \cite[Sec 6.1]{BHS} when $F=\Fo$. When $F/\Fo$ is ramified, essentially the same arguments \cite[Sec 6.6]{BHS} apply directly to prove (i) and (ii), bearing Remark \ref{ramified quadratic extension over F/Fo} and the observation (i) of \ref{first subsubsection Self-duality} in mind. 
\hfill$\blacksquare$

\subsection{The wild part}
\label{subsection The wild part}

From now on, we call a Langlands parameter of a general linear group or a classical group an {\bf epipelagic parameter} if its corresponding representation is epipelagic. For a general linear group, we call a {\bf quasi-epipelagic parameter} analogously.

The first objective of this section is Proposition \ref{adjoint swan conductor simple supercuspidal for classical groups} in the first subsection, which computes the adjoint Swan conductor of an epipelagic parameter $\varphi$ of a classical group whose lift $\tilde\varphi$ (as defined in Subsection \ref{subsubsection Lifting parameters}) is also epipelagic. Previously, this was computed for split simply-connected simple groups \cite[Sec 9]{Gross-Reeder} and for certain classical groups under a condition on $p$ \cite[Sec 9]{Oi-Sp-and-SO-even}. Here, we apply results from \cite{BK-epipelagic} to carry out the calculations for all classical groups, requiring only that $p$ be odd. The remaining subsections are devoted to comparing the wild part of a quasi-epipelagic parameter $\tilde\varphi$ with Reeder-Yu's template in \cite[Sec 7.2 and 7.3]{Reeder-Yu}. The main result is Proposition \ref{ramification theorem in RY template}, establishing a bijection between a line of stable functionals and a set of extensions of representations of Heisenberg type, again for all odd $p$.

\subsubsection{Conductors}

We begin with recalling the definitions of the Artin and Swan conductors of a finite dimensional smooth representation $(\phi,\mathbf V)$ of $\mathcal W_F$ such that $\phi(\mathrm{Fr})$ is semisimple. By convention, we put $\mathrm{ar}(\phi) =\mathrm{sw}(\phi)=0$ if $\phi$ is an unramified character; otherwise, for $i\in \mathbb Z_{\geq 0}$, denote $D_0:=\phi(\mathcal I_F)$, where $I_F$ is the inertia subgroup of $\mathcal W_F$, and let $F_0$ be the fixed field of $D_0$, which is finite and unramified over $F$. For $i\in \mathbb Z_{\geq 0}$, let $D_i = D_i(\phi)$ be the kernel of the action 
of $\phi(D_0)$ on $\mathfrak o_{F_0}/\mathfrak p^{i+1}_{F_0}$, and define $\mathrm{ar} = \mathrm{ar}_F$ by 
$$\mathrm{ar}(\phi) = \sum_{i\geq 0}\frac{\dim(\mathbf V/\mathbf V^{D_i})}{[D_0:D_i]}.$$ 
This is the Artin conductor appearing in the Langlands-Deligne local constant $\epsilon(\phi,s,\psi) = \epsilon(\phi,0,\psi) q^{- \mathrm{ar}(\phi)s }$ for $s\in \mathbb C$, when $\phi$ is a Langlands parameter. We also define the Swan conductor $\mathrm{sw} = \mathrm{sw}_F$ 
$$\mathrm{sw}(\phi) = \mathrm{ar}(\phi) - \dim(\mathbf V/\mathbf V^{D_0}).$$ 
Hence, $\mathrm{sw}(\phi)=0$ if $\phi|_{D_0}$ is a direct sum of trivial representations.

We have the following induction and restriction properties of Swan conductors. If $K/F$ is a tamely ramified extension, then 
 \begin{subequations}
\begin{align}
& \mathrm{sw}_F(\Ind_{K/F}\sigma)=f_{K/F} \mathrm{sw}_K(\sigma); 
\label{Swan conductor induction}
\\
& \mathrm{sw}_F(\phi)= \mathrm{sw}_K(\Res_{K/F}\phi),
\label{Swan conductor restriction}
\end{align}
\end{subequations}
for every Frobenius-semisimple smooth representation $\sigma$ of $\mathcal W_K$ and $\phi$ of $\mathcal W_F$.

Returning to the representation side, given  a supercuspidal representation $\tilde \pi$ of a general linear group, we also define the Swan conductor $\mathrm{sw}(\tilde \pi)$ using the Godement-Jacquet local constant \cite[Sec 1.2]{BK-epipelagic}. If $\tilde\pi$ is a supercuspidal representation of $\GL_m(F)$ containing a maximal simple stratum $\mathbf{s} = [\Lambda,0,\beta]$, we put $\mathrm{sw}(\mathbf s)= \mathrm{sw}(\tilde\pi) $, which is equal to $ -v_\Lambda(\beta) m$ \cite[2.1]{BK-epipelagic}.  Because the LLC preserves conductors, we have $\mathrm{sw}(\tilde \pi) = \mathrm{sw}(\tilde\varphi)$ if $\tilde\varphi $ is the parameter of $\tilde \pi$.

Epipelagic representations of $\GL_m(F)$ can be characterized by conductors; by  \cite[2.1 Lem 1]{BK-epipelagic}, they are precisely those with Swan conductor 1. Let $\mathbf s = [\Lambda,0,\beta]$ be a quasi-epipelagic maximal simple stratum over $F$, 
then we have $\mathrm{sw}_F(\mathbf s) = f_{F[\beta]/F}$. If $U/F$ is the maximal unramified subextension of $F[\beta]/F$, viewing $\mathbf s$ as a maximal simple stratum over $U$, we have $\mathrm{sw}_U(\mathbf s)=1 $, i.e., $\mathbf s$ is epipelagic over $U$. This is parallel to property (\ref{Swan conductor induction}) on the Galois side.

\subsubsection{Total wildness and primitivity properties}
\label{subsection Total wildness and primitivity properties}

We continue with the discussion from Subsection \ref{subsubsection Admissible triples (BH descriptions of Langlands parameters)}: let $\tilde\varphi = \Ind_{K/F}(\tilde\rho\otimes\tilde\tau)$ for a class of admissible data $[K/F, \tilde\rho, \tilde\tau]$.

Suppose that $\tilde\rho$ is an irreducible representation of $\mathcal W_K$, of dimension $p^r$ for some $r\geq 1$, such that $\rho_{\mathcal P} := \rho|_{\mathcal P_F}$  has Swan conductor 1. The following consequences hold.
\begin{enumerate}[(i)]
\item 
$\tilde\rho$ is totally wild, i.e., $\rho_{\mathcal P} $  remains irreducible. 
\label{definition totally wild}

\item $\tilde\rho$ is primitive over $K$, i.e. it is not induced from a representation of any proper subgroup $\mathcal W_L$, where $L/K$ is a finite extension. 
\label{definition primitive}
\end{enumerate}

Let $E/K$ be the p-centric field \footnote{In \cite{BK-epipelagic}, this field is called the p-kernel field, but was renamed the p-centric field in subsequent works, e.g., \cite{BH-Carayol}. The author finds the latter terminology more appropriate, since we consider not only $\ker\overline{\tilde\rho}$ but also $\ker\tilde\rho$.}, defined by 
\begin{equation}
\label{p-centric field, definition}
\Gal(E/K) = \ker(\overline{\tilde\rho}:\mathcal W_K\xrightarrow{\tilde\rho} \GL_{p^r}(\mathbb C)\rightarrow \mathrm{PGL}_{p^r}(\mathbb C)),
\end{equation}
and $T/K$ be the maximal tamely ramified subextension of $E/K$, called the imprimitivity field of $\tilde\rho$. Hence $E/T$ is Galois, and the adjoint ramification filtration of ${\tilde\rho}$ is given by
\begin{equation}
\label{adjoint ramification filtration}
 D_0 (\overline{\tilde\rho}) =  \Gal(E/K) \supset  D_1(\overline{\tilde\rho}) =\Gal(E/T)  \supset  D_2 (\overline{\tilde\rho}) = 1.
\end{equation}
Let $E_{\tilde\rho}/K $ be the kernel field of $\tilde\rho$, i.e., $\Gal(E_{\tilde\rho}/K) = \ker{\tilde\rho}$. The following properties hold.
\begin{itemize}
\item $\Gal(E_{\tilde\rho}/T)$ is a two-step nilpotent group, i.e., $\Gal(E/T)$ is an elementary abelian $p$-group of order $p^{2r}$, and $\Gal(E_{\tilde\rho}/E)\cong \mathbb F_p$ is the commutator subgroup of $\Gal(E_{\tilde\rho}/T)$;

\item $\rho_{\mathcal P} $  is the irreducible representation of Heisenberg type determined by its central character $\hat\xi=\hat\xi_{\tilde\rho}$ on $\Gal(E_{\tilde\rho }/E)$.
\end{itemize}
Suppose that $T/K$ is Galois. There is hence an exact sequence
\begin{equation}
\label{Heisenberg group exact sequence}
1\rightarrow \Gal(E_{\tilde\rho}/T)\rightarrow \tilde\rho(\mathcal W_K)=\Gal(E_{\tilde\rho}/K)\rightarrow \Gal(T/K) \rightarrow 1,
\end{equation}
giving a semidirect product decomposition
$$\Gal(E_{\tilde\rho}/K) \cong \Gal(E_{\tilde\rho}/T) \rtimes \Gal(T/K).$$
 The image group $\Gal(E_{\tilde\rho}/K)$ in (\ref{Heisenberg group exact sequence}) is therefore of symplectic-Heisenberg type, i.e., viewing $ A:=\Gal(E/T)$ as an $\mathbb F_p$-space of dimension $2r$ equipped with the symplectic structure defined by the commutator, then there is an induced morphism $\Gal(T/K)\rightarrow \SP(A)$.

We will resume this discussion in Subsection \ref{Totally wild  primitive parameters in Reeder-Yu's template}.

\begin{rmk}
\label{Some numerical results on parameters}
For later use, we collect several numerical properties of the above construction.

\begin{enumerate}[(i)]
\item  By \cite[5.2 Th]{BK-epipelagic}, $T$ can be expressed as the splitting field of an explicit polynomial over $K$ (see Remark \ref{different fields E and T} below), and $E$ as determined by the common kernel of the characters $\hat \chi$ of $\mathcal W_T$ satisfying $\hat \chi \otimes \tilde\rho|_{\mathcal W_T} \cong \tilde\rho|_{\mathcal W_T} $. By \cite[5.3 Th]{BK-epipelagic}, one then deduces that  $e(T/K) = 1+p^r$, and $T$ contains all  $(p^{2r}-1)$th roots of unity. 

\label{Some numerical results on parameters, splitting field}

\item We list several properties of the character $\xi$ of $E^\times$ corresponding to $\hat \xi$ of $\mathcal W_E$, taken from \cite[6.1 Th]{BK-epipelagic}.\begin{enumerate}[(I)]
\item $\xi$ is fixed under conjugation by $\Gal(E/K')$, where $K'$ is the maximal unramified subextension of $T/K$.

\item The level of $\xi$, as a character of $E^\times$, is $p^r+1$.

\item There exists a unique element $\beta_E\in E^\times / U_E^1$, of $E$-valuation $-(p^r+1)$, such that 
$$\xi(1+x) = \psi_E(\beta_E^{p^r} x), \quad \text{ for all }1+x\in U_E^{p^r+1},$$
where $\psi_E = \psi_K\circ\mathrm{tr}_{E/K}$ is an additive character of $E$ of conductor $2-p^{2r}$.

\end{enumerate}
Conversely, these properties determine  $\xi|_{U_E^1}$ uniquely. The element $\beta_K$ will be used in Proposition \ref{ramification theorem in RY template} below. \hfill \qed
\label{numerical results about xi}
\end{enumerate}
\end{rmk}

\subsubsection{Adjoint Swan conductors}

Given a parameter $\varphi : \mathcal W_F\rightarrow {}^LG$ of a connected reductive group $G$, we define the {\bf adjoint Swan conductor}  $\mathrm{Adsw}(\varphi)$ of $\varphi$ to be the Swan conductor of $\mathrm{Ad}\circ \varphi$, where $\mathrm{Ad}: {}^LG\rightarrow \Aut(\hat {\mathfrak g})$ is the adjoint representation. If $\tilde\varphi$ is an irreducible representation of $\mathcal W_F$, then $\mathrm{Adsw}(\tilde\varphi)$ is the Swan conductor of the {adjoint} representation $\Ad\circ \tilde\varphi \cong \tilde\varphi^{\vee}\otimes \tilde\varphi $ of $\tilde\varphi$.


By restricting to the Weil group of an unramified extension of $F$ if necessary, let $\tilde\varphi$ be an epipelagic parameter of $\tilde G = \GL_m(\mathbb C)$. We compute the adjoint Swan conductor of $\tilde\varphi$ in this case.
\begin{prop}
\label{adjoint swan conductor GL-case}
If $\tilde\varphi$ is a self-dual epipelagic parameter of $\tilde G = \GL_m(\mathbb C)$, then $$\mathrm{Adsw }(\tilde\varphi) = m-1.$$ 
\end{prop}
\proof
We first use $\Ad\circ \tilde\varphi \cong \tilde\varphi^{\vee}\otimes \tilde\varphi $. Since the LLC preserves Rankin-Selberg conductors, we apply \cite[6.5 Th(i)]{BHK-RS} to obtain 
\begin{equation}
\label{Rankin-Selberg conductor, self-dual case}
\mathrm{ar}(\tilde\varphi^\vee \otimes \tilde\varphi) = \mathrm{ar}(\tilde\pi^\vee \times \tilde\pi) = m^2 +\mathfrak c(\beta)-1,\quad \text{when }\tilde\varphi^\vee\cong \tilde\varphi,
\end{equation}
where $\mathfrak c(\beta)$ is the `generalized discriminant' defined in \cite[6.4]{BHK-RS}, which equals $m-1$ in the minimal case, and $\mathrm{ar}(\tilde\pi^\vee \times \tilde\pi)$ is the conductor defined in \cite{BHK-RS}. Moreover, since $\tilde\varphi|_{\mathcal I_F}$ remains irreducible (\cite[3.1 Lem]{BK-epipelagic}, i.e., inertially discrete in the sense of \cite{Gross-Reeder}), the fixed space $(\mathbf V_{\Ad\circ \tilde\varphi })^{\tilde\varphi^\vee \otimes \tilde\varphi(\mathcal I)}$ is the center $\mathfrak z$ of $\mathfrak{gl}_m$. Therefore, 
$$\mathrm{sw}(\tilde\varphi^\vee \otimes \tilde\varphi)=\mathrm{ar}(\tilde\varphi^\vee \otimes \tilde\varphi) - \dim (\mathbf V(\mathfrak{gl}_m)/\mathfrak z) = (m^2+m-2)-(m^2-1)=m-1,$$
which is the desired result. $\hfill\blacksquare$

\begin{rmk}
If $\tilde\varphi$ is moreover totally wild (as in Subsection \ref{subsection Total wildness and primitivity properties}), then Proposition \ref{Rankin-Selberg conductor, self-dual case} can be verified directly from the {adjoint ramification filtration} in (\ref{adjoint ramification filtration}).
\qed\end{rmk}

Now let $G$ be a classical group and view $\hat G$ as a subgroup of a general linear group $\hat{\tilde G}$ in the standard way as in Subsection \ref{subsubsection Pinned automorphisms}. If $\varphi$ is a parameter of $G$, we denote its lift to $\tilde G$ via Subsection \ref{subsubsection Lifting parameters} by $\tilde\varphi$. By \cite[Sec 7]{GGP}, \begin{equation}
\label{GGP, Ad is just Sym, Alt, or Asai}
\Ad \circ\varphi\cong 
   \begin{cases}
      \wedge^2 \tilde\varphi {}& \text{$\hat G $ is orthogonal,} \\
       \Sym^2\tilde\varphi 
       & \text{$\hat G $ is symplectic,} \\
   \end{cases}
   \quad\text{and}\quad 
   \cong 
   \begin{cases}
      {\mathrm{As}^-\tilde\varphi}& \text{$\hat G $ is odd unitary,} \\
       \mathrm{As}^+\tilde\varphi & \text{$\hat G $ is even unitary.} \\
   \end{cases}
\end{equation}
We recall that $\Ind_{F/\Fo}(\tilde\varphi\otimes \tilde\varphi) \cong S_+\tilde\varphi \oplus S_-\tilde\varphi $ as a $\mathcal W_{\Fo}$-representation, where 
\begin{equation*}
\{S_+,  S_-\} = 
   \begin{cases}       \{\Sym^2, \wedge^2\} & \text{if }F=\Fo, \\
     \{\mathrm{As}^+, \mathrm{As}^-\} & \text{if }F/\Fo\text{ is quadratic.} \\
   \end{cases}
\end{equation*}
We can therefore define the adjoint Swan conductor $\mathrm{Adsw}(\varphi) $ of $\varphi$, which is equivalent to the definition given in \cite{Gross-Reeder}.

\begin{prop}
\label{adjoint swan conductor simple supercuspidal for classical groups}
If $\hat G$ is 
$$\text{ramified $
\mathrm {O}_{2n}$ \quad (resp. $\SP_{2n}$, \quad ramified $\mathrm U_{2n+1}$, \quad unramified $\mathrm U_{N}$)}$$
where $n,N\geq 1$, and $\varphi$ be an epipelagic parameter of $G(\Fo)$ whose lifting $\tilde\varphi$ is also epipelagic, then $\mathrm{Adsw}(\varphi) $ is equal to 
$$n-1 \quad\text{(resp.}\quad n, \quad n, \quad \text{$N-1$)}.$$
\end{prop}
\proof If $G$ is unramified unitary, then the proof is similar to Proposition \ref{adjoint swan conductor GL-case}, since $\mathrm{sw}_{\Fo}(\mathrm{Ad}\circ \tilde\varphi) = \mathrm{sw}_{\Fo}(\mathrm{As}^\pm\tilde \varphi) = \mathrm{sw}_{F}(\mathrm{As}^\pm\tilde \varphi|_{\mathcal W_F})$, and $\mathrm{As}^\pm\tilde \varphi|_{\mathcal W_F}$ is just $\tilde \varphi\otimes \tilde \varphi$. For the other types of $G$ mentioned above, the proof is similar to \cite[Sec 9]{Oi-Sp-and-SO-even}, where certain conditions on $p$ are imposed. To bypass those conditions and accommodate ramified unitary groups, we recall several facts from \cite{BK-epipelagic}. Write $\tilde\varphi = \Ind_{K/F}(\tilde\rho\otimes \tilde\tau)$ with $k = [K:F]$ and $\dim \tilde\rho = p^r$. We know that $\dim\tilde\tau = 1$ in the epipelagic case, and $k$ is even if and only if $F=\Fo$. Let $D_i= D_i(\varphi,F)$ be the $i$th lower ramification group of $\Ad \circ\varphi$, then 
$$\tilde\varphi|_{D_1} = \bigoplus_{g\in\mathcal  W_F/\mathcal W_K}{}^g(\tilde\rho|_{D_1}),$$
 with $\tilde\rho|_{D_1}$ an irreducible representation of Heisenberg type \cite[4.2 Prop]{BK-epipelagic}. Denoting by $\tilde\rho_{\chi_g}$ the Heisenberg representation with non-trivial central character $\chi_g$ parametrized by $g\in \mathcal W_F/\mathcal W_K$, we have 
$$(S\tilde\varphi)|_{D_1} = \left(\bigoplus_{g\in \mathcal W_F/\mathcal W_K} S\tilde\rho_{\chi_g}\right) \oplus \left(\bigoplus_{
   \begin{smallmatrix}\{g,h\}\subseteq \mathcal W_F/\mathcal W_K
    \end{smallmatrix}
} (\tilde\rho_{\chi_g}\otimes{}^c\tilde\rho_{\chi_h} + {}^c\tilde\rho_{\chi_g}\otimes\tilde\rho_{\chi_h}) \right)$$
for $S\in \{\wedge^2, \Sym^2,  \mathrm{As}^+\}$. Note that the characters $\{\chi_g\}_{g\in \mathcal W_F/\mathcal W_K}$ are distinct and non-quadratic when $p$ is odd. Similar to \cite[Lem 9.13]{Oi-Sp-and-SO-even}, requiring only that $p$ be odd, we have $S\tilde\rho_{\chi_g}   = (\tilde\rho_{\chi^2_g} )^{\oplus P} $ for all $g\in\mathcal W_F/\mathcal W_K$, where 
$$P = \frac{p^r-1}{2}, \quad \frac{p^r+1}{2},\quad  \text{and} \quad p^r$$ when $S = \wedge^2, \Sym^2$, and $\mathrm{As}^+$, respectively. Moreover, by \emph{loc. cit.}, for all $\{g,h\}\subseteq \mathcal W_F/\mathcal W_K$, 
\begin{equation*}
{\tilde\rho}_{\chi_{g}}\otimes {\tilde\rho}_{\chi_{h}} = 
\begin{cases}
({\tilde\rho}_{\chi_{g}\chi_{h}})^{\oplus p^r}  &  \text{when }\chi_{g}\chi_{h}\neq 1,
 \\
 \oplus_{\delta\in (D_1/D_2)^\wedge}\delta  &  \text{when }\chi_{g}\chi_{h}=1.\end{cases}
\end{equation*}
Here $(D_1/D_2)^\wedge$ denotes the Pontryagin dual. By \cite[6.3 Prop]{BK-epipelagic} (or \cite[Prop 9.17]{Oi-Sp-and-SO-even}), the Swan conductors of the above components are given by
 $$\mathrm{sw} _F({\tilde\rho}_{\chi_g}) = \frac{1}{k}
\quad\text{and}\quad \mathrm{sw} _F(\delta) = \frac{1}{k(p^r+1)}\text{ if $\delta$ is non-trivial}.$$
 By the  additivity of sw, we have $\mathrm{sw}_{\Fo}(S\tilde\varphi)$ equal to
 \begin{equation*}
   \begin{split}
        &
     \sum_{g\in \mathcal W_F/\mathcal W_K}  {P   \mathrm{sw}_{\Fo}(\tilde\rho_{\chi_g^2}) }+ \sum_{
      \begin{smallmatrix}
         \{g,h\}\subseteq \mathcal W_F/\mathcal W_K
   \\
   \chi_{g}\chi_{h}\neq 1
      \end{smallmatrix}
}
{p^r}\mathrm{sw}_{\Fo}(\tilde\rho_{\chi_g})+ \sum_{
 \begin{smallmatrix}
         \{g,h\}\subseteq \mathcal W_F/\mathcal W_K
   \\
   \chi_{g}\chi_{h}= 1
      \end{smallmatrix}
}\sum_{
   \begin{smallmatrix}
   \delta\in (D_1/D_2)^\wedge
   \\
   \delta\neq 1   \end{smallmatrix}
}\mathrm{sw}_{\Fo}(\delta)\\
   & = [F:\Fo]^{-1}\left([\mathcal W_F:\mathcal W_K]\frac{P}{k} + \sum_{
      \begin{smallmatrix}
         \{g,h\}\subseteq \mathcal W_F/\mathcal W_K
   \\
   \chi_{g}\chi_{h}\neq 1
      \end{smallmatrix}
   }
   \frac{p^r}{k}+ 
   \sum_{
 \begin{smallmatrix}
         \{g,h\}\subseteq \mathcal W_F/\mathcal W_K
   \\
   \chi_{g}\chi_{h}= 1
      \end{smallmatrix}
}\sum_{
   \begin{smallmatrix}
   \delta\in (D_1/D_2)^\wedge
   \\
   \delta\neq 1   \end{smallmatrix}
}
\frac{1}{k(p^r+1)} \right)
     \\
     & = [F:\Fo]^{-1}\left( {P} +  \sum_{
 \begin{smallmatrix}
         \{g,h\}\subseteq \mathcal W_F/\mathcal W_K
   \\
   \chi_{g}\chi_{h}\neq 1
      \end{smallmatrix}
}\frac{p^r}{k} +
      \sum_{
 \begin{smallmatrix}
         \{g,h\}\subseteq \mathcal W_F/\mathcal W_K
   \\
   \chi_{g}\chi_{h}= 1
      \end{smallmatrix}
} \frac{ p^r -1}{k }  \right).
   \end{split}
\end{equation*}
Similar to \cite[Prop 9.16]{Oi-Sp-and-SO-even}, we have
\begin{equation*}
   \begin{split}
   &\#\{  \{g,h\}\subseteq \mathcal W_F/\mathcal W_K:\chi_{g}\chi_{h}\neq  1\}  = \frac{k(k-2)}{2}[F:\Fo]\quad\text{ and }\quad  
     \\
  &\#\{  \{g,h\}\subseteq \mathcal W_F/\mathcal W_K:\chi_{g}\chi_{h}= 1\} = \frac{k}{2} [F:\Fo] , \end{split}
\end{equation*}
which remain valid even if $E_{\tilde\rho}/F$ is non-Galois (in contrast to \cite[Sec 9]{Oi-Sp-and-SO-even}) due to the symmetry between  $\{\chi_{g},\chi_{g}^{-1}\}$ and that $\chi_{g}^{-1} = \chi_h$ for some $h\in \mathcal W_F/\mathcal W_K$ with $g^{-1} h\notin \mathcal W_K$. The sum $  \mathrm{sw}_{\Fo}(S\tilde\varphi)$ computed above is therefore $n-1$ (resp. $n$, resp. $n$) when $\hat G$ is ramified $
\mathrm {O}_{2n}$ (resp. $ \SP_{2n}$, resp. ramified $\mathrm U_{2n+1}$). This completes the proof.
\hfill$\blacksquare$

\subsubsection{Reeder-Yu's template: the wild part}
\label{Reeder-Yu template the wild part}

Let's recall from \cite[Sec 7.2 and 7.3]{Reeder-Yu} the template for constructing epipelagic Langlands parameters. Since we focus on the wild part of a Langlands parameter, we denote our base field by $K$ until the end of Section \ref{subsection The wild part}.

Let $K$ be a non-Archimedean local field of characteristic 0, and $G$ be a connected classical group over $K$, assumed to be simply connected. We require that $p$ is not a torsion prime for $\hat G$.

There are certain conditions on our field extensions. Let $T/K$ be a tamely totally ramified extension. We assume $T/K$ to be Galois by requiring the residue field $\mathbb F_K$ of $K$ containing enough roots of unity. Let $E/T$ be a totally ramified extension which is totally wild (i.e., its degree is a power of $p$), defined such that $\Gal(E/T) \cong U_T^1/U_T^2 \cong \mathbb F_K$, i.e., $N_{E/T}(E^\times)\cap U^1_T = U^2_T$, via local class field theory.

The epipelagic parameters in \emph{loc. cit.} are constructed by stable functionals. We pick a point $x\in \mathcal B(G,K)$, and let $ {\mathfrak c} = \mathbb F_K \beta \subset {\mathsf{V}}^*_{x,0+}\otimes \mathbb F_K$ be the Cartan subspace containing the stable functional $\beta$. We require that the stable functional $\beta$ generates a tamely ramified field extension whose ramification index is coprime to $p$.

Let $\hat {\mathcal T}$ be the torus of $\hat G$ in the pinning fixed in Section \ref{subsection L-groups of classical groups}, and assumed it is $\mathcal W_K$-invariant. We view the $p$-torsion subgroup $\hat{\mathcal T}[p]$ of $\hat{\mathcal T}$ as an elementary abelian $p$-group in $\hat G$, or equivalently, as an $\mathbb F_p$-space, which can be identified with $X\otimes \mathbb F_p$, where $X = \Hom(\mathbb C^\times, \hat{\mathcal T})$. The Cartan subspace $ {\mathfrak c} = \mathbb F_K \beta \subset {\mathsf{V}}^*_{x,0+}\otimes \mathbb F_K$ is then viewed as a subspace of $\hat {\mathfrak t}:= {\mathbb F}_K\otimes X$.

Put $k=[T:K]$, and let $h_{\zeta_k}$ be the $\mathbb F_p$-minimal polynomial of a primitive $k$th root of unity ${\zeta_k}$ in $\overline {\mathbb{F}}_p$, and define 
$$\hat{\mathcal T}[p,{\zeta_k}]=\{y\in \hat{\mathcal T}[p]: h_{{\zeta_k}}(s)y=1\},$$
which is an $\mathbb F_p\Gal(T/K)$-module contained in $\hat{\mathcal T}[p]$. 

By \cite[(7.12)]{Reeder-Yu}, there is an isomorphism
\begin{equation}
\label{Reeder-Yu isomorphism on Cartan space}
{\mathfrak c}(\mathbb F_K)  = \mathbb F_K \beta  \rightarrow \Hom_{\Gal(T/K)}(\mathbb F_{T},\hat{\mathcal T}[p,\zeta_k]),\quad \alpha\mapsto \hat \psi_\alpha,
\end{equation}
characterized by the property that 
\begin{equation}
\label{characterizing condition of the Cartan space-Heisenberg space map}
(\hat \psi_\alpha(b),\check x) = \mathrm{tr}_{{\mathbb F}_{T}/\mathbb F_p}(b\left<\alpha, \check x\right>),\quad \text{for all $b\in \mathbb F_{T}$ and $\check x\in \check X = \Hom(\hat{\mathcal T},\mathbb C^\times)$. }
\end{equation}
To continue, fix a generator $s$ of the inertia subgroup and a Frobenius lift $t$ in $ \Gal(T/K)$. Define 
\begin{equation}
\label{pre Reeder-Yu template the tame part}
\hat \mu = \hat \mu_{(\hat n_s, \hat n_t)}: \Gal(T/K) \to N_{{}^LG}(\hat {\mathcal T}), \quad s\mapsto \hat n_s, \quad t \mapsto \hat n_t,
\end{equation}
by their corresponding actions. The parameter
\begin{equation}
\label{Reeder-Yu template, general form}
\hat \psi_\alpha \rtimes \hat \mu: \Gal(E/K) = \Gal(E/T)\rtimes \Gal(T/K) \to {}^LG
\end{equation}
then satisfies some desirable properties, such as the minimal conductor property and the formal degree property, as established in \cite[Sec 7.3]{Reeder-Yu}.

In the following subsections, we aim at describing a totally wild, primitive parameter for a general linear group as in Subsection \ref{subsection Total wildness and primitivity properties} via the above template, and establish an isomorphism analogous to (\ref{Reeder-Yu isomorphism on Cartan space}) in Proposition \ref{ramification theorem in RY template}. As we have just observed, while there is no condition on $p$ in describing a totally wild, primitive parameter, the conditions required to fit a parameter into Reeder-Yu's template are quite restrictive. Certain results from \cite{Gross-Levy-Reeder-Yu}, on which \cite{Reeder-Yu} is based, do not directly apply in our setting. We will appeal to \cite{BK-epipelagic} and \cite{BH-Eff} instead. Our main result, Proposition \ref{ramification theorem in RY template}, provides an alternative version of (\ref{Reeder-Yu isomorphism on Cartan space}).

\subsubsection{Totally wild  primitive parameters in Reeder-Yu's template}
\label{Totally wild  primitive parameters in Reeder-Yu's template}

We return to the setup in Subsection \ref{subsection Total wildness and primitivity properties}. Certain conditions, though much milder, are imposed in \cite{BK-epipelagic}. First, there is a normalization condition on the additive character \cite[(5.1)]{BK-epipelagic}, namely
$$\psi_{K}(\zeta^p) = \psi_{K}(\zeta)\quad \text{ for all }\zeta\in \boldsymbol{\mu}_K,$$ 
which can be guaranteed by requiring that $\psi_K$ factors through $\mathrm{tr}_{\mathbb F_K/\mathbb F_p}$. Moreover, we require $K$ to contain all $(p^r+1)$th roots of unity, i.e., 
$$\mathbb F_K\supseteq \mathbb F_{p^{2r}},$$
which is guaranteed by \cite[5.3 Th]{BK-epipelagic} (see Remark \ref{Some numerical results on parameters} about degrees), so that our subsequent tamely ramified extensions of $K$ of degree $p^r+1$ are automatically Galois.

We now reverse the discussion of Subsection \ref{subsection Total wildness and primitivity properties} by starting with a tower of extensions $K\hookrightarrow T\hookrightarrow E$ and a non-trivial character  $\hat \xi$ of $\mathcal W_E$, all satisfying the conditions in Remark \ref{Some numerical results on parameters}. Let $\eta = \eta_{\hat \xi}$ be the associated Heisenberg representation of $\mathcal W_T$. By \cite[Th 2.4]{Gerardin}, 
there is a unique isomorphism class of representations $\Omega = \Omega_{\hat \xi}$ of ${\mathcal W_T}\rtimes \mathrm{Sp}(A)$, where $A$ is $\Gal(E/T)$ but viewed as a symplectic space defined by commutators, such that $\Omega|_{\mathcal W_T}\cong \eta$. We then define $\Omega$ on $\mathcal W_K$ via the pullback along $\mathcal W_K \to \Gal(T/K)\to \mathrm{Sp}(A)$.

Note that in Reeder-Yu's template, the group $G$ is assumed to be simply-connected, i.e., $\hat G$ is adjoint, so we consider the projective representation of $\Omega$, i.e.,
$$\overline{\Omega}: \Gal(E/K)\to \hat G := \mathrm{PGL}_{p^r}(\mathbb C).$$
The image $J = \overline{\Omega}(\Gal(E/T))$ is abelian. Aassume it is contained in the pinned maximal torus $\hat {\mathcal T}$.

 The following lemma and its proof are analogous to \cite[Lem 7.2]{Reeder-Yu}. 
\begin{lem}
\label{image in a maximal torus}
We have $Z_{\hat {G}}(J) = \hat {\mathcal T}$. 
\end{lem}
\proof
The assertion follow from the standard realization of the Heisenberg representation. Let $L$ be the intermediate subfield of $E/T$ such that $\Gal(E/L)$ forms a Lagragian subspace of the symplectic space $\Gal(E/T)$. As in \cite[Sec 1.3]{Gerardin}, the space $A_-$ of $\mathbb C$-valued functions on $\Gal(L/T)$ provides such a realization. Noting that $\Gal(E/L)$ acts on $A_-$ by scalar multiplications and $\Gal(L/T)$ acts by (finite order) translations. We now show that 
\begin{equation}
\label{for all coroots there is a}
\text{for all coroots $\check\alpha$ of $\hat {\mathcal T}$, there exists $t\in J$ such that $\check\alpha(t) \neq 1$.}
\end{equation}
For $\hat G = \mathrm{PGL}_{p^r}(\mathbb C)$, each coroot can be realized by a pair of distinct group elements, say $g\neq h$, in $\Gal(L/T)$. Then the $g$th and $h$th eigenvalues of the translation $t = \overline{\Omega}(gh^{-1})$ are different, showing (\ref{for all coroots there is a}).
\hfill$\blacksquare$

Denote $\overline{\hat \mu} := \overline{\Omega}|_{\Gal(T/K)}$, and define elements $\hat n_s^1 = \overline{\hat \mu} (s)$ and $\hat n_t^1 = \overline{\hat \mu} (t)$, both belonging to $\hat N = \hat N_{\hat{G}}(\hat {\mathcal T})$ by the above Lemma. We can write the projective representation as in (\ref{Reeder-Yu template, general form}):
 $$\overline{\tilde\rho}:=\overline{\Omega}|_{\Gal(E/T)}= \overline{\eta}|_{\Gal(E/T)}\rtimes  \overline{\hat \mu}_{(\hat n_s^1,  \hat n_t^1)}: \Gal(E/K)\rightarrow \hat{G} .$$
Finally, it is well-known \cite[Sec 1]{Koch-primitive} that  $\overline{\tilde\rho}$ can be lifted to a parameter of $\hat{\tilde G}:=\mathrm{GL}_{p^r}(\mathbb C)$, with ambiguity given by a character of $\mathcal W_K$. If such a character is not tamely ramified, then the resulting lifting is not epipelagic. Therefore, letting 
${\hat \mu}_{(\hat n_s^1,  \hat n_t^1)}$ be a lift of $\overline{\hat \mu}_{(\hat n_s^1,  \hat n_t^1)}$, we can verify that
 $${\tilde\rho}:={\eta}|_{\Gal(E/T)}\rtimes  {\hat \mu}_{(\hat n_s^1,  \hat n_t^1)}: \mathcal W_K\rightarrow \hat{\tilde G} $$
is a totally wild, primitive parameter of $\GL_{p^r}(K)$ which has Swan conductor 1, i.e., epipelagic.


\subsubsection{Relations with the Ramification Theorem}

Continue from the above paragraph, we describe a relation between the bijection (\ref{Reeder-Yu isomorphism on Cartan space}) for $\tilde G = \GL_m$ and the restriction to epipelagic objects of the bijection $\Phi_K$ appearing in the Ramification Theorem \cite[6.1]{BH-Eff}. The inverse of $\Phi_K$ is denoted by 
$$\mathcal Y:\mathcal E(K) \rightarrow \mathcal W_K\backslash \widehat{\mathcal P_F}$$
in \cite[8.2 Th]{BH-LTL4}, where the sets on both sides are defined as follows.
\begin{enumerate}[(i)]
\item $\mathcal E(K)$ is the set of endo-classes of simple characters over $K$. Any such class is uniquely determined by an irreducible supercuspidal representation of an appropriate general linear group.

\item $\widehat{\mathcal P_F}$ is the set of isomorphism classes of irreducible representations of $\mathcal P_F$, on which $\mathcal W_K$ acts by conjugation. 
\end{enumerate}

Now let $\tilde\pi$ be an epipelagic representation of $\tilde G= \GL_{p^r}(K)$, which is totally wild and primitive by the definitions in  Subsection \ref{subsection Total wildness and primitivity properties}. Suppose that $[\Lambda,0, \beta]$ is an underlying epipelagic maximal simple stratum, where $ \beta$ is a stable functional, which gives rises to a simple character $\psi_{ \beta}$ over $K$ with endo-class $[\psi_\beta]$. By \cite[8.3 Lem 2]{BH-LTL4}, two such supercuspidal representations give rise to the same endo-class if and only if they differ by a twist by a tamely ramified character of $K^\times$, which does not affect the equivalence class of the underlying stratum. Therefore, we obtain a well-defined map 
$${\mathfrak c}(\mathbb F_K )  = \mathbb F_K \beta \rightarrow \mathcal E(K), \quad \alpha \mapsto [\psi_\alpha],$$
which is injective by \cite[6.1 Cor]{BK-epipelagic}. Indeed, aside from $0\in {\mathfrak c} $ (which give rises to the endo-class of the trivial character on $K^\times$), the character $\psi_\alpha$ determines and is determined by $(\det\alpha )U_K^1$. We hence see from the explicit form of $\tilde\beta$ in Subsection \ref{subsection Construction of quasi-simple supercuspidals GL-case}
that the cosets $(\det\alpha) U_K^1 = c \varpi_K^{-1} U_K^1$, for $c$ ranging over $\mathbb F_K^\times$ and any uniformizer $\varpi_K$ of $K$, are all distinct.

\begin{rmk} 
\label{different fields E and T}
The tamely ramified extension $T$ associated to the Langlands parameter of $\tilde\pi$ defined in Subsection \ref{subsection Total wildness and primitivity properties} is the splitting field of the $(p^{2r}-1)$th roots of $(-1)^p\det(\beta)^{p^r-1}$ (see  Remark \ref{Some numerical results on parameters}\ref{Some numerical results on parameters, splitting field}). By \cite[5.2 Th]{BK-epipelagic}, if $\alpha_i = c_i\beta$, for $c_i\in \mathbb F_K$ and $i\in \{1,2\}$, are two stable functionals on the same $\mathbb F_K$-line, then the following are equivalent:\begin{enumerate}[(i)]
\item $E_1 = E_2$.
\item $T_1 = T_2$. 
\item $c_1^{p^r}= \zeta c_2^{p^r}\bmod U_K^1$ for some (${p^r-1}$)th roots of unity $\zeta$ in $K$. \qedhere
\end{enumerate}
\end{rmk}

The image of the composition ${\mathfrak c} (\mathbb F_K )\smallsetminus \{0\} \rightarrow \mathcal E(K) \xrightarrow{\mathcal Y} \mathcal W_K\backslash \widehat{\mathcal P_F}$ consists of the symplectic-Heisenberg representations described in Subsection \ref{subsection Total wildness and primitivity properties}. The coset $\beta_E U_E^1$ is then uniquely determined, of $E$-valuation $-(p^r+1)$, under the injective composition of maps
$$ \beta \in F[\beta]^\times / U^1_{F[\beta]}  \xrightarrow{N_{F[\beta]/K}}
K^\times / U^1_K \hookrightarrow  T^\times / U_T^1 \xleftarrow{N_{E/T}} E^\times / U_E^1  \ni \beta_E $$ 
such that $N_{F[\beta]/K}(\beta) = N_{E/T}(\beta_E) \mod U_T^1$. (Note that both $N_{F[\beta]/K}$ and $N_{E/T}$ above are group isomorphisms.) If $E/K$ is totally ramified, then we obtain a bijection $\mathbb F_K^\times \beta \to \mathbb F_E^\times \beta_E$.

Denote the image $\mathcal Y([\psi_\alpha])$ by $[\hat \psi_\alpha]$, where $\hat \psi_\alpha$ is a representative in the $\mathcal W_F$-conjugacy class $[\hat \psi_\alpha]$ of representations of $\mathcal P_F$. The following proposition is analogous to the isomorphism in (\ref{Reeder-Yu isomorphism on Cartan space}).

\begin{prop}
\label{ramification theorem in RY template}
Assume that $T/K$ is totally ramified and Galois. The composition 
$${\mathfrak c} (\mathbb F_K )\smallsetminus \{0\}=  \mathbb F_K^\times\beta \rightarrow \mathcal E(K) \xrightarrow{\mathcal Y} \mathcal W_K\backslash \widehat{\mathcal P_F}, \quad  \alpha\mapsto [\hat \psi_{\alpha}],$$
for $c\in \mathbb F_K^\times  = \mathbb F_E^\times $, induces a bijection
$$\mathbb F_K^\times\beta\longrightarrow \bigsqcup_{T} \Hom_{\Gal(T/K)}(\Gal(E/T),\hat{\mathcal T}[p,\zeta_{1+p^r}]), \quad c\beta \mapsto \overline\eta_{\hat \xi_c} ,$$
where $T$ ranges over the splitting fields of the $(p^{2r}-1)$th roots of $(-1)^p\det(c\beta)^{p^r-1}$ for $c\in \mathbb F_K^\times$, and $\hat \xi_c$ is the character of $\mathcal W_E$ determined by $c\beta_EU_E^1$.
\end{prop}
\proof As in \cite[Lem 7.1]{Reeder-Yu}, each image under the above composition satisfies the characterizing relation (\ref{characterizing condition of the Cartan space-Heisenberg space map}), which is in our setting becomes
$$
(\overline\eta_{\hat \xi} (\hat c),\check x) = \mathrm{tr}_{{\mathbb F}_{T}/\mathbb F_p}(c\left<\beta, \check x\right>), \quad \text{for all $\check x\in \check X = \Hom(\hat{\mathcal T},\mathbb C^\times)$,}
$$
where $T/K$ is the field defined by $\beta$ as above, and $c\mapsto \hat c$ corresponds under the maps
$$c\in \mathbb F_{T} \cong U^1_T/U^2_T \twoheadrightarrow \frac{U^1_T}{N_{E/T}(E^\times) \cap U^1_T }\cong \frac{T^\times}{N_{E/T}(E^\times)}\cong \Gal(E/T)\ni \hat c.$$
Fix an identification of $\Gal(E/T)$ with $\mathbb F_{p^r}$ somehow. As $ c$ ranging over the subgroup $\mathbb F_{p^r}^\times$ of $\mathbb F_{T} $, we obtain all morphisms in $\Hom_{\Gal(T/K)} \allowbreak (\Gal(E/T), \allowbreak \hat{\mathcal T}[p,\zeta_{1+p^r}])$ for a fixed $T$ as described. Remark \ref{different fields E and T} then implies that all these $T$ form an $(\mathbb F_T^\times/\mathbb F_{p^r}^\times)$-principal homogeneous set, giving the desired bijection as $T$ varies. \hfill$\blacksquare$

\begin{rmk}
In the proof, any generator of $\Gal(T/K)$ acts on $\Gal(E/T)$ cyclically with full order $p^r+1$, and the restriction $\overline\Omega_{\hat \xi} |_{\Gal(T/K)}$ acts on the $p^r$-dimensional symplectic-Heisenberg representation, giving all the $p^r$ non-trivial characters of $\Gal(T/K)$ (see \cite[Lemma 2.2]{Imai-Tsushima-conductorone}).
\qed\end{rmk}

Proposition \ref{ramification theorem in RY template} therefore establishes a correspondence between stable functionals and representations of symplectic-Heisenberg type. To reiterate, a supercuspidal representation  $\tilde\pi$ of $\GL_{p^r}(K)$ is Swan conductor 1, and its twists by tamely ramified characters of $K^\times$ determine an endo-class of simple characters represented by a stable functional $\beta$. One can define field extensions $K\hookrightarrow T \hookrightarrow E$ and a character $\xi$ of $U_E^{p^r+1}$, giving rise to an element $\beta_E \in E^\times$ of valuation $-(p^r+1)$, uniquely determined mod $U_E^{1}$. The corresponding character $\hat{\xi}$ of $\mathcal W_E$ then induces the symplectic-Heisenberg representation $\rho_{\mathcal P}$, giving the wild part of the Langlands parameter of $\tilde\pi$.

\subsection{The tame part}
\label{subsection The tame part}

Continuing from the previous discussion where we described the wild part of a parameter, we now describe its tame part in this section.

Given a parameter $\tilde\varphi = \Ind_{K/F}(\tilde\rho \otimes \tilde\tau)$ for an equivalence class $[K/F,\tilde\rho,\tilde\tau]$ of admissible triples (as in Secsection \ref{subsubsection Admissible triples (BH descriptions of Langlands parameters)}), let $K^\sharp$ be the Galois closure of $K/F$, which is also tamely ramified. Recall from \cite[Sec 5]{Koch-primitive} that, for a fixed $q$, a tamely ramified Galois extension  $K^\sharp/F$ is classified by the numerical data $(k,f,r)$, where $k,f\in \mathbb Z_{>0}$ and $r$ is chosen mod $k$, such that $q^f-1 = r(q-1)=0\bmod e$. The Galois group $\Gal(K^\sharp/F)$ has two generators $(s,t)$ satisfying the relations 
\begin{equation*}
\text{$s^k=1$, \quad $t^f = s^r$, \quad $tst^{-1}=s^q$. }
\end{equation*}

\subsubsection{Reeder-Yu's template: the tame part}
\label{subsubsection A pair of generators}

Let's continue from the template in Subsection \ref{Reeder-Yu template the wild part} and explain the tame part (\ref{pre Reeder-Yu template the tame part}) in more detail.

Let $K^\sharp/F$ be a tamely ramified Galois extension such that $\Gal(K^\sharp/F)$ acts on the Langlands dual group $\hat G_0$ of a quasi-split $F$-form of a connected reductive group $G_0$ as pinned automorphisms, as described in Subsection \ref{subsubsection Pinned automorphisms}. We lift the generators $s,t\in \Gal(K^\sharp/F)\cong {\mathcal W}_F/{\mathcal W}_{K^\sharp}$ to $\dot s, \dot t\in {\mathcal W}_F$, knowing that $\dot s^k$, $\dot t^{f}\dot s^{-d}$, and $\dot t\dot s\dot t^{-1} \dot s^{-q}$ all belong to ${\mathcal W}_{K^\sharp}$. The image $\Gamma$ of $\left<\dot s, \dot t\right>$ in $\Aut(\hat G_0)$ under the morphism (\ref{outer action}) is hence generated by two elements 
\begin{equation*}
\vartheta\mapsfrom s\quad\text{and}\quad \phi\mapsfrom t,
\end{equation*}
leaving a pinned torus $\hat {\mathcal T}_0$ invariant. We assume that $\vartheta$ has order $k$.
We will define two elements,  $\hat n_s$ and $\hat n_t$, in the normalizer ${}^LN_0:= N_{{}^LG_0}(\hat{\mathcal T}_0)$ of $\hat{\mathcal T}_0$ in ${}^LG$ whose projections under ${}^LN_0 \rightarrow {\mathcal W}_F \rightarrow \Gamma$ are $\vartheta$ and $\phi$ respectively.

Put $W = W(\hat G_0,\hat{\mathcal T}_0):=N_{\hat G_0}(\hat{\mathcal T}_0)/\hat{\mathcal T}_0$. Via the theory of Cartan subspaces \cite[Sec 6.1]{Reeder-Yu}, the tamely ramified extension $K/F$ (or in the general setting, the maximal tamely ramified subextension of $F[\beta]/F$ generated by a stable functional $\beta$) determines a unique $W$-conjugacy class in $W\vartheta\times W\phi$, represented say by $(\sigma,\tau)$, such that $\sigma $ is Z-regular elliptic (in the sense of \cite[p.455 bottom]{Reeder-Yu}) of order $k$, and $\tau\sigma\tau = \sigma^q$.

Let $\hat Z_0$ be the center of $\hat G_0$. By \cite[Prop 8]{Gross-Levy-Reeder-Yu}, all lifts of $\sigma$ in ${}^LN_0$ are $\hat{ \mathcal T}_0$-conjugate and have order $k \bmod \hat Z_0$. Fix one such lift $\hat n_s$. Since $\sigma$ is elliptic, there exists $\hat n_t \in {}^LN_0$ with $\hat n_t^{kf}\in \hat Z_0$ and satisfying the relation $\hat n_t \hat n_s \hat n_t^{-1} \in \hat n_s^q\hat Z_0^{{\mathcal W}_{F}}$ \cite[Sec 7.3]{Reeder-Yu}. By construction, $(\hat n_s,\hat n_t)\mapsto (\vartheta,\phi)$ under ${}^LN_0\rightarrow \Gamma$. If $\left<\dot s,\dot t\right>$ denotes the subgroup in ${\mathcal W}_F$ generated by $\dot s,\dot t$, we define a morphism
\begin{equation}
\label{Gamma action on dual of T}
\hat \mu_{(\hat n_s,\hat n_t)}: \left<\dot s,\dot t\right>\rightarrow {}^LN_0, \quad \dot s\mapsto \hat n_s,\quad \dot t\mapsto \hat n_t,
\end{equation}
determined by the stable functional $\beta$ via the chosen Cartan subspace containing it, as well as the choices of $(\hat n_s,\hat n_t)$.

\subsubsection{Embeddings of L-groups}
\label{subsubsection Embeddings of L-tori}

We first consider $\tilde G = \GL_{N,F}$. Let $k$ be a divisor of $N$, and $K/F$ be a field extension of degree $k$. Denote 
$$\tilde G_K = \GL_{Nk^{-1},K} 
\quad\text{and}\quad 
\tilde { M }:= \Res_{K/F}\tilde {G}_K.$$
We see that $\tilde M$ is a Levi subgroup of $\tilde  G$ after base-change to $K$. 
The dual group of $\tilde M$ is 
$$ \hat{\tilde M} =  \Ind_{K/F}\hat {\tilde G}_K\cong \GL_{Nk^{-1}}(\mathbb C)^k .$$ 
Usually, we view $\hat{\tilde M}$ as embedded into $\hat{\tilde G}$ as a subgroup of block-diagonal matrices. The action of $\mathcal W_F$ on $\hat{\tilde M}$ is the induced action, which defines the L-group 
${}^L\tilde{M} = \hat{\tilde M}\rtimes {\mathcal W}_F$ of $\tilde M$.

To embed ${}^L\tilde M$ into ${}^L\tilde G $, we denote an element in $\hat {\tilde M}$ by blocked entries $(x_g)_{g\in {\mathcal W}_F/{\mathcal W}_K}$, where each $x_g\in \hat{\tilde{G}}_K$. Let $\Sigma_{K/F}$ be the permutation group of the set ${\mathcal W}_F/{\mathcal W}_K$, and consider the composition of morphisms
$${\mathcal W}_F\to \Sigma_{K/F} \to \mathrm{O}_k(\mathbb C), \quad w\mapsto (\hat {\tilde n}_w)_k, $$
where the first map is defined by the left action of ${\mathcal W}_F$ on ${\mathcal W}_F/{\mathcal W}_K$, and $(\hat {\tilde n}_w)_k$ is the permutation matrix (with entries with 1 and 0 only) corresponding to the action of $w\in {\mathcal W}_F$. We put 
\begin{equation}
\label{Levi embedding by tensor product, GL-case}
\hat {\tilde n}_w := I_{Nk^{-1}}\otimes (\hat {\tilde n}_w)_k \in \hat {\tilde{G}}_K \otimes \mathrm{O}_k(\mathbb C)\hookrightarrow \hat {\tilde G} ,
\end{equation}
then 
\begin{equation}
\label{L-embedding for GL}
\tilde \iota = \tilde \iota_{(\hat{\tilde n}_s,\hat {\tilde n}_t)}: {}^L\tilde M \rightarrow {}^L\tilde G, \quad (x_g)\rtimes w\mapsto \diag(x_g)\hat {\tilde n}_w \rtimes w,
\end{equation}
defines an L-embedding of ${}^L\tilde M$ into ${}^L\tilde G$, which is clearly Levi-type endoscopic.

If $K/F$ is a tamely ramified extension as in Subsection \ref{subsubsection A pair of generators}, then we can describe $\tilde \iota$ as the embedding (\ref{Gamma action on dual of T}) by taking $\hat {\tilde n}_s$ and $\hat {\tilde n}_t$ to be given by the actions of $\dot s$ and $ \dot t$ in ${\mathcal W}_F$, and so $\tilde \iota:{}^L \tilde {{ M}}\rightarrow {}^L\tilde G$ is determined by the diagonal inclusion $\tilde \iota|_{\hat {\tilde M}}:\hat{\tilde{ M}}\rightarrow \hat {\tilde{G}}$
and $\tilde \iota|_{\mathcal W_F} = \hat \mu_{(\hat{\tilde n}_s,\hat {\tilde n}_t)}$ .

We now consider classical groups on the dual side. Let ${}^L\hat G $ be the following L-groups in Galois form (Remark \ref{Remark on Galois form}):
\begin{equation}
\label{types of classical groups, dual side}
  \mathrm{O}_{2n,F^\sharp/F}(\mathbb C), \quad 
   \SP_{2n}(\mathbb C),\quad \text{ or }\quad \mathrm{GL}_{2n+1,F/F_\bullet}(\mathbb C).
\end{equation}
Let $K$ be a cyclic tamely ramified extension of $F$ of degree $k$, and $s\in \mathcal W_{\Fo}$ be a generator for $\mathcal W_{\Fo}/\mathcal W_K$, such that $s^{[F:\Fo]}\in \mathcal W_F$ and $s^{k[F:\Fo]}\in \mathcal W_K$. Under the classification of epipelagic strata in Section \ref{subsection Epipelagic strata}, we further assume the following two cases:
\begin{equation}
\label{two cases for simple supercuspidals}
  \text{$k$ is even when $F=\Fo$} \quad \text{and}\quad  \text{$k$ is odd when $F/\Fo$ is ramified quadratic.}
\end{equation}
Let $K^\flat$ be defined by $\mathcal W_{K^\flat} = \left<s^{k[F:\Fo]/2},\mathcal W_{K} \right>$, so that $K/K^\flat$ is always quadratic. Here ${K^\flat}$ is an extension over $F$ in the first case, while $K/{K^\flat}$ extends $F/\Fo$ in the second.

In both cases of (\ref{two cases for simple supercuspidals}), we define
$$
 G_{K} = \mathrm{GL}_{{Nk^{-1},K}},\quad G_{K^\flat} = \mathrm U_{{Nk^{-1},K/K^\flat}}, \quad \text{and}\quad M = \Res_{K^\flat/\Fo}G_{K^\flat} , \\
$$
whose L-groups are given by 
\begin{equation*}
   \begin{split}
     & {}^LG_{K} = \hat G_K \rtimes \mathcal W_K = \mathrm{GL}_{Nk^{-1}}(\mathbb C)\times \mathcal W_K,\quad
      \\
& {}^L G_{K^\flat} = \Ind_{K/K^\flat}\hat G_K\rtimes \mathcal W_{\Fo},\quad
\text{and}\quad 
{}^LM = \Ind_{K^\flat/\Fo}\hat G_{K^\flat}\rtimes \mathcal W_{\Fo}   \end{split}
\end{equation*}
When $F = \Fo$, the inclusion $\iota_M: {}^LM \to {}^L\tilde M$ is just the identity map. When $F/\Fo$ is ramified quadratic, we  denote by  
$j_{G_{K^\flat}}: {}^L G_K \hookrightarrow {}^LG_{K^\flat}$ the unitary-type endoscopic embedding, and apply $\Ind_{K/F}$ and $\Ind_{K^\flat/\Fo}$ to both sides, respectively, to define $j_M: {}^L\tilde M \hookrightarrow {}^LM$.

Until the end of Subsection \ref{subsubsection The rectifying characters}, we assume moreover that $K/F$ is totally (tamely) ramified. In the following two subsections, similar to (\ref{Levi embedding by tensor product, GL-case}), we will define a $\mathcal W_{\Fo}$-invariant  complex group $\hat G_0$ and elements $\hat n_{w,0}\in  N_{{}^L{G_0}}(\hat{\mathcal T}_0)$, for $w\in \{s,t\}$, such that 
\begin{equation*}
\hat n_w := I_{Nk^{-1}}\otimes \hat n_{w,0} \in \hat {G}_K \otimes {}^L G_0 \hookrightarrow {}^L{G} ,
\end{equation*}
induces an embedding $\iota: {}^L M \to {}^LG$ as in (\ref{L-embedding for GL}).

\subsubsection{The Weyl group elements by Galois actions}
\label{subsection The Galois element n_s lifting sigma}


When ${}^LG$ is as in (\ref{types of classical groups, dual side}), we define an L-group ${}^LG_0$ and an element $\hat n_{s,0} \in {}^LN_0:= N_{{}^L{G_0}}(\hat{\mathcal T}_0)$ in the respective cases, as follows.

\begin{enumerate}[(i)]

\item When ${}^L G = \mathrm{O}_{2n,F^\sharp/F}(\mathbb C)$, put ${}^LG_0 = \mathrm{O}_{k,K/K^\flat}(\mathbb C)$ in the Galois form, and let $ \hat n_{\mathrm O} $ be the permutation matrix switching the middle two coordinates, then the element $\hat \sigma$ belongs to the coset $\hat G_0 \hat n_{\mathrm O} = \hat G_0 \vartheta$, which can be embedded into $
 \hat  G_0^\sharp:=\mathrm O_{k}(\mathbb C) $. We again put $\hat n_{s,0} = \hat\sigma$, whose order is $k$.

\item When ${}^L G = \SP_{2n}(\mathbb C)= \times \mathcal W_F$, put ${}^L G_0 = \SP_{k}(\mathbb C)\times \mathcal W_K$ where $k=2m$, we verify that 
$$\hat n_{s,0} = \hat \sigma \diag(1_{m},-1,1_{m-1}) $$ belongs to $\hat G_0$, where $\hat\sigma$ is as in case (i). The order of $\hat n_{s,0}$ is $2k=4m$ in this case.

\item When ${}^LG = \mathrm{GL}_{2n+1,F/F_\bullet}(\mathbb C)$, put ${}^LG_0= \GL_{k,K/K^\flat}(\mathbb C)$ in the Galois-form, we have $s\in {\mathcal W}_{\Fo}\smallsetminus {\mathcal W}_F$. Lifting the Weyl group element $\sigma$ to $\hat \sigma\in \hat G_0 = \GL_k(\mathbb C)$, we put $\hat n_{s,0} = \hat\sigma \hat H_0 \vartheta\in \hat G_0 \vartheta$, then $\hat n_{s,0}^2 = \hat \sigma^2\in \hat G_0 $. Note that $\hat n_{s,0}^k = \hat H_0 \vartheta$, so that the order of $\hat n_{s,0}$ is $2k$.
\end{enumerate}
Finally, we define 
$$\hat n_s := I_{Nk^{-1}}\otimes \hat n_{s,0}\in  {}^LN.$$
In all cases, we can choose other lifts of $\sigma$ in ${}^LN$ that are also monomial matrices with the same non-zero entry positions as $\hat n_s$ but containing other suitable values. It is then straightforward to show that all such lifts are $\hat {M}$-conjugate to an element in $\hat n_s \hat Z^{\mathcal W_\Fo}$.

We then determine the conditions on the non-zero entries of the possible lifts $\hat n_t \in {}^LG_0$. It turns out that there are only few possibilities.

For $j,k\in \mathbb Z$ with $k>0$, denote by $r_k(j) \in \{0,\dots,k-1\}$ such that $ j \equiv r_k(j ) \bmod k$.

\begin{lem}
\label{characterization of n_t}
With $\hat n_s$ as constructed in Subsection \ref{subsection The Galois element n_s lifting sigma} and a sign $z\in \hat Z^{\left<\vartheta,\phi\right>}= \{\pm 1\}$ both fixed, there exist exactly two possible $\hat n_t \in \hat G$, differ only by a sign, such that $\hat n_t \hat n_s \hat n_t^{-1} = z\hat n_s^{q} $. \end{lem}
\proof 
First, after fixing an ordering of $\mathcal W_\Fo/\mathcal W_K$, the conjugation action $t$ on the basis determines a permutation matrix $\hat\tau$ in $\mathrm {O}(V)$. The aim is to modify its entries so that it lies in the classical group $\hat G_0$. 

We separate into two case as in (\ref{two cases for simple supercuspidals}). In the first case, where $k$ is even and $F=\Fo$, put $m=k/2$. The matrix $\hat n_t $ is of the form $$\hat \tau \diag(d_0,\dots,d_{m-1},u_{m-1}d_{m-1}^{-1},\dots,u_0d_0^{-1}),$$ 
for some $d_i\in \mathbb C^\times$ and signs $u_i\in \{\pm 1\}$ determined by the condition that $\hat n_t \in \hat G$, i.e., $\hat H_0^{-1}{}^t\hat n_t^{-1} \hat H_0 = \hat n_t$. If $\hat G$ is orthogonal, then $u_i=1$ for all $i$. If $\hat G$ is symplectic, then for each $i\in \{1,\dots,m-1\}$, the condition becomes \begin{equation}
\label{the condition for n_t being symplectic}
u_i=1
\quad\Leftrightarrow\quad
0<r_k( q i )<m.
\end{equation}
Note that $u_0=1$ since $t$ always fixes the first and the last diagonal entries. With all $\{u_i\}_{i=0}^{m-1}$ now determined, the conjugation condition $\hat n_t \hat n_s \hat n_t^{-1} = z \hat n_s^{q} $ gives equations of the form 
$$d_{i}/d_{i+1} = v_i\text{ for }i\in \{0,\dots,m-2\},\quad d_{m-1}d_0 = v_{m-1},$$
again for some fixed signs $\{v_i\}_{i=0}^{m-1}$. There are two possibilities, determined by $d_0$ under the condition $d_0^2 = v_0\dots v_{m-1}$ (which is equal to 1 if $\hat G$ is orthogonal, and $(-1)^{(q-1)/2}$ if $\hat G$ is symplectic). Each possibility leads to two solutions for $(d_i)_{i=0}^{m-1}$, which are negative of each other. The two possible matrices $\hat n_t$ are now determined, which proves the lemma in the first case of (\ref{two cases for simple supercuspidals}). We remark that either $d_i\in \{\pm 1\}$ for all $i$, or $d_i\in \{\pm \sqrt{-1}\}$ for all $i$; the latter occurs only when $\hat G$ is symplectic.

The proof in the second case of (\ref{two cases for simple supercuspidals}), where $k$ is odd and $F/\Fo$ is ramified quadratic, is similar to the orthogonal case above. The condition $\hat H^{-1}{}^t\hat n_t^{-1} \hat H = \hat n_t$ implies that $d_i \in \{\pm 1\}$ for each $i$, and $\hat n_t \hat n_s \hat n_t^{-1} =\pm \hat n_s^{q} $ implies that either  $d_i=1$ for all $i$, or $d_i=-1$ for all $i$.
\hfill$\blacksquare$

Similar to $\hat n_s $, we define 
$$\hat n_t := I_{Nk^{-1}}\otimes \hat n_{t,0}\in  {}^LN,$$
and an embedding
$$\iota = \iota_{(\hat n_s,\hat n_t)} :{}^L{ M}\rightarrow {}^LG$$ 
such that $\iota|_{\hat M}:\hat M\rightarrow \hat G$ is the block-diagonal inclusion and $\iota|_{\mathcal W_F} = \hat \mu_{(\hat n_s,\hat n_t)}$ as defined in (\ref{Gamma action on dual of T}) with respect to the choices $(\hat n_s,\hat n_t)$.

To construct an epipelagic Langlands parameter of $G$ from a stable functional $\beta$, 
let $K$ be the maximal tamely ramified subextension of $F[\beta]/F$ and $K^\sharp$ be its Galois closure, and define $K^{\musDoubleSharp}$ by $\mathcal W_{K^{\musDoubleSharp}} = \ker\hat \mu_{(\hat n_s,\hat n_t)}$, so that $K^{\musDoubleSharp}\supseteq K^\sharp$. Let $E/K$ be the p-centric field defined as in (\ref{p-centric field, definition}), $T/K$ be the maximal tamely ramified subsection of $E/K$. Putting $T^{\musDoubleSharp} = T K^{\musDoubleSharp}$ and $E_1^{\musDoubleSharp} = E_1 K^{\musDoubleSharp}$, we have $\Gal(E_1^{\musDoubleSharp}/T^{\musDoubleSharp})\cong \Gal(E_1/T)$.

We begin with the representation $\hat \psi_{2\beta}$ defined by the composition in Proposition \ref{ramification theorem in RY template} with $\alpha = 2\beta$ whose kernel is given by $\mathcal P_{E_1}$ for some extension $E_1/K$ containing $E$.  Denote the extension of $\hat \psi_{2\beta}$  to $\mathcal W_{T}$ by 
\begin{equation}
\label{tilde rho beta}
\eta = \eta_{\beta}: \mathcal W_{T} \to {}^LG_K,
\end{equation}
which can be viewed as both a $\Gal(E_1/T)$- and a $\Gal(E_1^{\musDoubleSharp}/T^{\musDoubleSharp})$-representation since $\Gal(E_1^{\musDoubleSharp}/E_1)$ is Galois. Denote by $\rho = \rho_\beta$ the symplectic-Heisenberg $ \Gal(E_1^{\musDoubleSharp}/K^{\musDoubleSharp})$-representation extending 
$\eta_{\beta}$, and define a morphism $\varphi$ by gluing together $\hat \mu_{(\hat n_s,\hat n_t)}$ in (\ref{Gamma action on dual of T}) and $ \rho$ as follows.
\begin{equation}
\label{Reeder-Yu morphism}
\varphi = \varphi_{(\beta,\hat n_s,\hat n_t)}:= \rho_{\beta}\rtimes \hat \mu_{(\hat n_s,\hat n_t)}:{\mathcal W}_\Fo\rightarrow {\mathcal W}_\Fo/{\mathcal W}_{E_1^{\musDoubleSharp}}\cong \Gal(E_1^{\musDoubleSharp}/T^{\musDoubleSharp})\rtimes \left<\dot s,\dot t\right>  \rightarrow {}^LN\subset {}^LG. \end{equation}
The lower ramification filtration subgroups  inside the image of $\varphi$ are given by
$$D_2 = \hat\psi_{2\beta}( \Gal(E_1^{\musDoubleSharp}/E^{\musDoubleSharp})) \hookrightarrow D_1 =  \rho_\beta( \Gal(E_1^{\musDoubleSharp}/K^{\musDoubleSharp})) \hookrightarrow D_0 = \left<D_1,\hat n_ s\right> 
\hookrightarrow
D = \left<D_0,\hat n_t\right> .
$$

\subsubsection{The rectifying characters}
\label{subsubsection The rectifying characters}

Suppose that $\iota_G:{}^LG\rightarrow {}^L\tilde G$ and $j_G:{}^L\tilde G\rightarrow {}^L G$ (resp. $ \iota_M: {}^LM\rightarrow {}^L\tilde M$ and $j_M:{}^L\tilde M\rightarrow {}^L M$, resp. $\tilde \iota:{}^L \tilde {{ M}}\rightarrow {}^L\tilde G$) are the embeddings defined in the Subsection \ref{subsubsection Lifting parameters} (resp. \ref{subsubsection Embeddings of L-tori}, resp. (\ref{L-embedding for GL})). The two diagrams
\begin{equation}
\label{generally non-commutative diagrams}
  \xymatrixcolsep{8pc} \xymatrixrowsep{1.5pc}\xymatrix{
{}^L{ M}
\ar[r]^{\iota\text{ using }{(\hat n_s,\hat n_t})}  
\ar[d]_{\iota_{ M}}  
&{}^LG
\ar[d]_{\iota_G}
\\
{}^L{\tilde { M}}
\ar[r]^{\tilde \iota \text{ using }(\hat{\tilde n}_s,\hat {\tilde n}_t)}
&{}^L{\tilde G}
}
\quad \text{and} \quad 
 \xymatrixcolsep{8pc} \xymatrixrowsep{1.5pc}\xymatrix{
{}^L{ M}
\ar[r]^{\iota\text{ using }{(\hat n_s,\hat n_t})}  
&{}^LG
\\
{}^L{\tilde { M}}
\ar[u]_{j_{ M}}  
\ar[r]^{\tilde \iota \text{ using }(\hat{\tilde n}_s,\hat {\tilde n}_t)}
&{}^L{\tilde G}
\ar[u]_{j_G}
}
\end{equation}
when $F=\Fo$ and  $F/\Fo$ is ramified quadratic, respectively, are in general not commutative. One way to  ``commutize'' these diagrams is to modify $(\hat \sigma, \hat \tau)$ to match $(\hat n_s,\hat n_t)$. A more systematic approach to this matching is to introduce a twist by a character of $K^\times$, which has the advantage of describing Langlands parameters by twisted admissible triples. 

Let $\hat \mu$ be the character of ${\mathcal W}_K = \left<\dot s^k, \dot t\right>$ satisfying the following conditions, depending on the cases in (\ref{two cases for simple supercuspidals}):
 \begin{subequations}
  \label{conditions on hat mu}
\begin{align}
 &\text{$F=\Fo$}: &&\hat\mu(\dot s^{k})=-1, \quad \hat\mu(\dot t) \in \{\pm  \mathfrak n_\psi \}  =
\begin{cases}
\{\pm 1\} & q\equiv 1\bmod 4,
\\
\{\pm \sqrt{-1}\}& q\equiv 3\bmod 4.
\end{cases}
 \label{conditions on hat mu, F=Fo}
\\
   &\text{$F/\Fo$ is quadratic}: &&\hat\mu(\dot s^{k})=1, \quad \hat\mu(\dot t) \in \{\pm  1 \}. 
    \label{conditions on hat mu, F/Fo quadratic ram}
   \end{align}
 \end{subequations}
To state the following lemma in the case $F/\Fo$ is ramified quadratic, suppose we have a morphism 
$\mu: \mathcal W_{\Fo}\to {}^LG$ whose image lies in $j_G({}^L\tilde G)$. We denote by $j_G^*(\mu)$ its restriction to $\mathcal W_F$, viewed as taking values in ${}^L\tilde G$. 

\begin{lem}
\label{mu_ns_nt as induced representation}
When $F=\Fo$ (resp. $F/\Fo$ ramified quadratic), fixing $(\hat n_s,\hat n_t)$ as in Lemma \ref{characterization of n_t}, there exists a unique character $\hat \mu$ satisfying (\ref{conditions on hat mu}) such that  $$\iota_G\circ \mu _{(\hat n_s,\hat n_t)} \text{ (resp. $j_G^*(\mu _{(\hat n_s,\hat n_t)} )$)}\cong \Ind_{K/F}\hat \mu $$ 
as a $\mathcal W_F$-representation, where $ \hat \mu _{(\hat n_s,\hat n_t)}$ is constructed as in (\ref{Gamma action on dual of T}).
\end{lem}
\proof
We prove the lemma in the case where $F=\Fo$ and $\hat G_0$ is symplectic (the proof for $\hat G_0$ orthogonal or for $F/\Fo$ ramified quadratic proceeds along similar but simpler arguments). Using the group transfer from ${\mathcal W}_F$ to ${\mathcal W}_K$, the matrix of the monomial representation $\Ind_{K/F}\hat \mu$ can be expressed by the map $$P:{\mathcal W}_F\rightarrow \GL_k(\mathbb C),\quad g\mapsto P (g):= \hat w_g \diag(\hat\mu(u_{s^i}(g)))_i,$$ 
where $\hat w_g $ is the permutation matrix associated with the action of $g$ on ${\mathcal W}_F/{\mathcal W}_K = \{s^i\}_{i=0}^{k-1}$, viewed as an ordered set, and $u_{s^i}(g)$ denotes the transfer of $g\in {\mathcal W}_F$ to ${\mathcal W}_K$. We first compute $P(\dot s)$. The transfer of $\dot s$ is straightforward: 
$$u_{s^i}(\dot s) = 1\text{ for all }i\in \{0,\dots,k-2\},\quad u_{s^{k-1}}(\dot s) = \dot s^k,$$
which yields $P(\dot s) = \hat n_s$. The transfer of $\dot t$ is given by
$$\text{$u_{s^i}(\dot t) = \dot s^{q i -r_k(q i)} \dot t$ \quad for all $i=0,\dots,k-1$.}$$
Putting $k = 2m$, the values above satisfy  $u_1(\dot t) = u_{s^m}(\dot t)= \dot t$ and $u_{s^i}(\dot t)=\pm u_{s^{i+m}}(\dot t)$ for $i\in \{0,\dots,m-1\}$. For $i\neq 0$, that $u_{s^i}(\dot t)/u_{s^{i+m}}(\dot t)=1$ equivalent to the condition that:
\begin{equation}
\label{the condition for group transfer of t}
   \begin{split}
   &  \text{either both $r_{2k}(q i )$ and $r_{2k}( q( i+k))$ belong to $\{0,\dots,k-1\}$}
\\
&\text{or both belong to $\{k,\dots,2k-1\}$ \quad when $q\equiv 1 \bmod 4$;}
\\
 &  \text{one of $r_{2k}(q i )$ and $r_{2k}( q( i+k))$ belongs to $\{0,\dots,k-1\}$}
\\
&\text{and the other belongs to $\{k,\dots,2k-1\}$ \quad when $q\equiv 3\bmod 4$.}
   \end{split}
\end{equation}
Under the condition (\ref{conditions on hat mu}) on the possible values of $\hat \mu(\dot t)$, it is easy to verify that (\ref{the condition for group transfer of t}) is equivalent to (\ref{the condition for n_t being symplectic}). Moreover, the element $P(\dot t) = \hat w_t \diag(u_{s^i}(\dot t))_i $ satisfies $$P(\dot t) P(\dot s) P(\dot t) ^{-1} =\pm P(\dot s)^q,$$ so it must be one of the two lifts $\hat n_t$ (differing by a sign) determined in Lemma \ref{characterization of n_t}. 
\hfill$\blacksquare$

We call the characters in (\ref{conditions on hat mu}) the rectifying characters. In each case, there are two possible values of $\hat\mu(\dot t)$ differing by a sign. We distinguish (non-canonically) the two characters by denoting them $\hat\mu_+$ and $\hat\mu_-.$

We now follow \cite[Sec 2]{BK-epipelagic} to construct, from the class of the triple $[K/F,\tilde\rho,1_{\mathcal W_K}]$, an epipelagic representation of $\tilde G = \GL_{m}(F)$, where the $\mathcal W_F$-conjugacy class of $\tilde\rho|_{\mathcal P_F}$ corresponds to  $2\beta$ via the inverse of the composition in Proposition \ref{ramification theorem in RY template}. Let $\Lambda$ be the $\mathfrak o_{F[\beta]}$-invariant lattice chain corresponding to $x\in \mathcal B(\tilde G,F)$, and $\psi_{2\beta}$ be the same character defined on $\mathcal U^{0_+}(\Lambda)$.

\begin{lem}There are exactly two self-dual extensions $\tilde{\boldsymbol{\lambda}}$ of $\psi_{2\beta}$ to $\tilde{\boldsymbol{\mathcal{J}}}:=\mathcal U^{0_+}(\Lambda) F[\beta]^\times$.\end{lem}
\proof
The proof is identical to \cite[Prop 3.2]{BT-ramified}. The extension of $\psi_{2\beta}$ to $\boldsymbol{\mu}_{F[\beta]} $ is just trivial by self-duality, leaving only two choices for $\tilde{\boldsymbol{\lambda}}(\varpi_{F[\beta]})$ determined by $\tilde{\boldsymbol{\lambda}}(\varpi_{F[\beta]})^2=\tilde{\boldsymbol{\lambda}}(-1)$.
\hfill$\blacksquare$

We denote the two self-dual extensions by $\tilde{\boldsymbol{\lambda}}_+$ and $\tilde{\boldsymbol{\lambda}}_-$, although there is no canonical way to distinguish them.

When $F=\Fo$, let $\rho_{\beta\bullet}:=\rho_{\beta}: \mathcal W_{K} \to {}^LG_K$ be defined as in (\ref{tilde rho beta}). When $F/\Fo$ is ramified quadratic, viewing $\rho_{\beta}$ as a representation of $\mathcal W_K$, the representation $ \rho_{\beta}\oplus {}^{s^k}\rho_{\beta}^\vee $ of $\mathcal W_{K^\flat}$ (the Asai-induction) defines a morphism 
$\rho_{\beta\bullet}: \mathcal W_{K^\flat} \to {}^LG_{K^\flat} $. Induce $\rho_{\beta\bullet}$ to 
$$\Ind_{G_K}^M\rho_{\beta\bullet}: \mathcal W_{F} \to {}^LM.$$ 
For $\delta\in \{-,+\}$, let $\varphi_{(\beta,\hat n_s,\delta\hat n_t)}: \mathcal W_\Fo \to {}^LG$ be defined as in (\ref{Reeder-Yu morphism}). With these preparations, we obtain the following results.

\begin{prop}
\label{prop ReederYu parameter as induced representation}
\begin{enumerate}[(i)]
\item When $F=\Fo$ (resp. $F/\Fo$ is ramified quadratic), as a representation of ${\mathcal W}_F$,
\begin{equation*}
   \begin{split}
     \iota_G\circ \varphi_{(\beta,\hat n_s,\delta\hat n_t)} = \tilde\iota\circ\iota_M\circ \Ind_{G_K}^M \rho_{\beta} 
     \quad (\text{resp. $ j_G^*(\varphi_{(2\beta,\hat n_s,\delta\hat n_t)}) = \tilde\iota\circ j_M^*( \Ind_{G_K}^M \rho_{\beta\bullet} )$})
   \end{split}
\end{equation*}
is isomorphic to $\Ind_{K/F}(\tilde\rho_{2\beta}\otimes \hat \mu_\delta)$, where $\tilde\rho_{2\beta}|_{\mathcal P_F} $ corresponds to $2\beta$ via the inverse of the composition in Proposition \ref{ramification theorem in RY template} and $\hat \mu_\delta$ is the rectifying character in Lemma \ref{mu_ns_nt as induced representation}.

\label{prop ReederYu parameter as induced representation, isomorphism}

\item 

Putting $\tilde \varphi_{(2\beta,\hat n_s,\delta\hat n_t)} := \iota_G\circ \varphi_{(\beta,\hat n_s,\delta\hat n_t)}$ and 
$\epsilon : = {\epsilon(\tilde\pi_{\tilde{\boldsymbol{\lambda}}_+}, 1/2,\psi)}{\epsilon(\tilde\varphi_{(2\beta,\hat n_s,\hat n_t)}, 1/2,\psi)}^{-1}$, there is a unique bijection 
$$\{\tilde\pi_{\tilde{\boldsymbol{\lambda}}_+} , \tilde\pi_{\tilde{\boldsymbol{\lambda}}_-}\}\rightarrow \{ \tilde\varphi_{(2\beta,\hat n_s,\hat n_t)}, \tilde\varphi_{(2\beta,\hat n_s,-\hat n_t)}\},\quad 
\tilde\pi_{\tilde{\boldsymbol{\lambda}}_\delta}\mapsto \tilde\varphi_{(2\beta,\hat n_s,\epsilon \delta\hat n_t)}\quad  \text{for }\delta\in \{-,+\},$$
obtained by restricting the LLC for $\tilde G = \GL_{m}(F)$.
\end{enumerate}
\end{prop}
\proof
Part (i) follows by combining Lemma \ref{mu_ns_nt as induced representation} with Artin reciprocity $\alpha_E:E^\times \rightarrow {\mathcal W}_E^{ab}$, under which $\zeta\mapsto s^k$. To prove (ii), recall that the map above is the restriction of the explicit LLC in \cite[6.4 Cor]{BH-Eff} to the subset of self-dual representations of $\tilde G$. Therefore, for each $\delta\in \{-,+\}$, the local constant ${\epsilon(\tilde\pi_{\tilde{\boldsymbol{\lambda}}_\delta}, 1/2,\psi)}$ equals either ${\epsilon(\tilde\varphi_{(2\beta,\hat n_s,\hat n_t)}, 1/2,\psi)}$ or ${\epsilon(\tilde\varphi_{(2\beta,\hat n_s,-\hat n_t)}, 1/2,\psi)}$. Moreover, \cite[2.2 Lemma or 2.3 Lemma]{BK-epipelagic} shows that these two epsilon-factors are negatives of each other. \hfill$\blacksquare$

\subsubsection{Remarks on unramified unitary groups}
\label{subsubsection remarks Unramified unitary groups}

As a final remark, we record the corresponding results for the unramified unitary group $\mathrm U_{\hat N}(F/\Fo)$. When $F/\Fo$ is unramified, we have $s\in {\mathcal W}_F$ and $ t\in {\mathcal W}_{\Eo}\smallsetminus {\mathcal W}_E$. We define 
\begin{equation}
\label{ns and nt for unramified unitary group}
   \begin{split}
      \hat n_s &= \antidiag(u,\diag(1,\dots,1))\in \hat G, \quad \text{where }u\in {\boldsymbol{\mu}}(E)_1;
      \\
      \hat n_t &= \hat \tau \hat H \phi \cdot \diag(u_{s^i}(t)) \in \hat G\phi,
   \end{split}
\end{equation}
  where $\hat \tau$ is the permutation matrix determined by the relation $tst^{-1} = s^q$, and $u_{s^i}$ is the transfer defined in the proof of Lemma \ref{mu_ns_nt as induced representation}. The pair $(\hat n_s,\hat n_t)$ satisfies ${}^\phi\hat n_s = \hat n_t^{-k} \hat n_s \hat n_t^k$ for some divisor $k$ of $f$.

Since $q_F = q_{\Fo}^2 \equiv 1\bmod 4$, we have $\left(\frac{-1}{{\boldsymbol{\mu}}_F}\right)=1$ and $\mathfrak n_\psi\in \{\pm 1\}$. The character $\hat \mu$ in (\ref{conditions on hat mu}) is therefore either trivial or quadratic. Lemmas \ref{characterization of n_t} and \ref{mu_ns_nt as induced representation} and their proofs apply to $(\hat n_s,\hat n_t)$ defined in (\ref{ns and nt for unramified unitary group}) if we assume $u\in \{\pm 1\}$ in $\hat n_s$. Proposition \ref{prop ReederYu parameter as induced representation} remains valid with the same character $\hat \mu$ in (\ref{conditions on hat mu}). 

\subsection{Epipelagic parameters and related results}
\label{subsection Epipelagic parameters}

In the first two subsections, we complete the description of a general epipelagic parameter of a classical group via endoscopic embeddings, admissible triples, and Reeder-Yu's template. Since this discussion is a straightforward analogue of that in Section \ref{subsection The tame part} (notably around Subsections \ref{subsubsection Embeddings of L-tori} and \ref{subsubsection The rectifying characters}), we treat it very briefly by omitting routine calculations. We then present the remaining main results in the final two subsections, namely the correct parity of parameters obtained by twisting by rectifying characters, and the calculation of adjoint Swan conductors.



\subsubsection{Unramified induction}
\label{subsubsection Unramified induction}

We again assume that $\hat G$ is of the form in (\ref{types of classical groups, dual side}). Put $\hat N = 2n$ (resp. $2n+1$) when $F=\Fo$ (resp. $F/\Fo$ is ramified quadratic). It is convenient to define the L-group ${}^LG$ in its Galois form, as in Remark \ref{Remark on Galois form}.
$$
      \mathrm{O}_{\hat N,F_\bullet^\sharp/F_\bullet}(\mathbb C), \quad \mathrm{Sp}_{\hat N}(\mathbb C),\quad \text{ or} \quad \mathrm{GL}_{\hat N,F/F_\bullet}(\mathbb C).
$$
Let $U/F$ be an unramified extension of degree $f$, lying over $U_\bullet/F_\bullet$ either trivially or ramified quadratically. Of course, we have $\Gal(U/F)\cong \Gal(U_\bullet/F_\bullet)$. Similar to ${}^LG$, we define the L-group of ${}^LG_U$ in Galois form, as follows.
$$
  \mathrm{O}_{\hat N/f,U_\bullet^\sharp/U_\bullet}(\mathbb C), \quad \mathrm{Sp}_{\hat N/f}(\mathbb C),
  \quad \text{ or}\quad 
\mathrm{GL}_{\hat N/f,U/U_\bullet}(\mathbb C).
$$
Define ${}^LM:=\Ind_{U_\bullet/F_\bullet}{\hat G_U} \rtimes \Gal(U_\bullet/F_\bullet)$. To embed ${}^LM$ into ${}^LG$, first consider
$${\mathcal W}_F\to \Gal(U/F) \cong \mathbb Z/{f}\to \mathrm{O}_f(\mathbb C), \quad w\mapsto (\hat {n}_w)_f,$$
where $(\hat {n}_w)_f$ is the permutation matrix corresponding to the action of $w\in {\mathcal W}_F$. We then have $(\hat n_s)_f= I_f$, while $(\hat n_t)_f$ arises from the Frobenius of $U_\bullet/F_\bullet$. The injective morphism $\hat G_U\otimes \mathrm{O}_f(\mathbb C) \hookrightarrow \hat G$ is therefore well-defined. Putting $\hat n_t = I_{\hat N/f}\otimes (\hat n_t)_f$, the morphism determined by 
$$
\iota = \iota_{\hat n_t}: {}^L M \rightarrow {}^L G, \quad \diag(x_g)\rtimes t\mapsto \diag(x_g)\hat { n}_t \rtimes t \quad \text{ for all }\diag(x_g)\in \hat M,
$$
 defines an L-embedding of ${}^LM$ into ${}^LG$.

Moreover, we define $\hat {\tilde G}_U= \GL_{\hat N/f}(\mathbb C)$. Put $\hat {\tilde M} = \Ind_{U/F}{\tilde G}_U \cong \hat {\tilde G}_U^f$ and ${}^L{\tilde M} = \hat {\tilde M} \rtimes \Gal(U/F)$. The embedding $\iota:{}^L\tilde M\rightarrow {}^L\tilde G$ is defined analogously to (\ref{L-embedding for GL}).

Recall the standard embeddings $\iota_G: {}^LG\to {}^L\tilde G$ (resp. $j_G: {}^L\tilde G\to {}^LG$) from Subsection \ref{subsubsection Lifting parameters}, define $\iota_{G_U}$ (resp. $j_{G_U}$) similarly with $G_U$ in place of $G$, and induce $\iota_{G_U}$ to $\iota_M : = \Ind_{U/F}\iota_{G_U}: {}^L M\to {}^L\tilde M$ (resp. $j_{G_U}$ to $j_M : = \Ind_{U/F}j_{G_U}: {}^L \tilde M\to {}^LM$). The two diagrams
\begin{equation*}
  \xymatrixcolsep{6pc} \xymatrixrowsep{1.5pc}\xymatrix{
{}^L{ M}
\ar[r]^{\iota\text{ using }{\hat n_t}}  
\ar[d]_{\iota_{ M}}  
&{}^LG
\ar[d]_{\iota_G}
\\
{}^L{\tilde { M}}
\ar[r]^{\tilde \iota \text{ using }\hat {\tilde n}_t}
&{}^L{\tilde G}
}
\quad \quad \text{and} \quad \quad 
 \xymatrixcolsep{6pc} \xymatrixrowsep{1.5pc}\xymatrix{
{}^L{ M}
\ar[r]^{\iota\text{ using }{\hat n_t}}  
&{}^LG
\\
{}^L{\tilde { M}}
\ar[u]_{j_{ M}}  
\ar[r]^{\tilde \iota \text{ using }\hat {\tilde n}_t}
&{}^L{\tilde G}
\ar[u]_{j_G}
}
\end{equation*}
when $F=\Fo$ and  $F/\Fo$ is ramified quadratic, respectively, are now commutative (in contrast to the generally non-commutative diagrams in (\ref{generally non-commutative diagrams})).

Let $\nu$ be the unramified quadratic character of $U^\times$ such that $\nu(\varpi)$ equals $(-1)^{f-1}$ (i.e., the signature of any generator of $\mathbb Z/f$). Suppose that $\varphi_{U}:\mathcal W_{U_\bullet}\to {}^LG_U$ is a parameter defined as in (\ref{Reeder-Yu morphism}), but now with the base field $U_\bullet$ in place of $F_\bullet$. Define the morphism $\varphi_{(\varphi_U,\hat n_t)}$ by 
\begin{equation}
\label{a component of a parameter, in the full form}
\varphi = \varphi_{(\varphi_U,\hat n_t)}:=\varphi_U\rtimes \hat \mu_{(I_{f},\hat n_t)}:{\mathcal W}_\Fo \rightarrow {\mathcal W}_\Fo/{\mathcal W}_{E^{\musDoubleSharp}_1}\cong \Gal(E^{\musDoubleSharp}_1/U_\bullet)\rtimes \Gal(U_\bullet/F_\bullet) \rightarrow {}^LG.
\end{equation}
As in Proposition \ref{prop ReederYu parameter as induced representation}\ref{prop ReederYu parameter as induced representation, isomorphism}, we readily verify that 
\begin{equation}
\label{a component of a parameter, in the full form, as a representation}
\iota_G\circ \varphi_{(\varphi_U,\hat n_t)} = \tilde\iota\circ\iota_M\circ \Ind_{G_U}^M\varphi_U
\quad \text{(resp. $j_G^*( \varphi_{(\varphi_U,\hat n_t)} ) = \tilde\iota\circ j^*_M(\Ind_{G_U}^M\varphi_U$))}
\end{equation}
is isomorphic to 
$$\Ind_{U/F}(\iota_{G_U}\circ \varphi_U\otimes \hat \nu) = \Ind_{K/F}(\tilde\rho_{2\beta} \otimes \hat \alpha)$$ as a representation of $\mathcal W_F$, where $\alpha$ is the tamely ramified character $\mu (\nu\circ N_{K/F})$ of $K^\times$.

\begin{rmk}
The rectifying character $\nu$ appears also in the Unramified Induction Theorem in \cite[9.1]{BH-Eff}, which asserts that if $\tilde\varphi_U= \Ind_{K/U}\tilde\varphi_K$ is the Langlands parameter of $\tilde\pi_{\tilde{\boldsymbol{\lambda}}_U}:= \cInd_{\tilde{\boldsymbol{\mathcal J}}_U}^{\tilde G_U}\tilde{\boldsymbol{\lambda}}_U$ 
for some irreducible representation $\tilde\varphi_K$ of $\mathcal W_K$, then the Langlands parameter of $\tilde\pi_{\tilde{\boldsymbol{\lambda}}}:=\cInd_{\tilde{\boldsymbol{\mathcal J}}}^{\tilde G}\tilde{\boldsymbol{\lambda}}$ 
is $\Ind_{K/F}(\tilde\varphi_K \otimes \hat \nu) $. Here both $\tilde{\boldsymbol{\lambda}}_U$ and $\tilde{\boldsymbol{\lambda}}$ have the same underlying stratum and match under the Glauberman correspondence (see \cite[5.6 Prop]{BH-Eff}).
\qed\end{rmk}


\subsubsection{Epipelagic parameters}
\label{subsubsection Epipelagic parameters}

Given an index set $I$, let $\hat I$ denote its dual set defined at the beginning of Section \ref{subsection Reducibility results for different classical groups}. The results of that section show that the endoscopic lift of an epipelagic parameter of $G$ has the form
$$
\bigoplus_{i\in \hat I}\tilde\varphi_i,
$$
where each $\tilde\varphi_i $ is an irreducible representation such that $\tilde\varphi_i = \tilde\varphi_{\tilde\pi_i}$, where ${\tilde\pi_i}$ is either a self-dual depth-zero character of $F^\times$ (in which case $i\in \{o,o'\}$) or a self-dual quasi-epipelagic representation of $\GL_{\dim \tilde\varphi_i}(F)$. The underlying strata of the representations of the latter kind are subject to the conditions in \ref{Table GL and unram-U}-\ref{Table SO-even} of Section \ref{subsection Epipelagic strata}.

We define ${}^LG_i $ to be the following L-groups in Galois form (Remark \ref{Remark on Galois form}),
$$
\mathrm{O}_{2n_i,F_i^\sharp/F}(\mathbb C),\quad 
\mathrm{Sp}_{2n_i}(\mathbb C), \quad \text{ or}\quad 
\mathrm{GL}_{(2n_i+1),F/\Fo}(\mathbb C).
$$
Define ${}^L{ H} $ and its embedding $\Delta$ into ${}^LG$ as follows.

\begin{enumerate}[(i)]
\item The case when $G$ is symplectic is the simplest. We define
$$\Delta: {}^L{ H} = \prod_{i\in \hat I}\mathrm{Sp}_{2n_i}(\mathbb C) \rightarrow {}^LG = \mathrm{Sp}_{2n}(\mathbb C).$$

\item When $G$ is even orthogonal and $o\notin I$, let $F_i^\sharp = F[\sqrt{d_i}]$ for some $d_i\in  F^\times/(F^\times)^2 $, and 
$$(d_i)_i\mapsto d:=\prod_i d_i 
\text{ via the product map } (F^\times/(F^\times)^2 )^{\#\hat I}\mapsto F^\times/(F^\times)^2. $$
Putting $F^\sharp = F[\sqrt{d}]$, we define
$$\Delta:{}^L{ H} = \prod_{i\in \hat I}\mathrm{O}_{2n_i,F_i^\sharp/F}(\mathbb C) \rightarrow {}^LG = \mathrm{O}_{2n,F^\sharp/F}(\mathbb C).$$ 
Note that when the product $d=1$, the image lies in the split $\mathrm{SO}_{2n}(\mathbb C)$.

\item When $G$ is even orthogonal and $o\in I$, we define
$$\Delta:{}^L{ H} = \mathrm{O}_{2,F_o^\sharp/F}(\mathbb C) \times \prod_{i\in \hat I\smallsetminus \{o,o'\}}\mathrm{O}_{2n_i,F_i^\sharp/F}(\mathbb C) \rightarrow {}^LG = \mathrm{O}_{2n,F^\sharp/F}(\mathbb C),$$ 
where $d_o $ is determined by the ramified character $\tilde{\boldsymbol{\lambda}}_o$ in (\ref{the two extra characters in the general epipelagic even orthogonal case}) and $d = \prod_{i\in \hat I\smallsetminus \{o'\}}d_i$.

\item When $G$ is odd special orthogonal, we define $$\Delta:{}^L{ H} = \mathrm{O}_{1}(\mathbb C)  \times \prod_{i\in  I }\mathrm{O}_{2n_i,F_i^\sharp/F}(\mathbb C) \rightarrow {}^LG = \mathrm{O}_{2n+1}(\mathbb C).$$ 
Since $\tilde\varphi_o  = \prod_{i\in I}\det\tilde\varphi_i \in \mathrm{O}_{1}(\mathbb C) =\{\pm 1\}$, the image of $\bigoplus_{i\in I}\tilde\varphi_i$ lies in $\mathrm{SO}_{2n+1}(\mathbb C)$.

\item When $G$ is ramified unitary, define $\hat H = \prod_{i\in \hat I}\hat G_i$ and ${}^L H = \hat H\rtimes \Gal(F/\Fo)$, where $\Gal(F/\Fo)$ acts on each factor $\hat G_i$ as prescribed by ${}^LG_i$. We define two embeddings
$$ \Delta^{\epsilon}: {}^L{ H}  \rightarrow {}^LG = \mathrm{GL}_{\hat N,F/\Fo}(\mathbb C)$$
as in \cite[(2.1.13)]{Mok-unitary}, and define $\Delta$ to be the one with the correct parity, i.e., $\Delta = \Delta^{(-1)^{\hat N-1}}$.
\end{enumerate}
In contrast to the single-component cases treated previously, these L-embeddings are not endoscopic in general.

Put ${}^L\tilde G_i  = \mathrm{GL}_{n_i}(\mathbb C)$ and  ${}^L{ \tilde H}  = \prod_{i\in \hat I}{}^L\tilde G_i $. We  define the standard embedding
$\iota_H:{}^LH\to {}^L\tilde H$ (resp. $j_H:{}^L\tilde H\to {}^LH$) as the product of the standard embeddings $\iota_{G_i}: {}^LG_i \to {}^L\tilde G_i$ (resp. the unitary-type endoscopic embeddings $j_{G_i}: {}^L\tilde G_i \to {}^LG_i$). We also define $\tilde\Delta: {}^L\tilde H\to {}^L\tilde G$ to be the block-diagonal embedding, which is endoscopic of Levi-type.

The two diagrams
\begin{equation*}
  \xymatrixcolsep{5pc} \xymatrixrowsep{1.5pc}\xymatrix{
{}^L{ H} 
\ar[r]^{\Delta}  
\ar[d]_{\iota_{ H}}  
&{}^LG
\ar[d]_{\iota_G}
\\
{}^L{\tilde {H}}
\ar[r]^{\tilde \Delta}
&{}^L{\tilde G}
}
\quad \quad \text{and} \quad \quad 
 \xymatrixcolsep{5pc} \xymatrixrowsep{1.5pc}\xymatrix{
{}^L{ H}
\ar[r]^{\Delta}  
&{}^LG
\\
{}^L{\tilde { H}}
\ar[u]_{j_{ H}}  
\ar[r]^{\tilde \Delta}
&{}^L{\tilde G}
\ar[u]_{j_G}
}
\end{equation*}
are readily seen to be commutative.

Let $\varphi_i: \mathcal W_\Fo\to {}^LG_i$ be a parameter as in (\ref{a component of a parameter, in the full form}), and denote $(\varphi_i)_{i} := (\varphi_i)_{i\in \hat I}:\mathcal W_\Fo\to {}^LH$ by juxtaposition. Moreover, simply putting 
$$\hat n_s = \Delta(\hat n_s^i)_i, \quad \hat n_t = \Delta(\hat n_t^i)_i 
\quad\text{and defining}\quad \hat \mu_{(\hat n_s,\hat n_t)}  = \Delta(\hat \mu_{\hat n_s^i,\hat n_t^i})_i,$$
we express the parameter $\Delta(\varphi_i)_{i}$ of $G$ as  
$$\varphi_{(\beta, \hat n_s, \hat n_t)}:=\hat \rho_{\beta}\rtimes \hat \mu_{(\hat n_s,\hat n_t)}$$ as in Reeder-Yu's template in (\ref{Reeder-Yu morphism})

The endoscopic lift of each $\varphi_i$ is given by the LHS of (\ref{a component of a parameter, in the full form, as a representation}), which we denote simply by $\tilde\varphi_i$. Each  $\tilde\varphi_i$ is of the form $\tilde \varphi_{(\varphi_{U_i},\hat n_t^i)} = \tilde \varphi_{(2\beta_i,\hat n_s^i ,\hat n_t^i)} = \Ind_{K_i/F}(\tilde\rho_i \otimes \hat \mu_{i}) $. We readily verify that 
$$\iota_G\circ \varphi_{(\beta,\hat n_s ,\hat n_t)} \quad 
\text{(resp. $j_G^*( \varphi_{(\beta,\hat n_s ,\hat n_t)})$)} 
$$
is isomorphic to 
$ \bigoplus_{i\in \hat I}\tilde \varphi_{(2\beta_i,\hat n_s^i ,\hat n_t^i)}$
as a representation of $\mathcal W_F$.

Up to the choice of sign for each $i\in \hat I\smallsetminus\{o,o'\}$, we can formulate the LLC for epipelagic representations of $G$ as
$$
 \Pi_{\delta }
 \xrightarrow{\text{endoscopic lift}}
\prod_{i\in\hat I} \tilde\pi_{\tilde{\boldsymbol{\lambda}}_{i,\delta}}
 \xrightarrow{
    \begin{smallmatrix} \text{Prop \ref{prop ReederYu parameter as induced representation}(ii)}
    \\
    \text{and Subsec \ref{subsubsection Unramified induction}}    \end{smallmatrix}
 }
 \bigoplus_{i\in \hat I} \tilde \varphi_{(2\beta_i,\hat n_s^i ,\delta\hat n_t^i)}
 \longrightarrow
  \varphi_{(\beta,\hat n_s,\delta \hat n_t)} 
$$
for $\delta\in \{+,-\}$. Each $ \Pi_{\delta}$ is the union of at most two packets, one for each pure inner form of $G$.

\subsubsection{A parity result}
\label{subsubsection A parity result}

We record the following key observation: 
\begin{equation*}
   \begin{split}
&\text{a quasi-epipelagic parameter has the expected parity only if }
\\
&\text{it is twisted by the rectifying character in (\ref{conditions on hat mu}).}
   \end{split}
\end{equation*}
The following proposition is the precise statement. 
\begin{prop}
\label{last parity result}
Let $\tilde\varphi$ be an irreducible component (as an irreducible representation of $\mathcal W_F$) of an epipelagic parameter of $G$. Assume that the quadratic extension $K/K_\flat$ in Proposition \ref{prop effective parameter, self-dual}\ref{prop effective parameter, self-dual, same parity} is ramified (and so is $F/\Fo$ if it is quadratic).
\begin{enumerate}[(i)]
\item If $G$ is symplectic or even orthogonal (resp. odd unitary), then $\tilde\varphi$ is $+$-selfdual (resp. $(+,F/\Fo)$-selfdual).
\label{last parity result, conjugate-orthogonal}

\item If $G$ is odd orthogonal (resp. even  unitary), then $\tilde\varphi$ is $-$-selfdual (resp. $(-,F/\Fo)$-selfdual).

\label{last parity result, conjugate-symplectic}

\end{enumerate}

\end{prop}
\proof
We first consider an epipelagic parameter  $\tilde\varphi= \Ind_{K/F}(\tilde\rho\otimes \hat \mu_{}) $ given by a class of admissible triples $[K/F,\tilde\rho, \hat \mu]$. Since $\tilde\rho$ is always $+$-selfdual by \cite[7.5 Prop]{BHS}, the parity of $\tilde\varphi$ follows from that of $\mu $. Now recall that the LLC for $\GL_n$ is given by
$$\Ind_{K/F}(\tilde\rho\otimes \hat \mu_{})\leftrightarrow \tilde\pi_{\tilde{\boldsymbol{\lambda}}}, \quad\text{ or }\quad 
\Ind_{K/F}(\tilde\rho\otimes \mathbf{1}_{\mathcal W_K})\leftrightarrow \tilde\pi_{\tilde{\boldsymbol{\lambda}}(\mu\circ N_{F[\beta]/K})},
$$
Here $\mu\circ N_{F[\beta]/K} = \mu$ on ${\boldsymbol{\mu}}_{F[\beta]} = {\boldsymbol{\mu}}_K$ since $F[\beta]/K$ is totally wild and $p$ is odd. We therefore only need to check whether the product of $\tilde\lambda$ (given by our endoscopic lifting results in Section \ref{subsection Reducibility results for different classical groups}) and $\mu$ is trivial or quadratic on ${\boldsymbol{\mu}}_K$. The values of $\tilde\lambda|_{{\boldsymbol{\mu}}_K}$ are calculated in Sections \ref{subsection Constructions simple supercuspidals} and \ref{subsection Ramified unitary groups},  
$$\tilde\lambda|_{{\boldsymbol{\mu}}_K}  =    \begin{cases} \text{ trivial} 
\\
\text{ quadratic} 
 \end{cases}
 \text{when $G$ is}\quad
  \begin{cases} 
\text{odd orthogonal, or odd ramified unitary, }    
\\
\text{symplectic, even orthogonal, or even ramified unitary.}
 \end{cases}
$$
We then recall from Proposition \ref{prop effective parameter, self-dual} that, when $K/K^\flat$ is ramified, a self-dual character $\mu$ of $K^\times$ is $(+,K/K^\flat)$-selfdual (resp. $(-,K/K^\flat)$-selfdual) if and only if it is trivial (resp. quadratic) on ${{\boldsymbol{\mu}}_K}$. The results from (\ref{conditions on hat mu}) tell us that
$$\mu|_{{\boldsymbol{\mu}}_K}  =    \begin{cases} \text{ trivial} 
\\
\text{ quadratic} 
 \end{cases}
 \text{when $F/\Fo$ is}\quad
  \begin{cases} 
\text{ramified quadratic, }    
\\
\text{trivial.}
 \end{cases}
$$
This implies that the product $\tilde\lambda\mu|_{{\boldsymbol{\mu}}_K} $ is trivial in case \ref{last parity result, conjugate-orthogonal} and quadratic in case \ref{last parity result, conjugate-symplectic}, and hence proves the proposition in the case of epipelagic components.

As for quasi-epipelagic components, each of them is of the form $\Ind_{K/F}(\tilde\rho\otimes \hat \mu_{}\hat \nu) $ for an unramified character $ \nu$. Since $\tilde\lambda\mu\nu|_{{\boldsymbol{\mu}}_K} =\tilde\lambda\mu|_{{\boldsymbol{\mu}}_K} $, the parity of the parameter remains unchanged. \hfill $\blacksquare$

We expect Proposition \ref{last parity result} to hold for arbitrary supercuspidal representations of $G$, and not only for the quasi-epipelagic ones, except that the general rectifying characters should depend on the underlying simple characters of the representations. If $G$ is symplectic and $p\neq 2$, this can be verified using results from \cite[Sec 6]{BHS}. For other types of classical groups, we expect this to hold in the tame case (i.e., when $p\nmid e_{F[\beta]/F}$), and conjecture that the rectifying characters concerning self-duality, whenever defined, may be related to those appearing in the essentially tame case (see \cite{BH-ET1,BH-ET2,BH-ET3}). Beyond that case, the question remains open since neither the explicit endoscopic lifting nor the rectifying characters are known in complete generality.

\subsubsection{Minimal conductor property}
\label{subsubsection Minimal conductor property}

The equality stated in Proposition \ref{adjoint swan is dim of Gx} below is related to computing the formal degree of a representation \cite{Hiraga-Ichino-Ikeda}, known as the Hiraga-Ichino-Ikeda conjecture (HII for short). Here we prove Proposition \ref{adjoint swan is dim of Gx} for epipelagic representations of connected classical groups by a direct calculation, using Proposition \ref{adjoint swan conductor simple supercuspidal for classical groups}.

Following the construction in Section \ref{subsection Epipelagic inducing types for classical groups}, let $\pi = \cInd_{\mathcal J}^{G(F)}\lambda$ be an epipelagic representation, where ${\mathcal{J}} = G(F)_{x,\beta}$, and $\varphi$ be the parameter of $\pi$.

\begin{prop}
\label{adjoint swan is dim of Gx}
We have $\mathrm{Adsw}(\varphi)= \dim \mathsf G_{\mathrm{ss},x} .$
\end{prop}


\proof
When $\tilde G= \GL_m(F)$ and $x$ is a barycenter of an alcove in the Bruhat-Tits building of $\tilde G$, we have $\tilde{\mathsf G}_{\mathrm{ss},x} = \GL_1(\mathbb F)^m \cap \mathrm{SL}_m(\mathbb F)\cong \GL_1(\mathbb F)^{m-1}$, and the result follows from Proposition \ref{adjoint swan conductor GL-case}. The same arguments apply when $G = \mathrm{U}_m(F/\Fo)$ with $F/\Fo$ unramified, using the last case of Proposition \ref{adjoint swan conductor simple supercuspidal for classical groups}.



For the other classical groups, first notice that the statement remains unchanged under base-change to any unramified extension over $F$, in particular to the composite of the maximal unramified subextensions in $F[\beta_i]/F$. We therefore assume that all $F[\beta_i]/F$ are totally ramified. We put $J = I\smallsetminus \{o\}$ for convenience, and use the tables in \cite[Sec 8.2]{Gross-Levy-Reeder-Yu} (the labelled Dynkin diagrams) to compute $\mathsf G_x$ as follows:
\begin{center}
\begin{tabular}{ |c|c|c| } 
 \hline
 $G$ & $\mathsf G_x$ & $\dim \mathsf G_x$
 \\
 \hline
 $\mathrm{SO}_{2n+1}$ & $\mathsf S(\mathsf O_{\#J}\times \mathsf O_{\#J+1})\times (\mathsf{GL}_{\#J})^{n/\#J-1}$ & $n\#J$
\\
 $\mathrm{SP}_{2n}$ & $(\mathsf{GL}_{\#J})^{n/\#J}$ & $n\#J$
\\
 $\mathrm{SO}_{2n}$, $o\notin I$ & $\mathsf S(\mathsf O_{\#J})^2\times (\mathsf{GL}_{\#J})^{n/\#J-1}$ & $(n-1)\#J$
\\
 $\mathrm{SO}_{2n}$, $o\in I$ &  $\mathsf S(\mathsf O_{\#J+1})^2\times (\mathsf{GL}_{\#J})^{(n-1)/\#J-1}$ & $n\#J$
\\
 $\mathrm{U}_{N}$, $o\notin I$ & $\mathsf O_{\#J}\times (\mathsf{GL}_{\#J})^{(N/\#J-1)/2}$ & $(N-1)\#J/2$
\\
 $\mathrm{U}_{N}$, $o\in I$ & $\mathsf O_{\#J+1}\times (\mathsf{GL}_{\#J})^{((N-1)/\#J-1)/2}$ & $N\#J/2$
 \\
   \hline
\end{tabular}
\end{center}
To compute the adjoint Swan conductors, we first observe that $\dim (\tilde\varphi_i \otimes \tilde\varphi_j)^{ D_0} = 0$ for all $i,j\in \hat I$. This is clear when $i\neq j$, since $\dim\Hom_{ D_0}(\tilde\varphi_i,\tilde\varphi_j)=0$. When $i=j$, we recall from (\ref{GGP, Ad is just Sym, Alt, or Asai}) that $\Ad\circ \varphi$ is given by $S\tilde\varphi$ for suitable $S\in \{\wedge^2, \Sym^2,  \mathrm{As}^+\}$, and observe that 
$(\Ad\circ\varphi)^{D_0(\Fo)}\subset (\Ind_{F/\Fo}(\tilde\varphi\otimes \tilde\varphi))^{D_0(\Fo)} = (\tilde\varphi\otimes \tilde\varphi)^{D_0(F)}$,
which is 0 by \cite[3.1 Lem]{BK-epipelagic}. Therefore, the swan is additive on the terms in the direct sum
$$S(\oplus_{i\in \hat I}\tilde \varphi_i) = \bigoplus_{i\in \hat I}S\tilde \varphi_i \oplus  \bigoplus_{
 \begin{smallmatrix}
\{i,j\}\subset  \hat I
 \end{smallmatrix}
}(\tilde\varphi_i\otimes {}^c\tilde\varphi_j + {}^c\tilde\varphi_i\otimes \tilde\varphi_j).$$
If the common dimension of $\tilde\varphi_i$ for $i\in \hat I\smallsetminus \{o,o'\}$ is $m$, we use 
$$ \mathrm{sw}(S\tilde \varphi_i) = \frac{m}{2}-1, \frac{m}{2},\text{ or }\frac{m-1}{2}\text{ resp., \quad and} \quad 
\mathrm{sw}(\tilde\varphi_i\otimes {}^c\tilde\varphi_j)=m $$
by Proposition \ref{adjoint swan conductor simple supercuspidal for classical groups} and \cite[6.5 Th(ii)]{BHK-RS}, respectively, for all $\{i, j\}\subset \hat I\smallsetminus\{o,o'\}$, as well as 
$$\mathrm{sw}(\tilde\varphi_i\otimes {}^c\tilde\varphi_o) = \mathrm{sw}(\tilde\varphi_i\otimes {}^c\tilde\varphi_{o'})=1 \text{ for all $i\in  \hat I\smallsetminus\{o,o'\}$}
\quad \text{and} \quad \mathrm{sw}(\tilde\varphi_o\otimes {}^c\tilde\varphi_{o'})=0$$
to obtain the following table:
\begin{center}
\begin{tabular}{ |c|c|c| } 
 \hline
 $\hat G$ &  $m$ &  Adsw$(\varphi)$
 \\
 \hline
 $\mathrm{SP}_{2n}$ & $2n/\#J$  &  $\#J (m/2) + \#J ( \frac{\#J -1}{2} ) m  = n\#J$ \\
 $\mathrm{SO}_{2n+1}$ & $2n/\#J$  &  $
 \#J (\frac{m}{2}-1)+\#J ( \frac{\#J -1}{2} ) m  +\#J = n\#J$ \\
 $\mathrm{SO}_{2n}$, $o\notin I$ & $2n/\#J$  &  $\#J  (\frac{m}{2}-1) + \#J ( \frac{\#J -1}{2} ) m  = (n-1)\#J$
 \\
  $\mathrm{SO}_{2n}$, $o\in I$ & $2(n-1)/\#J$  &  $\#J  (\frac{m}{2}-1) + \#J ( \frac{\#J -1}{2} ) m + 2\#J  = n\#J$
  \\
  ramified $\mathrm{U}_{N}$, $o\notin I$ & $ N/\#J$ &  $\#J  (\frac{m-1}{2}) + \#J ( {\#J -1}) m = (N-1)\#J/2$
\\
 ramified $\mathrm{U}_{N}$, $o\in I$ & $(N-1)/\#J$ & $ \#J  (\frac{m-1}{2}) + \#J ( {\#J -1}) m + \# J =  N\#J/2$
 \\
    \hline
\end{tabular}
\end{center}
(In the last two rows for ramified $U_N$, the degree $m$ is denoted by $d$ in \cite[Sec 8.2]{Gross-Levy-Reeder-Yu}.) We obtain the proposition by comparing the last columns of the two tables.
\hfill$\blacksquare$

In general, we can prove the HII if we also have \cite[(7.4)]{Reeder-Yu}:
\begin{equation}
\label{7.4 of Reeder-Yu}
\dim \xi / \#  A_\varphi  = \dim \lambda /\# {\mathsf A}_{x,\beta},
\end{equation}
where $ {\mathsf A}_{x,\beta}:=\mathcal J/ G(F)_{x,0+}$ as defined in Section \ref{subsection Epipelagic inducing types for classical groups}, $ A_\varphi:=Z_{\hat G}(\varphi)$, and $\xi$ is an irreducible representation of $A_\varphi$ such that $(\varphi,\xi)$ is the enhanced parameter of $\pi $. Indeed, the HII is equivalent to Proposition \ref{adjoint swan is dim of Gx} together with (\ref{7.4 of Reeder-Yu}), as explained in \emph{loc. cit.}.

When $\pi$ is an epipelagic representation of a connected classical group $G(F)$, the inducing type $\lambda$ is just a character. Also, since $\hat{\mathcal T}= Z_{\hat G}(\hat\psi_\beta(\Gal(E/T)))$ by \cite[Lem 7.2]{Reeder-Yu}, we have  $A_{\varphi}\cong \hat{\mathcal T}^{\left<\dot s,\dot t\right>_{}}\cong \{\pm 1\}^{\#\hat I}$; in particular, $A_{\varphi}$ is abelian and $\dim \xi=1$. We can see that (\ref{7.4 of Reeder-Yu}) follows directly by comparing $\#I$ and $\#\hat I$:
\begin{center}
\begin{tabular}{ |c|c|c| } 
 \hline
 $G$ & $\hat k$ where $A_\varphi\cong \{\pm 1\}^{\hat k}$ & $ k$ where ${\mathsf A}_{x,\beta}\cong \{\pm 1\}^{ k}$ \\ 
 \hline
 Sp & $\#\hat I -1$ & $\#I$  \\ 
 $\SO_{\mathrm{odd}}$ & $\#\hat I$  & $\#I - 1 $ \\ 
  $\SO_{\mathrm{even}}$ with $o\notin I$ & $\#\hat I$  & $\#I  $ \\ 
  $\SO_{\mathrm{even}}$ with $o\in I$  & $\#\hat I - 1$  & $\#I  $ \\ 
    ramified unitary & $\#\hat I$  & $\#I  $ \\ 

 \hline
\end{tabular}
\end{center}
The numbers in the last two columns are equal by the definitions of $\hat I$ in various subsections of Section \ref{subsection Reducibility results for different classical groups} and the construction of $\mathsf A_{x,\beta}$ the end of Section \ref{subsection Epipelagic inducing types for classical groups}.

\addcontentsline{toc}{section}{References} 
\bibliographystyle{alpha}

\end{document}